\documentclass[12pt,twosides]{amsart}
\usepackage{amssymb,amsmath,amsthm, amscd, enumerate, mathrsfs}
\usepackage{graphicx, hhline}
\usepackage[all]{xy}
\usepackage{braket}
\usepackage[usenames]{color}
\usepackage{hyperref}
\usepackage{fancyhdr}
\usepackage[top=30truemm,bottom=30truemm,left=30truemm,right=30truemm]{geometry}

\hypersetup{colorlinks=true}

\title{Minimal model program for normal pairs along log canonical locus}
\author{Kenta Hashizume}
\date{2025/08/01}
\address{Department of 
Mathematics, Faculty of Science, Niigata University, Niigata 950-2181, Japan}
\address{Institute for Research Administration, Niigata University, Niigata 950-2181, Japan}
\email{hkenta@math.sc.niigata-u.ac.jp}

\newtheorem{thm}{Theorem}[section]

\newtheorem{lem}[thm]{Lemma}
\newtheorem{cor}[thm]{Corollary}
\newtheorem{prop}[thm]{Proposition}

\newtheorem{ques}[thm]{Question}

\theoremstyle{definition}
\newtheorem{defn}[thm]{Definition}
\newtheorem{rem}[thm]{Remark}

\newtheorem{exam}[thm]{Example}
\newtheorem*{ack}{Acknowledgments} 
\newtheorem*{divisor}{Divisors and morphisms} 
\newtheorem*{b-divisor}{b-divisors} 
 
\newtheorem*{g-pair}{Generalized pairs} 
\newtheorem*{adj-g-pair}{Divisorial adjunction for generalized pairs} 
\newtheorem*{mmp-g-pair}{MMP for generalized pairs}

\newtheorem{step1}{Step}
\newtheorem{step2}{Step}
\newtheorem{step3}{Step}
\newtheorem{step4}{Step}

\newtheorem*{claim*}{Claim}

\begin{document}

\begin{abstract}
Let $(X,\Delta)$ be a normal pair with a projective morphism $X \to Z$ and let $A$ be a relatively ample $\mathbb{R}$-divisor on $X$. 
We prove the termination of some minimal model program on $(X,\Delta+A)/Z$ and the abundance conjecture for its minimal model under assumptions that the non-nef locus of $K_{X}+\Delta+A$ over $Z$ does not intersect the non-lc locus  of $(X,\Delta)$ and that the restriction of $K_{X}+\Delta+A$ to the non-lc locus of $(X,\Delta)$ is semi-ample over $Z$. 
\end{abstract}

\maketitle

\tableofcontents

\section{Introduction}
Throughout this paper, we work over the complex number field $\mathbb{C}$. 
\subsection{Minimal model theory}
The minimal model theory is a fundamental tool to construct a variety with good geometrical properties. 
The theory is indispensable to recent development of birational geometry, especially, partial resolutions of singularities, boundedness and finiteness of various invariants, and moduli problems. 
For inductive arguments on the dimension of varieties, we usually deal with pairs of a  variety (or a scheme) and a linear combination of divisors. 
Currently, the minimal model theory is discussed in the category of log canonical (lc, for short) pairs. 

The minimal model theory consists of three principal pieces; running a minimal model program (MMP, for short), the termination of the MMP, and the abundance conjecture. 
The first one has been completely established (\cite{bchm}, \cite{fujino-fund}, \cite{birkar-flip}, \cite{haconxu-lcc}, \cite{has-mmp}). 
Though the other two pieces are still open, it is expected that they hold true for all lc pairs, and remarkable progress has been made (\cite{bchm}, \cite{hashizumehu}).  

There are research works on the minimal model theory in the framework beyond lc pairs. 
Semi-log canonical pairs appear in the compactification of moduli of lc pairs and inductive arguments to study the abundance conjecture for lc pairs. 
Partial results of minimal model theory for semi-log canonical pairs hold (\cite{fujino-fund-slc}, \cite{ambrokollar}). 
On the other hand, generalized pairs, introduced by Birkar--Zhang \cite{bz}, appear as the structures of base varieties of lc-trivial fibrations and normal pairs whose log canonical $\mathbb{R}$-divisor is anti-nef. 
The minimal model theory for generalized lc pairs is developed (\cite{bz}, \cite{clx--vanishing-g-pair}, \cite{hacon-liu}, \cite{has-iitakafibration}, \cite{has-nonvan-gpair}, \cite{ltj}, \cite{liuxie-relative-nakayama}, \cite{liuxie-semiample-gpair}, \cite{tsakanikas-xie--remarks-mmp}, \cite{xie-contraction-gpair}), and some relationships between the minimal model theory for generalized lc pairs and that for lc pairs are known (\cite{lp-g-abund-I}). 
Unfortunately, there exist some examples showing that the minimal model theory cannot be established in full generality for pairs
beyond lc pairs. 
For example, we cannot construct a step of an MMP for semi-log canonical pairs in general (\cite{fujino-fund-slc}, \cite{ambrokollar}), and we have to take a polarization when we discuss the abundance conjecture or the non-vanishing conjecture for generalized lc pairs (\cite{has-nonvan-gpair}, \cite{xie-contraction-gpair}, \cite{clx--vanishing-g-pair}).  

In this paper, we discuss the minimal model theory for normal pairs that are not necessarily lc. 
As shown in Section \ref{sec--example}, the existence of a step of an MMP, the existence of a minimal model, and the abundance conjecture do not always hold for normal pairs in general. 
For a projective normal pair $(X,\Delta)$, to construct a step of a $(K_{X}+\Delta)$-MMP, we need the cone and contraction theorem and the existence of a relative lc model with respect to a $(K_{X}+\Delta)$-negative extremal contraction.  
The cone and contraction theorem for normal pairs (more generally, quasi-log schemes) have been established by Fujino \cite{fujino-book}, and the existence of a relative lc model with respect to a $(K_{X}+\Delta)$-negative extremal contraction is a local problem (cf.~\cite[Corollary 6.7]{kollar-mori}) and known for lc pairs by the author \cite{has-mmp} (see \cite{birkar-flip} by Birkar and \cite{haconxu-lcc} by Hacon--Xu in the case of $\mathbb{Q}$-divisors). 
Hence, if any $(K_{X}+\Delta)$-negative curve does not intersect the non-lc locus of $(X,\Delta)$, then we can construct a step $(X,\Delta) \dashrightarrow (X',\Delta')$ of a $(K_{X}+\Delta)$-MMP. 
If, furthermore, any $(K_{X'}+\Delta')$-negative curve is disjoint from the non-lc locus of $(X',\Delta')$, then we can construct a next step $(X',\Delta') \dashrightarrow (X'',\Delta'')$ of a $(K_{X}+\Delta)$-MMP. 
By repeating this discussion, if $K_{X}+\Delta$ has a special positivity around the non-lc  locus of $(X,\Delta)$, we can run a $(K_{X}+\Delta)$-MMP. 

The special positivity condition in the previous paragraph is satisfied if the non-nef locus (see Definition \ref{defn--nonneflocus}) of the log canonical $\mathbb{R}$-divisor is disjoint from the non-lc locus. 
The non-nef locus of the log canonical $\mathbb{R}$-divisor contains all curves whose intersection number with the log canonical $\mathbb{R}$-divisor is negative (Theorem \ref{thm--nonnef-negativecurve}), and any MMP does not modify the complement of the non-nef locus of the log canonical $\mathbb{R}$-divisor. 
For a projective normal pair $(X,\Delta)$, assuming that the non-nef locus of  $K_{X}+\Delta$ does not intersect the non-lc locus of $(X,\Delta)$, then we can always construct a sequence of steps of a $(K_{X}+\Delta)$-MMP (Corollary \ref{cor--mmpwithscaling-normalpair}). 
Then it is natural to consider the termination of the $(K_{X}+\Delta)$-MMP. 
More specifically, we may consider the following question. 

\begin{ques}\label{ques--termination-mmp-normalpair}Let $\pi \colon X \to Z$ be a projective morphism of normal quasi-projective varieties. Let $(X,\Delta)$ be a normal pair such that $K_{X}+\Delta$ is $\pi$-pseudo-effective. Suppose that the non-nef locus of $K_{X}+\Delta$ over $Z$ (Definition \ref{defn--nonneflocus}) does not intersect the non-lc locus of $(X,\Delta)$. 
Then, is there a finite sequence of steps of a $(K_{X}+\Delta)$-MMP over $Z$ $$(X,\Delta)=:(X_{1},\Delta_{1}) \dashrightarrow\cdots \dashrightarrow  (X_{i},\Delta_{i}) \dashrightarrow \cdots \dashrightarrow (X_{m},\Delta_{m})$$ such that $K_{X_{m}}+\Delta_{m}$ is nef over $Z$?
\end{ques}

This question is natural since we can give an affirmative answer under the assumption on the termination of all MMP for klt pairs (see Remark \ref{rem--mmp-reduction}).

\subsection{Results}

The following theorem is the main result of this paper.

\begin{thm}[cf.~Theorem \ref{thm--termination-lcmmp-main}]\label{thm--termination-lcmmp-intro}
Let $\pi \colon X \to Z$ be a projective morphism of normal quasi-projective varieties. 
Let $(X,\Delta)$ be a normal pair and let $A$ be a $\pi$-ample $\mathbb{R}$-divisor on $X$ such that $K_{X}+\Delta+A$ is $\pi$-pseudo-effective. 
Suppose that the non-nef locus of $K_{X}+\Delta+A$ over $Z$ does not intersect the non-lc locus ${\rm Nlc}(X,\Delta)$ of $(X,\Delta)$. Suppose in addition that $(K_{X}+\Delta+A)|_{{\rm Nlc}(X,\Delta)}$, which we think of an $\mathbb{R}$-line bundle on ${\rm Nlc}(X,\Delta)$, is a finite $\mathbb{R}_{>0}$-linear combination of $\pi|_{{\rm Nlc}(X,\Delta)}$-globally generated invertible sheaves on ${\rm Nlc}(X,\Delta)$. 
We put $(X_{1},B_{1}):=(X,\Delta+A)$. 
Then there exists a diagram
$$
\xymatrix
{
(X_{1},B_{1})\ar@{-->}[r]&\cdots \ar@{-->}[r]& (X_{i},B_{i})\ar@{-->}[rr]\ar[dr]_{\varphi_{i}}&& (X_{i+1},B_{i+1})\ar[dl]^{\varphi'_{i}} \ar@{-->}[r]&\cdots\ar@{-->}[r]&(X_{m},B_{m}) \\
&&&W_{i}
}
$$
 over $Z$, where all varieties are normal and projective over $Z$, satisfying the following. 
\begin{itemize}
\item
For each $1 \leq i <m$, the diagram $(X_{i},B_{i}) \overset{\varphi_{i}}{\longrightarrow} W_{i} \overset{\varphi'_{i}}{\longleftarrow} (X_{i+1},B_{i+1})$ is a usual step of a $(K_{X_{1}}+B_{1})$-MMP over $Z$, i.e. $\varphi_{i}$ is birational, $\rho(X_{i}/W_{i})=1$, $\varphi'_{i}$ is small birational, and $-(K_{X_{i}}+B_{i})$ and $K_{X_{i+1}}+B_{i+1}$ are ample over $W_{i}$, 
\item
$(X_{1},B_{1}) \dashrightarrow (X_{m},B_{m})$ is an isomorphism on a neighborhood of ${\rm Nlc}(X,\Delta)$, 
\item
$K_{X_{m}}+B_{m}$ is semi-ample over $Z$. 
\end{itemize}
Moreover, if $X$ is $\mathbb{Q}$-factorial, then all $X_{i}$ in the MMP are also $\mathbb{Q}$-factorial. 
\end{thm}

If the non-lc locus of $(X,\Delta)$ in Theorem \ref{thm--termination-lcmmp-intro} is empty, then the statement follows from \cite[Theorem 1.5 and Theorem 1.7]{hashizumehu}. 
Thus, Theorem \ref{thm--termination-lcmmp-intro} is a generalization of \cite{hashizumehu} and \cite{bchm}. 
We note that the scheme structure of ${\rm Nlc}(X,\Delta)$ in Theorem \ref{thm--termination-lcmmp-intro} is not defined to be the reduced scheme structure. 
For the definition of ${\rm Nlc}(X,\Delta)$ as a closed subscheme of $X$, see Definition \ref{defn--nlc-nklt}. 
We also note that the MMP in Theorem \ref{thm--termination-lcmmp-intro} and the results in the introduction are natural generalizations of the non-$\mathbb{Q}$-factorial MMP in \cite[4.9.1]{fujino-book} to normal pairs. 
However, in arguments throughout this paper we usually adopt a definition of MMP weaker than \cite[4.9.1]{fujino-book}, and the definition is different from \cite[4.9.1]{fujino-book} or \cite{kollar-nonQfac-mmp}. 
More specifically, we do not assume the relative Picard number to be one in each extremal contraction of MMP. 
For details, see Definition \ref{defn--mmp-fullgeneral} and Remark \ref{rem--mmp-fullgeneral}. 
This is the reason why we say ``a usual step of a $(K_{X_{1}}+B_{1})$-MMP'' and  we explicitly state the properties of the diagram in Theorem \ref{thm--termination-lcmmp-intro}.

The following results are corollaries of Theorem \ref{thm--termination-lcmmp-intro}. 

\begin{cor}[= Corollary \ref{cor--nonvan-lc-main}]\label{cor--nonvan-intro}
Let $\pi \colon X \to Z$ be a projective morphism of normal quasi-projective varieties. 
Let $(X,\Delta)$ be a normal pair and let $A$ be a $\pi$-ample $\mathbb{R}$-divisor on $X$ such that $K_{X}+\Delta+A$ is $\pi$-pseudo-effective. 
Suppose that the non-nef locus of $K_{X}+\Delta+A$ over $Z$ does not intersect ${\rm Nlc}(X,\Delta)$ and that $(K_{X}+\Delta+A)|_{{\rm Nlc}(X,\Delta)}$, which we think of an $\mathbb{R}$-line bundle on ${\rm Nlc}(X,\Delta)$, is a finite $\mathbb{R}_{>0}$-linear combination of $\pi|_{{\rm Nlc}(X,\Delta)}$-globally generated invertible sheaves on ${\rm Nlc}(X,\Delta)$. 
Then the stable base locus ${\rm Bs}|K_{X}+\Delta+A/Z|_{\mathbb{R}}$ of $K_{X}+\Delta+A$ over $Z$ is disjoint from ${\rm Nlc}(X,\Delta)$. 
\end{cor}

\begin{cor}[= Corollary \ref{cor--mmp-from-nonvan-fullgeneral}]\label{cor--mmp-from-nonvan-intro}
Let $\pi \colon X \to Z$ be a projective morphism of normal quasi-projective varieties. 
Let $(X,\Delta)$ be a normal pair and let $A$ be a $\pi$-ample $\mathbb{R}$-divisor on $X$ such that $K_{X}+\Delta+A$ is $\pi$-pseudo-effective. 
Suppose that ${\rm Bs}|K_{X}+\Delta+A/Z|_{\mathbb{R}}$ is disjoint from ${\rm Nlc}(X,\Delta)$. 
We put $(X_{1},B_{1}):=(X,\Delta+A)$. 
Then there exists a diagram
$$
\xymatrix
{
(X_{1},B_{1})\ar@{-->}[r]&\cdots \ar@{-->}[r]& (X_{i},B_{i})\ar@{-->}[rr]\ar[dr]_{\varphi_{i}}&& (X_{i+1},B_{i+1})\ar[dl]^{\varphi'_{i}} \ar@{-->}[r]&\cdots\ar@{-->}[r]&(X_{m},B_{m}) \\
&&&W_{i}
}
$$
 over $Z$, where all varieties are normal and projective over $Z$, satisfying the following. 
\begin{itemize}
\item
For each $1 \leq i <m$, the diagram $(X_{i},B_{i}) \overset{\varphi_{i}}{\longrightarrow} W_{i} \overset{\varphi'_{i}}{\longleftarrow} (X_{i+1},B_{i+1})$ is a usual step of a $(K_{X_{1}}+B_{1})$-MMP over $Z$ (cf.~Theorem \ref{thm--termination-lcmmp-intro}), 
\item
$(X_{1},B_{1}) \dashrightarrow (X_{m},B_{m})$ is an isomorphism on a neighborhood of ${\rm Nlc}(X,\Delta)$, 
\item
$K_{X_{m}}+B_{m}$ is semi-ample over $Z$. 
\end{itemize}
Moreover, if $X$ is $\mathbb{Q}$-factorial, then all $X_{i}$ in the MMP are also $\mathbb{Q}$-factorial. 
\end{cor}

\begin{cor}[= Corollary \ref{cor--mmp-lc-strictnefthreshold-main}]\label{cor--mmp-lc-strictnefthreshold-intro}
Let $\pi \colon X \to Z$ be a projective morphism of normal quasi-projective varieties. 
Let $(X,\Delta)$ be a normal pair such that $K_{X}+\Delta$ is $\pi$-pseudo-effective. 
Suppose that the non-nef locus of $K_{X}+\Delta$ over $Z$ does not intersect ${\rm Nlc}(X,\Delta)$. 
Let $A$ be a $\pi$-ample $\mathbb{R}$-divisor on $X$. 
Then there exists a diagram
$$
\xymatrix{
(X,\Delta)=:(X_{1},\Delta_{1})\ar@{-->}[r]&\cdots \ar@{-->}[r]& (X_{i},\Delta_{i})\ar@{-->}[rr]\ar[dr]_{\varphi_{i}}&& (X_{i+1},\Delta_{i+1})\ar[dl]^{\varphi'_{i}} \ar@{-->}[r]&\cdots \\
&&&W_{i}
}
$$
 over $Z$, where all varieties are normal and projective over $Z$, satisfying the following.
\begin{itemize}
\item
$(X_{i},\Delta_{i}) \overset{\varphi_{i}}{\longrightarrow} W_{i} \overset{\varphi'_{i}}{\longleftarrow} (X_{i+1},\Delta_{i+1})$ is a usual step of a $(K_{X}+\Delta)$-MMP over $Z$ with scaling of $A$ (cf.~Theorem \ref{thm--termination-lcmmp-intro}), 
\item
the non-isomorphic locus (Definition \ref{defn--mmp-fullgeneral}) of the $(K_{X}+\Delta)$-MMP is disjoint from ${\rm Nlc}(X,\Delta)$, and 
\item
if we put 
$$\lambda_{i}:={\rm inf}\{\mu \in \mathbb{R}_{\geq 0}\,| \,\text{$K_{X_{i}}+\Delta_{i}+\mu A_{i}$ is nef over $Z$} \}$$  
for each $i \geq 1$, then ${\rm lim}_{i \to \infty}\lambda_{i}=0$. 
\end{itemize}
Moreover, if $X$ is $\mathbb{Q}$-factorial, then all $X_{i}$ in the MMP are also $\mathbb{Q}$-factorial. 
\end{cor}

For applications of Corollary \ref{cor--mmp-lc-strictnefthreshold-intro}, see Corollary \ref{cor--mmp-numericaldim-zero} and Corollary \ref{cor--nonnef-diminished}.

\begin{cor}[= Corollary \ref{cor--finite-generation-main}]{\label{cor--finite-generation-intro}}
Let $\pi \colon X \to Z$ be a projective morphism of normal quasi-projective varieties. 
Let $(X,\Delta)$ be a normal pair such that $\Delta$ is a $\mathbb{Q}$-divisor on $X$. 
Let $A$ be a $\pi$-ample $\mathbb{Q}$-divisor on $X$ such that $K_{X}+\Delta+A$ is $\pi$-pseudo-effective. 
Suppose that the non-nef locus of $K_{X}+\Delta+A$ over $Z$ does not intersect ${\rm Nlc}(X,\Delta)$ and that $(K_{X}+\Delta+A)|_{{\rm Nlc}(X,\Delta)}$, which we think of a $\mathbb{Q}$-line bundle on ${\rm Nlc}(X,\Delta)$, is semi-ample over $Z$. 
Then the sheaf of graded $\pi_{*}\mathcal{O}_{X}$-algebra
$$\underset{m \in \mathbb{Z}_{\geq 0}}{\bigoplus}\pi_{*}\mathcal{O}_{X}(\lfloor m(K_{X}+\Delta+A)\rfloor)$$
is finitely generated. 
\end{cor}

\subsection{Structure of this paper and idea of proofs}
The contents of this paper are as follows.  

In Section \ref{sec--preliminaries}, we collect some definitions and basic results used in this paper. 
Relative non-nef locus for $\mathbb{R}$-Cartier $\mathbb{R}$-divisors (Definition \ref{defn--nonneflocus}) and quasi-log schemes induced by normal pairs (Definition \ref{defn--lc-trivial-fib-quasi-log}) will be defined in this section. 

In Section \ref{sec--running-mmp}, we prove fundamental results to run an MMP for normal pairs. 
We first define minimal models and a step of an MMP for normal pairs, and then we collect their basic properties. 
After that, using the cone and contraction theorem for quasi-log schemes (\cite[Chapter 6]{fujino-book}), we prove that we can run an MMP for normal pairs if the non-nef locus of the log canonical $\mathbb{R}$-divisor is disjoint from the non-lc locus. 
We also show constructions of special kinds of MMP with scaling; MMP for normal pairs with scaling whose nef threshold is not stationary, and another type is so called ``MMP for quasi-log schemes induced by normal pairs with scaling''. 
For details, see Theorem \ref{thm--mmp-nefthreshold-strict}. 
These two kinds of MMP will play crucial roles in  Section \ref{sec--mmp-klt} and  Section \ref{sec--mmp-lc}. 
Furthermore, the second kind of the MMP is used to construct the first kind of the MMP. 

In Section \ref{sec--mmp-klt}, we study the minimal model theory for polarized normal pairs such that the non-nef locus of the log canonical $\mathbb{R}$-divisor is disjoint from the non-klt locus. 
This section can be divided into three pieces. 
Firstly, we will prove the existence of a good minimal model for polarized normal pairs assuming the non-vanishing theorem. 
In the proof, the circle of ideas in \cite[Section 5]{bchm} and the special termination in \cite{fujino-sp-ter} play important roles. 
Secondary, we will prove the termination of some MMP with scaling assuming
 the existence of good minimal models. 
This result is a variant of \cite[Theorem 4.1 (iii)]{birkar-flip}, and we partially borrow the idea in its proof. 
Despite the bad singularities of a given normal pair, an appropriate partial resolution of singularities as in Proposition \ref{prop--dlt-resol-linsystem} enables us to make use of the ideas in \cite[Section 5]{bchm}, \cite{fujino-sp-ter}, and \cite[Proof of Theorem 4.1 (iii)]{birkar-flip}. 
Finally, we will prove the non-vanishing theorem. 
Using the idea from \cite[Section 7]{bchm}, we will reduce the non-vanishing theorem to the extension of sections from the non-klt locus. 
To explain the strategy, consider a special case that does not appear in the proof, which we introduce for convenience, where we are given a projective polarized normal pair $(X,\Delta+A)$ such that $\Delta$ and $A$ are $\mathbb{Q}$-divisors, ${\rm Bs}|K_{X}+\Delta+(1+t)A|_{\mathbb{R}}$ is disjoint from the non-klt locus ${\rm Nklt}(X,\Delta)$ of $(X,\Delta)$ for any $t \in \mathbb{R}_{>0}$, and the $\mathbb{Q}$-line bundle $(K_{X}+\Delta+A)|_{{\rm Nklt}(X,\Delta)}$ is semi-ample. 
The goal is to prove ${\rm Bs}|K_{X}+\Delta+ A|_{\mathbb{R}} \cap {\rm Nklt}(X,\Delta) = \emptyset$. 
We run a $(K_{X}+\Delta+A)$-MMP with scaling of $A$ whose nef threshold is not stationary, constructed in Section \ref{sec--running-mmp}. 
By an argument, we may assume that the nef threshold goes to zero, and for any integer $p \geq 2$, we get a normal pair $(X',\Delta'+A')$ such that $K_{X'}+\Delta'+\frac{p}{p-1}A'$ is nef and log big with respect to $(X',\Delta')$. 
Choosing $p$ so that $p(K_{X}+\Delta+A)$ is a Weil divisor, we can apply a vanishing theorem for quasi-log schemes (Theorem \ref{thm--vanishing-quasi-log}) to $p(K_{X'}+\Delta'+A')$. 
Then we can extend the sections from ${\rm Nklt}(X',\Delta')$ to $X'$. 
From this and the construction of the MMP, we obtain ${\rm Bs}|K_{X}+\Delta+ A|_{\mathbb{R}} \cap {\rm Nklt}(X,\Delta) = \emptyset$, as desired. 

In Section \ref{sec--mmp-lc}, we prove Theorem \ref{thm--termination-lcmmp-intro} and corollaries. 
For a polarized normal pair as in Theorem \ref{thm--termination-lcmmp-intro}, we will run an MMP and use the argument of the special termination as in \cite{fujino-sp-ter} to reduce an MMP along klt locus, proved in Section \ref{sec--mmp-klt}. 
The main difficulties to carry out the idea are that the ampleness of the polarization is not preserved under an MMP and that we cannot directly use the adjunction for lc centers. 
Hence, we deal with quasi-log schemes induced by normal pairs such that lc centers of the normal pair are geometrically well shaped, and we consider MMP for quasi-log schemes induced by normal pairs with scaling. 
Roughly speaking, this is a sequence of quasi-log schemes induced by normal pairs $f_{i}\colon (Y_{i},\Delta_{i}+f^{*}_{i}A_{i}) \to [X_{i},\omega_{i}+A_{i}]$ ($i=1,\,2,\,\cdots$) such that $(Y_{i},\Delta_{i}+f^{*}_{i}A_{i})$ (resp.~$X_{i}$) form a sequence of steps of a $(K_{Y_{1}}+\Delta_{1}+f^{*}_{1}A_{1})$-MMP with scaling (resp.~an $(\omega_{1}+A_{1})$-MMP with scaling). 
These two MMP are deeply linked, and thus the terminations of the two MMP are equivalent.  
By making use of the MMP proved in Section \ref{sec--running-mmp}, we construct an MMP for quasi-log scheme induced by normal pair with scaling such that the nef threshold of the MMP is not stationary, $A_{i}$ is ample (after replacing suitably), and we can apply the adjunction to  any lc center of each $(Y_{i},\Delta_{i})$. 
The $(K_{Y_{1}}+\Delta_{1}+f^{*}_{1}A_{1})$-MMP is used for the special termination. 
On the other hand, the $(\omega_{1}+A_{1})$-MMP is used for the abundance by using the base point free theorem for quasi-log schemes. 
Regarding the given normal pair $(X, \Delta+A)$ as a quasi-log scheme induced by a normal pair $(X,\Delta+A) \to [X,K_{X}+\Delta+A]$, we prove Theorem \ref{thm--termination-lcmmp-intro} by using the above MMP. 
Corollaries are direct consequences of Theorem \ref{thm--termination-lcmmp-intro}. 

In Section \ref{sec--example}, we collect some examples. 
We construct examples for which a step of an MMP does not exist, a minimal model does not exist, and the abundance conjecture does not hold.  

\begin{ack}
The author was partially supported by JSPS KAKENHI Grant Number JP22K13887. 
The author thanks Professor Osamu Fujino for a comment and pointing out a typo. 
The author is grateful to the referee(s) for reading the manuscript carefully and lots of suggestions. 
\end{ack}

\section{Preliminaries}\label{sec--preliminaries}

Throughout this paper, a {\em scheme} means a separated scheme of finite type over $\mathbb{C}$. 
A {\em variety} means an integral scheme, that is, an irreducible reduced separated scheme of finite type over $\mathbb C$.

\begin{divisor}
Let $\pi \colon X \to Z$ be a projective morphism from a normal variety  to a variety. 
We will use the standard definitions of $\pi$-nef $\mathbb{R}$-divisor, $\pi$-ample $\mathbb{R}$-divisor, $\pi$-semi-ample $\mathbb{R}$-divisor, and $\pi$-pseudo-effective $\mathbb{R}$-Cartier $\mathbb{R}$-divisor. 
The set of $\mathbb{R}$-divisors on $X$ is denoted by ${\rm WDiv}_{\mathbb{R}}(X)$. 
For a prime divisor $P$ over $X$, the image of $P$ on $X$ is denoted by $c_{X}(P)$. 

A {\em contraction} $f\colon X\to Y$ is a projective morphism of varieties such that $f_{*}\mathcal{O}_{X}\cong \mathcal{O}_{Y}$. 
For a variety $X$ and an $\mathbb{R}$-divisor $D$ on $X$, a {\em log resolution of} $(X,D)$ is a projective birational morphism $g \colon W\to X$ from a smooth variety $W$ such that the exceptional locus ${\rm Ex}(g)$ of $g$ is pure codimension one and ${\rm Ex}(g)\cup {\rm Supp}\,g_{*}^{-1}D$ is a simple normal crossing divisor. 

A birational map $\phi\colon X \dashrightarrow X'$ of varieties is called a {\em birational contraction} if $\phi^{-1}$ does not contract any divisor. 
We say that $\phi$ is {\em small} if $\phi$ and $\phi^{-1}$ are birational contractions.  
\end{divisor}

\begin{defn}[$\mathbb{R}$-line bundle on scheme, relative ampleness, relative semi-ampleness]\label{defn--R-line-bundle-scheme}
Let $X$ be a (not necessarily reduced or irreducible) scheme and let ${\rm Pic}(X)$ be the Picard group of $X$. 
A {\em $\mathbb{Q}$-line bundle on $X$} is an element of ${\rm Pic}(X)\otimes_{\mathbb{Z}}\mathbb{Q}$. 
Let $\pi \colon X \to Z$ be a projective morphism to a scheme $Z$. 
A $\mathbb{Q}$-line bundle $\mathcal{L}$ on $X$ is {\em $\pi$-ample} or {\em ample over $Z$} if $\mathcal{L}$ is a finite $\mathbb{Q}_{>0}$-linear combination of $\pi$-ample invertible sheaves on $X$. 
We say that a $\mathbb{Q}$-line bundle $\mathcal{L}$ on $X$ is {\em $\pi$-semi-ample} or {\em semi-ample over $Z$} if we can write $\mathcal{L}=\sum_{i=1}^{k}q_{i} \mathcal{L}_{i}$ as an element of ${\rm Pic}(X)\otimes_{\mathbb{Z}}\mathbb{Q}$ such that $q_{1},\,\cdots,\,q_{k}$ are positive rational numbers and $\mathcal{L}_{1},\,\cdots,\,\mathcal{L}_{k}$ are globally generated over $Z$, in other words, the natural morphism
$\pi^{*}\pi_{*}\mathcal{L}_{i} \to \mathcal{L}_{i}$ 
is surjective for all $1 \leq i \leq k$. 

Similarly, we define an {\em $\mathbb{R}$-line bundle on $X$} to be an element of ${\rm Pic}(X)\otimes_{\mathbb{Z}}\mathbb{R}$, and we say that an $\mathbb{R}$-line bundle $\mathcal{L}$ on $X$ is {\em $\pi$-ample} or {\em ample over $Z$} (resp.~{\em $\pi$-semi-ample} or {\em semi-ample over $Z$}) if $\mathcal{L}$ is a finite $\mathbb{R}_{>0}$-linear combination of invertible sheaves on $X$ that are $\pi$-ample (resp.~globally generated over $Z$). 
\end{defn}

\begin{defn}
Let $a$ be a real number. 
We define the {\em round up} $\lceil a \rceil$ to be the smallest integer not less than $a$, and we define the {\em round down} $\lfloor a \rfloor$ to be the largest integer not greater than $a$. 
It is easy to see that $\lfloor a \rfloor=- \lceil -a \rceil$. 

Let $D$ be an $\mathbb{R}$-divisor on a variety, and let $D=\sum_{i}d_{i}D_{i}$ be the decomposition of $D$ into distinct prime divisors. 
We define
\begin{equation*}
\begin{split}
\lceil D\rceil:=\sum_{i}\lceil d_{i} \rceil D_{i}, \qquad \lfloor D \rfloor:= \sum_{i}\lfloor d_{i} \rfloor D_{i}, \qquad {\rm and} \qquad \{D\}:=D-\lfloor D \rfloor.
\end{split}
\end{equation*}
We also define
\begin{equation*}
D^{<1} :=\sum_{d_{i}<1}d_{i}D_{i}, \quad 
D^{= 1}:=\sum_{d_{i}= 1} D_{i}, \quad 
D^{>1} :=\sum_{d_{i}>1}d_{i}D_{i}, \quad {\rm and} \quad 
D^{\geq 1}:=\sum_{d_{i}\geq 1}d_{i} D_{i}. 
\end{equation*}
By definition, we have $\lfloor D \rfloor = -\lceil -D \rceil$. 
\end{defn}

\subsection{$\mathbb{R}$-linear system}\label{subsec--linear-system}

In this subsection, we define the relative $\mathbb{R}$-linear system and the relative stable base locus, and we prove some basic results. 
Afterwards, we define the relative diminished base locus. 
Almost all results in this subsection are analogue of the results proved in \cite[Subsection 3.5]{bchm}.

\begin{defn}[Relative $\mathbb{R}$-linear system, relative stable base locus]\label{defn--R-linear-system}
Let $\pi \colon X \to Z$ be a projective morphism of normal varieties. 
Let $D$ be a (not necessarily $\mathbb{Q}$-Cartier) $\mathbb{Q}$-divisor on $X$. 
Then the {\em $\mathbb{Q}$-linear system of $D$ over $Z$}, denoted by $|D/Z|_{\mathbb{Q}}$, is defined by
$$|D/Z|_{\mathbb{Q}}:=\{E \geq 0\,|\, E \sim_{\mathbb{Q},\, Z}D\}.$$ 
The {\em stable base locus of $D$ over $Z$}, denoted by ${\rm Bs}|D/Z|_{\mathbb{Q}}$, is defined by
$${\rm Bs}|D/Z|_{\mathbb{Q}}:=\bigcap_{E \in |D/Z|_{\mathbb{Q}}}{\rm Supp}\,E.$$ 
If $|D/Z|_{\mathbb{Q}}$ is empty, then we set ${\rm Bs}|D/Z|_{\mathbb{Q}}=X$ by convention. 

For an $\mathbb{R}$-divisor $D'$ on $X$, the {\em $\mathbb{R}$-linear system of $D'$ over $Z$}, denoted by $|D'/Z|_{\mathbb{R}}$, is defined by
$$|D'/Z|_{\mathbb{R}}:=\{E' \geq 0\,|\, E' \sim_{\mathbb{R},\, Z}D'\},$$ 
 and the {\em stable base locus of $D'$ over $Z$}, denoted by ${\rm Bs}|D'/Z|_{\mathbb{R}}$, is defined by
$${\rm Bs}|D'/Z|_{\mathbb{R}}:=\bigcap_{E' \in |D'/Z|_{\mathbb{R}}}{\rm Supp}\,E'.$$ 
If $|D'/Z|_{\mathbb{R}}$ is empty, then we set ${\rm Bs}|D'/Z|_{\mathbb{R}}=X$ by convention. 

We say that an $\mathbb{R}$-divisor $D''$ on $X$ is {\em movable over $Z$} if the codimension of ${\rm Bs}|D''/Z|_{\mathbb{R}}$ in $X$ is greater than or equal to two.
\end{defn}

\begin{lem}[cf.~{\cite[Lemma 3.5.6]{bchm}}]\label{lem--movablediv-decom}
Let $X \to Z$ be a projective morphism of normal varieties and let $D$ be an $\mathbb{R}$-divisor on $X$ which is movable over $Z$. 
Then there exist positive real numbers $r_{1},\cdots,\,r_{l}$ and effective $\mathbb{Q}$-divisors $M_{1},\cdots,\,M_{l}$ on $X$ such that $D \sim_{\mathbb{R},\,Z}\sum_{i=1}^{l}r_{i}M_{i}$ and ${\rm Bs}|M_{i}/Z|_{\mathbb{Q}} \subset {\rm Bs}|D/Z|_{\mathbb{R}}$ for every $1 \leq i \leq l$. 
In particular, $M_{i}$ is movable over $Z$ for every $1 \leq i \leq l$. 
\end{lem}

\begin{proof}
The argument is similar to \cite[Proof of Lemma 3.5.6]{bchm}. 
We fix elements $D_{1},\cdots,\,D_{p}$ of $|D/Z|_{\mathbb{R}}$ such that $\bigcap_{i=1}^{p}{\rm Supp}\,D_{i}={\rm Bs}|D/Z|_{\mathbb{R}}$. 
We prove Lemma \ref{lem--movablediv-decom} by induction on the sum of the numbers of components of the divisors $D_{1},\cdots,\,D_{p}$. 
For each $2 \leq i \leq p$, we consider the set
$$\mathcal{W}_{i}:=\left\{(E_{1}, E_{i})\in {\rm WDiv}_{\mathbb{R}}(X)\times {\rm WDiv}_{\mathbb{R}}(X)\,\middle|\begin{array}{l}
\text{$E_{1} \geq 0$, $E_{i} \geq0$, $E_{1}\sim_{\mathbb{R},\,Z}E_{i}$,}\\
\text{${\rm Supp}\,E_{1}={\rm Supp}\,D_{1}$, and}\\
\text{${\rm Supp}\,E_{i}={\rm Supp}\,D_{i}$}
\end{array}\right\}.$$
By the argument from convex geometry, we can find a rational polytope $\mathcal{C}_{i} \subset \mathcal{W}_{i}$ that contains $(D_{1},D_{i})$. 
Then there are effective $\mathbb{Q}$-divisors $N_{1},\cdots,\, N_{p}$ on $X$ such that $(N_{1},N_{i}) \in \mathcal{C}_{i}$ for every $2 \leq i \leq p$. 
Then 
${\rm Supp}\,N_{i}={\rm Supp}\,D_{i}$ for all $1 \leq i \leq p$, and it follows that 
$N_{1}\sim_{\mathbb{Q},\,Z}N_{i}$ for all $2 \leq i \leq p$.  
Since we have
$${\rm Bs}|N_{1}/Z|_{\mathbb{Q}}\subset \bigcap_{i=1}^{p}{\rm Supp}\,N_{i}=  \bigcap_{i=1}^{p}{\rm Supp}\,D_{i}={\rm Bs}|D/Z|_{\mathbb{R}},$$ 
it follows that $N_{1}$ is movable over $Z$ and ${\rm Bs}|N_{1}/Z|_{\mathbb{Q}} \subset {\rm Bs}|D/Z|_{\mathbb{R}}$. 
We set 
$$M_{1}:=N_{1} \quad {\rm and} \quad r_{1}:={\rm sup}\{t \in \mathbb{R}_{\geq 0}\,|\, \text{$D_{i}-t N_{i}\geq 0$ for every $1\leq i \leq p$}\}.$$
Putting $D'_{i}=D_{i}-r_{1} N_{i}$ for each $1\leq i \leq p$ and $D':=D-r_{1}M_{1}$, then the sum of the numbers of components of $D'_{1},\cdots,\,D'_{p}$ is strictly less than the sum of the numbers of components of $D_{1},\cdots,\,D_{p}$. 
By construction, the relation $D'\sim_{\mathbb{R},\,Z} D'_{i}$ holds for every $1 \leq i \leq p$, and 
$${\rm Bs}|D'/Z|_{\mathbb{R}} \subset \bigcap_{i=1}^{p}{\rm Supp}\,D'_{i}\subset \bigcap_{i=1}^{p}{\rm Supp}\,D_{i}={\rm Bs}|D/Z|_{\mathbb{R}}.$$
Thus $D'$ is movable over $Z$.
We apply the induction hypothesis to $D'$. 
Then we can find positive real numbers $r'_{1},\cdots,\,r'_{q}$ and effective $\mathbb{Q}$-divisors $N'_{1},\cdots,\,N'_{q}$ on $X$ such that $D' \sim_{\mathbb{R},\,Z}\sum_{j=1}^{q}r'_{j}N'_{j}$ and ${\rm Bs}|N'_{j}/Z|_{\mathbb{Q}} \subset {\rm Bs}|D'/Z|_{\mathbb{R}}$ for every $1 \leq j \leq q$. 
Since we have $D'=D-r_{1}N_{1}$ and ${\rm Bs}|D'/Z|_{\mathbb{R}} \subset {\rm Bs}|D/Z|_{\mathbb{R}}$, by renaming $r'_{1},\cdots,\,r'_{q}$ and $N'_{1},\cdots,\,N'_{q}$ we obtain positive real numbers $r_{1},\cdots,\,r_{l}$ and effective $\mathbb{Q}$-divisors $M_{1},\cdots,\,M_{l}$ on $X$ satisfying the conditions of Lemma \ref{lem--movablediv-decom}. 
\end{proof}

\begin{lem}[cf.~{\cite[Proposition 3.5.4]{bchm}}]\label{lem--rdiv-decom}
Let $X \to Z$ be a projective morphism of normal varieties and let $D$ be an $\mathbb{R}$-divisor on $X$ such that $|D/Z|_{\mathbb{R}} \neq \emptyset$.  
Then there exist effective $\mathbb{R}$-divisors $M$ and $F$ on $X$ satisfying the following.
\begin{itemize}
\item
$D \sim_{\mathbb{R},\, Z}M+F$,
\item
every component of $F$ is an irreducible component of ${\rm Bs}|D/Z|_{\mathbb{R}}$, and
\item
we may write $M=\sum_{i=1}^{l}r_{i}M_{i}$ for some positive real numbers $r_{1},\cdots,\,r_{l}$ and effective $\mathbb{Q}$-divisors $M_{1},\cdots,\,M_{l}$ on $X$ such that every $M_{i}$ is movable over $Z$ and ${\rm Bs}|M_{i}/Z|_{\mathbb{Q}} \subset {\rm Bs}|D/Z|_{\mathbb{R}}$ for every $1\leq i \leq l$. 
\end{itemize} 
\end{lem}

\begin{proof}
Let $D_{1},\cdots,\,D_{k}$ be elements of $|D/Z|_{\mathbb{R}}$ such that $\bigcap_{i=1}^{k}{\rm Supp}\,D_{i}={\rm Bs}|D/Z|_{\mathbb{R}}$. 
For any prime divisor $P$ on $X$, we define $r_{P}:=\underset{1\leq i \leq k}{\rm min}\, {\rm coeff}_{P}(D_{i})$, and we set
$$F:=\sum_{P}r_{P}P,$$
where $P$ runs over prime divisors on $X$. 
By construction, $F$ is well defined as an effective $\mathbb{R}$-divisor on $X$ and every component of $F$ is an irreducible component of ${\rm Bs}|D/Z|_{\mathbb{R}}$. 
We put 
$$D':=D-F.$$
Then $D' \sim_{\mathbb{R},\,Z}D_{i}-F$ for all $1\leq i \leq k$, and construction of $F$ implies that any prime divisor is not contained in $\bigcap_{i=1}^{k}{\rm Supp}\,(D_{i}-F)$. 
Hence, $D'$ is movable over $Z$. 
We also have
$${\rm Bs}|D'/Z|_{\mathbb{R}}\subset \bigcap_{i=1}^{k}{\rm Supp}\,(D_{i}-F) \subset \bigcap_{i=1}^{k}{\rm Supp}\,D_{i}={\rm Bs}|D/Z|_{\mathbb{R}}.$$
By applying Lemma \ref{lem--movablediv-decom} to $D'$, we get positive real numbers $r_{1},\cdots,\,r_{l}$ and effective $\mathbb{Q}$-divisors $M_{1},\cdots,\,M_{l}$ on $X$ such that 
\begin{itemize}
\item
$D' \sim_{\mathbb{R},\,Z}\sum_{i=1}^{l}r_{i}M_{i}$, and 
\item
$M_{i}$ is movable over $Z$ and ${\rm Bs}|M_{i}/Z|_{\mathbb{Q}} \subset {\rm Bs}|D'/Z|_{\mathbb{R}}$ for every $1 \leq i \leq l$. 
\end{itemize}
Putting $M=\sum_{i=1}^{l}r_{i}M_{i}$, then $M$ and $F$ satisfy the conditions of Lemma \ref{lem--rdiv-decom}. 
\end{proof}

\begin{thm}\label{thm--resol-R-linear-system}
Let $\pi \colon X \to Z$ be a projective morphism of normal quasi-projective varieties and let $\Delta$ be a reduced divisor on $X$. 
Let $D$ be an $\mathbb{R}$-Cartier $\mathbb{R}$-divisor on $X$ such that $|D/Z|_{\mathbb{R}} \neq \emptyset$. 
Then there exist a log resolution $f \colon Y \to X$ of $(X,\Delta)$ and effective $\mathbb{R}$-divisors $M_{Y}$ and $F_{Y}$ on $Y$ satisfying the following. 
\begin{itemize}
\item
$f^{*}D\sim_{\mathbb{R},\,Z}M_{Y}+F_{Y}$, 
\item
every component of $F_{Y}$ is an irreducible component of ${\rm Bs}|f^{*}D/Z|_{\mathbb{R}}$,  
\item
$M_{Y}$ is semi-ample over $Z$, and 
\item
${\rm Supp}(F_{Y}+f_{*}^{-1}\Delta)\cup{\rm Ex}(f)$ is a simple normal crossing divisor on $Y$. 
\end{itemize}
\end{thm}

\begin{proof}
Replacing $X$ by a resolution, we may assume that $X$ is smooth. 
By Lemma \ref{lem--rdiv-decom}, there exist effective $\mathbb{R}$-divisors $M$ and $F$ on $X$ satisfying the following.
\begin{itemize}
\item
$D \sim_{\mathbb{R},Z}M+F$,
\item
every component of $F$ is an irreducible component of ${\rm Bs}|D/Z|_{\mathbb{R}}$, and
\item
we may write $M=\sum_{i=1}^{l}r_{i}M_{i}$ for some positive real numbers $r_{1},\cdots,\,r_{l}$ and effective $\mathbb{Q}$-divisors $M_{1},\cdots,\,M_{l}$ on $X$ such that every $M_{i}$ is movable over $Z$ and ${\rm Bs}|M_{i}/Z|_{\mathbb{Q}} \subset {\rm Bs}|D/Z|_{\mathbb{R}}$ for every $1\leq i \leq l$. 
\end{itemize} 
We take a positive integer $k$ such that $kM_{i}$ is Cartier for every $1 \leq i \leq l$. 
Since $Z$ is quasi-projective, choosing $k$ sufficiently large and divisible, we may assume
$${\rm Bs}|M_{i}/Z|_{\mathbb{Q}}={\rm Supp}\bigl({\rm Coker}(\pi^{*}\pi_{*}\mathcal{O}_{X}(kM_{i})\otimes_{\mathcal{O}_{X}} \mathcal{O}_{X}(-kM_{i}) \longrightarrow \mathcal{O}_{X})\bigr)$$
as Zariski closed subsets of $X$. 
Define $\mathcal{I}_{i}:={\rm Im}(\pi^{*}\pi_{*}\mathcal{O}_{X}(kM_{i})\otimes_{\mathcal{O}_{X}} \mathcal{O}_{X}(-kM_{i}) \to \mathcal{O}_{X})$ for each $1 \leq i \leq l$. 
We take a log resolution $f \colon Y \to X$ of $(X,\Delta+F)$ such that $\mathcal{I}_{i}\cdot \mathcal{O}_{Y}$ is an invertible sheaf on $Y$ for every $1 \leq i \leq l$. 
We can write $\mathcal{I}_{i}\cdot \mathcal{O}_{Y}=\mathcal{O}_{Y}(-E_{i})$ for some effective Cartier divisor $E_{i}$ on $Y$. 
Since ${\rm Bs}|M_{i}/Z|_{\mathbb{Q}} \subset {\rm Bs}|D/Z|_{\mathbb{R}}$, every component of $E_{i}$ is an irreducible component of ${\rm Bs}|f^{*}D/Z|_{\mathbb{R}}$.  
By taking some blow-ups if necessary, we may assume that ${\rm Supp}(f_{*}^{-1}(\Delta+F)+\sum_{i=1}^{l}E_{i})\cup{\rm Ex}(f)$ is a simple normal crossing divisor on $Y$. 
Now we put 
$$M_{Y}:=f^{*}M-\frac{1}{k}\sum_{i=1}^{l}r_{i}E_{i} \quad {\rm and} \quad F_{Y}:=f^{*}F+\frac{1}{k}\sum_{i=1}^{l}r_{i}E_{i}.$$
Then $M_{Y}=\frac{1}{k}\sum_{i=1}^{l}r_{i}\bigl( f^{*}(kM_{i})-E_{i}\bigr)$. 
Hence $M_{Y}$ is semi-ample over $Z$. 
By construction, $f \colon Y \to X$, $M_{Y}$, and $F_{Y}$ satisfy the conditions of Theorem \ref{thm--resol-R-linear-system}. 
\end{proof}

Finally, we define the relative diminished base locus. 
We note that we only use the diminished base locus in the projective case as in \cite{bbp}, \cite{tsakanikas-xie}. However, we define the relative version of the diminished base locus for the future use.

\begin{defn}[Relative diminished base locus, cf.~{\cite{bbp}}, {\cite{tsakanikas-xie}}]\label{defn--dinimished-base-locus}
Let $\pi \colon X \to Z$ be a projective morphism from a normal quasi-projective variety $X$ to a quasi-projective scheme $Z$, and let $D$ be a $\pi$-pseudo-effective $\mathbb{R}$-Cartier $\mathbb{R}$-divisor on $X$. 
Then the {\em diminished base locus of $D$ over $Z$}, which we denote by $\boldsymbol{\rm B}_{-}(D/Z)$, is defined to be
$$\boldsymbol{\rm B}_{-}(D/Z):=\bigcup_{\epsilon>0}{\rm Bs}|D+\epsilon A/Z|_{\mathbb{R}}$$
for a $\pi$-ample $\mathbb{R}$-divisor $A$ on $X$. 
Note that $\boldsymbol{\rm B}_{-}(D/Z)$ does not depend on $A$. 
When $Z$ is a point, we denote the diminished base locus over $Z$ by $\boldsymbol{\rm B}_{-}(D)$. 
\end{defn}

\subsection{Non-nef locus}\label{subsec--non-nef-locus}

The goal of this subsection is to define the relative non-nef locus and prove some properties used in this paper. 

\begin{defn}[Relative asymptotic vanishing order, {\cite[III, \S 4.a]{nakayama}}]\label{defn--asymp-van-order}
Let $\pi \colon X \to Z$ be a projective morphism from a normal quasi-projective variety $X$ to a quasi-projective scheme $Z$. 
Let $D$ be a $\pi$-pseudo-effective $\mathbb{R}$-Cartier $\mathbb{R}$-divisor on $X$. 
For a prime divisor $P$ over $X$, we define the {\em asymptotic vanishing order of $D$ along $P$ over $Z$}, denoted by $\sigma_{P}(D/Z)$, as follows (\cite[III, \S 4.a]{nakayama}): 
We take a resolution of $f \colon X' \to X$ of $X$ on which $P$ appears as a prime divisor. 
We put 
$$m_{f^{*}D}:={\rm max}\left\{m \in \mathbb{Z}_{\geq 0}\,\middle|\begin{array}{l}(\pi\circ f)_{*}\mathcal{O}_{Y}(\lfloor f^{*}D \rfloor-mP)\hookrightarrow(\pi\circ f)_{*}\mathcal{O}_{Y}(\lfloor f^{*}D \rfloor)\\ 
\text{is an isomorphism}\end{array}\right\}$$
if $(\pi\circ f)_{*}\mathcal{O}_{Y}(\lfloor f^{*}D \rfloor)$ is not the zero sheaf. 
When $D$ is big over $Z$, then
$$\sigma_{P}(f^{*}D; X'/Z)_{\mathbb{Z}}:=\left\{\begin{array}{ll}
+\infty& ((\pi\circ f)_{*}\mathcal{O}_{Y}(\lfloor f^{*}D \rfloor)=0)\\
m_{f^{*}D}+{\rm coeff}_{P}(\{f^{*}D\})& ((\pi\circ f)_{*}\mathcal{O}_{Y}(\lfloor f^{*}D \rfloor)\neq0)
\end{array}\right.$$
and $\sigma_{P}(D/Z)$ is defined to be
$$\sigma_{P}(D/Z):=\underset{m \to \infty}{\rm lim}\frac{1}{m}\sigma_{P}(mf^{*}D; X'/Z)_{\mathbb{Z}}.$$
Note that $\sigma_{P}(D/Z)$ is independent of the resolution $f \colon X' \to X$. 
If $D$ is not necessarily $\pi$-big, then
$$\sigma_{P}(D/Z)=\underset{\epsilon\to 0+}{\rm lim}\sigma_{P}(D+\epsilon A/Z)$$
for a $\pi$-ample $\mathbb{R}$-divisor $A$ on $X$. 
Note that $\sigma_{P}(D/Z)$ does not depend on $A$.  
However, we may have $\sigma_{P}(D/Z)=\infty$ for some $P$. 
This is the main difference between the relative asymptotic vanishing order (\cite[III, \S 4.a]{nakayama}) and the asymptotic vanishing order in the projective case (\cite[1.6 Definition]{nakayama}). 
We regard $\sigma_{P}(D/Z)$ as a value in $[0,\infty]$. 
\end{defn}

\begin{defn}[Relative Nakayama--Zariski decomposition, cf.~{\cite[III, \S 4.a]{nakayama}}]\label{defn--relative-nakayama-zariski-decom}
Let $\pi \colon X \to Z$ be a projective morphism from a normal quasi-projective variety $X$ to a quasi-projective scheme $Z$, and let $D$ be a $\pi$-pseudo-effective $\mathbb{R}$-Cartier $\mathbb{R}$-divisor on $X$. 

We first define the {\em negative part of the Nakayama--Zariski decomposition of $D$ over $Z$}, denoted by $N_{\sigma}(D; X/Z)$, as follows: 
We take a resolution of $f \colon X' \to X$ of $X$, and consider the formal sum
$$\sum_{\substack {P':{\rm \,prime\,divisor}\\{\rm on\,}X'}}\sigma_{P'}(f^{*}D/Z)P'$$
of prime divisors on $X'$ with the coefficients in $[0,\infty]$ (see \cite[III, \S 4.a]{nakayama}). 
By regarding $\sigma_{P'}(f^{*}D/Z)f_{*}P':=0$ whenever $f_{*}P'=0$, we define $N_{\sigma}(D; X'/Z)$ by 
\begin{equation*}
\begin{split}
N_{\sigma}(D; X/Z):&=\sum_{\substack {P':{\rm \,prime\,divisor}\\{\rm on\,}X'}}\sigma_{P'}(f^{*}D/Z)f_{*}P'
=\sum_{\substack {P:{\rm \,prime\,divisor}\\{\rm on\,}X}}\sigma_{P}(D/Z)P.
\end{split}
\end{equation*}
We note that $N_{\sigma}(D; X/Z)$ is independent of the resolution $f \colon X' \to X$. 

Suppose that $N_{\sigma}(D; X/Z)$ is well defined as an $\mathbb{R}$-divisor on $X$. 
Then we define
$$P_{\sigma}(D; X/Z):=D-N_{\sigma}(D; X/Z)$$
and call it the {\em positive part of the Nakayama--Zariski decomposition of $D$ over $Z$}. 
We call the relation $D=P_{\sigma}(D; X/Z)+N_{\sigma}(D; X/Z)$ the {\em Nakayama--Zariski decomposition of $D$ over $Z$}. 
\end{defn}

\begin{rem}\label{rem--relative-nakayama-zariski-decom}
With notations as in Definition \ref{defn--relative-nakayama-zariski-decom}, if $X$ is smooth then the relative Nakayama--Zariski decomposition in Definition \ref{defn--relative-nakayama-zariski-decom} coincide with that in \cite[III, \S 4.a]{nakayama}. 
If $Z$ is a point, then $N_{\sigma}(D; X/Z)$ is always well defined (\cite[III, 1.5 Lemma]{nakayama}), and $N_{\sigma}(D; X/Z)$ and $P_{\sigma}(D; X/Z)$ are the same as those in \cite[Section 4]{bhzariski}. 
When $Z$ is a variety, Definition \ref{defn--relative-nakayama-zariski-decom} coincides with \cite[Definition 3.1]{liuxie-relative-nakayama}. 
Note that we may use \cite[Lemma 3.2]{liuxie-relative-nakayama} since the base scheme is quasi-projective. 
Although not explicitly mentioned in \cite[Lemma 3.2]{liuxie-relative-nakayama}, the quasi-projectivity of the base scheme is necessary in its proof. 
Anyway, we may freely use the results in \cite[Section 3]{liuxie-relative-nakayama}. 
\end{rem}

In this paper, we discuss the relative Nakayama--Zariski decomposition only when  the relative Nakayama--Zariski decomposition is well defined. 
Moreover, we only deal with the relative Nakayama--Zariski decompositions whose negative parts and positive parts are well defined as $\mathbb{R}$-Cartier $\mathbb{R}$-divisors.

\begin{defn}[Relative non-nef locus, cf.~{\cite[III, 2.6 Definition]{nakayama}}]\label{defn--nonneflocus}
Let $\pi \colon X \to Z$ be a projective morphism from a normal quasi-projective variety $X$ to a quasi-projective scheme $Z$, and let $D$ be a $\pi$-pseudo-effective $\mathbb{R}$-Cartier $\mathbb{R}$-divisor on $X$. 
We define the {\em non-nef locus of $D$ over $Z$}, denoted by ${\rm NNef}(D/Z)$, to be
$${\rm NNef}(D/Z):=\underset{\sigma_{P}(D/Z)>0}{\bigcup}c_{X}(P),$$
where $P$ runs over prime divisors over $X$ and $\sigma_{P}(D/Z)$ is the asymptotic vanishing order of $D$ along $P$ over $Z$ in Definition \ref{defn--asymp-van-order}.   
\end{defn}

\begin{rem}\label{rem--nonnef-semiample}
Let $\pi \colon X \to Z$, $P$, and $D$ be as in Definition \ref{defn--asymp-van-order}. 
Then
$$\sigma_{P}(D+H/Z) \leq \sigma_{P}(D/Z)$$
for every $\pi$-semi-ample $\mathbb{R}$-divisor $H$ on $X$. 
This implies that
$${\rm NNef}(D/Z)=  \bigcup_{\epsilon>0}{\rm NNef}(D+\epsilon H/Z).$$
\end{rem}

\begin{rem}[cf.~{\cite[Lemma 2.6]{bbp}}, {\cite[Remark 2.9]{tsakanikas-xie}}]\label{rem--nonnef-relation}
By definition, we have
$${\rm NNef}(D/Z) \subset \boldsymbol{\rm B}_{-}(D/Z) \subset {\rm Bs}|D/Z|_{\mathbb{R}}$$
for every $\pi \colon X \to Z$ and $D$ as in Definition \ref{defn--nonneflocus}. 
Furthermore, we have 
$${\rm NNef}(D/Z)=f({\rm NNef}(f^{*}D/Z))$$ 
for any projective birational morphism $f \colon Y \to X$ from a normal variety $Y$.  
\end{rem}

In this paper we use the following results without mentioning explicitly.

\begin{thm}\label{thm--nonnef-negativecurve}
Let $\pi \colon X \to Z$ be a projective morphism of normal quasi-projective varieties. 
Let $D$ be a $\pi$-pseudo-effective $\mathbb{R}$-Cartier $\mathbb{R}$-divisor on $X$. 
Let $C$ be a curve on $X$ such that $\pi(C)$ is a point and $C \not\subset {\rm NNef}(D/Z)$. 
Then $(D \,\cdot\, C) \geq 0$. 
\end{thm}

\begin{proof}
Let $f \colon Y \to X$ be a resolution of $X$ such that $f^{-1}(C)$ is pure codimension one. 
Let $A$ be an ample Cartier divisor on $Y$. 
Since $C \not\subset {\rm NNef}(D/Z)$, by the same argument as in \cite[Proof of III, 1.7 Lemma]{nakayama}, for any $k \in \mathbb{Z}_{>0}$ we see that ${\rm Bs}|f^{*}D+\frac{1}{k}A/Z|_{\mathbb{R}}$ does not contain any irreducible component of $f^{-1}(C)$ mapping onto $C$. 
We can find a curve $C' \subset Y$ such that $f(C')=C$ and $C' \not\subset {\rm Bs}|f^{*}D+\frac{1}{k}A/Z|_{\mathbb{R}}$ for any $k \in \mathbb{Z}_{>0}$. 
Then $(f^{*}D+\frac{1}{k}A)\cdot C' \geq 0$ for all $k \in \mathbb{Z}_{>0}$, and therefore
$(f^{*}D\,\cdot\, C') \geq 0$. 
Thus $(D \,\cdot\, C) \geq 0$. 
\end{proof}

\begin{lem}
Let $\pi \colon X \to Z$ be a projective morphism of normal quasi-projective varieties, and let $D$ be a $\pi$-pseudo-effective $\mathbb{R}$-Cartier $\mathbb{R}$-divisor on $X$. 
Let $f \colon Y \to X$ be a projective surjective morphism from a normal variety $Y$. 
Then we have the inclusion ${\rm NNef}(D/Z)\supset f({\rm NNef}(f^{*}D/Z))$. 
\end{lem}

\begin{proof}
By Remark \ref{rem--nonnef-semiample}, it is enough to prove ${\rm NNef}(D+A/Z)\supset f({\rm NNef}(f^{*}(D+A)/Z))$ for any $\pi$-ample divisor $A$ on $X$. 
In particular, we may assume that $D$ is $\pi$-big. 

We will prove $X\setminus {\rm NNef}(D/Z)\subset X \setminus f({\rm NNef}(f^{*}D/Z))$. 
We fix $x \in X\setminus {\rm NNef}(D/Z)$. 
We pick $y \in f^{-1}(x)$ and an arbitrary prime divisor $P$ over $Y$ such that $c_{Y}(P)\ni y$. 
We will prove that $\sigma_{P}(f^{*}D/Z)=0$. 
Let $g \colon X' \to X$ be a resolution of $X$ such that $g^{-1}(f(c_{Y}(P)))$ is pure codimension one in $X'$, and let $h \colon Y' \to Y$ be a resolution of $Y$ such that $P$ appears as a prime divisor on $Y'$ and the map $f' \colon Y' \dashrightarrow X'$ is a morphism. 
Since $D$ is $\pi$-big and $Z$ is quasi-projective, by Definition \ref{defn--asymp-van-order}, the equality
$$\sigma_{P}(f^{*}D/Z)={\rm inf}\left\{{\rm coeff}_{P}(D')\,\middle|\begin{array}{l}\text{$D'\geq 0$ and $\mathbb{R}$-Cartier,}\, D'\sim_{\mathbb{Q}}h^{*}f^{*}D+h^{*}f^{*}\pi^{*}H_{Z}\\ 
\text{for some $\mathbb{Q}$-Cartier $\mathbb{Q}$-divisor $H_{Z}$ on $Z$}\end{array}\right\}
$$
holds. 
Let $g^{-1}(f(c_{Y}(P)))=\sum_{i=1}^{l}Q_{i}$ be the prime decomposition. 
As above, we have 
$$\sigma_{Q_{i}}(D/Z)={\rm inf}\left\{{\rm coeff}_{P}(E')\,\middle|\begin{array}{l}\text{$E'\geq 0$ and $\mathbb{R}$-Cartier,}\, E'\sim_{\mathbb{Q}}g^{*}D+g^{*}\pi^{*}H'_{Z}\\ 
\text{for some $\mathbb{Q}$-Cartier $\mathbb{Q}$-divisor $H'_{Z}$ on $Z$}\end{array}\right\}
$$
for all $i$. 
Since $x \in X\setminus {\rm NNef}(D/Z)$, we have $\sigma_{Q_{i}}(D/Z)=0$ for all $i$ such that $g(Q_{i})\ni x$. 
For any $\epsilon \in \mathbb{R}_{>0}$, by taking a general effective $\mathbb{R}$-Cartier $\mathbb{R}$-divisor $E_{\epsilon}$ on $X'$ such that $E_{\epsilon}\sim_{\mathbb{Q}}g^{*}D+g^{*}\pi^{*}H'_{Z}$ for some $\mathbb{Q}$-Cartier $\mathbb{Q}$-divisor $H'_{Z}$ on $Z$, we have ${\rm coeff}_{Q_{i}}(E_{\epsilon}) \leq \epsilon$ for all $i$ such that $g(Q_{i})\ni x$. 
By construction, $P$ is a component of $f'^{*}(\sum_{j}Q_{j})$, where $j$ runs over the indices such that $g(Q_{j})\ni x$. 
Therefore, for any $\epsilon \in \mathbb{R}_{>0}$, we can find $D_{\epsilon} \geq 0$ such that $D_{\epsilon}\sim_{\mathbb{Q}}h^{*}f^{*}D+h^{*}f^{*}\pi^{*}H_{Z}$ for some $\mathbb{Q}$-Cartier $\mathbb{Q}$-divisor $H_{Z}$ on $Z$ and ${\rm coeff}_{P}(D_{\epsilon}) \leq \epsilon$. 
This implies $\sigma_{P}(f^{*}D/Z)=0$. 
Therefore, $y \not\in {\rm NNef}(f^{*}D/Z)$ for any $y \in f^{-1}(x)$. 
Then $f^{-1}(y) \cap {\rm NNef}(f^{*}D/Z) =\emptyset$. 
Hence $x \in X \setminus f({\rm NNef}(f^{*}D/Z))$. 
We finish the proof. 
\end{proof}

\subsection{Singularities of pairs}

In this subsection, we collect some definitions and basic results on singularities of pairs.

A {\em sub normal pair} $(X,\Delta)$ consists of a normal variety $X$ and an $\mathbb{R}$-divisor $\Delta$ on $X$ such that $K_{X}+\Delta$ is $\mathbb{R}$-Cartier. 
A {\em normal pair} is a sub normal pair $(X,\Delta)$ such that $\Delta$ is effective. 
We use these terms to explicitly state the normality of $X$. 
For a sub normal pair $(X,\Delta)$ and a prime divisor $P$ over $X$, the {\em discrepancy of $P$ with respect to $(X,\Delta)$} is denoted by $a(P,X,\Delta)$. 
A normal pair $(X,\Delta)$ is {\em Kawamata log terminal} ({\em klt}, for short) if $a(P,X,\Delta)>-1$ for all prime divisors $P$ over $X$. 
A normal pair $(X,\Delta)$ is {\em log canonical} ({\em lc}, for short) if $a(P,X,\Delta) \geq -1$ for all prime divisors $P$ over $X$. 
A normal pair $(X,\Delta)$ is {\em divisorially log terminal} ({\em dlt}, for short) if $(X,\Delta)$ is lc and there exists a log resolution $f \colon Y \to X$ of $(X,\Delta)$ such that every $f$-exceptional prime divisor $E$ on $Y$ satisfies $a(E,X,\Delta) > -1$. 
When $(X,\Delta)$ is an lc pair, an {\em lc center} of $(X,\Delta)$ is $c_{X}(P)$ for some prime divisor $P$ over $X$ such that $a(P,X,\Delta)=-1$. 
We freely use properties of lc centers of dlt pairs in \cite[Proposition 3.9.2]{fujino-what-log-ter} and \cite[Theorem 4.16]{kollar-mmp}.  

\begin{defn}[Non-lc locus, non-klt locus]\label{defn--nlc-nklt}
Let $(X,\Delta)$ be a normal pair. 
We define the {\em non-lc locus} and the {\em non-klt locus} of $(X,\Delta)$, denoted by ${\rm Nlc}(X,\Delta)$ and ${\rm Nklt}(X,\Delta)$ respectively, to be closed subschemes of $X$ by the following construction: 
We take a log resolution $f \colon Y \to X$ of $(X,\Delta)$. 
We may write $K_{Y}+\Gamma = f^{*}(K_{X}+\Delta)$ with an $\mathbb{R}$-divisor $\Gamma$ on $Y$. 
Then the natural isomorphism $\mathcal{O}_{X} \to f_{*}\mathcal{O}_{Y}(\lceil -(\Gamma^{<1})\rceil)$
defines ideal sheaves
\begin{equation*}
\begin{split}
&\mathcal{I}_{{\rm Nlc}(X,\Delta)}:=f_{*}\mathcal{O}_{Y}(\lceil -(\Gamma^{<1})\rceil-\lfloor \Gamma^{>1}\rfloor)=f_{*}\mathcal{O}_{Y}(-\lfloor \Gamma\rfloor+\Gamma^{=1}),\quad {\rm and}\\
&\mathcal{I}_{{\rm Nklt}(X,\Delta)}:=f_{*}\mathcal{O}_{Y}(-\lfloor \Gamma\rfloor)=f_{*}\mathcal{O}_{Y}(\lceil -(\Gamma^{<1})\rceil-\lfloor \Gamma^{>1}\rfloor-\Gamma^{=1})
\end{split}
\end{equation*}
of $\mathcal{O}_{X}$. 
These ideal sheaves are independent of the log resolution $f \colon Y \to X$ (\cite{fst-suppli-nonlc-ideal}). 
Then ${\rm Nlc}(X,\Delta)$ and ${\rm Nklt}(X,\Delta)$ are closed subschemes of $X$ defined by $\mathcal{I}_{{\rm Nlc}(X,\Delta)}$ and $\mathcal{I}_{{\rm Nklt}(X,\Delta)}$, respectively. 
\end{defn}

${\rm Nlc}(X,\Delta)$ and ${\rm Nklt}(X,\Delta)$ sometimes mean the support of the non-lc locus and the non-klt locus of $(X,\Delta)$ respectively if there is no risk of confusion.

\begin{defn}[Lc center]\label{defn--lccenter}
Let $(X,\Delta)$ be a normal pair. 
A subset $S \subset X$ is called an {\em lc center} of $(X,\Delta)$ if $S \not\subset {\rm Nlc}(X,\Delta)$ and there exists a prime divisor $P$ over $X$ such that $S=c_{X}(P)$ and $a(P,X,\Delta)=-1$. 
Unless otherwise stated, the scheme structure of an lc center is the naturally induced reduced scheme structure.
\end{defn}

\begin{thm}[Dlt blow-up, {\cite[Theorem 3.9]{fujino-morihyper}}]\label{thm--dlt-blowup}
Let $(X,\Delta)$ be a normal pair such that $X$ is quasi-projective.  
Then there exists a projective birational morphism $f \colon Y \to X$ from a normal quasi-projective variety $Y$ with the following properties.
\begin{itemize}
\item
$Y$ is $\mathbb{Q}$-factorial,
\item
$a(E,X,\Delta)\leq -1$ for every $f$-exceptional prime divisor $E$ on $Y$, and
\item
if we define an $\mathbb{R}$-divisor $\Gamma$ on $Y$ by 
$$K_{Y}+\Gamma=f^{*}(K_{X}+\Delta),$$
then $(Y,\Gamma^{<1}+{\rm Supp}\,\Gamma^{\geq 1})$ is dlt.
\end{itemize}
We call the morphism $f\colon (Y,\Gamma) \to (X,\Delta)$ a dlt blow-up. 
\end{thm}

We will prove some results on the defining ideal sheaf of the non-lc locus of a normal pair.

\begin{lem}\label{lem--nonlc-sheaf-isom}
Let $(Y,\Delta)$ be a sub normal pair such that $K_{Y}+\{\Delta\}+\Delta^{=1}$ is $\mathbb{R}$-Cartier. 
Let $g \colon W \to Y$ be a projective birational morphism from a normal variety $W$ such that all $g$-exceptional prime divisors $P$ on $W$ satisfy $a(P,Y,\{\Delta\}+\Delta^{=1})>-1$. 
We define an $\mathbb{R}$-divisor $\Gamma$ on $W$ by 
$K_{W}+\Gamma=g^{*}(K_{Y}+\Delta).$ 
We put 
$$D=\lceil -(\Delta^{<1}) \rceil-\lfloor \Delta^{>1}\rfloor \qquad {\rm and} \qquad G=\lceil -(\Gamma^{<1}) \rceil-\lfloor \Gamma^{>1} \rfloor.$$
Then $D$ is a $\mathbb{Q}$-Cartier divisor and $G-\lfloor g^{*}D \rfloor$ is effective and $g$-exceptional. 
\end{lem}

\begin{proof}
We have 
$$\Delta=\{\Delta\}+\Delta^{=1}+\lfloor \Delta^{<1}\rfloor+\lfloor \Delta^{>1}\rfloor=\{\Delta\}+\Delta^{=1}-D.$$
Since $K_{Y}+\Delta$ and $K_{Y}+\{\Delta\}+\Delta^{=1}$ are $\mathbb{R}$-Cartier, $D$ is $\mathbb{Q}$-Cartier. 

We will prove that $G-\lfloor g^{*}D \rfloor$ is effective and $g$-exceptional. 
Clearly we have $g_{*}\Gamma=\Delta$, $g_{*}\{\Gamma\}= \{\Delta\}$, and $g_{*}\Gamma^{=1}=\Delta^{=1}$. 
Thus $g_{*}G=D$. 
It is also clear that $g_{*}^{-1}D$ is a Weil divisor, and therefore $\lfloor g^{*}D\rfloor-g^{*}D$ is $g$-exceptional. 
From these facts, it follows that $G-\lfloor g^{*}D\rfloor$ is $g$-exceptional. 
We also have
\begin{equation*}
\begin{split}
K_{W}+\{\Gamma\}+\Gamma^{=1}+\lfloor \Gamma^{<1}\rfloor+\lfloor \Gamma^{>1} \rfloor=g^{*}(K_{Y}+\{\Delta\}+\Delta^{=1})+g^{*}(\lfloor \Delta^{<1}\rfloor+\lfloor \Delta^{>1}\rfloor).
\end{split}
\end{equation*}
By the definitions of $D$ and $G$, we obtain
$$G-g^{*}D=K_{W}+\{\Gamma\}+\Gamma^{=1}-g^{*}(K_{Y}+\{\Delta\}+\Delta^{=1}).$$
Since all $g$-exceptional prime divisors $P$ on $W$ satisfy $a(P,Y,\{\Delta\}+\Delta^{=1})>-1$, we have $\lceil (K_{W}+\{\Gamma\}+\Gamma^{=1}-g^{*}(K_{Y}+\{\Delta\}+\Delta^{=1}))\rceil \geq 0$. 
Therefore, we obtain
\begin{equation*}
\begin{split}
\lceil G \rceil+\lceil -g^{*}D\rceil\geq
\lceil(G-g^{*}D)\rceil=\lceil (K_{W}+\{\Gamma\}+\Gamma^{=1}-g^{*}(K_{Y}+\{\Delta\}+\Delta^{=1}))\rceil\geq0.
\end{split}
\end{equation*}
Since $G$ is a Weil divisor, it follows that $G-\lfloor g^{*}D\rfloor \geq 0$. 

By the above arguments, $D$ is $\mathbb{Q}$-Cartier and $G-\lfloor g^{*}D\rfloor$ is effective and $g$-exceptional. 
Thus, Lemma \ref{lem--nonlc-sheaf-isom} holds. 
\end{proof}

\begin{cor}\label{cor--dlt-nonlc-ideal}
Let $(Y,\Delta)$ be a normal pair such that $(Y,\Delta^{<1}+{\rm Supp}\,\Delta^{\geq 1})$ is a dlt pair.
Let $g \colon W \to Y$ be a projective birational morphism. 
We define an $\mathbb{R}$-divisor $\Gamma$ on $W$ by $K_{W}+\Gamma=g^{*}(K_{Y}+\Delta).$
Suppose that $(W,{\rm Supp}\,\Gamma)$ is log smooth. 
Then the natural isomorphism $\mathcal{O}_{Y}\to g_{*}\mathcal{O}_{W}(\lceil -(\Gamma^{<1})\rceil)$ induces an isomorphism
$$\mathcal{O}_{Y}(-\lfloor \Delta^{>1}\rfloor)\longrightarrow g_{*}\mathcal{O}_{W}(\lceil -(\Gamma^{<1})\rceil-\lfloor \Gamma^{>1} \rfloor)$$
as subsheaves of $\mathcal{O}_{Y} \cong g_{*}\mathcal{O}_{W}(\lceil -(\Gamma^{<1})\rceil)$. 
\end{cor}

\begin{proof}
As in \cite[Lemma 5.2]{fst-suppli-nonlc-ideal}, the sheaf $g_{*}\mathcal{O}_{W}(\lceil -(\Gamma^{<1})\rceil-\lfloor \Gamma^{>1} \rfloor)$ does not depend on $g \colon W \to Y$. 
From this fact, we may freely replace $g$ without loss of generality. 

Since the problem is local, we may shrink $Y$, and therefore we may assume that $Y$ is quasi-projective. 
Since $(Y,\Delta^{<1}+{\rm Supp}\,\Delta^{\geq 1})$ is dlt, there is a small $\mathbb{Q}$-factorialization $h \colon Y' \to Y$. 
The quasi-projectivity of $Y$ was used to construct this $h \colon Y' \to Y$. 
Put $\Delta'=h_{*}^{-1}\Delta$. 
Since $h$ is small, $h$ is an isomorphism over the smooth locus of $Y$ (see, for example, \cite[Corollary 2.63]{kollar-mori}). 
Thus, the dlt property (\cite[Definition 2.37]{kollar-mori}) of $(Y,\Delta^{<1}+{\rm Supp}\,\Delta^{\geq 1})$ is preserved after we replace $(Y,\Delta)$ by $(Y',\Delta')$. 
Since $h$ is a small projective birational morphism, we have
$$h_{*}\mathcal{O}_{Y'}(-\lfloor \Delta'^{>1} \rfloor)= \mathcal{O}_{Y}(-\lfloor \Delta^{>1}\rfloor).$$
From these facts, we may replace $(Y,\Delta)$ and $g \colon W\to Y$ by $(Y',\Delta')$ and some log resolution of $(Y',\Delta')$, respectively. 
Therefore, we may assume that $Y$ is $\mathbb{Q}$-factorial. 

Since $(Y,\Delta^{<1}+{\rm Supp}\,\Delta^{\geq 1})$ is a $\mathbb{Q}$-factorial dlt pair, $K_{Y}+\{\Delta\}+\Delta^{=1}$ is an $\mathbb{R}$-Cartier $\mathbb{R}$-divisor and there is a log resolution $g' \colon W' \to Y$ of $(Y,\Delta)$ such that all $g'$-exceptional prime divisors $P$ on $W'$ satisfy $a(P,Y,\Delta^{<1}+{\rm Supp}\,\Delta^{\geq 1})>-1$. 
It is easy to check the relation $\{\Delta\}+\Delta^{=1} \leq \Delta^{<1}+{\rm Supp}\,\Delta^{\geq 1}$, and thus all the $g'$-exceptional prime divisors $P$ on $W'$ satisfy $a(P,Y,\{\Delta\}+\Delta^{=1})>-1$. 
By replacing $g$ by $g'$, we may assume that all $g$-exceptional prime divisors $P$ on $W$ satisfy $a(P,Y,\{\Delta\}+\Delta^{=1})>-1$. 
We put 
$$D=-\lfloor \Delta^{>1}\rfloor \qquad {\rm and} \qquad G=\lceil -(\Gamma^{<1})\rceil-\lfloor \Gamma^{>1} \rfloor.$$
By Lemma \ref{lem--nonlc-sheaf-isom}, $D$ is $\mathbb{Q}$-Cartier and $G-\lfloor g^{*}D\rfloor$ is effective and $g$-exceptional. 
Then
$$\mathcal{O}_{Y}(-\lfloor \Delta^{>1}\rfloor) = \mathcal{O}_{Y}(D)  \cong g_{*}\mathcal{O}_{W}(G) = g_{*}\mathcal{O}_{W}(\lceil -(\Gamma^{<1}) \rceil-\lfloor \Gamma^{>1} \rfloor)$$ by the standard argument of divisorial sheaves. 
\end{proof}

\begin{thm}\label{thm--dlt-nonlc-locus}
Let $(Y,\Delta)$ be a normal pair such that $(Y,\Delta^{<1}+{\rm Supp}\,\Delta^{\geq 1})$ is a dlt pair.
Let $S$ be an lc center of $(Y,\Delta)$. 
We define $\mathbb{R}$-divisors $G_{S}$ and $B_{S}$ on $S$ by 
\begin{equation*}
\begin{split}
&G_{S}=(\Delta^{>1}-{\rm Supp}\,\Delta^{>1})|_{S}=(K_{Y}+\Delta-(K_{Y}+\Delta^{<1}+{\rm Supp}\,\Delta^{\geq 1}))|_{S}, \qquad {\rm and} \\
&K_{S}+B_{S}=(K_{Y}+\Delta^{<1}+{\rm Supp}\,\Delta^{\geq 1})|_{S},
\end{split}
\end{equation*}
respectively. 
We put $\Delta_{S}=B_{S}+G_{S}$. 
Then $\Delta_{S}$ satisfies the following properties.
\begin{itemize}
\item
$(S,\Delta_{S})$ is a normal pair satisfying $K_{S}+\Delta_{S}=(K_{Y}+\Delta)|_{S}$, 
\item 
$\Delta^{<1}_{S}+{\rm Supp}\,\Delta^{\geq 1}_{S}=B_{S}$, in particular,  $(S, \Delta^{<1}_{S}+{\rm Supp}\,\Delta^{\geq 1}_{S})$ is a dlt pair, and
\item
the morphism $ \mathcal{O}_{Y} \to  \mathcal{O}_{Y}\otimes_{\mathcal{O}_{Y}} \mathcal{O}_{S}= \mathcal{O}_{S}$ induces a morphism 
$$\mathcal{O}_{Y}(- \lfloor \Delta^{>1} \rfloor)\longrightarrow \mathcal{O}_{S}( - \lfloor \Delta_{S}^{>1} \rfloor).$$ 
\end{itemize}
\end{thm}

\begin{proof}
We prove the theorem by induction on ${\rm dim}\,Y$. 
The theorem is clear if ${\rm dim}\,Y=1$. 
From now on, we assume that ${\rm dim}\,Y>1$. 
Since $S$ is not contained in ${\rm Supp}\,\Delta^{>1}$, it follows that $S$ is an lc center of $(Y,\Delta^{<1}+{\rm Supp}\,\Delta^{\geq 1})$. 
Then we can find a component $T$ of $\Delta^{=1}$ containing $S$. 
Since $T$ is not contained in ${\rm Supp}\,\Delta^{>1}$, the divisor 
$$G_{T}:=(\Delta^{>1}-{\rm Supp}\,\Delta^{>1})|_{T}=(K_{Y}+\Delta-(K_{Y}+\Delta^{<1}+{\rm Supp}\,\Delta^{\geq 1}))|_{T}$$
 is well defined as an effective $\mathbb{R}$-Cartier $\mathbb{R}$-divisor on $T$. 
Let $(T,B_{T})$ be a dlt pair defined by adjunction $K_{T}+B_{T}=(K_{Y}+\Delta^{<1}+{\rm Supp}\,\Delta^{\geq 1})|_{T}$. 
We set
$$\Delta_{T}:=B_{T}+G_{T}.$$
We will prove that $\Delta_{T}$ satisfies all the properties of Theorem \ref{thm--dlt-nonlc-locus}. 
We have
$$K_{T}+\Delta_{T}=(K_{Y}+\Delta^{<1}+{\rm Supp}\,\Delta^{\geq 1})|_{T}+(\Delta^{>1}-{\rm Supp}\,\Delta^{>1})|_{T}=(K_{Y}+\Delta)|_{T}.$$
Hence $K_{T}+\Delta_{T}$ is $\mathbb{R}$-Cartier and we see that $(T,\Delta_{T})$ is a pair. 
Therefore $\Delta_{T}$ satisfies the first property of Theorem \ref{thm--dlt-nonlc-locus}. 
We may write 
$$\Delta^{\geq 1}=T+\sum_{i}D_{i}+\sum_{j}(1+\gamma_{j})D'_{j}$$
for some prime divisors $D_{i}$ and $D'_{j}$ and positive real numbers $\gamma_{j}$. 
By the dlt property of $(Y,\Delta^{<1}+{\rm Supp}\,\Delta^{\geq 1})$, we see that $(Y,\Delta)$ is log smooth at the generic point of $T\cap D_{i}$ and $T\cap D_{j}$ for all $i$ and $j$. 
By considering the locus on which $(Y,\Delta)$ is log smooth, we see that any component $P$ of $G_{T}=(\Delta^{>1}-{\rm Supp}\,\Delta^{>1})|_{T}$ is an irreducible component of $T \cap D'_{j}$ for some $j$. 
This fact implies that any component $P$ of $G_{T}$ satisfies
$${\rm coeff}_{P}(\Delta_{T})=1+\gamma_{j}>1, \quad {\rm coeff}_{P}(B_{T})=1, \quad {\rm and} \quad {\rm coeff}_{P}(\Delta_{T}^{<1}+{\rm Supp}\,\Delta^{\geq 1}_{T})=1.$$
Therefore, any component $P$ of $G_{T}$ satisfies 
$${\rm coeff}_{P}(B_{T})={\rm coeff}_{P}(\Delta_{T}^{<1}+{\rm Supp}\,\Delta^{\geq 1}_{T}).$$ 
Since $\Delta_{T}=B_{T}+G_{T}$, if a prime divisor $Q$ is not a component of $G_{T}$, then we have 
$${\rm coeff}_{Q}(\Delta_{T})={\rm coeff}_{Q}(B_{T}) \leq 1,$$ 
where the inequality follows from the dlt property of $(T,B_{T})$. 
This shows that $Q$ is not a component of $\Delta_{T}^{>1}$. 
Hence we have ${\rm coeff}_{Q}(\Delta_{T})={\rm coeff}_{Q}(\Delta_{T}^{<1}+{\rm Supp}\,\Delta^{\geq 1}_{T}).$
Then
$${\rm coeff}_{Q}(B_{T})={\rm coeff}_{Q}(\Delta_{T}^{<1}+{\rm Supp}\,\Delta^{\geq 1}_{T}).$$
By the discussion, we obtain
$$B_{T}=\Delta_{T}^{<1}+{\rm Supp}\,\Delta^{\geq 1}_{T},$$
which is the second property of Theorem \ref{thm--dlt-nonlc-locus}. 

In this paragraph, we will prove the third property of Theorem \ref{thm--dlt-nonlc-locus} for $T$ and $\Delta_{T}^{>1}$. 
Since $(Y,\Delta^{<1}+{\rm Supp}\,\Delta^{\geq 1})$ is dlt, there is a log resolution $g \colon W \to Y$ of $(Y,\Delta)$ such that all $g$-exceptional prime divisor $E$ on $W$ satisfy $a(E,Y,\Delta^{<1}+{\rm Supp}\,\Delta^{\geq 1})>-1$. 
We put $T'=g_{*}^{-1}T$ and $g_{T'}=g|_{T'}\colon T' \to T$, and we define an $\mathbb{R}$-divisor $\Gamma$ on $W$ by 
$$K_{W}+T'+\Gamma=g^{*}(K_{Y}+\Delta).$$ 
By the first property of Theorem \ref{thm--dlt-nonlc-locus} for $(T,\Delta_{T})$, we have $K_{T'}+\Gamma|_{T'}=g_{T'}^{*}(K_{T}+\Delta_{T})$. 
By the exact sequence
\begin{equation*}
\begin{split}
0 \longrightarrow \mathcal{O}_{W}(\lceil -(\Gamma^{<1})\rceil-\lfloor \Gamma^{>1} \rfloor-T')&\longrightarrow \mathcal{O}_{W}(\lceil -(\Gamma^{<1})\rceil-\lfloor \Gamma^{>1} \rfloor)\\
& \longrightarrow \mathcal{O}_{T'}(\lceil -(\Gamma^{<1})|_{T'}\rceil-\lfloor \Gamma^{>1}|_{T'} \rfloor) \longrightarrow 0,
\end{split}
\end{equation*}
we get the exact sequence 
\begin{equation*}
\begin{split}
 0 \longrightarrow g_{*}\mathcal{O}_{W}(\lceil -(\Gamma^{<1})\rceil-\lfloor \Gamma^{>1} \rfloor-T')&\longrightarrow g_{*}\mathcal{O}_{W}(\lceil -(\Gamma^{<1})\rceil-\lfloor \Gamma^{>1} \rfloor)\\ & \longrightarrow g_{T'*}\mathcal{O}_{T'}(\lceil -(\Gamma^{<1})|_{T'}\rceil-\lfloor \Gamma^{>1}|_{T'} \rfloor). 
 \end{split}
 \end{equation*}
By Corollary \ref{cor--dlt-nonlc-ideal}, the isomorphism $\mathcal{O}_{Y} \to g_{*}\mathcal{O}_{W}(\lceil -(\Gamma^{<1})\rceil)$ induces an isomorphism
$\mathcal{O}_{Y}(- \lfloor \Delta^{>1} \rfloor) \cong g_{*}\mathcal{O}_{W}(\lceil -(\Gamma^{<1})\rceil-\lfloor \Gamma^{>1} \rfloor)$. 
This isomorphism and $T'=g_{*}^{-1}T$ imply that $\mathcal{O}_{Y}(- \lfloor \Delta^{>1} \rfloor-T) \to g_{*}\mathcal{O}_{W}(\lceil -(\Gamma^{<1})\rceil-\lfloor \Gamma^{>1} \rfloor-T')$ is an isomorphism. 
Now $(T',\Gamma|_{T'})$ is log smooth by construction, and $(T, \Delta^{<1}_{T}+{\rm Supp}\,\Delta^{\geq 1}_{T})$ is dlt by the second property of Theorem \ref{thm--dlt-nonlc-locus} for $T$. 
Since $g_{T'} \colon T' \to T$ is birational and $K_{T'}+\Gamma|_{T'}=g_{T'}^{*}(K_{T}+\Delta_{T})$, Corollary \ref{cor--dlt-nonlc-ideal} implies that the isomorphism $\mathcal{O}_{T} \to g_{T'*}\mathcal{O}_{T'}(\lceil -(\Gamma^{<1})|_{T'}\rceil)$ induces an isomorphism
$\mathcal{O}_{T}( - \lfloor \Delta_{T}^{>1} \rfloor) \cong g_{T'*}\mathcal{O}_{T'}(\lceil -(\Gamma^{<1})|_{T'}\rceil-\lfloor \Gamma^{>1}|_{T'} \rfloor)$. 
By these isomorphisms, we get the exact sequence
$$0 \longrightarrow \mathcal{O}_{Y}(- \lfloor \Delta^{>1} \rfloor-T)\longrightarrow \mathcal{O}_{Y}(- \lfloor \Delta^{>1} \rfloor) \longrightarrow \mathcal{O}_{T}( - \lfloor \Delta_{T}^{>1} \rfloor)$$
such that the morphism $\mathcal{O}_{Y}(- \lfloor \Delta^{>1} \rfloor) \to \mathcal{O}_{T}( - \lfloor \Delta_{T}^{>1} \rfloor)$ is induced by $\mathcal{O}_{Y} \to \mathcal{O}_{T}$. 
This shows that $T$ and $\Delta_{T}^{>1}$ satisfy the third property of Theorem \ref{thm--dlt-nonlc-locus}. 

The arguments in the previous paragraphs show that $\Delta_{T}$ satisfies all the properties of Theorem \ref{thm--dlt-nonlc-locus}. 
We recall that $T$ is a component of $\Delta^{=1}$ containing $S$. 
Therefore $S$ is an lc center of $(T, \Delta_{T})$. 
By the second property of Theorem \ref{thm--dlt-nonlc-locus} for $\Delta_{T}$, we obtain
$$B_{T}=\Delta_{T}^{<1}+{\rm Supp}\,\Delta^{\geq 1}_{T} \qquad {\rm and} \qquad G_{T}=\Delta_{T}^{>1}-{\rm Supp}\,\Delta_{T}^{>1}.$$
Therefore, we have
\begin{equation*}
\begin{split}
&G_{S}=\bigl((\Delta^{>1}-{\rm Supp}\,\Delta^{>1})|_{T}\bigr)|_{S}=(\Delta_{T}^{>1}-{\rm Supp}\,\Delta_{T}^{>1})|_{S},\quad {\rm and}\\
&K_{S}+B_{S}=\bigl((K_{Y}+\Delta^{<1}+{\rm Supp}\,\Delta^{\geq 1})|_{T}\bigr)|_{S}
=(K_{T}+\Delta_{T}^{<1}+{\rm Supp}\,\Delta^{\geq 1}_{T})|_{S}. \\
\end{split}
\end{equation*}
By the induction hypothesis of Theorem \ref{thm--dlt-nonlc-locus}, the divisor $\Delta_{S}$ satisfies the first and the second properties of Theorem \ref{thm--dlt-nonlc-locus}, and moreover we obtain the morphism
$$\mathcal{O}_{Y}(- \lfloor \Delta^{>1} \rfloor) \longrightarrow \mathcal{O}_{T}( - \lfloor \Delta_{T}^{>1} \rfloor) \longrightarrow \mathcal{O}_{S}( - \lfloor \Delta_{S}^{>1} \rfloor)$$
induced by $\mathcal{O}_{Y} \to \mathcal{O}_{T} \to\mathcal{O}_{S}$. 
From this, $\Delta_{S}$ satisfies all properties of Theorem \ref{thm--dlt-nonlc-locus}. 
We finish the proof. 
\end{proof}

\subsection{Quasi-log scheme}\label{subsec--quasi-log}

The goal of this subsection is to define a {\em quasi-log scheme induced by a normal pair} (Definition \ref{defn--lc-trivial-fib-quasi-log}) and prove some results used in this paper. 

\begin{defn}[{Quasi-log scheme, \cite[Definition 6.2.2]{fujino-book}}]\label{defn--quasi-log}
A {\em{quasi-log scheme}} is a scheme $X$ endowed with an $\mathbb R$-Cartier $\mathbb{R}$-divisor (or an $\mathbb R$-line bundle) $\omega$ on $X$, a closed subscheme ${\rm Nqlc}(X, \omega)\subsetneq X$, and a finite collection $\{C\}$ of reduced and irreducible subschemes of $X$, such that there is a proper morphism $f\colon (Y, B_Y)\to X$ from a globally embedded simple normal crossing pair satisfying the following properties: 
\begin{itemize}
\item $f^*\omega\sim_{\mathbb R}K_Y+B_Y$, 

\item the natural map $\mathcal O_X \to f_*\mathcal O_Y(\lceil -(B_Y^{<1})\rceil)$ induces an isomorphism 
$$
\mathcal I_{{\rm Nqlc}(X, \omega)}\overset{\cong}{\longrightarrow} f_*\mathcal O_Y(\lceil -(B_Y^{<1})\rceil-\lfloor B_Y^{>1}\rfloor),  
$$ 
where $\mathcal I_{{\rm Nqlc}(X, \omega)}$ is the defining ideal sheaf of ${\rm Nqlc}(X, \omega)$, and

\item the collection of reduced and irreducible subschemes 
$\{C\}$ coincides with the images 
of the strata of $(Y, B_Y)$ that are not included in ${\rm Nqlc}(X, \omega)$. 
\end{itemize}

We simply write $[X, \omega]$ to denote the above data 
$$
\left(X, \omega, f\colon (Y, B_Y)\to X\right)
$$ 
if there is no risk of confusion. 
An element of $\{C\}$ is called a {\em qlc center of $[X,\omega]$}. 

The {\rm non-qklt locus of $X$}, denoted by ${\rm Nqklt}(X, \omega)$, is the union of ${\rm Nqlc}(X, \omega)$ and all qlc centers of $[X,\omega]$. 
We note that ${\rm Nqklt}(X, \omega)$ has the scheme structure naturally induced by ${\rm Nqlc}(X, \omega)$ and all qlc centers of $[X,\omega]$ (cf.~\cite[Notation 6.3.10]{fujino-book}).  
\end{defn}

\begin{thm}\label{thm--abundance-quasi-log}
Let $[X,\omega]$ be a quasi-log scheme and $\pi \colon X \to Z$ a projective morphism to a scheme. 
Let $A$ be a $\pi$-ample $\mathbb{R}$-divisor on $X$. 
Suppose that $\omega+A$ is $\pi$-nef and $(\omega+A)|_{{\rm Nqlc}(X, \omega)}$, which we think of an $\mathbb{R}$-line bundle on ${\rm Nqlc}(X, \omega)$, is semi-ample over $Z$. 
Then $\omega+A$ is $\pi$-semi-ample.
\end{thm}

\begin{proof}
We may assume that $Z$ is affine. 
By using \cite[Lemma 4.25]{fujino-morihyper} and the argument from convex geometry, we can find positive real numbers $r_{1},\,\cdots, \,r_{m}$, $\mathbb{Q}$-line bundles $\omega_{1},\,\cdots,\,\omega_{m}$ on $X$, and $\pi$-ample $\mathbb{Q}$-divisors $A_{1},\,\cdots,\,A_{m}$ on $X$ such that
\begin{itemize}
\item
$\sum_{i=1}^{m}r_{i}=1$, $\sum_{i=1}^{m}r_{i}\omega_{i}=\omega$, and $\sum_{i=1}^{m}r_{i}A_{i}=A$, 
\item
for every $1\leq i \leq m$, there is the structure of a quasi-log scheme $[X,\omega_{i}]$ that has the same non-qlc locus as ${\rm Nqlc}(X, \omega)$, and
\item
all $(\omega_{i}+A_{i})|_{{\rm Nqlc}(X, \omega)}$ are semi-ample over $Z$. 
\end{itemize}
For each $1\leq i \leq m$, the cone theorem \cite[Theorem 6.7.4]{fujino-book} implies that we may write
$$\overline{\rm NE}(X/Z)=\overline{\rm NE}(X/Z)_{(\omega_{i}+A_{i}) \geq 0}+\overline{\rm NE}(X/Z)^{(i)}_{-\infty}+\sum_{j=1}^{p_{i}}\mathbb{R}_{\geq 0}[C_{i,j}]$$
for some $p_{i} \in \mathbb{Z}_{>0}$ and curves $C_{i,j} \subset X$ contained in a fiber of $\pi$, where 
$$\overline{\rm NE}(X/Z)^{(i)}_{-\infty}:={\rm Im}(\overline{\rm NE}({\rm Nqlc}(X, \omega_{i})/Z) \to \overline{\rm NE}(X/Z))\qquad (\text{\cite[Definition 6.7.1]{fujino-book}}).$$
Since $(\omega_{i}+A_{i})|_{{\rm Nqlc}(X, \omega)}$ is semi-ample over $Z$, we have $\overline{\rm NE}(X/Z)^{(i)}_{-\infty} \subset \overline{\rm NE}(X/Z)_{(\omega_{i}+A_{i}) \geq 0}$ for any $1 \leq i \leq m$. 
For each $1 \leq i \leq m$ and $1 \leq j \leq p_{i}$, we consider the set
$$\mathcal{H}_{i,j}:=\left\{(t_{1},\cdots,\,t_{m})\in (\mathbb{R}_{\geq 0})^{m}\,\middle|\,\sum_{k=1}^{m}t_{k}=1, \, \Bigl(\sum_{k=1}^{m}t_{k} (\omega_{k}+A_{k})\Bigr)\cdot C_{i,j}\geq 0\right\}.$$
Since all $\omega_{k}+A_{k}$ are $\mathbb{Q}$-Cartier and all $\mathcal{H}_{i,j}$ contain $(r_{1}, \cdots,\,r_{m})$, we see that $\bigcap_{i,j}\mathcal{H}_{i,j}$ is a rational polytope that contains $(r_{1}, \cdots,\,r_{m})$. 
By the standard argument from convex geometry, we can find positive real numbers $r'_{1},\,\cdots, \,r'_{n}$, $\mathbb{Q}$-line bundles $\omega'_{1},\,\cdots,\,\omega'_{n}$ on $X$, and $\pi$-ample $\mathbb{Q}$-divisors $A'_{1},\,\cdots,\,A'_{n}$ on $X$ such that
\begin{itemize}
\item
$\sum_{l=1}^{n}r'_{l}=1$, $\sum_{l=1}^{n}r'_{l}\omega'_{l}=\omega$, and $\sum_{l=1}^{n}r'_{l}A'_{l}=A$, 
\item
every $\omega'_{l}+A'_{l}$ is a convex linear combination of $\omega_{1}+A_{1}, \,\cdots,\,\omega_{m}+A_{m}$, 
\item
for every $1\leq l \leq n$, there is the structure of a quasi-log scheme $[X,\omega'_{l}]$ that has the same non-qlc locus as ${\rm Nqlc}(X, \omega)$, 
\item
all $(\omega'_{l}+A'_{l})|_{{\rm Nqlc}(X, \omega)}$ are semi-ample over $Z$, and
\item
$(\omega'_{l}+A'_{l})\cdot C_{i,j} \geq 0$ for all $l$, $i$, and $j$.  
\end{itemize}
Then $\omega'_{l}+A'_{l}$ is $\pi$-nef for all $l$. 
Indeed, if $\omega'_{l}+A'_{l}$ is not $\pi$-nef for some $l$, then there exists an $(\omega'_{l}+A'_{l})$-negative extremal ray $R$ of $\overline{\rm NE}(X/Z)$. 
By construction, there is an index $i$ such that $R$ is an $(\omega_{i}+A_{i})$-negative extremal ray of $\overline{\rm NE}(X/Z)$. 
Then $R$ is generated by $C_{i,j}$ for some $j$. 
However, it contradicts the condition $(\omega'_{l}+A'_{l})\cdot C_{i,j} \geq 0$. 
Therefore, $\omega'_{l}+A'_{l}$ is $\pi$-nef for all $l$. 
By the base point free theorem \cite[Theorem 6.5.1]{fujino-book}, every $\omega'_{l}+A'_{l}$ is $\pi$-semi-ample. 
Therefore, $\omega+A=\sum_{l=1}^{n}r'_{l}(\omega'_{l}+A'_{l})$ is also $\pi$-semi-ample. 
\end{proof}

In this paper, we use the notion of quasi-log scheme induced by normal pair. 

\begin{lem}\label{lem--str-quasi-log}
Let $f \colon Y \to X$ be a contraction of normal varieties and let $(Y,\Delta)$ be a normal pair such that $f({\rm Nlc}(Y,\Delta))\subsetneq X$ and $K_{Y}+\Delta\sim_{\mathbb{R}}f^{*}\omega$ for some $\mathbb{R}$-Cartier $\mathbb{R}$-divisor $\omega$ on $X$. 
Let $g \colon W \to Y$ be a log resolution of $(Y,\Delta)$. 
We define an $\mathbb{R}$-divisor $\Gamma$ on $W$ by 
$K_{W}+\Gamma=g^{*}(K_{Y}+\Delta).$ 
Let $(X,\omega, f\circ g \colon (W,\Gamma) \to X)$ be the structure of a quasi-log scheme. 
Then the structure does not depend on the choice of $g \colon W \to Y$, in other words, ${\rm Nqlc}(X, \omega)$ and the set of qlc centers do not depend on $g$.  
\end{lem}

\begin{proof}
This is a direct consequence of \cite[Proposition 6.3.1]{fujino-book}. 
\end{proof}

\begin{defn}[Quasi-log scheme induced by normal pair]\label{defn--lc-trivial-fib-quasi-log}
A {\em quasi-log scheme induced by a normal pair}, which we denote by $f\colon (Y,\Delta) \to [X,\omega]$ in this paper, consists of a normal pair $(Y,\Delta)$, a contraction $f \colon Y \to X$ of normal varieties, and the structure of a quasi-log scheme $[X,\omega]$ on $X$ which is defined with a log resolution of $(Y,\Delta)$. 
By definition, it follows that
\begin{itemize}
\item
$f({\rm Nlc}(Y,\Delta))\subsetneq X$, and 
\item
$K_{Y}+\Delta\sim_{\mathbb{R}}f^{*}\omega$. 
\end{itemize}
By Lemma \ref{lem--str-quasi-log}, the structure of $[X,\omega]$ does not depend on the log resolution of $(Y,\Delta)$.

In the case of $X=Y$ and $\omega=K_{Y}+\Delta$, we may identify $[Y,K_{Y}+\Delta]$ with the normal pair $(Y,\Delta)$. 
In this situation, it follows that ${\rm Nqlc}(Y,K_{Y}+\Delta)={\rm Nlc}(Y,\Delta)$ and ${\rm Nqklt}(Y,K_{Y}+\Delta)={\rm Nklt}(Y,\Delta)$ as closed subschemes of $Y$. 
\end{defn}

By definition, $(Y,\Delta)$ and $f \colon Y \to X$ in a quasi-log scheme induced by a normal pair $f\colon (Y,\Delta) \to [X,\omega]$ also form an lc-trivial fibration.  
However, we will often focus on the structure of the quasi-log scheme of $[X,\omega]$ and use ${\rm Nqklt}(X, \omega)$ and ${\rm Nqlc}(X, \omega)$. 
Hence, we regard $f\colon (Y,\Delta) \to [X,\omega]$ as a quasi-log scheme rather than an lc-trivial fibration in this paper. 

\begin{rem}
Let $f\colon (Y,\Delta) \to [X,\omega]$ be a quasi-log scheme induced by a normal pair $(Y,\Delta)$. 
Let $\mathcal{I}_{{\rm Nklt}(Y,\Delta)}$ and $\mathcal{I}_{{\rm Nlc}(Y,\Delta)}$ be the defining ideal sheaves of the non-klt locus and non-lc locus of $(Y,\Delta)$, respectively. 
By definition, ${\rm Nqklt}(X,\omega)$ and ${\rm Nqlc}(X,\omega)$ are the closed subschemes defined by $f_{*}\mathcal{I}_{{\rm Nklt}(Y,\Delta)}$ and $f_{*}\mathcal{I}_{{\rm Nlc}(Y,\Delta)}$, respectively. 
In particular, we have $f({\rm Nklt}(Y,\Delta))={\rm Nqklt}(X,\omega)$ and $f({\rm Nlc}(Y,\Delta))={\rm Nqlc}(X,\omega)$ set-theoretically.  
\end{rem}

\begin{thm}\label{thm--pair-inductive-adjunction}
Let $f\colon (Y,\Delta) \to [X,\omega]$ be a quasi-log scheme induced by a normal pair such that $(Y,\Delta^{<1}+{\rm Supp}\,\Delta^{\geq 1})$ is a dlt pair. 
Let $S$ be an lc center of $(Y,\Delta)$ such that $f(S)\not\subset{\rm Nqlc}(X, \omega)$. 
Let $(S,\Delta_{S})$ be a normal pair defined by using adjunction $K_{S}+\Delta_{S}=(K_{Y}+\Delta)|_{S}$. 
Let $f_{S} \colon S \to T$ be the Stein factorization of $f|_{S} \colon S \to X$, and let $\omega_{T}$ be the pullback of $\omega$ to $T$. 
Then there exists the structure of a quasi-log scheme induced by a normal pair $f_{S} \colon (S,\Delta_{S}) \to [T,\omega_{T}]$ as in Definition \ref{defn--lc-trivial-fib-quasi-log} such that the natural morphism $\tau \colon T \to X$ induces a morphism ${\rm Nqlc}(T, \omega_{T}) \to {\rm Nqlc}(X, \omega)$ of closed subschemes and the image of any qlc center of $[T ,\omega_{T}]$ by $\tau$ is a qlc center of $[X,\omega]$. 
\end{thm}

\begin{proof}
Let $\tau \colon T \to X$ and $f_{S}\colon S \to T$ be as in the theorem. 
To define the structure of a quasi-log scheme induced by a normal pair $f_{S} \colon (S,\Delta_{S}) \to [T,\omega_{T}]$, we need to check $f_{S}({\rm Nlc}(S,\Delta_{S})) \neq T$. 
By Theorem \ref{thm--dlt-nonlc-locus}, the pair $(S,\Delta_{S}^{<1}+{\rm Supp}\,\Delta_{S}^{\geq 1})$ is a dlt pair and
the natural morphism $\mathcal{O}_{Y} \to \mathcal{O}_{S}$ induces a morphism $\mathcal{O}_{Y}(- \lfloor \Delta^{>1} \rfloor)\to \mathcal{O}_{S}( - \lfloor \Delta_{S}^{>1} \rfloor)$. 
By Corollary \ref{cor--dlt-nonlc-ideal}, the support of ${\rm Nlc}(S,\Delta_{S})$ is contained in ${\rm Nlc}(Y,\Delta)\cap S$. 
Then $f({\rm Nlc}(S,\Delta_{S})) \subset {\rm Nqlc}(X, \omega) \cap f(S)$. 
This fact and the hypothesis $f(S)\not\subset{\rm Nqlc}(X, \omega)$ imply $f_{S}({\rm Nlc}(S,\Delta_{S})) \neq T$. 
Therefore, we may define the structure of a quasi-log scheme induced by a normal pair $f_{S} \colon (S,\Delta_{S}) \to [T,\omega_{T}]$. 

Since $(Y,\Delta^{<1}+{\rm Supp}\,\Delta^{\geq 1})$ is a dlt pair and $S$ is an lc center of $(Y,\Delta)$, any lc center of $(S,\Delta_{S})$ is an lc center of $(Y,\Delta)$ contained in $S$. 
From this, the image $\tau(C)$ of any qlc center $C$ of $[T ,\omega_{T}]$ is a qlc center of $[X,\omega]$. 

By Corollary \ref{cor--dlt-nonlc-ideal}, the defining ideal sheaf $\mathcal{I}_{X}$ of ${\rm Nqlc}(X, \omega)$ is $f_{*}\mathcal{O}_{X}(-\lfloor \Delta^{>1}\rfloor)$, and the defining ideal sheaf $\mathcal{I}_{T}$ of ${\rm Nqlc}(T, \omega_{T})$ is $f_{S*}\mathcal{O}_{S}(-\lfloor \Delta^{>1}_{S}\rfloor)$. 
By Theorem \ref{thm--dlt-nonlc-locus}, the natural morphism $\mathcal{O}_{Y} \to \mathcal{O}_{S}$ induces a morphism $\mathcal{I}_{X} \to \tau_{*}\mathcal{I}_{T}$. 
By the following diagram
 $$
\xymatrix@C=16pt{
0 \ar[r]& \tau_{*}\mathcal{I}_{T} \ar[r]& \tau_{*}\mathcal{O}_{T} \ar[r]& \tau_{*}\mathcal{O}_{{\rm Nqlc}(T, \omega_{T})}\\
0 \ar[r]& \mathcal{I}_{X} \ar[r]\ar[u]& \mathcal{O}_{X} \ar[r]\ar[u]& \mathcal{O}_{{\rm Nqlc}(X, \omega)}  \ar[r]& 0,
}
$$
there exists a natural morphism $\mathcal{O}_{{\rm Nqlc}(X, \omega)}\to \tau_{*}\mathcal{O}_{{\rm Nqlc}(T, \omega_{T})}$. 
This induces the desired morphism ${\rm Nqlc}(T, \omega_{T}) \to {\rm Nqlc}(X, \omega)$ as a morphism between closed subschemes. 
\end{proof}

\begin{lem}\label{lem--can-bundle-formula}
Let $f\colon (Y,\Delta) \to [X,\omega]$ be a quasi-log scheme induced by a normal pair. 
Let $\pi \colon X \to Z$ be a projective morphism to a quasi-projective scheme $Z$, and let $A$ be a $\pi$-ample $\mathbb{R}$-divisor on $X$. 
Then there exists a normal pair $(X,G)$ such that the relation $K_{X}+G \sim_{\mathbb{R},\,Z}\omega+A$ holds and ${\rm Nklt}(X,G) = {\rm Nqklt}(X,\omega)$ as closed subschemes of $X$. 
\end{lem}

\begin{proof}
By the definitions of the discriminant $\mathbb{R}$-b-divisors and the moduli $\mathbb{R}$-b-divisors as in \cite[Corollary 5.2]{fujino-hashizume-adjunction} (see \cite{ambro1} for the case of $\mathbb{Q}$-divisors), we obtain a generalized pair $(X,B+M)$ as in \cite{bz} such that $\omega \sim_{\mathbb{R},\,Z} K_{X}+B+M$ and the generalized non-klt locus of $(X,B+M)$ is equal to ${\rm Nqklt}(X,\omega)$ (cf.~\cite[Lemma 3.2]{ambro1}). 
Let $g \colon X' \to X$ be a log resolution of $(X,B)$ such that the moduli part $M'$ on $X'$ is nef over $Z$. 
We define an $\mathbb{R}$-divisor $B'$ on $X'$ by $K_{X'}+B'+M'=g^{*}(K_{X}+B+M)$. 
Since $g^{*}A+M'$ is nef and big over $Z$, we can find an effective $\mathbb{R}$-divisor $E'$ on $X'$ and a general $(\pi\circ g)$-ample $\mathbb{R}$-divisor $H'$ on $X'$ such that $E'+H' \sim_{\mathbb{R},\,Z}g^{*}A+M'$ and $\lfloor B' \rfloor= \lfloor (B'+E'+H') \rfloor$. 
By replacing $X'$ with a log resolution of $(X',B'+E')$, we may assume that $(X',B'+E'+H')$ is log smooth. 
Put $G:=g_{*}(B'+E'+H')$. 
By construction, we have $K_{X}+G\sim_{\mathbb{R},\,Z}\omega+A$ and ${\rm Nklt}(X,G)$ coincides with the generalized non-klt locus of $(X,B+M)$, and therefore we have ${\rm Nklt}(X,G)={\rm Nqklt}(X,\omega)$ as closed subschemes of $X$. 
\end{proof}

\begin{thm}\label{thm--vanishing-quasi-log}
Let $f\colon (Y,\Delta) \to [X,\omega]$ be a quasi-log scheme induced by a normal pair. 
Fix $X'$ a union of ${\rm Nqlc}(X, \omega)$ and (possibly empty) some qlc centers of $[X,\omega]$. 
Let $\pi \colon X \to Z$ be a projective morphism to a scheme $Z$, and let $L$ be a $\mathbb{Q}$-Cartier Weil divisor on $X$ such that 
\begin{itemize}
\item
$L$ is Cartier on a neighborhood of $X'$, 
\item
$f^{*}L$ is a Weil divisor on $Y$, and 
\item
$L-\omega$ is $\pi$-nef and $\pi$-log big with respect to $[X,\omega]$ (\cite[Definition 3.7]{fujino-bpf-quasi-log}). 
In other words, $L-\omega$ is $\pi$-nef and $\pi$-big and the restriction of $L-\omega$ to any qlc center of $[X,\omega]$ is $\pi$-big. 
\end{itemize}
Then $R^{i}\pi_{*}(\mathcal{I}_{X'}\otimes_{\mathcal{O}_{X}} \mathcal{O}_{X}(L))=0$ for every $i>0$, where $\mathcal{I}_{X'}$ is the defining ideal sheaf of $X'$.
\end{thm}

\begin{proof}
We may assume that $Z$ is affine. 
Let $U \supset X'$ be an open subset of $X$ on which $L$ is Cartier. 
Let $g \colon W \to Y$ be a log resolution of $(Y,\Delta)$, and we define $\Gamma$ by 
$$K_{W}+\Gamma=g^{*}(K_{Y}+\Delta).$$
Then $g^{*}f^{*}L=\lfloor g^{*}f^{*}L \rfloor + \{g^{*}f^{*}L\}$ and $\{g^{*}f^{*}L\}$ is an effective $g$-exceptional $\mathbb{Q}$-divisor such that $\{g^{*}f^{*}L\}|_{(f \circ g)^{-1}(U)}=0$. 
By taking some blow-ups on $W$ if necessary, we may assume that the union of all strata of $(W,{\rm Supp}\,\Gamma^{\geq 1})$ mapping into $X'$ is ${\rm Supp}\,(\Gamma^{>1}+S)$  for some union $S$ of components of $\Gamma^{= 1}$. 
We put
$$T:=S+\lfloor \Gamma^{>1}\rfloor \qquad {\rm and}\qquad B:= \Gamma^{=1}+\{\Gamma\}-S.$$
Then $B$ is an snc boundary $\mathbb{R}$-divisor, and $B^{=1}$ and $T$ have no common components. 
Moreover, $\Gamma=T+B-\lceil-(\Gamma^{<1})\rceil$ and $\mathcal{I}_{X'}=(f\circ g)_{*}\mathcal{O}_{W}(-T+\lceil-(\Gamma^{<1})\rceil)$.  
We have
\begin{equation*}
\begin{split}
\lfloor g^{*}f^{*}L \rfloor=&g^{*}f^{*}(L-\omega)+g^{*}f^{*}\omega-\{g^{*}f^{*}L\}\\
\sim_{\mathbb{R}}&g^{*}f^{*}(L-\omega)+K_{W}+\Gamma-\{g^{*}f^{*}L\}\\
\sim_{\mathbb{R}}&g^{*}f^{*}(L-\omega)+K_{W}+T+B-\lceil-(\Gamma^{<1})\rceil-\{g^{*}f^{*}L\}.
\end{split}
\end{equation*}
We take a reduced divisor $E$ on $W$ such that if we define
$$B':=B-\{g^{*}f^{*}L\}+E,$$
then $B'$ is a boundary $\mathbb{R}$-divisor and $B'^{=1}\leq B^{=1}$. 
More precisely, for each prime divisor $P$ on $W$, the coefficient of $P$ in $E$ is defined by 
\begin{equation*}
{\rm coeff}_{P}(E)=\left\{  \begin{array}{l}{0 \qquad ({\rm coeff}_{P}(B-\{g^{*}f^{*}L\})\geq 0)} \\ {1 \qquad  ({\rm coeff}_{P}(B-\{g^{*}f^{*}L\})< 0)} \end{array}\right. \end{equation*}
We note that $E$ is $g$-exceptional and $E|_{(f \circ g)^{-1}(U)}=0$ because any component of $E$ is a component of $\{g^{*}f^{*}L\}$. 
By the above relation, we have
\begin{equation*}\tag{$\spadesuit$}\label{thm--vanishing-quasi-log-(spade)}
\lfloor g^{*}f^{*}L \rfloor+E+\lceil-(\Gamma^{<1})\rceil-T\sim_{\mathbb{R}}g^{*}f^{*}(L-\omega)+K_{W}+B'.
\end{equation*}

We put 
$$A:=\lfloor g^{*}f^{*}L \rfloor+E+\lceil-(\Gamma^{<1})\rceil\sim_{\mathbb{R}}g^{*}f^{*}(L-\omega)+K_{W}+T+B',$$ and we consider the exact sequence
$$0 \longrightarrow \mathcal{O}_{W}(A-T) \longrightarrow \mathcal{O}_{W}(A) \longrightarrow \mathcal{O}_{T}(A|_{T}) \longrightarrow 0.$$
Now $T$ is an snc Weil divisor, and $B^{=1}$ and $T$ have no common components. 
By this fact and the log smoothness of $(W,{\rm Supp}\,\Gamma)$, any stratum of $(W,B^{=1})$ is not contained in ${\rm Supp}\,T$. 
By these facts, we can check that the image of any lc center of $(W,B)$ by $f \circ g$ is not contained in $X'$. 
Indeed, if some lc center $C$ of $(W,B)$ is mapped into $X'$, then $C$ is a stratum of $(W,{\rm Supp}\,\Gamma^{\geq 1})$ by $B^{=1} \leq {\rm Supp}\,\Gamma^{\geq 1}$.
Then $C \subset {\rm Supp}\,B^{=1}\cap {\rm Supp}\,T$ since the union of all strata of $(W,{\rm Supp}\,\Gamma^{\geq 1})$ mapping into $X'$ is ${\rm Supp}\,T$, which follows from the definitions of $T$ and $S$. 
This contradicts the log smoothness of $(W,{\rm Supp}\,\Gamma)$ and the fact that $B^{=1}$ and $T$ have no common components. 
Thus, the image of any lc center of $(W,B)$ is not mapped into $X'$. 
Since $B'^{=1}\leq B^{=1}$, any lc center of $(W,B')$ is not mapped into $X'$. 
By the torsion-free theorem \cite[Theorem 5.6.2 (i)]{fujino-book}, the connecting morphism
$(f\circ g)_{*}\mathcal{O}_{T}(A|_{T}) \to R^{1}(f\circ g)_{*}\mathcal{O}_{W}(A-T)$ is the zero morphism. 
Therefore,
$$0 \longrightarrow (f\circ g)_{*}\mathcal{O}_{W}(A-T) \longrightarrow (f\circ g)_{*}\mathcal{O}_{W}(A) \longrightarrow (f\circ g)_{*}\mathcal{O}_{T}(A|_{T}) \longrightarrow 0$$
is exact. 

Now we have the exact sequence
$$0 \longrightarrow \mathcal{I}_{X'} \longrightarrow \mathcal{O}_{X} \longrightarrow \mathcal{O}_{X'} \longrightarrow 0.$$
We recall that $L|_{U}$ is Cartier and $U \supset X'$. 
From this, we have the exact sequence
$$0 \longrightarrow \mathcal{I}_{X'}\otimes_{\mathcal{O}_{X}} \mathcal{O}_{X}(L) \longrightarrow \mathcal{O}_{X}(L) \longrightarrow \mathcal{O}_{X'}(L|_{X'}) \longrightarrow 0.$$
We recall that $f^{*}L$ is a Weil divisor on $Y$, which is the hypothesis of Theorem \ref{thm--vanishing-quasi-log}, and $A=\lfloor g^{*}f^{*}L \rfloor+E+\lceil-(\Gamma^{<1})\rceil$ such that $E+\lceil-(\Gamma^{<1})\rceil$ is an effective $g$-exceptional divisor. 
Thus, we have 
$$(f\circ g)_{*}\mathcal{O}_{W}(A) \cong f_{*}\mathcal{O}_{Y}(f^{*}L)\cong \mathcal{O}_{X}(L)$$
via the natural morphism $\mathcal{O}_{X}(L) \to (f\circ g)_{*}\mathcal{O}_{W}(A)$ (cf.~\cite[II, 2.11. Lemma]{nakayama}). 

Recall that $\mathcal{I}_{X'}=(f\circ g)_{*}\mathcal{O}_{W}(-T+\lceil-(\Gamma^{<1})\rceil)$.  
Since $U \supset X'=(f \circ g)({\rm Supp}\,T)$, we have ${\rm Supp}\,T \subset (f \circ g)^{-1}(U)$. 
By recalling the fact $E|_{(f \circ g)^{-1}(U)}=0$, we have
\begin{equation*}
\begin{split}
(f\circ g)_{*}\mathcal{O}_{T}(A|_{T}) =&(f\circ g)_{*}\mathcal{O}_{T}\bigl((\lfloor g^{*}f^{*}L \rfloor+E+\lceil-(\Gamma^{<1})\rceil)|_{T}\bigr)\\
\cong&(f\circ g)_{*}\mathcal{O}_{T}\bigl(\lceil-(\Gamma^{<1})\rceil|_{T}\bigr)\otimes_{\mathcal{O}_{X}} \mathcal{O}_{X}(L).
\end{split}
\end{equation*}
Note that this isomorphism is induced by the restriction of $\mathcal{O}_{X}(L) \to (f\circ g)_{*}\mathcal{O}_{W}(A)$ to $X'$. 
By the definitions of $T$ and $B$, we also have
$$\lceil-(\Gamma^{<1})\rceil-T\sim_{\mathbb{R}}-g^{*} f^{*}\omega+K_{W}+\Gamma+\lceil-(\Gamma^{<1})\rceil-T\sim_{\mathbb{R},\,X}K_{W}+B.$$
In the second paragraph of this proof, we have checked that any lc center of $(W,B)$ is not mapped into $X'$. 
By the torsion-free theorem \cite[Theorem 5.6.2 (i)]{fujino-book} and applying $(f \circ g)_{*}$ to the exact sequence
$$0 \longrightarrow \mathcal{O}_{W}(\lceil-(\Gamma^{<1})\rceil-T) \longrightarrow \mathcal{O}_{W}(\lceil-(\Gamma^{<1})\rceil) \longrightarrow \mathcal{O}_{T}(\lceil-(\Gamma^{<1})\rceil|_{T}) \longrightarrow 0,$$
we have the exact sequence
\begin{equation*}
\begin{split}
0 \longrightarrow \mathcal{I}_{X'} \longrightarrow \mathcal{O}_{X} \longrightarrow (f \circ g)_{*}\mathcal{O}_{T}(\lceil-(\Gamma^{<1})\rceil|_{T}) \longrightarrow 0. 
\end{split}
\end{equation*}
This shows $(f \circ g)_{*}\mathcal{O}_{T}(\lceil-(\Gamma^{<1})\rceil|_{T})= \mathcal{O}_{X'}$. 
Therefore we have 
$$(f\circ g)_{*}\mathcal{O}_{T}(A|_{T}) \cong (f\circ g)_{*}\mathcal{O}_{T}\bigl(\lceil-(\Gamma^{<1})\rceil|_{T}\bigr)\otimes_{\mathcal{O}_{X}} \mathcal{O}_{X}(L) \cong \mathcal{O}_{X'}\otimes_{\mathcal{O}_{X}} \mathcal{O}_{X}(L) \cong \mathcal{O}_{X'}(L|_{X'}).$$
Here, the final isomorphism follows from the Cartier property of $L$ around $X'$. 

From the above argument, we obtain the diagram
 $$
\xymatrix@C=16pt{
0 \ar[r]& (f\circ g)_{*}\mathcal{O}_{W}(A-T) \ar[r]& (f\circ g)_{*}\mathcal{O}_{W}(A) \ar[r]& (f\circ g)_{*}\mathcal{O}_{T}(A|_{T}) \ar[r]& 0\\
0 \ar[r]& \mathcal{I}_{X'}\otimes_{\mathcal{O}_{X}} \mathcal{O}_{X}(L) \ar[r]& \mathcal{O}_{X}(L) \ar[r]\ar[u]^{\cong}& \mathcal{O}_{X'}(L|_{X'}) \ar[r]\ar[u]^{\cong}& 0.
}
$$
From this, we have $(f\circ g)_{*}\mathcal{O}_{W}(A-T) \cong \mathcal{I}_{X'}\otimes_{\mathcal{O}_{X}} \mathcal{O}_{X}(L)$. 
By (\ref{thm--vanishing-quasi-log-(spade)}) and the definition of $A$, we have $A-T \sim_{\mathbb{R}}g^{*}f^{*}(L-\omega)+K_{W}+B'.$
By construction of $B'$, any stratum of $B'^{=1}$ is a stratum of $\Gamma^{=1}$. 
Since $L-\omega$ is $\pi$-nef and $\pi$-log big with respect to $[X,\omega]$ (see \cite[Definition 3.7]{fujino-bpf-quasi-log}), which is the hypothesis of Theorem \ref{thm--vanishing-quasi-log}, the vanishing theorem for simple normal crossing varieties \cite[Theorem 5.7.3 (ii)]{fujino-book} implies 
$$R^{i}\pi_{*}(\mathcal{I}_{X'}\otimes_{\mathcal{O}_{X}} \mathcal{O}_{X}(L)) \cong R^{i}\pi_{*}\bigl((f\circ g)_{*}\mathcal{O}_{W}(A-T)\bigr)=0$$
for all $i>0$. 
This is what we wanted to prove. 
\end{proof}

\section{Running minimal model program}\label{sec--running-mmp}

In this section, we define a minimal model and a good minimal model for normal pairs, and we study how to construct a sequence of an MMP for normal pairs whose non-nef locus is disjoint from the non-lc locus. 
Corollary \ref{cor--mmpwithscaling-normalpair} and Corollary \ref{cor--mmp-nomralpair-Qfacdlt} are the main results of this section.

\subsection{Minimal model}

In this subsection, we define minimal models for normal pairs and prove some basic results. 

\begin{defn}[Minimal model]\label{defn--minmodel}
Let $X \to Z$ be a projective morphism from a normal quasi-projective variety $X$ to a quasi-projective scheme $Z$, and let $(X,\Delta)$ be a normal pair. 
Let $(X',\Delta')$ be a normal pair with a projective morphism $X' \to Z$, and let $\phi \colon X \dashrightarrow X'$ be a birational map over $Z$. 
We say that $(X',\Delta')$ is a {\em minimal model of $(X,\Delta)$ over $Z$} if 
\begin{itemize}
\item
for any prime divisor $P$ on $X$, we have 
$$a(P,X,\Delta) \leq a(P,X',\Delta'),$$
and the strict inequality holds if $P$ is $\phi$-exceptional,   
\item
for any prime divisor $P'$ on $X'$, we have 
$${\rm coeff}_{P'}(\Delta')=-a(P',X,\Delta),$$  
and the inequality $a(P',X,\Delta)\leq -1$ holds if $P'$ is $\phi^{-1}$-exceptional, and
\item
$K_{X'}+\Delta'$ is nef over $Z$. 
\end{itemize}
We say that a minimal model $(X',\Delta')$ of $(X,\Delta)$ over $Z$ is a {\em good minimal model} if $K_{X'}+\Delta'$ is semi-ample over $Z$. 
A {\em $\mathbb{Q}$-factorial minimal model} (resp.~{\em a $\mathbb{Q}$-factorial good minimal model}) of $(X,\Delta)$ over $Z$ is a  minimal model (resp.~a good minimal model) $(X',\Delta')$ of $(X,\Delta)$ over $Z$ such that $X'$ is $\mathbb{Q}$-factorial. 
\end{defn}

Even if $(X,\Delta)$ is lc, the above definition of minimal model is different from the definition of log minimal model in the sense of Birkar--Shokurov because $(X',\Delta')$ is not necessarily $\mathbb{Q}$-factorial dlt. 
However, if $(X,\Delta)$ is an lc pair then $(X',\Delta')$ is also an lc pair and any $\mathbb{Q}$-factorial dlt model $(X'',\Delta'')$ of $(X',\Delta')$ is a log minimal model of $(X,\Delta)$ over $Z$ in the sense of Birkar--Shokurov.

\begin{lem}\label{lem--min-model-nonlc-crepant}
Let $X \to Z$ be a projective morphism of normal quasi-projective varieties and $(X,\Delta)$ a normal pair. 
Let $(X',\Delta')$ be a normal pair with a projective birational morphism $f \colon X' \to X$ such that $K_{X'}+\Delta'=f^{*}(K_{X}+\Delta)$ and any $f$-exceptional prime divisor $E$ on $X'$ satisfies $a(E,X,\Delta) \leq -1$. 
If a normal pair $(X'',\Delta'')$ is a minimal model (resp.~a good minimal model) of $(X',\Delta')$ over $Z$, then $(X'',\Delta'')$ is a minimal model (resp.~a good minimal model) of $(X,\Delta)$ over $Z$. 
\end{lem}

\begin{proof}
We will prove the case of minimal model because the case of good minimal model can be proved similarly. 

Let $(X'',\Delta'')$ be a minimal model of $(X',\Delta')$ over $Z$. 
For any prime divisor $P$ on $X$, we have 
$$a(P,X,\Delta)=a(P,X', \Delta') \leq a(P,X'',\Delta'').$$
Moreover, if $P$ is exceptional over $X''$ then $f_{*}^{-1}P$ is exceptional over $X''$. 
Hence
$$a(P,X,\Delta)< a(P,X'',\Delta'').$$
Therefore, the first condition of a minimal model in Definition \ref{defn--minmodel} holds. 
For any prime divisor $Q''$ on $X''$, we have
$${\rm coeff}_{Q''}(\Delta'')=-a(Q'',X',\Delta')=-a(Q'',X,\Delta).$$
Suppose that $Q''$ is exceptional over $X$. 
If $Q''$ is not exceptional over $X'$, then the birational transform $Q'$ of $Q''$ on $X'$ is $f$-exceptional, and therefore 
$$a(Q'',X,\Delta)=a(Q',X,\Delta)\leq -1.$$ 
If $Q''$ is exceptional over $X'$, then 
$$a(Q'',X,\Delta)=a(Q'',X',\Delta')\leq -1.$$ 
In any case, we have $a(Q'',X,\Delta)\leq -1$. 
Thus the second condition of a minimal model in Definition \ref{defn--minmodel} holds. 
By definition, $K_{X''}+\Delta''$ is nef over $Z$, and thus the third condition of a minimal model in Definition \ref{defn--minmodel} holds. 
Therefore, $(X'',\Delta'')$ is a minimal model of $(X,\Delta)$ over $Z$. 
From this, we see that Lemma \ref{lem--min-model-nonlc-crepant} holds. 
\end{proof}

\begin{lem}\label{lem--good-min-model-makayama-zariski-decomp}
Let $X \to Z$ be a projective morphism of normal quasi-projective varieties and $(X,\Delta)$ a normal pair. 
Suppose that $(X,\Delta)$ has a good minimal model over $Z$. 
Then $K_{X}+\Delta$ birationally has the Nakayama--Zariski decomposition over $Z$ whose positive part is semi-ample over $Z$. 
In other words, there exists a projective birational morphism $f \colon X' \to X$ from a normal variety $X'$ such that $N_{\sigma}(f^{*}(K_{X}+\Delta); X'/Z)$ is well defined as an $\mathbb{R}$-divisor on $X'$ and $P_{\sigma}(f^{*}(K_{X}+\Delta); X'/Z)$ is semi-ample over $Z$. 
In particular, $N_{\sigma}(f^{*}(K_{X}+\Delta); X'/Z)$ and $P_{\sigma}(f^{*}(K_{X}+\Delta); X'/Z)$ are both $\mathbb{R}$-Cartier. 
\end{lem}

\begin{proof}
The proof is the same as the lc case. 
Let $(X'',\Delta'')$ be a good minimal model of $(X,\Delta)$ over $Z$. 
Take a common resolution $f \colon X' \to X$ and $g \colon X' \to X''$ of $X \dashrightarrow X''$. 
By Definition \ref{defn--minmodel}, we can write
$$f^{*}(K_{X}+\Delta)=g^{*}(K_{X''}+\Delta'')+E$$
for some effective $g$-exceptional $\mathbb{R}$-divisor on $X'$. 
By \cite[Lemma 3.4]{liuxie-relative-nakayama}, we have
$$\sigma_{P}(f^{*}(K_{X}+\Delta)/Z)=\sigma_{P}(g^{*}(K_{X''}+\Delta'')+E/Z)={\rm coeff}_{P}(E)$$
for any prime divisor $P$ on $X'$, where the final equality follows from the semi-ampleness of $K_{X''}+\Delta''$. 
By Definition \ref{defn--relative-nakayama-zariski-decom}, we have $N_{\sigma}(f^{*}(K_{X}+\Delta); X'/Z)=E$, and therefore we have $P_{\sigma}(f^{*}(K_{X}+\Delta); X'/Z)=g^{*}(K_{X''}+\Delta'')$, which is semi-ample over $Z$. 
\end{proof}

\begin{prop}\label{prop--minmodel-biratmodel}
Let $X \to Z$ be a projective morphism of normal quasi-projective varieties and let $(X,\Delta)$ be a normal pair such that $|K_{X}+\Delta/Z|_{\mathbb{R}} \neq \emptyset$.  
Let $f \colon Y \to X$ be a projective birational morphism from a normal variety $Y$ such that any $f$-exceptional prime divisor $P$ on $Y$ satisfies $a(P,X,\Delta) \leq -1$ or $P \subset {\rm Bs}|f^{*}(K_{X}+\Delta)/Z|_{\mathbb{R}}$. 
We write
$$K_{Y}+\Gamma=f^{*}(K_{X}+\Delta)+E$$
for some effective $\mathbb{R}$-divisors $\Gamma$ and $E$ on $Y$ having no common components. 
Let $B$ be an effective $\mathbb{R}$-divisor on $Y$ such that ${\rm Supp}\,B \subset {\rm Bs}|f^{*}(K_{X}+\Delta)/Z|_{\mathbb{R}}$. 
Suppose that $(Y,\Gamma+B)$ is a normal pair. 
In other words, suppose that $K_{Y}+\Gamma+B$ is $\mathbb{R}$-Cartier. 
Let $\phi \colon Y \dashrightarrow Y'$ be a birational contraction over $Z$, where $Y'$ is a normal variety and projective over $Z$, such that $\phi_{*}(K_{Y}+\Gamma+B)$ is semi-ample over $Z$ and $\phi$ only contracts some prime divisors contained in ${\rm Bs}|K_{Y}+\Gamma+B/Z|_{\mathbb{R}}$. 
Then $\phi$ exactly contracts all prime divisors contained in ${\rm Bs}|f^{*}(K_{X}+\Delta)/Z|_{\mathbb{R}}$, and $(Y',\phi_{*}(\Gamma+B))$ is a good minimal model of $(X,\Delta)$ over $Z$.  
\end{prop}

\begin{proof}
We have
$${\rm Bs}|K_{Y}+\Gamma+B/Z|_{\mathbb{R}} \subset {\rm Bs}|K_{Y}+\Gamma-E/Z|_{\mathbb{R}}\cup{\rm Supp}\,(B+E).$$
Since $K_{Y}+\Gamma=f^{*}(K_{X}+\Delta)+E$ and $\Gamma \geq 0$ and $E \geq 0$ have no common components, any component $\tilde{E}$ of $E$ satisfies $a(\tilde{E},X,\Delta)>0$. 
Then ${\rm Supp}\,E \subset  {\rm Bs}|f^{*}(K_{X}+\Delta)/Z|_{\mathbb{R}}$ by the assumption of $f \colon Y \to X$. 
Since ${\rm Supp}\,B \subset {\rm Bs}|f^{*}(K_{X}+\Delta)/Z|_{\mathbb{R}}$, we obtain
\begin{equation*}\tag{$*$}\label{prop--minmodel-biratmodel-(*)}
\begin{split}
{\rm Supp}\,(B+E) \subset  {\rm Bs}|f^{*}(K_{X}+\Delta)/Z|_{\mathbb{R}}. 
\end{split}
\end{equation*}
From this fact and the relation $K_{Y}+\Gamma-E=f^{*}(K_{X}+\Delta)$, we have
$${\rm Bs}|K_{Y}+\Gamma+B/Z|_{\mathbb{R}} \subset {\rm Bs}|K_{Y}+\Gamma-E/Z|_{\mathbb{R}}\cup{\rm Supp}\,(B+E) \subset {\rm Bs}|f^{*}(K_{X}+\Delta)/Z|_{\mathbb{R}}.$$
Therefore, any prime divisor contracted by $\phi$ is contained in ${\rm Bs}|f^{*}(K_{X}+\Delta)/Z|_{\mathbb{R}}$. 
We put $\Gamma'=\phi_{*}\Gamma$ and $B'=\phi_{*}B$.
Let $g \colon W \to Y$ and $g' \colon W \to Y'$ be a common resolution of $\phi$. 
We may write
\begin{equation*}\tag{$**$}\label{prop--minmodel-biratmodel-(**)}
\begin{split}
g^{*}f^{*}(K_{X}+\Delta)+F=g'^{*}(K_{Y'}+\Gamma'+B')+F'
\end{split}
\end{equation*}
for some effective $\mathbb{R}$-divisors $F$ and $F'$ on $W$ that have no common components. 
Then $F'$ is $g'$-exceptional since $g'_{*}(F-F')=B'+\phi_{*}E$ and $g'_{*}F$ and $g'_{*}F'$ have no common components. 

In this paragraph, we will prove that $F=0$. 
We assume that $F$ is $(f \circ g)$-exceptional. 
Then the semi-ampleness of $K_{Y'}+\Gamma'+B'$ over $Z$ implies
\begin{equation*}
\begin{split}
{\rm Supp}\,F'={\rm Bs}|g'^{*}(K_{Y'}+\Gamma'+B')+F'/Z|_{\mathbb{R}} =&{\rm Bs}|g^{*}f^{*}(K_{X}+\Delta)+F/Z|_{\mathbb{R}} \\
=& {\rm Bs}|g^{*}f^{*}(K_{X}+\Delta)/Z|_{\mathbb{R}}\cup{\rm Supp}\,F.
\end{split}
\end{equation*}
From this relation and the fact that $F$ and $F'$ have no common components, we have $F=0$.  
Therefore, it is sufficient to prove that $F$ is $(f \circ g)$-exceptional. 
By definition, we have $g'_{*}F=B'+\phi_{*}E$. 
Hence every component of $g_{*}F$ is $\phi$-exceptional or a component of $B+E$. 
By the inclusion (\ref{prop--minmodel-biratmodel-(*)}) and the fact that $\phi$ only contracts divisors contained in ${\rm Bs}|f^{*}(K_{X}+\Delta)/Z|_{\mathbb{R}}$, we have ${\rm Supp}\,g_{*}F\subset {\rm Bs}|f^{*}(K_{X}+\Delta)/Z|_{\mathbb{R}}$. 
Thus 
$${\rm Supp}\,(f \circ g)_{*}F\subset {\rm Bs}|K_{X}+\Delta/Z|_{\mathbb{R}}.$$
From now on, we suppose that $(f \circ g)_{*}F \neq 0$ and we will get a contradiction. 
By (\ref{prop--minmodel-biratmodel-(**)}) and the semi-ampleness of $K_{Y'}+\Gamma'+B'$ over $Z$, any component of $(f \circ g)_{*}F$ is not an irreducible component of ${\rm Bs}|K_{X}+\Delta+(f \circ g)_{*}F/Z|_{\mathbb{R}}$. 
Note that $K_{X}+\Delta+(f \circ g)_{*}F$ may not be $\mathbb{R}$-Cartier.  
We pick $D \in {\rm Bs}|K_{X}+\Delta/Z|_{\mathbb{R}}$, and let $D=\sum_{i=1}^{l}a_{i}D_{i}$ be the prime decomposition, where $a_{1},\cdots,\, a_{l}$ are positive real numbers. 
Since every component of ${\rm Supp}\,(f \circ g)_{*}F$ is contained in ${\rm Bs}|K_{X}+\Delta/Z|_{\mathbb{R}}$, we can write $(f \circ g)_{*}F=\sum_{i=1}^{l}b_{i}D_{i}$ for some nonnegative real numbers $b_{1},\cdots,\, b_{l}$. 
Relabeling the indices, we may assume $\frac{b_{1}}{a_{1}}\geq \frac{b_{i}}{a_{i}}$ for all $2 \leq i \leq l$. 
Then $b_{1}>0$ because otherwise the relation $\frac{b_{1}}{a_{1}}\geq \frac{b_{i}}{a_{i}}$ implies $b_{i}=0$ for all $2 \leq i \leq l$, which shows $(f \circ g)_{*}F=0$. 
Moreover
\begin{equation*}
\begin{split}
\left(1+\frac{b_{1}}{a_{1}}\right)D=\left(1+\frac{b_{1}}{a_{1}}\right)\sum_{i=1}^{l}a_{i}D_{i}=&D+b_{1}D_{1}+\sum_{i=2}^{l}\frac{a_{i}b_{1}}{a_{1}}D_{i}\\
=&D+(f \circ g)_{*}F+\sum_{i=2}^{l}a_{i}\left(\frac{b_{1}}{a_{1}}-\frac{b_{i}}{a_{i}}\right)D_{i}. 
\end{split}
\end{equation*}
Using $D \sim_{\mathbb{R},\,Z}K_{X}+\Delta$ and ${\rm Bs}|K_{X}+\Delta/Z|_{\mathbb{R}}={\rm Bs}|(1+\frac{b_{1}}{a_{1}})(K_{X}+\Delta)/Z|_{\mathbb{R}}$, we obtain
$${\rm Supp}\,(f \circ g)_{*}F\subset {\rm Bs}|K_{X}+\Delta/Z|_{\mathbb{R}}\subset {\rm Bs}|K_{X}+\Delta+(f \circ g)_{*}F/Z|_{\mathbb{R}} \cup \bigcup_{i=2}^{l}{\rm Supp}\,D_{i}.$$
Recalling the fact that any component of $(f \circ g)_{*}F$ is not an irreducible component of ${\rm Bs}|K_{X}+\Delta+(f \circ g)_{*}F/Z|_{\mathbb{R}}$, we see that the right hand side does not contain ${\rm Supp}\,D_{1}$ as an irreducible component. 
Then we get a contradiction because $(f \circ g)_{*}F=\sum_{i=1}^{l}b_{i}D_{i}$ and $b_{1}>0$. 
Therefore we have $(f \circ g)_{*}F = 0$, and therefore $F$ is $(f \circ g)$-exceptional. 
Then $F=0$ as discussed above. 

By (\ref{prop--minmodel-biratmodel-(**)}) and the argument in the previous paragraph, we may write 
\begin{equation*}\tag{$*\!*\!*$}\label{prop--minmodel-biratmodel-(***)}
\begin{split}
g^{*}f^{*}(K_{X}+\Delta)=g'^{*}(K_{Y'}+\Gamma'+B')+F'
\end{split}
\end{equation*}
such that $F'$ is effective and $g'$-exceptional. 
By the semi-ampleness of $K_{Y'}+\Gamma'+B'$ over $Z$, the equality
${\rm Supp}\,F'={\rm Bs}|g^{*}f^{*}(K_{X}+\Delta)/Z|_{\mathbb{R}}$ holds. 
This shows that $\phi \colon Y \dashrightarrow Y'$ contracts all divisorial components of ${\rm Bs}|f^{*}(K_{X}+\Delta)/Z|_{\mathbb{R}}$. 
Since $\phi$ only contracts some divisors contained in ${\rm Bs}|f^{*}(K_{X}+\Delta)/Z|_{\mathbb{R}}$, it follows that $\phi$ exactly contracts divisors contained in ${\rm Bs}|f^{*}(K_{X}+\Delta)/Z|_{\mathbb{R}}$. 
This is the first assertion of Proposition \ref{prop--minmodel-biratmodel}. 
By this fact and (\ref{prop--minmodel-biratmodel-(*)}), we also see that $B+E$ is contracted by $\phi$. 

Finally, we check that $(Y',\Gamma')$ is a good minimal model of $(X,\Delta)$ over $Z$. 
Let $Q$ be a prime divisor on $X$. 
By (\ref{prop--minmodel-biratmodel-(***)}), we have 
$$a(Q,X,\Delta) \leq a(Q,Y',\Gamma').$$ 
If $Q$ is exceptional over $Y'$, then $f^{-1}_{*}Q$ is contracted by $\phi$. 
Therefore, $f^{-1}_{*}Q$ is contained in ${\rm Bs}|f^{*}(K_{X}+\Delta)/Z|_{\mathbb{R}}$ by the first assertion of Proposition \ref{prop--minmodel-biratmodel}. 
Then $(f \circ g)^{-1}_{*}Q$ is a component of $F'$, and therefore 
$$a(Q,X,\Delta)<a(Q,Y',\Gamma').$$ 
This implies the first condition of the good minimal model in Definition \ref{defn--minmodel}. 
We recall the hypothesis of Proposition \ref{prop--minmodel-biratmodel} that $K_{Y}+\Gamma=f^{*}(K_{X}+\Delta)+E$ and any $f$-exceptional prime divisor $P$ on $Y$ satisfies $a(P,X,\Delta) \leq -1$ or $P \subset {\rm Bs}|f^{*}(K_{X}+\Delta)/Z|_{\mathbb{R}}$. 
Since $\phi$ exactly contracts divisors contained in ${\rm Bs}|f^{*}(K_{X}+\Delta)/Z|_{\mathbb{R}}$, for any prime divisor $Q'$ on $Y'$, we have
$$a(Q',Y',\Gamma')=-{\rm coeff}_{Q'}(\Gamma')=-{\rm coeff}_{\phi^{-1}_{*}Q'}(\Gamma-E)=a(Q',X,\Delta)$$
and $a(Q',Y',\Gamma') \leq -1$ if $Q'$ is exceptional over $X$. 
This is the second condition of the good minimal model in Definition \ref{defn--minmodel}. 
The third condition of the good minimal model is clear. 
From these facts, $(Y',\Gamma')$ is a good minimal model of $(X,\Delta)$ over $Z$. 
\end{proof}

\subsection{Minimal model program}

In this subsection, we define a step of a minimal model program and we discuss construction of a minimal model program for normal pairs.

\begin{defn}[Minimal model program]\label{defn--mmp-fullgeneral}
Let $\pi \colon X \to Z$ be a projective morphism from a normal quasi-projective variety $X$ to a quasi-projective scheme $Z$, and let $D$ be an $\mathbb{R}$-Cartier $\mathbb{R}$-divisor on $X$. 

A {\em step of a $D$-MMP over $Z$} is a diagram
 $$
\xymatrix@R=16pt{
X\ar@{-->}[rr]^-{\phi}\ar[dr]\ar[ddr]_-{\pi}&&X'\ar[dl]\ar[ddl]^-{\pi'}\\
&V\ar[d]\\
&Z
}
$$
consisting of normal quasi-projective varieties $X$, $X'$, and $V$, which are projective over $Z$, such that
\begin{itemize}
\item
$X \to V$ is a birational morphism and $X' \to V$ is a small birational morphism, 
\item
$-D$ is ample over $V$, and 
\item
$\phi_{*}D$ is $\mathbb{R}$-Cartier and ample over $V$. 
\end{itemize}
Sometimes we call $X \to V$ a {\em $D$-negative extremal contraction}.

A {\em sequence of steps of a $D$-MMP over $Z$} is a sequence of birational contractions 
$$X=:X_{1}\dashrightarrow X_{2} \dashrightarrow \cdots \dashrightarrow X_{i} \dashrightarrow \cdots$$
such that each birational contraction $X_{i}\dashrightarrow X_{i+1}$ forms a step of a $D_{i}$-MMP over $Z$, where $D_{i}$ is the birational transform of $D$ on $X_{i}$. 
When $D$ is of the form $K_{X}+B$ for some normal pair $(X,B)$, then a sequence of steps of a $(K_{X}+B)$-MMP over $Z$ is often denoted by
$$(X,B)=:(X_{1},B_{1})\dashrightarrow (X_{2},B_{2}) \dashrightarrow \cdots \dashrightarrow (X_{i},B_{i}) \dashrightarrow \cdots.$$
With notation as above, the {\em non-isomorphic locus} of a $D$-MMP is the union of points $x \in X$ such that $X \dashrightarrow X_{i}$ is not an isomorphism at $x$ for some $i$. 
The non-isomorphic locus is a countable union of closed subsets of $X$, and furthermore, in the case of finitely many steps of a $D$-MMP, the non-isomorphic locus is a closed subset of $X$. 

Let $A$ be an $\mathbb{R}$-Cartier $\mathbb{R}$-divisor on $X$ such that $D+\lambda A$ is nef over $Z$ for some $\lambda \in \mathbb{R}_{\geq 0}$. 
We say that a sequence of steps of a $D$-MMP over $Z$ 
$$X=:X_{1}\dashrightarrow X_{2} \dashrightarrow \cdots \dashrightarrow X_{i} \dashrightarrow \cdots$$
is a {\em sequence of steps of a $D$-MMP over $Z$ with scaling of $A$} if we put $D_{i}$ (resp.~$A_{i}$) as the birational transform of $D$ and (resp.~$A$) on $X_{i}$, then the following conditions hold.
\begin{itemize}
\item
$A_{i}$ is $\mathbb{R}$-Cartier, 
\item
the nonnegative real number
$$\lambda_{i}:={\rm inf}\{\mu \in \mathbb{R}_{\geq 0}\,|\, \text{$D_{i}+\mu A_{i}$ is nef over $Z$}\}$$
is well defined, and 
\item
$(D_{i}+\lambda_{i}A_{i})\cdot C_{i}=0$ for any curve $C_{i} \subset X_{i}$ that is contracted by the $D_{i}$-negative extremal contraction of the MMP. 
\end{itemize}
\end{defn}

\begin{rem}\label{rem--mmp-fullgeneral}
With notation as in Definition \ref{defn--mmp-fullgeneral}, the morphism $X \to V$ in a step of a $D$-MMP over $Z$ does not necessarily satisfy $\rho(X/V)=1$. 
This is the difference between a step of an MMP in Definition \ref{defn--mmp-fullgeneral} and a usual step of an MMP as in \cite[4.9.1]{fujino-book}. 
In particular, $X \to V$ and $X' \to V$ in Definition \ref{defn--mmp-fullgeneral} can be isomorphisms. 
We adopt this definition for the convenience of proofs of results in this paper. 
\end{rem}

\begin{rem}\label{rem--mmp-basic}
Let 
$$X=:X_{1}\dashrightarrow X_{2} \dashrightarrow \cdots \dashrightarrow X_{i} \dashrightarrow \cdots$$
be a sequence of steps of a $D$-MMP over $Z$. 
Then the following statements hold. 
\begin{itemize}
\item
For any $\mathbb{R}$-Cartier $\mathbb{R}$-divisor $D'$ on $X$ such that $D' \sim_{\mathbb{R},\,Z}uD$ for some $u \in \mathbb{R}_{>0}$, the $D$-MMP is also a sequence of steps of a $D'$-MMP over $Z$. 

\item
Suppose that the $D$-MMP is a sequence of steps of a $D$-MMP over $Z$ with scaling of an $\mathbb{R}$-Cartier $\mathbb{R}$-divisor $A$. 
We set
$$\lambda_{i}:={\rm inf}\{\mu \in \mathbb{R}_{\geq 0}\,|\, \text{$D_{i}+\mu A_{i}$ is nef over $Z$}\}$$
for each $i \geq 1$. 
By taking a common resolution of $X_{i} \dashrightarrow X_{i+1}$ and using the negativity lemma, we see that $D_{i+1}+\lambda_{i} A_{i+1}$ is nef over $Z$. 
In particular, we have $\lambda_{i} \geq \lambda_{i+1}$. 

\item
Suppose that the $D$-MMP is a sequence of steps of a $D$-MMP over $Z$ with scaling of an $\mathbb{R}$-Cartier $\mathbb{R}$-divisor $A$. 
Then, for any $t \in \mathbb{R}_{\geq 0}$, the $D$-MMP is a sequence of steps of a $(D-tA)$-MMP over $Z$ with scaling of $A$. 

\item
Suppose that the $D$-MMP is a sequence of steps of a $D$-MMP over $Z$ with scaling of an $\mathbb{R}$-Cartier $\mathbb{R}$-divisor $A$. 
Let $\lambda$ be a nonnegative real number such that $D+\lambda A$ is nef over $Z$. 
Let $\lambda'$ be a positive real number such that $\lambda' > \lambda$.
Then the equality 
$D+ \mu A=\frac{\lambda'-\mu}{\lambda'}(D+\frac{\mu}{\lambda'-\mu}(D+\lambda' A))$ 
holds for all $\mu < \lambda'$, and therefore the $D$-MMP is also a $D$-MMP over $Z$ with scaling of $D+\lambda'A$. 
\item
By the same argument as in \cite[Proof of Lemma 3.38]{kollar-mori}, we can check that for any $i \geq 1$, the birational map $X \dashrightarrow X_{i}$ is an isomorphism on an open subset $U_{i} \subset X$ whose complement is contained in ${\rm NNef}(D/Z)$.  
\end{itemize}
\end{rem}

\begin{lem}\label{lem--extremal-ray} 
Let $[X,\omega]$ be a quasi-log scheme such that $X$ is a normal variety. 
Let $\pi \colon X \to Z$ be a projective morphism to a quasi-projective scheme $Z$. 
Suppose that $\omega$ is $\pi$-pseudo-effective and ${\rm NNef}(\omega/Z) \cap {\rm Nqlc}(X, \omega)=\emptyset$. 
Let $A$ be an effective $\mathbb{R}$-Cartier $\mathbb{R}$-divisor on $X$ such that $\omega+A$ is $\pi$-nef and we have ${\rm Nqlc}(X,\omega+A) = {\rm Nqlc}(X, \omega)$ set theoretically. 
We put 
$$\lambda:={\rm inf}\{t\in \mathbb{R}_{\geq0}\,|\, \omega+tA \text{ is nef over $Z$}\,\}.$$
Then $\lambda=0$ or there exists an $\omega$-negative extremal ray $R$ of $\overline{\rm NE}(X/Z)$ such that $R$ is rational and relatively ample at infinity (\cite[Definition 6.7.2]{fujino-book}) and $(\omega+\lambda A)\cdot R=0$. 
\end{lem}

\begin{proof}
The proof is very similar to the argument in the lc case.
We may assume $\lambda>0$ because otherwise there is nothing to prove. 
By the cone theorem \cite[Theorem 6.7.4]{fujino-book}, we may write
$$\overline{\rm NE}(X/Z)=\overline{\rm NE}(X/Z)_{\omega \geq 0}+\overline{\rm NE}(X/Z)_{-\infty}+\sum_{j}R_{j}$$
for $\omega$-negative extremal rays $R_{j}$ of $\overline{\rm NE}(X/Z)$ that are rational and relatively ample at infinity. 
Then $\overline{\rm NE}(X/Z)_{-\infty} \subset \overline{\rm NE}(X/Z)_{\omega \geq 0}$ since ${\rm NNef}(\omega/Z) \cap {\rm Nqlc}(X, \omega)=\emptyset$.  
From this, we see that any $\omega$-negative extremal ray of $\overline{\rm NE}(X/Z)$ is rational and relatively ample at infinity. 
Since $\lambda>0$, there exists at least one $\omega$-negative extremal ray $R_{j}$ of $\overline{\rm NE}(X/Z)$ that is rational and relatively ample at infinity. 

We fix a $\pi$-ample Cartier divisor $L$ on $X$. 
For each index $j$, let $C_{j}$ be a curve on $X$ such that the numerical class of $C_{j}$ lies in $R_{j}$ and 
$$(L \cdot C_{j})={\rm min}\{(L \cdot C)\,|\, \text{the numerical class of $C$ lies in $R_{j}$}\}.$$
We call $C_{j}$ a minimal curve of $R_{j}$. 
Such $C_{j}$ exists and 
$0<-(\omega\cdot C_{j})\leq 2\cdot {\rm dim}\,X$ 
because \cite[Theorem 1.6 (iii)]{fujino-morihyper} shows the existence of a rational curve $C'_{j}$ spanning $R_{j}$ such that $0<-(\omega\cdot C'_{j})\leq 2\cdot {\rm dim}\,X$. 

By applying \cite[Lemma 4.25]{fujino-morihyper} to $[X,\omega+\lambda A]$, we can find positive real numbers $r_{1},\,\cdots, \,r_{m}$ and $\mathbb{Q}$-Cartier $\mathbb{Q}$-divisors $\Theta_{1},\,\cdots,\,\Theta_{m}$ on $X$ such that
\begin{itemize}
\item
$\sum_{i=1}^{m}r_{i}=1$ and $\sum_{i=1}^{m}r_{i}\Theta_{i}=\omega+\lambda A$, and
\item
for all $1\leq i \leq m$, the structure of a quasi-log scheme $[X,\Theta_{i}]$ has the same non-qlc locus as ${\rm Nqlc}(X, \omega+\lambda A)$. 
\end{itemize} 
By the conditions $\lambda \leq 1$ and ${\rm Nqlc}(X,\omega+A)={\rm Nqlc}(X,\omega)$ set theoretically, for any indices $i$ and $j$, the second condition implies that $R_{j}$ is rational and relatively ample at infinity with respect to $[X,\Theta_{i}]$. 
By \cite[Theorem 1.6 (iii)]{fujino-morihyper}, if $R_{j}$ is $\Theta_{i}$-negative then there is a rational curve $C_{j}^{(i)}$ spanning $R_{j}$ such that $0<-(\Theta_{i}\cdot C_{j}^{(i)})\leq 2 \cdot {\rm dim}\,X$. 
By the definition of the minimal curve $C_{j}$ of $R_{j}$, we have $C_{j}^{(i)}= \alpha \cdot C_{j}$ in $N_{1}(X/Z)$ for some real number $\alpha \geq 1$. 
Hence, 
$-(\Theta_{i}\cdot C_{j})\leq 2 \cdot {\rm dim}\,X.$ 
If $R_{j}$ is not $\Theta_{i}$-negative, then clearly we have $(\Theta_{i}\cdot C_{j})\geq -2 \cdot {\rm dim}\,X$. 
From this, we have 
$(\Theta_{i}\cdot C_{j})\geq -2 \cdot {\rm dim}\,X$
for any $i$ and $j$. 

Let $V \subset {\rm WDiv}_{\mathbb{R}}(X)$ be a rational polytope spanned by $\Theta_{1},\,\cdots,\,\Theta_{m}$. 
We will prove that the set
$$V':=\{\Theta' \in V\,|\, \text{$(\Theta'\cdot C_{j})\geq 0$ for all $j$}\}$$
is a rational polytope.
Fix $p \in \mathbb{Z}_{>0}$ such that all $p\Theta_{i}$ are Cartier. 
For each $j$, we set
$$\mathcal{H}_{j}:=\left\{(t_{1},\cdots,\,t_{m})\in (\mathbb{R}_{\geq 0})^{m}\,\middle|\,\sum_{i=1}^{m}t_{i}=1, \, \sum_{i=1}^{m}t_{i} (p\Theta_{i}\cdot C_{j})\geq 0\right\}.$$
Then it is sufficient to prove that $\bigcap_{j}\mathcal{H}_{j}$ is a rational polytope because the linear map 
$$\mathbb{R}^{m}\ni (t_{1},\cdots,\,t_{m}) \mapsto \sum_{i=1}^{m}t_{i}\Theta_{i} \in {\rm WDiv}_{\mathbb{R}}(X)$$
induces a surjective map from $\bigcap_{j}\mathcal{H}_{j}$ to $V'$. 
We will prove that $\bigcap_{j}\mathcal{H}_{j}$ is the intersection of finitely many $\mathcal{H}_{j}$. 
Suppose by contradiction that $\bigcap_{j}\mathcal{H}_{j}$ cannot be represented by the intersection of any finitely many $\mathcal{H}_{j}$. 
By taking a subset of $\{\mathcal{H}_{j}\}_{j}$, we get an infinite sequence $\{\mathcal{H}_{k}\}_{k \in \mathbb{Z}_{>0}}$ such that $\bigcap_{k=1}^{n}\mathcal{H}_{k} \supsetneq \bigcap_{k=1}^{n+1}\mathcal{H}_{k}$ for all $n$. 
For every $1 \leq i \leq m$, by the facts $(\Theta_{i}\,\cdot\, C_{k})\geq -2 \cdot {\rm dim}\,X$ and $(p\Theta_{i}\,\cdot\, C_{k}) \in \mathbb{Z}$,
the set $\{ (p\Theta_{i}\cdot C_{k})\, |\, k \in \mathbb{Z}_{>0} \}$ satisfies the descending chain condition. 
Replacing $\{\mathcal{H}_{k}\}_{k \in \mathbb{Z}_{>0}}$ by a subsequence, we may assume $(p\Theta_{1}\cdot C_{k}) \leq (p\Theta_{1}\cdot C_{k+1})$ for all $k$. 
The property $\bigcap_{k=1}^{n}\mathcal{H}_{k} \supsetneq \bigcap_{k=1}^{n+1}\mathcal{H}_{k}$ is preserved after this replacement. 
Replacing $\{\mathcal{H}_{k}\}_{k \in \mathbb{Z}_{>0}}$ by a subsequence again, we may assume $(p\Theta_{2}\cdot C_{k}) \leq (p\Theta_{2}\cdot C_{k+1})$ for all $k$. 
Repeating this discussion and replacing $\{\mathcal{H}_{k}\}_{k \in \mathbb{Z}_{>0}}$ with a subsequence, we may assume that
$(p\Theta_{i}\cdot C_{k}) \leq (p\Theta_{i}\cdot C_{k+1})$
for every $i$ and $k$. 
Then $\mathcal{H}_{1} \subset \mathcal{H}_{2}$ by the definition of $\mathcal{H}_{k}$, which is a contradiction. 
From this argument, $\bigcap_{j}\mathcal{H}_{j}$ is the intersection of finitely many $\mathcal{H}_{j}$. 
Thus $\bigcap_{j}\mathcal{H}_{j}$ is a rational polytope, and so is $V'$. 

By the above argument, we can find positive real numbers $r'_{1},\,\cdots, \,r'_{l}$ and $\mathbb{Q}$-Cartier $\mathbb{Q}$-divisors $\Theta'_{1},\,\cdots,\,\Theta'_{l}$ on $X$ such that 
\begin{itemize}
\item
$\sum_{i=1}^{l}r'_{i}=1$ and $\sum_{i=1}^{l}r'_{i}\Theta'_{i}=\omega+\lambda A$, and 
\item
$(\Theta'_{i}\cdot C_{j})\geq 0$ for any $1 \leq i \leq l$ and $j$.
\end{itemize}
For any $\lambda'<\lambda$, there is an index $j$ such that $(\omega+\lambda'A) \cdot C_{j}<0$. 
We also recall that $0<-(\omega\cdot C_{j})\leq 2\cdot {\rm dim}\,X$. 
From these facts, we have 
\begin{equation*}
\begin{split}
(\omega+\lambda A)\cdot C_{j}=&(\omega+\lambda' A)\cdot C_{j}+(\lambda-\lambda')(A\cdot C_{j})<\frac{(\lambda-\lambda')}{\lambda'}(\lambda' A\cdot C_{j})\\
=& \frac{(\lambda-\lambda')}{\lambda'}\bigl((\omega+\lambda' A)\cdot C_{j}-(\omega\cdot C_{j})\bigr)<2\frac{(\lambda-\lambda')}{\lambda'}\cdot {\rm dim}\,X
\end{split}
\end{equation*}
for any $\lambda'$ and $C_{j}$ as above. 
On the other hand, putting $p' \in \mathbb{Z}_{>0}$ so that all $p'\Theta'_{i}$ are Cartier, then the two conditions stated at the start of this paragraph imply that
$$(\omega+\lambda A)\cdot C_{j}=\frac{1}{p'}\sum_{i=1}^{l}r'_{i}(p'\Theta'_{i}\cdot C_{j})$$
is zero or not less than $\frac{1}{p'}{\rm min}\{r'_{i}\,|\,1\leq i \leq l\}$. 
By choosing $\lambda'< \lambda$ sufficiently close to $\lambda$, we obtain an index $j$ such that $(\omega+\lambda A)\cdot C_{j}=0$. 
Then the corresponding $\omega$-negative extremal ray $R:=R_{j}$ of $\overline{\rm NE}(X/Z)$ satisfies the condition of Lemma \ref{lem--extremal-ray}. 
\end{proof}

\begin{thm}\label{thm--mmpstep-quasi-log}
Let $f\colon (Y,\Delta) \to [X,\omega]$ be a quasi-log scheme induced by a normal pair. 
Let $\pi \colon X \to Z$ be a projective morphism to a quasi-projective scheme $Z$. 
Let $\varphi\colon X \to V$ be a birational morphism over $Z$, where $V$ is normal and projective over $Z$, such that $-\omega$ is $\varphi$-ample and $\varphi$ is an isomorphism on a neighborhood of ${\rm Nqlc}(X, \omega)$. 
Then we can construct a diagram
$$
\xymatrix@R=16pt{
(Y,\Delta)\ar[d]_{f}\ar@{-->}[rr]&& (Y',\Delta')\ar[d]^{f'} \\
[X,\omega] \ar@{-->}[rr]\ar[dr]_{\varphi}&& [X',\omega'] \ar[dl]^{\varphi'}\\
&V
}
$$
over $Z$ such that
\begin{itemize}
\item
$f'\colon (Y',\Delta') \to [X',\omega']$ is a quasi-log scheme induced by a normal pair such that $Y'$ and $X'$ are projective over $V$, 
\item
$(Y,\Delta)\dashrightarrow (Y',\Delta')$ is a sequence of steps of a $(K_{Y}+\Delta)$-MMP over $V$, and 
\item
$\varphi'\colon X' \to V$ is a projective small birational morphism and the $\mathbb{R}$-divisor $\omega'$ is the birational transform of $\omega$ on $X'$ and $\varphi'$-ample. 
\end{itemize}
Furthermore, if $Y$ is $\mathbb{Q}$-factorial, then $Y'$ is also $\mathbb{Q}$-factorial. 
\end{thm}

\begin{proof}
Let $U_{1} \subset V$ be the largest open subset over which $\varphi$ is an isomorphism. 
Then $U_{1} \supset \varphi({\rm Nqlc}(X, \omega))$ by our assumption. 
We put $U_{2}=V\setminus \varphi({\rm Nqlc}(X, \omega))$. 
Note that $U_{2}$ is open and $V=U_{1} \cup U_{2}$.
We put
\begin{equation*}
\begin{split}
Y_{1}:=(\varphi \circ f)^{-1}(U_{1}),\quad Y_{2}:=(\varphi \circ f)^{-1}(U_{2}), \quad \Delta_{1}:=\Delta|_{Y_{1}},\quad {\rm and} \quad \Delta_{2}:=\Delta|_{Y_{2}}.
\end{split}
\end{equation*}
Then $(Y_{2},\Delta_{2})$ is lc and $(K_{Y}+\Delta)|_{Y_{1}\cap Y_{2}}\sim_{\mathbb{R},\,U_{1}\cap U_{2}}0$. 
By our assumption on $\varphi\colon X \to V$, there is a $\pi$-ample $\mathbb{R}$-divisor $H$ on $X$ such that $\omega +H \sim_{\mathbb{R},\,V}0$. 
Put $H_{2}:=f^{*}H|_{Y_{2}}$. 
Then $K_{Y_{2}}+\Delta_{2}+H_{2}\sim_{\mathbb{R},\,U_{2}}0$. 
Taking $H$ generally, we may assume that $H$ is effective and $(Y_{2},\Delta_{2}+H_{2})$ is lc. 
By \cite[Theorem 1.1]{has-mmp}, \cite[Remark 2.7]{birkar-flip}, and \cite[Theorem 1.7]{hashizumehu}, there exists a sequence $(Y_{2},\Delta_{2})\dashrightarrow (Y'_{2},\Delta'_{2})$ of steps of a $(K_{Y_{2}}+\Delta_{2})$-MMP over $U_{2}$ 
to a good minimal model $(Y'_{2},\Delta'_{2})$ over $U_{2}$. 
If $Y$ is $\mathbb{Q}$-factorial, then $Y_{2}$ is  $\mathbb{Q}$-factorial, and therefore $Y'_{2}$ is also $\mathbb{Q}$-factorial by construction of a sequence of steps of the standard log MMP for $\mathbb{Q}$-factorial lc pairs \cite[4.8.16]{fujino-book}. 
Since the relation $(K_{Y}+\Delta)|_{Y_{1}\cap Y_{2}}\sim_{\mathbb{R},\,U_{1}\cap U_{2}}0$ holds, the $(K_{Y_{2}}+\Delta_{2})$-MMP does not modify $Y_{1}\cap Y_{2}$. 
$$
\xymatrix
{
Y_{1}\ar[d]\ar@{}[r]|*{\hspace{-13pt}\supset}&Y_{1} \cap Y_{2}\ar[d]\ar@{}[r]|*{\hspace{13pt}\subset}&Y'_{2}\ar[d]\\
U_{1}\ar@{}[r]|*{\hspace{-13pt}\supset}&U_{1} \cap U_{2}\ar@{}[r]|*{\hspace{13pt}\subset}&U_{2}
}
$$
Therefore, we can glue $Y_{1} \to U_{1}$ and $Y'_{2} \to U_{2}$ along $Y_{1}\cap Y_{2}\to U_{1} \cap U_{2}$ and we obtain a projective morphism $g\colon (Y',\Delta') \to V$ such that the inverse image of $U_{1}$ (resp.~$U_{2}$) is $(Y_{1},\Delta_{1})$ (resp.~$(Y'_{2},\Delta'_{2})$). 
By the same argument, over $U_{1} \cap U_{2}$ we may glue $Y_{1} \to U_{1}$ and any variety appearing in the $(K_{Y_{2}}+\Delta_{2})$-MMP over $U_{2}$. 
Thus, the map
$$(Y,\Delta)\dashrightarrow (Y',\Delta')$$ is a finite sequence of steps of a $(K_{Y}+\Delta)$-MMP over $V$. 
If $Y$ is $\mathbb{Q}$-factorial, then $Y_{1}$ is $\mathbb{Q}$-factorial. 
Since $Y'_{2}$ is also $\mathbb{Q}$-factorial, we can directly check that $Y'=Y_{1}\cup Y'_{2}$ is $\mathbb{Q}$-factorial. 

By construction, $K_{Y'}+\Delta'$ is semi-ample over $V$.  
Let $f' \colon Y' \to X'$  be the contraction over $V$ induced by $K_{Y'}+\Delta'$. 
Let $\varphi' \colon X' \to V$ be the induced morphism.
We will check that $\varphi'$ is a projective small birational morphism. 
By construction, $\varphi'$ is an isomorphism over $U_{1}$, and therefore $\varphi'$ is birational. 
If there is a $\varphi'$-exceptional prime divisor $Q'$ on $X'$, then there is a prime divisor $P'$ on $Y'$ such that $f'(P')=Q'$. 
Then the inverse map $Y' \dashrightarrow Y$ is not an isomorphism on the generic point of $P'$. 
This contradicts the fact that $(Y,\Delta)\dashrightarrow (Y',\Delta')$ is a sequence of steps of a $(K_{Y}+\Delta)$-MMP over $V$. 
Thus, there is no $\varphi'$-exceptional prime divisor, and $\varphi'$ is a projective small birational morphism. 

Let $\omega'$ be the birational transform of $\omega$ on $X'$. 
We will check that the diagram
  $$
\xymatrix@R=16pt{
Y\ar[d]_{f}\ar@{-->}[rr]&& Y'\ar[d]^{f'} \\
X \ar@{-->}[rr]\ar[dr]_{\varphi}&& X' \ar[dl]^{\varphi'}\\
&V
}
$$
satisfies all the conditions of Theorem \ref{thm--mmpstep-quasi-log}. 
The birational map $(Y,\Delta)\dashrightarrow (Y',\Delta')$ is a finite sequence of steps of a $(K_{Y}+\Delta)$-MMP over $V$, which is the second condition of Theorem \ref{thm--mmpstep-quasi-log}. 
We have already check that $\varphi'$ is small, and this fact also shows that $X \dashrightarrow X'$ is a birational contraction. 
Let $\omega_{Y'}$ be the birational transform of $f^{*}\omega$ on $Y'$. 
Then $K_{Y'}+\Delta' \sim_{\mathbb{R}}\omega_{Y'}$. 
Therefore, $\omega_{Y'}$ is $\mathbb{R}$-Cartier and semi-ample over $V$, and $f' \colon Y' \to X'$ is the contraction over $V$ induced by $\omega_{Y'}$. 
Thus, there exists a $\varphi'$-ample $\mathbb{R}$-divisor $D'$ on $X'$ such that $\omega_{Y'} \sim_{\mathbb{R}}f'^{*}D'$. 
We recall that $U_{1} \subset V$ is the largest open subset over which $\varphi$ is an isomorphism. 
This implies that $\varphi'$ is an isomorphism over $U_{1}$ and ${\rm codim}_{V}(V \setminus U_{1}) \geq 2$. 
By restricting the above diagram over $U_{1}$, we have
$$f'|_{Y_{1}}^{*}(\omega'|_{\varphi'^{-1}(U_{1})})=\omega_{Y'}|_{Y_{1}}\sim_{\mathbb{R}}f'|_{Y_{1}}^{*}(D'|_{\varphi'^{-1}(U_{1})}).$$
This implies $\omega'|_{\varphi'^{-1}(U_{1})}\sim_{\mathbb{R}} D'|_{\varphi'^{-1}(U_{1})}$. 
Moreover, we have ${\rm codim}_{X'}(X' \setminus \varphi'^{-1}(U_{1}))\geq 2$ since $\varphi'$ is small. 
From these facts, we have $\omega' \sim_{\mathbb{R}} D'$. 
Thus shows that $\omega'$ is $\mathbb{R}$-Cartier and $\varphi'$-ample. 
Therefore, the third condition of Theorem \ref{thm--mmpstep-quasi-log} holds.
We also see that $K_{Y'}+\Delta' \sim_{\mathbb{R}}f'^{*}\omega'$. 
This induces the structure of a quasi-log scheme induced by a normal pair $f'\colon (Y',\Delta') \to [X',\omega']$, which is the first condition of Theorem \ref{thm--mmpstep-quasi-log}. 

We have already checked that the $\mathbb{Q}$-factoriality of $Y$ implies the $\mathbb{Q}$-factoriality of $Y'$.
By these arguments, we obtain the desired diagram of Theorem \ref{thm--mmpstep-quasi-log}.
\end{proof}

\begin{rem}\label{rem--mmp-fibration}
Let
$$
\xymatrix@R=16pt{
(Y,\Delta)\ar[d]_{f}\ar@{-->}[rr]&& (Y',\Delta')\ar[d]^{f'} \\
[X,\omega] \ar@{-->}[rr]\ar[dr]_{\varphi}&& [X',\omega'] \ar[dl]^{\varphi'}\\
&V
}
$$
be the diagram over $Z$ in Theorem \ref{thm--mmpstep-quasi-log}. 
By construction in the proof of Theorem \ref{thm--mmpstep-quasi-log}, we can write $V=U_{1} \cup U_{2}$ with $Y_{2}:=(\varphi \circ f)^{-1}(U_{2})$ and $Y'_{2}:=(\varphi' \circ f')^{-1}(U_{2})$
such that the $(K_{Y}+\Delta)$-MMP $(Y,\Delta) \dashrightarrow (Y',\Delta')$ over $V$ is an isomorphism on $(\varphi \circ f)^{-1}(U_{1})$ and $(Y_{2},\Delta|_{Y_{2}}) \dashrightarrow (Y'_{2},\Delta'|_{Y'_{2}})$ is the usual log MMP over $U_{2}$ for the lc pair $(Y_{2},\Delta|_{Y_{2}})$. 
The  $i$-th step of the $(K_{Y}+\Delta)$-MMP over $V$ can be written as
$$
\xymatrix{
 (Y^{(i)},\Delta^{(i)})\ar@{-->}[rr]\ar[dr]&& (Y^{(i+1)},\Delta^{(i+1)})\ar[dl]\\
&W^{(i)},
}
$$
 where $-(K_{Y^{(i)}}+\Delta^{(i)})$ and $K_{Y^{(i+1)}}+\Delta^{(i+1)}$ are ample over $W^{(i)}$. 
Then $Y^{(i)} \to W^{(i)}$ is an isomorphism on the inverse image of $U_{1}$. 
For any $\mathbb{R}$-Cartier divisor $D$ on $Y^{(i)}$, we can find $r \in \mathbb{R}$ such that $D-r(K_{Y^{(i)}}+\Delta^{(i)})$ is numerically trivial over $W^{(i)}$. 
In particular, the equality $\rho(Y^{(i)}/W^{(i)})=1$ holds. 
\end{rem}

\begin{thm}\label{thm--mmp-nefthreshold-strict}
Let $f\colon (Y,\Delta) \to [X,\omega]$ be a quasi-log scheme induced by a normal pair. 
Let $\pi \colon X \to Z$ be a projective morphism to a quasi-projective scheme $Z$. 
Suppose that $\omega$ is $\pi$-pseudo-effective and ${\rm NNef}(\omega/Z) \cap {\rm Nqlc}(X, \omega)=\emptyset$. 
Let $A$ be an $\mathbb{R}$-Cartier $\mathbb{R}$-divisor on $X$ such that $\omega+\lambda_{0} A$ is $\pi$-ample for some positive real number $\lambda_{0}$. 
Then there exists a diagram
$$
\xymatrix{
(Y,\Delta)=:(Y_{k_{1}},\Delta_{k_{1}})\ar@<4.5ex>[d]_{f=:f_{1}}\ar@{-->}[r]& (Y_{k_{2}},\Delta_{k_{2}})\ar[d]_{f_{2}} \ar@{-->}[r]&\cdots \ar@{-->}[r]&(Y_{k_{i}},\Delta_{k_{i}})\ar[d]_{f_{i}}\ar@{-->}[r]&\cdots\\
[X,\omega]=:[X_{1},\omega_{1}] \ar@{-->}[r]& [X_{2},\omega_{2}] \ar@{-->}[r]&\cdots \ar@{-->}[r]&[X_{i},\omega_{i}]\ar@{-->}[r]&\cdots
}
$$
over $Z$, where all $Y_{k_{i}}$ and $X_{i}$ are projective over $Z$, such that 
\begin{itemize}
\item
$f_{i}\colon (Y_{k_{i}},\Delta_{k_{i}}) \to [X_{i},\omega_{i}]$ are quasi-log schemes induced by normal pairs,  
\item
the sequence of upper horizontal maps is a sequence of steps of a $(K_{Y}+\Delta)$-MMP over $Z$ with scaling of $f^{*}A$, 
\item
the lower horizontal sequence of maps is a sequence of steps of an $\omega$-MMP over $Z$ with scaling of $A$, and
\item
if we put 
$$\lambda_{i}:={\rm inf}\{\mu \in \mathbb{R}_{\geq 0}\,| \,\text{$\omega_{i}+\mu A_{i}$ is nef over $Z$} \}$$ 
for each $i \geq 1$, then the following properties hold.
\begin{itemize}
\item[$\circ$]
$\lambda_{i}>\lambda_{i+1}$ for all $i \geq 1$, 
\item[$\circ$]
$\omega_{i}+\lambda_{i-1} A_{i}$ is semi-ample over $Z$ for all $i \geq 1$, and  
\item[$\circ$]
$\omega_{i}+t A_{i}$ is ample over $Z$ for all $i \geq 1$ and $t \in (\lambda_{i},\lambda_{i-1})$.  
\end{itemize}
\end{itemize}
Furthermore, if $Y$ is $\mathbb{Q}$-factorial, then $Y_{k_{i}}$ is also $\mathbb{Q}$-factorial for every $i \geq 1$. 
\end{thm}

\begin{proof}
We put $X_{1}:=X$, $\omega_{1}:=\omega$, $Y_{k_{1}}:=Y$, $\Delta_{k_{1}}:=\Delta$, $f_{1}:=f$,  and $A_{1}:=A$. 
We define
$$\lambda_{1}:={\rm inf}\{\mu \in \mathbb{R}_{\geq 0}\,| \,\text{$\omega_{1}+\mu A_{1}$ is nef over $Z$} \}.$$
Then $\lambda_{1}<\lambda_{0}$. 
If $\lambda_{1}=0$, there is nothing to prove. 
Thus we may assume $\lambda_{1}>0$. 
Then
$$\omega_{1}+\lambda_{1} A_{1}=\frac{\lambda_{0}-\lambda_{1}}{\lambda_{0}}\left(\omega_{1}+\frac{\lambda_{1}}{\lambda_{0}-\lambda_{1}}(\omega_{1}+\lambda_{0} A_{1})\right). $$
Since $\omega_{1}+\lambda_{0} A_{1}$ is ample over $Z$ and  the non-nef locus of $\omega_{1}$ over $Z$ does not intersect ${\rm Nqlc}(X_{1}, \omega_{1})$, the $\mathbb{R}$-line bundle $(\omega_{1}+\lambda_{1} A_{1})|_{{\rm Nqlc}(X_{1}, \omega_{1})}$ is ample over $Z$. 
By applying Theorem \ref{thm--abundance-quasi-log} to $[X_{1},\omega_{1}]$ and $\frac{\lambda_{1}}{\lambda_{0}-\lambda_{1}}(\omega_{1}+\lambda_{0} A_{1})$, we see that $\omega_{1}+\lambda_{1} A_{1}$ is $\pi$-semi-ample. 
Thus we get a contraction $\varphi \colon X_{1} \to V$ over $Z$ induced by $\omega_{1}+\lambda_{1} A_{1}$. 
Then $-\omega_{1}$ is $\varphi$-ample since $\omega_{1}+\lambda_{0} A_{1}$ is $\pi$-ample and $0<\lambda_{1}<\lambda_{0}$. 
Moreover, since the non-nef locus of $\omega_{1}$ over $Z$ does not intersect ${\rm Nqlc}(X_{1}, \omega_{1})$, any curve intersecting ${\rm Nqlc}(X_{1}, \omega_{1})$ has a positive intersection number with $\omega_{1}+\lambda_{1}A_{1}$. 
Hence $\varphi$ is birational and an isomorphism on a neighborhood of ${\rm Nqlc}(X_{1}, \omega_{1})$. 
By Theorem \ref{thm--mmpstep-quasi-log}, we get the diagram
$$
\xymatrix@R=16pt{
(Y_{k_{1}},\Delta_{k_{1}})\ar[d]_{f_{1}}\ar@{-->}[rr]&& (Y_{k_{2}},\Delta_{k_{2}})\ar[d]^{f_{2}} \\
[X_{1},\omega_{1}] \ar@{-->}[rr]\ar[dr]_{\varphi}&& [X_{2},\omega_{2}] \ar[dl]^{\varphi'}\\
&V
}
$$
over $Z$ such that $\varphi' \colon X_{2} \to V$ is a projective small birational morphism, $\omega_{2}$ is $\varphi'$-ample, and the upper horizontal map is a sequence of steps of a $(K_{Y_{k_{1}}}+\Delta_{k_{1}})$-MMP over $V$. 
Then $K_{Y_{k_{1}}}+\Delta_{k_{1}}+\lambda_{1}f_{1}^{*}A_{1} \sim_{\mathbb{R},\,V}0$. 
This implies that $(Y_{k_{1}},\Delta_{k_{1}}) \dashrightarrow (Y_{k_{2}},\Delta_{k_{2}})$ is a sequence of steps of a $(K_{Y_{k_{1}}}+\Delta_{k_{1}})$-MMP over $Z$ with scaling of $f_{1}^{*}A_{1}$.  

Let $A_{2}$ be the birational transform of $A_{1}$ on $X_{2}$. 
By construction, there exists an $\mathbb{R}$-Cartier $\mathbb{R}$-divisor $L$ on $V$ such that $L$ is ample over $Z$ and $\omega_{2}+\lambda_{1} A_{2} \sim_{\mathbb{R}}\varphi'^{*}L$. 
Since $\omega_{2}$ is $\varphi'$-ample, there exists $\lambda'_{1}< \lambda_{1}$, which is sufficiently close to $\lambda_{1}$, such that $\omega_{2}+\lambda'_{1}A_{2}$ is ample over $Z$. 
By this discussion and the property in Remark \ref{rem--mmp-basic} on the non-nef locus over $Z$ (Definition \ref{defn--nonneflocus}), we see that $(Y_{k_{2}},\Delta_{k_{2}}) \to [X_{2},\omega_{2}] \to Z$ and $A_{2}$ satisfy the hypothesis of Theorem \ref{thm--mmp-nefthreshold-strict}. 
We define
$$\lambda_{2}:={\rm inf}\{\mu \in \mathbb{R}_{\geq 0}\,| \,\text{$\omega_{2}+\mu A_{2}$ is nef over $Z$} \}.$$
Then $\lambda_{2} \leq \lambda'_{1}<\lambda_{1}$ by construction. 
We apply the argument of the previous paragraph to $(Y_{k_{2}},\Delta_{k_{2}}) \to [X_{2},\omega_{2}] \to Z$ and $A_{2}$, and we get the diagram 
$$
\xymatrix{
(Y_{k_{2}},\Delta_{k_{2}})\ar[d]_{f_{2}}\ar@{-->}[r]& (Y_{k_{3}},\Delta_{k_{3}})\ar[d]^{f_{3}} \\
[X_{2},\omega_{2}] \ar@{-->}[r]& [X_{3},\omega_{3}]
}
$$
over $Z$. 
By repeating this discussion, we obtain the desired diagram. 
\end{proof}

\begin{cor}\label{cor--mmpwithscaling-normalpair}
Let $\pi \colon Y \to Z$ be a projective morphism of normal quasi-projective varieties. 
Let $(Y,\Delta)$ be a normal pair such that $K_{Y}+\Delta$ is $\pi$-pseudo-effective and ${\rm NNef}(K_{Y}+\Delta/Z) \cap {\rm Nlc}(Y,\Delta) = \emptyset$. 
Let $A$ be an effective $\mathbb{R}$-Cartier $\mathbb{R}$-divisor on $Y$ such that $K_{Y}+\Delta+A$ is $\pi$-nef and ${\rm Nlc}(Y, \Delta)={\rm Nlc}(Y, \Delta+A)$ set theoretically. 
Then there exists a sequence of steps of a $(K_{Y}+\Delta)$-MMP over $Z$ with scaling of $A$ 
$$
\xymatrix{
(Y,\Delta)=:(Y_{1},\Delta_{1})\ar@{-->}[r]&\cdots \ar@{-->}[r]& (Y_{i},\Delta_{i})\ar@{-->}[rr]\ar[dr]&& (Y_{i+1},\Delta_{i+1})\ar[dl] \ar@{-->}[r]&\cdots, \\
&&&W_{i}
}
$$
where $(Y_{i},\Delta_{i}) \to W_{i} \leftarrow (Y_{i+1},\Delta_{i+1})$ is a step of a $(K_{Y_{i}}+\Delta_{i})$-MMP over $Z$, such that the non-isomorphic locus of the MMP is disjoint from ${\rm Nlc}(Y,\Delta)$ and $\rho(Y_{i}/W_{i})=1$ for every $i \geq 1$. 
\end{cor}

\begin{proof}
This follows from Lemma \ref{lem--extremal-ray}, \cite[Theorem 6.7.3]{fujino-book}, and Theorem \ref{thm--mmpstep-quasi-log}. 
\end{proof}

\begin{cor}\label{cor--mmp-nomralpair-Qfacdlt}
Let $\pi \colon Y \to Z$ be a projective morphism of normal quasi-projective varieties. 
Let $(Y,\Delta)$ be a normal pair such that $K_{Y}+\Delta$ is $\pi$-pseudo-effective and ${\rm NNef}(K_{Y}+\Delta/Z) \cap {\rm Nlc}(Y,\Delta) = \emptyset$. 
Let $A$ be an $\mathbb{R}$-Cartier $\mathbb{R}$-divisor on $Y$ such that $K_{Y}+\Delta+\lambda_{0} A$ is $\pi$-ample for some positive real number $\lambda_{0}$. 
Then there exists a sequence of steps of a $(K_{Y}+\Delta)$-MMP over $Z$ with scaling of $A$ 
$$
\xymatrix{
(Y,\Delta)=:(Y_{1},\Delta_{1})\ar@{-->}[r]&\cdots \ar@{-->}[r]& (Y_{i},\Delta_{i})\ar@{-->}[rr]\ar[dr]&& (Y_{i+1},\Delta_{i+1})\ar[dl] \ar@{-->}[r]&\cdots, \\
&&&W_{i}
}
$$
where $(Y_{i},\Delta_{i}) \to W_{i} \leftarrow (Y_{i+1},\Delta_{i+1})$ is a step of a $(K_{Y_{i}}+\Delta_{i})$-MMP over $Z$, satisfying the following.
\begin{itemize}
\item
The non-isomorphic locus of the MMP is disjoint from ${\rm Nlc}(Y,\Delta)$, 
\item
$\rho(Y_{i}/W_{i})=1$ for every $i \geq 1$, and
\item
if we put 
$$\lambda_{i}:={\rm inf}\{\mu \in \mathbb{R}_{\geq 0}\,| \,\text{$K_{Y_{i}}+\Delta_{i}+\mu A_{i}$ is nef over $Z$} \}$$  
for each $i \geq 1$ and $\lambda:={\rm lim}_{i \to \infty}\lambda_{i}$, then the following properties hold. 
\begin{itemize}
\item[$\circ$]
The MMP terminates after finitely many steps or otherwise $\lambda \neq \lambda_{i}$ for every $i \geq 1$, and 
\item[$\circ$]
$K_{Y_{i}}+\Delta_{i}+t A_{i}$ is semi-ample over $Z$ for all $i \geq 1$ and any $t \in (\lambda_{i},\lambda_{i-1}]$. 
\end{itemize}
\end{itemize}
Furthermore, if $Y$ is $\mathbb{Q}$-factorial, then $Y_{i}$ is also $\mathbb{Q}$-factorial for every $i \geq 1$. 
\end{cor}

\begin{proof}
Put $X:=Y$ and $\omega:=K_{Y}+\Delta$. 
We apply Theorem \ref{thm--mmp-nefthreshold-strict} to $(Y,\Delta)\overset{{\rm id}_{Y}}{\longrightarrow}[X,\omega]$, and we get a diagram 
$$
\xymatrix{
(Y,\Delta)=:(Y_{k_{1}},\Delta_{k_{1}})\ar@<4.5ex>[d]_{{\rm id}_{Y}=:f_{1}}\ar@{-->}[r]& (Y_{k_{2}},\Delta_{k_{2}})\ar[d]_{f_{2}} \ar@{-->}[r]&\cdots \ar@{-->}[r]&(Y_{k_{i}},\Delta_{k_{i}})\ar[d]_{f_{i}}\ar@{-->}[r]&\cdots\\
[X,\omega]=:[X_{1},\omega_{1}] \ar@{-->}[r]& [X_{2},\omega_{2}] \ar@{-->}[r]&\cdots \ar@{-->}[r]&[X_{i},\omega_{i}]\ar@{-->}[r]&\cdots
}
$$
over $Z$ satisfying the properties of Theorem \ref{thm--mmp-nefthreshold-strict}. 
We check that the sequences of upper horizontal maps
$$(Y,\Delta) \dashrightarrow \cdots \dashrightarrow (Y_{j},\Delta_{j})\dashrightarrow \cdots$$
satisfies the properties of Corollary \ref{cor--mmp-nomralpair-Qfacdlt}. 
By the second property of Theorem \ref{thm--mmp-nefthreshold-strict}, the sequence of maps are the sequence of steps of a $(K_{Y}+\Delta)$-MMP over $Z$ with scaling of $A$. 
By Remark \ref{rem--mmp-basic}, the first property of Corollary \ref{cor--mmp-nomralpair-Qfacdlt} holds. 
By construction of the above diagram (see Theorem \ref{thm--mmpstep-quasi-log} and Theorem \ref{thm--mmp-nefthreshold-strict}) and Remark \ref{rem--mmp-fibration}, the second property of Corollary \ref{cor--mmp-nomralpair-Qfacdlt} holds. 
For each $i$, let $A_{k_{i}}$ and $A_{X_{i}}$ be the birational transforms of $A$ on $Y_{k_{i}}$ and $X_{i}$ respectively. 
Since the sequences of upper horizontal maps and the lower horizontal maps are sequences of steps of $(K_{Y}+\Delta)$-MMP over $Z$ with scaling of $A$, both $A_{k_{i}}$ and $A_{X_{i}}$ are $\mathbb{R}$-Cartier, and we can easily check that $A_{k_{i}}=f_{i}^{*}A_{X_{i}}$ by an induction on $i$. 
This implies
$$K_{Y_{k_{i}}}+\Delta_{k_{i}}+tA_{k_{i}}\sim_{\mathbb{R}}f_{i}^{*}(\omega_{i}+tA_{X_{i}})$$
for all $i \geq 1$ and $t \in \mathbb{R}_{\geq 0}$. 
We put
$$\lambda_{j}:={\rm inf}\{\mu \in \mathbb{R}_{\geq 0}\,| \,\text{$K_{Y_{j}}+\Delta_{j}+\mu A_{j}$ is nef over $Z$} \}$$  
for each $j \geq 1$.
By the properties in Theorem \ref{thm--mmp-nefthreshold-strict}, $\lambda_{k_{i}}>\lambda_{k_{i+1}}$ and $K_{Y_{k_{i}}}+\Delta_{k_{i}}+t A_{k_{i}}$ is semi-ample over $Z$ for all $i \geq 1$ and $t \in (\lambda_{k_{i}},\lambda_{k_{i-1}}]$. 
By construction of the MMP  (see Theorem \ref{thm--mmpstep-quasi-log}), we have $\lambda_{j}=\lambda_{k_{i}}$ for every $k_{i} \leq j < k_{i+1}$. 
This implies that $K_{Y_{j}}+\Delta_{j}+t A_{j}$ is semi-ample over $Z$ for all $t \in (\lambda_{j},\lambda_{j-1}]$. 
By this discussion, the MMP
$$(Y,\Delta) \dashrightarrow \cdots \dashrightarrow (Y_{j},\Delta_{j})\dashrightarrow \cdots$$
satisfies the properties of Corollary \ref{cor--mmp-nomralpair-Qfacdlt}. 
The preservation of the $\mathbb{Q}$-factoriality can be checked by the argument as in \cite[Proposition 3.36 and Proposition 3.37]{kollar-mori}.
\end{proof}

\begin{rem}\label{rem--difference-mmp-scaling}
The main difference between the $(K_{Y}+\Delta)$-MMP over $Z$ with scaling of $A$ in Corollary \ref{cor--mmpwithscaling-normalpair} and that in Corollary \ref{cor--mmp-nomralpair-Qfacdlt} is that we have $\lambda_{i}\neq \lambda$ for any $i \geq 1$ in the MMP of Corollary \ref{cor--mmp-nomralpair-Qfacdlt} if the MMP does not terminate. 
This property is crucial to the proofs in Section \ref{sec--mmp-klt} and Section \ref{sec--mmp-lc}. 
If $(Y,\Delta)$ is klt and $A$ is ample over $Z$, then this property holds true for any MMP with scaling in Corollary \ref{cor--mmpwithscaling-normalpair} (see \cite{bchm} and \cite[Theorem 4.1]{birkar-flip}). 
In general, even if $A$ is ample over $Z$, it is not clear that this property holds for any $(K_{Y}+\Delta)$-MMP over $Z$ with scaling of $A$ as in Corollary \ref{cor--mmpwithscaling-normalpair}. 
\end{rem}

\begin{rem}
The following properties of a normal pair $(Y,\Delta)$ are preserved under the $(K_{Y}+\Delta)$-MMP in Corollary \ref{cor--mmpwithscaling-normalpair} and Corollary \ref{cor--mmp-nomralpair-Qfacdlt}. 
\begin{itemize} \item $Y$ is $\mathbb{Q}$-factorial, \item $(Y,0)$ is $\mathbb{Q}$-factorial klt, and \item $(Y,\Delta^{<1}+{\rm Supp}\,\Delta^{\geq 1})$ is dlt.  \end{itemize}
Indeed, the preservation of the $\mathbb{Q}$-factoriality is clear, and the preservation of the dlt property is easy because any $(K_{Y}+\Delta)$-MMP is a $(K_{Y}+\Delta^{<1}+{\rm Supp}\,\Delta^{\geq 1})$-MMP when $(Y,\Delta^{<1}+{\rm Supp}\,\Delta^{\geq 1})$ is dlt (see Corollary \ref{cor--dlt-nonlc-ideal}). 
For the second situation, it is enough to prove that if $(Y_{j},\Delta_{j}) \dashrightarrow (Y_{j+1},\Delta_{j+1})$ is a step of a $(K_{Y_{j}}+\Delta_{j})$-MMP, $(Y_{j},0)$ is $\mathbb{Q}$-factorial klt, and $Y_{j+1}$ is $\mathbb{Q}$-factorial, then $(Y_{j+1},0)$ is klt. 
Fix $t \in [0,1)$ such that $Y_{j} \dashrightarrow Y_{j+1}$ is a step of a $(K_{Y_{j}}+t\Delta_{j})$-MMP. 
On a neighborhood of ${\rm Nlc}(Y_{j},\Delta_{j})$, the map $Y_{j} \dashrightarrow Y_{j+1}$ is an isomorphism, and therefore $(Y_{j+1},0)$ is klt on a neighborhood of ${\rm Nlc}(Y_{j+1},\Delta_{j+1})$. 
On the other hand, $(Y_{j}, t\Delta_{j})$ is klt on $Y_{j}\setminus  {\rm Nlc}(Y_{j},\Delta_{j})$ because $(Y_{j},0)$ is klt and $t<1$. 
By taking a common resolution of $Y_{j} \dashrightarrow Y_{j+1}$ and the standard argument using negativity lemma, we see that $(Y_{j+1}, t\Delta_{j+1})$ is klt on $Y_{j+1}\setminus  {\rm Nlc}(Y_{j+1},\Delta_{j+1})$. 
Thus, $(Y_{j+1}, 0)$ is klt on $Y_{j+1}\setminus  {\rm Nlc}(Y_{j+1},\Delta_{j+1})$. 
 Therefore, $(Y_{j+1},0)$ is klt. 
\end{rem}

We close this section with a reduction of Question \ref{ques--termination-mmp-normalpair} to the termination of all MMP for klt pairs. 

\begin{rem}\label{rem--mmp-reduction}
Assume the termination of all MMP for all klt pairs. 
Let $\pi \colon X \to Z$ be a projective morphism of normal quasi-projective varieties. Let $(X,\Delta)$ be a normal pair such that $K_{X}+\Delta$ is $\pi$-pseudo-effective. Suppose that the non-nef locus of $K_{X}+\Delta$ over $Z$ does not intersect the non-lc locus of $(X,\Delta)$. 
Let
$$(X,\Delta)=:(X_{1},\Delta_{1}) \dashrightarrow\cdots \dashrightarrow  (X_{i},\Delta_{i}) \dashrightarrow \cdots $$
be a sequence of steps a $(K_{X}+\Delta)$-MMP over $Z$. 
The existence of the $(K_{X}+\Delta)$-MMP follows from Corollary \ref{cor--mmpwithscaling-normalpair}. 
Let $f\colon (Y,\Gamma) \to (X,\Delta)$ be a dlt blow-up in Theorem \ref{thm--dlt-blowup}. 
By using the lift of MMP as in \cite[Remark 2.9]{birkar-flip} or using Theorem \ref{thm--mmpstep-quasi-log} repeatedly, we get a diagram
$$
\xymatrix{
(Y,\Gamma)=:(Y_{k_{1}},\Gamma_{k_{1}})\ar@<4.5ex>[d]_{f=:f_{1}}\ar@{-->}[r]& (Y_{k_{2}},\Gamma_{k_{2}})\ar[d]_{f_{2}} \ar@{-->}[r]&\cdots \ar@{-->}[r]&(Y_{k_{i}},\Gamma_{k_{i}})\ar[d]_{f_{i}}\ar@{-->}[r]&\cdots\\
(X,\Delta)=:(X_{1},\Delta_{1}) \ar@{-->}[r]& (X_{2},\Delta_{2}) \ar@{-->}[r]&\cdots \ar@{-->}[r]&(X_{i},\Delta_{i})\ar@{-->}[r]&\cdots
}
$$
over $Z$ such that the upper horizontal sequence of birational maps
$$(Y,\Gamma) \dashrightarrow \cdots \dashrightarrow (Y_{j},\Gamma_{j})\dashrightarrow \cdots$$
is a sequence of steps of a $(K_{Y}+\Gamma)$-MMP. 
By construction, $(Y, \Gamma^{<1}+{\rm Supp}\,\Gamma^{\geq 1})$ is dlt, and ${\rm NNef}(K_{X}+\Delta/Z) \cap {\rm Nlc}(X,\Delta) = \emptyset$ implies ${\rm NNef}(K_{Y}+\Gamma/Z) \cap {\rm Nlc}(Y,\Gamma) = \emptyset$. 
By Remark \ref{rem--mmp-basic}, this $(K_{Y}+\Gamma)$-MMP is a $(Y, \Gamma^{<1}+{\rm Supp}\,\Gamma^{\geq 1})$-MMP over $Z$. 
By the special termination \cite{fujino-sp-ter}, after finitely many steps the $(K_{Y}+\Gamma^{<1}+{\rm Supp}\,\Gamma^{\geq 1})$-MMP will be a $(K_{Y}+\Gamma^{<1})$-MMP. 
Since the pair $(Y,\Gamma^{<1})$ is $\mathbb{Q}$-factorial klt, this MMP terminates by the termination of all MMP for all klt pairs. 
\end{rem}

\section{Minimal model program along Kawamata log terminal locus}\label{sec--mmp-klt}

In this section we study the minimal model theory for normal pairs such that the non-nef locus of the log canonical $\mathbb{R}$-divisor is disjoint from the non-klt locus of the normal pair. 

\subsection{From non-vanishing to existence of minimal model}

In this subsection, we study the existence of minimal models under the assumption on the non-vanishing theorem. 

\begin{prop}\label{prop--dlt-resol-linsystem}
Let $\pi \colon X \to Z$ be a projective morphism of normal quasi-projective varieties.  
Let $(X,\Delta)$ be a normal pair and let $D$ be an $\mathbb{R}$-Cartier $\mathbb{R}$-divisor on $X$ such that $|D/Z|_{\mathbb{R}} \neq \emptyset$. 
Then there exist a projective birational morphism $f \colon \widetilde{X} \to X$ from a $\mathbb{Q}$-factorial variety $\widetilde{X}$, an $\mathbb{R}$-divisor $\widetilde{M}$ on $\widetilde{X}$, and an effective $\mathbb{R}$-divisor $\widetilde{F}$ on $\widetilde{X}$ satisfying the following properties.
\begin{itemize}
\item
$f^{*}D\sim_{\mathbb{R},\,Z}\widetilde{M}+\widetilde{F}$, 

\item
$\widetilde{M}$ is semi-ample over $Z$ and ${\rm Supp}\,\widetilde{F}= {\rm Bs}|f^{*}D/Z|_{\mathbb{R}}$, 

\item
putting
$$\widetilde{\Gamma}=f_{*}^{-1}\Delta+\sum_{\text{$E_{i}$:$f$-exceptional}}E_{i}+{\rm Supp}\,\widetilde{F},$$
then $(\widetilde{X},\widetilde{\Gamma}^{<1}+{\rm Supp}\,\widetilde{\Gamma}^{\geq 1})$ is a $\mathbb{Q}$-factorial dlt pair, and

\item
for any $f$-exceptional prime divisor $E_{i}$ on $\widetilde{X}$, at least one of $a(E_{i},X,\Delta) \leq  -1$ and $E_{i} \subset {\rm Bs}|f^{*}D/Z|_{\mathbb{R}}$ holds. 
\end{itemize}
\end{prop}

\begin{proof}
By applying Theorem \ref{thm--resol-R-linear-system} to $\pi \colon X \to Z$, $\Delta$, and $D$, we get a log resolution $h \colon \overline{X} \to X$ of $(X,\Delta)$ and effective $\mathbb{R}$-divisors $\overline{M}$ and $\overline{F}$ on $\overline{X}$ such that
\begin{itemize}
\item
$h^{*}D\sim_{\mathbb{R},\,Z}\overline{M}+\overline{F}$, 
\item
every component of $\overline{F}$ is an irreducible component of ${\rm Bs}|h^{*}D/Z|_{\mathbb{R}}$,  
\item
$\overline{M}$ is semi-ample over $Z$, and 
\item
${\rm Supp}(\overline{F}+h_{*}^{-1}\Delta)\cup{\rm Ex}(h)$ is a simple normal crossing divisor on $\overline{X}$. 
\end{itemize}
We put
$$\overline{\Gamma}=h_{*}^{-1}\Delta+\sum_{\text{$E'_{j}$:$h$-exceptional}}E'_{j}+{\rm Supp}\,\overline{F}.$$
Then $(\overline{X}, \overline{\Gamma}^{<1}+{\rm Supp}\,\overline{\Gamma}^{\geq 1})$ is a log smooth lc pair. 
By construction, we may write
$$\overline{\Gamma}^{<1}+{\rm Supp}\,\overline{\Gamma}^{\geq 1}=h_{*}^{-1}\Delta+\sum_{\text{$E'_{j}$:$h$-exceptional}}E'_{j}+\overline{G}_{+}-\overline{G}_{-}$$
for some effective $\mathbb{R}$-divisors $\overline{G}_{+}$ and $\overline{G}_{-}$ on $\overline{X}$ such that $\overline{G}_{+}$ and $\overline{G}_{-}$ have no common components, ${\rm Supp}\,\overline{G}_{+}\subset {\rm Supp}\,\overline{F}$, and
${\rm Supp}\,\overline{G}_{-}\subset {\rm Supp}\,h_{*}^{-1}\Delta$. 

We can write $\overline{M}=\sum_{k}r_{k}\overline{M}_{k}$ for some positive real numbers $r_{k}$ and Cartier divisors $\overline{M}_{k}$ on $\overline{X}$ that are semi-ample over $Z$. We fix $\alpha \in \mathbb{R}_{>0}$ such that $\alpha r_{k} >2 \cdot {\rm dim}\,X$ for all $k$ and $\alpha \overline{F}-\overline{G}_{+} \geq 0$. 
Since we can write
$$K_{\overline{X}}+h_{*}^{-1}\Delta+\sum_{\text{$E'_{j}$:$h$-exceptional}}E'_{j}=h^{*}(K_{X}+\Delta)+\sum_{\text{$E'_{j}$:$h$-exceptional}}(a(E'_{j},X,\Delta)+1)E'_{j},$$
this relation and $\overline{M} \sim_{\mathbb{R},\,X}-\overline{F}$ imply
\begin{equation*}
\begin{split}
K_{\overline{X}}+\overline{\Gamma}^{<1}+{\rm Supp}\,\overline{\Gamma}^{\geq 1}+\alpha \overline{M}=&K_{\overline{X}}+h_{*}^{-1}\Delta+\sum_{\text{$E'_{j}$:$h$-exceptional}}E'_{j}+\overline{G}_{+}-\overline{G}_{-}+\alpha \overline{M}\\
\sim_{\mathbb{R},\,X}&\sum_{\text{$E'_{j}$:$h$-exceptional}}(a(E'_{j},X,\Delta)+1)E'_{j}-(\alpha \overline{F}-\overline{G}_{+}+\overline{G}_{-}),
\end{split}
\end{equation*}
where $\alpha \overline{F}-\overline{G}_{+}+\overline{G}_{-} \geq 0$ by construction of $\alpha$. 
Since $\overline{M}$ is semi-ample over $Z$, we can run a $(K_{\overline{X}}+\overline{\Gamma}^{<1}+{\rm Supp}\,\overline{\Gamma}^{\geq 1}+\alpha \overline{M})$-MMP over $X$ with scaling of an ample divisor, and therefore we get a birational contraction
$$\phi \colon (\overline{X}, \overline{\Gamma}^{<1}+{\rm Supp}\,\overline{\Gamma}^{\geq 1}+\alpha \overline{M}) \dashrightarrow (\widetilde{X}, \widetilde{\Gamma}^{<1}+{\rm Supp}\,\widetilde{\Gamma}^{\geq 1}+\alpha \widetilde{M})$$
over $X$, where $\widetilde{\Gamma}=\phi_{*} \overline{\Gamma}$ and $\widetilde{M}=\phi_{*}\overline{M}$, such that $K_{\widetilde{X}}+\widetilde{\Gamma}^{<1}+{\rm Supp}\,\widetilde{\Gamma}^{\geq 1}+\alpha \widetilde{M}$ is the limit of movable $\mathbb{R}$-divisors over $X$.  
By the choice of $\alpha$ and the argument of the length of extremal rays (\cite[Theorem 3.8.1]{bchm}), we see that $\widetilde{M}$ is semi-ample over $Z$. 

We denote the induced birational morphism $\widetilde{X} \to X$ by $f$, and we put $\widetilde{F}=\phi_{*}\overline{F}$. 
We check that $f \colon \widetilde{X} \to X$, $\widetilde{M}$, and $\widetilde{F}$ satisfy all the conditions of Proposition \ref{prop--dlt-resol-linsystem}. 
By construction, $\widetilde{X}$ is $\mathbb{Q}$-factorial. 
Since $h^{*}D\sim_{\mathbb{R},\,Z}\overline{M}+\overline{F}$, we have 
$$f^{*}D \sim_{\mathbb{R},\,Z} \widetilde{M}+\widetilde{F}.$$
Therefore, the first condition of Proposition \ref{prop--dlt-resol-linsystem} holds. 
We pick an arbitrary element $\widetilde{L} \in |f^{*}D/Z|_{\mathbb{R}}$. 
Putting $\overline{L}:=h^{*}f_{*}\widetilde{L}$, then $\overline{L} \in |h^{*}D/Z|_{\mathbb{R}}$, and we have $\phi_{*}\overline{L}=\widetilde{L}$ by the negativity lemma. 
Since every component of $\overline{F}$ is contained in ${\rm Bs}|h^{*}D/Z|_{\mathbb{R}}$, we have ${\rm Supp}\,\overline{L}\supset {\rm Supp}\,\overline{F}$. 
Therefore we have ${\rm Supp}\,\widetilde{L}\supset {\rm Supp}\,\widetilde{F}$.  Since $\widetilde{L}$ is arbitrary, it follows that ${\rm Supp}\,\widetilde{F}\subset {\rm Bs}|f^{*}D/Z|_{\mathbb{R}}$. 
The inverse inclusion also holds by the semi-ampleness of $\widetilde{M}$ over $Z$. 
Therefore
$${\rm Supp}\,\widetilde{F}= {\rm Bs}|f^{*}D/Z|_{\mathbb{R}},$$
and we see that the second condition of Proposition \ref{prop--dlt-resol-linsystem} holds. 
Now
$$\widetilde{\Gamma}=\phi_{*} \overline{\Gamma}=f_{*}^{-1}\Delta+\sum_{\text{$E_{i}$:$f$-exceptional}}E_{i}+{\rm Supp}\,\widetilde{F}$$
and $(\widetilde{X},\widetilde{\Gamma}^{<1}+{\rm Supp}\,\widetilde{\Gamma}^{\geq 1})$ is a $\mathbb{Q}$-factorial dlt pair by construction, which is the third condition of Proposition \ref{prop--dlt-resol-linsystem}. 
We put $\widetilde{G}_{+}=\phi_{*}\overline{G}_{+}$ and $\widetilde{G}_{-}=\phi_{*}\overline{G}_{-}$.  
Then
\begin{equation*}
\begin{split}
K_{\widetilde{X}}+\widetilde{\Gamma}^{<1}+{\rm Supp}\,\widetilde{\Gamma}^{\geq 1}+\alpha \widetilde{M}=&K_{\widetilde{X}}+f_{*}^{-1}\Delta+\sum_{\text{$E_{i}$:$f$-exceptional}}E_{i}+\widetilde{G}_{+}-\widetilde{G}_{-}+\alpha \widetilde{M}\\
\sim_{\mathbb{R},\,X}&\sum_{\text{$E_{i}$:$f$-exceptional}}(a(E_{i},X,\Delta)+1)E_{i}-(\alpha \widetilde{F}-\widetilde{G}_{+}+\widetilde{G}_{-}),
\end{split}
\end{equation*}
and $\alpha \widetilde{F}-\widetilde{G}_{+}+\widetilde{G}_{-} \geq 0$ by construction. 
Since $K_{\widetilde{X}}+\widetilde{\Gamma}^{<1}+{\rm Supp}\,\widetilde{\Gamma}^{\geq 1}+\alpha \widetilde{M}$ is the limit of movable $\mathbb{R}$-divisors over $X$, the negativity lemma implies that 
$$(\alpha \widetilde{F}-\widetilde{G}_{+}+\widetilde{G}_{-})-\sum_{\text{$E_{i}$:$f$-exceptional}}(a(E_{i},X,\Delta)+1)E_{i} \geq 0.$$
From this fact, for any $f$-exceptional prime divisor $E_{i}$, if $a(E_{i},X,\Delta)>-1$ then $E_{i}$ is an irreducible component of $\alpha \widetilde{F}-\widetilde{G}_{+}+\widetilde{G}_{-}$. 
By recalling that ${\rm Supp}\,\overline{G}_{+}\subset {\rm Supp}\,\overline{F}$ and that ${\rm Supp}\,\overline{G}_{-}\subset {\rm Supp}\,h_{*}^{-1}\Delta$, we have
$${\rm Supp}\,(\alpha \widetilde{F}-\widetilde{G}_{+}+\widetilde{G}_{-}) \subset {\rm Supp}\,(f_{*}^{-1}\Delta+\widetilde{F})={\rm Supp}\,f_{*}^{-1}\Delta \cup {\rm Bs}|f^{*}D/Z|_{\mathbb{R}}.$$
Therefore, any $f$-exceptional prime divisor $E_{i}$ satisfying $a(E_{i},X,\Delta)>-1$ is contained in ${\rm Supp}\,f_{*}^{-1}\Delta \cup {\rm Bs}|f^{*}D/Z|_{\mathbb{R}}$. 
Since $E_{i}$ is $f$-exceptional, $E_{i}$ is contained in ${\rm Bs}|f^{*}D/Z|_{\mathbb{R}}$. 
This shows the final condition of Proposition \ref{prop--dlt-resol-linsystem}. 
\end{proof}

The following result is a variant of Proposition \ref{prop--dlt-resol-linsystem}.

\begin{prop}\label{prop--dlt-resol-linsystem-2}
Let $\pi \colon X \to Z$ be a projective morphism of normal quasi-projective varieties.  
Let $(X,\Delta)$ be a normal pair and let $D$ be an $\mathbb{R}$-Cartier divisor on $X$. 
Suppose that there exists a resolution $g \colon X' \to X$ of $X$ for which $g^{*}D$ has the Nakayama--Zariski decomposition over $Z$ whose positive part is semi-ample over $Z$. 
Then there exists a projective birational morphism $f \colon \widetilde{X} \to X$ from a $\mathbb{Q}$-factorial variety $\widetilde{X}$ such that putting $\widetilde{M}$ and $\widetilde{F}$ as the positive part and the negative part of the Nakayama--Zariski decomposition of $f^{*}D$ over $Z$, respectively, then the following properties hold.
\begin{itemize}
\item
$f^{*}D\sim_{\mathbb{R},\,Z}\widetilde{M}+\widetilde{F}$, 

\item
$\widetilde{M}$ is semi-ample over $Z$ and ${\rm Supp}\,\widetilde{F}= {\rm Bs}|f^{*}D/Z|_{\mathbb{R}}$, 

\item
putting
$$\widetilde{\Gamma}=f_{*}^{-1}\Delta+\sum_{\text{$E_{i}$:$f$-exceptional}}E_{i}+{\rm Supp}\,\widetilde{F},$$
then $(\widetilde{X},\widetilde{\Gamma}^{<1}+{\rm Supp}\,\widetilde{\Gamma}^{\geq 1})$ is a $\mathbb{Q}$-factorial dlt pair, and

\item
for any $f$-exceptional prime divisor $E_{i}$ on $\widetilde{X}$, at least one of $a(E_{i},X,\Delta) \leq  -1$ and $E_{i} \subset {\rm Bs}|f^{*}D/Z|_{\mathbb{R}}$ holds. 
\end{itemize}
\end{prop}

\begin{proof}
Let $g \colon X' \to X$ be as in the proposition.
By \cite[Lemma 3.4 (4)]{liuxie-relative-nakayama}, for any projective birational morphism $g' \colon X'' \to X'$ from a smooth variety $X''$, we see that the Nakayama--Zariski decomposition of $g'^{*}g^{*}D$ over $Z$ is well defined and 
$$P_{\sigma}(g'^{*}g^{*}D ; X''/Z)=g'^{*}P_{\sigma}(g^{*}D ; X'/Z).$$
In particular, $g'^{*}g^{*}D$ has the Nakayama--Zariski decomposition over $Z$ whose positive part is semi-ample over $Z$. 
Therefore, for any $X'' \overset{g'}{\longrightarrow} X' \overset{g}{\longrightarrow} X$ that is a log resolution of $(X,\Delta)$, the divisor $g'^{*}g^{*}D$ has the Nakayama--Zariski decomposition over $Z$ whose positive part is semi-ample over $Z$. 

By the above argument, we may take a log resolution $h \colon  \overline{X} \to X$ of $(X,\Delta)$ such that $P_{\sigma}(h^{*}D ; \overline{X}/Z)$ is semi-ample over $Z$. 
Put $\overline{F}:=N_{\sigma}(h^{*}D ; \overline{X}/Z)$. 
Then any component of $\overline{F}$ is an irreducible component of ${\rm Bs}|h^{*}D/Z|_{\mathbb{R}}$. 
By replacing $h \colon \overline{X} \to X$ with a higher log resolution, we may assume that ${\rm Supp}\,(\overline{F}+h_{*}^{-1}\Delta) \cup {\rm Ex}(h)$ is a simple normal crossing divisor on $\overline{X}$. 
Putting $\overline{M}:=P_{\sigma}(h^{*}D ; \overline{X}/Z)$, then the proof of Proposition \ref{prop--dlt-resol-linsystem} works with no changes by using $h \colon \overline{X} \to X$, $\overline{M}$, and $\overline{F}$.  
\end{proof}

\begin{prop}\label{prop--exist-minmodel-boundary-change}
Let $\pi \colon X \to Z$ be a projective morphism of normal quasi-projective varieties. 
Let $(X,\Delta)$ be a normal pair such that ${\rm Bs}|K_{X}+\Delta/Z|_{\mathbb{R}} \cap {\rm Nlc}(X,\Delta)=\emptyset$. 
Let $(X,\Delta')$ be a normal pair such that $\Delta' \sim_{\mathbb{R},\,Z} \Delta$. 
If $(X,\Delta')$ has a good minimal model over $Z$, then $(X,\Delta)$ has a $\mathbb{Q}$-factorial good minimal model over $Z$. 
\end{prop}

\begin{proof}
Since $K_{X}+\Delta' \sim_{\mathbb{R},\,Z} K_{X}+\Delta$ and $(X,\Delta')$ has a good minimal model over $Z$, using Lemma \ref{lem--good-min-model-makayama-zariski-decomp}, we see that $K_{X}+\Delta$ birationally has the Nakayama--Zariski decomposition over $Z$ whose positive part is semi-ample over $Z$. 
By Proposition \ref{prop--dlt-resol-linsystem-2}, we get a projective birational morphism $f \colon \overline{X} \to X$ from a $\mathbb{Q}$-factorial variety $\overline{X}$ such that putting $\overline{M}$ and $\overline{F}$ as the positive and the negative parts of the Nakayama--Zariski decomposition of $f^{*}(K_{X}+\Delta)$ over $Z$, respectively, then the following properties hold.
\begin{itemize}
\item
$f^{*}(K_{X}+\Delta)\sim_{\mathbb{R},\,Z}\overline{M}+\overline{F}$, 

\item
$\overline{M}$ is semi-ample over $Z$ and ${\rm Supp}\,\overline{F}= {\rm Bs}|f^{*}(K_{X}+\Delta)/Z|_{\mathbb{R}}$, 

\item
putting
$$\overline{\Gamma}=f_{*}^{-1}\Delta+\sum_{\text{$E_{i}$:$f$-exceptional}}E_{i}+{\rm Supp}\,\overline{F},$$
then $(\overline{X},\overline{\Gamma}^{<1}+{\rm Supp}\,\overline{\Gamma}^{\geq 1})$ is a $\mathbb{Q}$-factorial dlt pair, and

\item
for any $f$-exceptional prime divisor $E_{i}$ on $\overline{X}$, at least one of $a(E_{i},X,\Delta) \leq  -1$ and $E_{i} \subset {\rm Bs}|f^{*}(K_{X}+\Delta)/Z|_{\mathbb{R}}$ holds. 
\end{itemize}
We may write
$$K_{\overline{X}}+\overline{\Delta}=f^{*}(K_{X}+\Delta)+\overline{E}$$
for some effective $\mathbb{R}$-divisors $\overline{\Delta}$ and $\overline{E}$ that have no common components. 
We put 
$$\overline{B}:=\sum_{F_{j}\subset {\rm Supp}\,\overline{F}}(1-{\rm coeff}_{F_{j}}(\overline{\Delta}))F_{j}.$$
Since ${\rm Bs}|K_{X}+\Delta/Z|_{\mathbb{R}} \cap {\rm Nlc}(X,\Delta)=\emptyset$,  we have $a(F_{j},X,\Delta) \geq -1$ for any component $F_{j}$ of $\overline{F}$. 
By construction, the discrepancy of any component of $\overline{E}$ with respect to $(X,\Delta)$ is positive. 
Thus ${\rm Supp}\,\overline{E}\subset {\rm Bs}|f^{*}(K_{X}+\Delta)/Z|_{\mathbb{R}}$. 
Since ${\rm Supp}\,\overline{F}= {\rm Bs}|f^{*}(K_{X}+\Delta)/Z|_{\mathbb{R}}$, the following properties hold.
\begin{itemize}
\item
${\rm Supp}\,(\overline{B}+\overline{E}) \subset {\rm Supp}\,\overline{F}$, 
\item
$\overline{B}$ is effective, 
\item 
${\rm Supp}\,\overline{B} \cap {\rm Supp}\,\overline{\Delta}^{>1} = \emptyset$, and 
\item
${\rm coeff}_{F_{j}}(\overline{\Delta}+\overline{B})=1$ for any component $F_{j}$ of $\overline{F}$. 
\end{itemize}
Proposition \ref{prop--minmodel-biratmodel} for $\pi \colon X \to Z$, $(X,\Delta)$, $f\colon \overline{X} \to X$, and $(\overline{X},\overline{\Delta}+\overline{B})$ implies that to prove the existence of a $\mathbb{Q}$-factorial good minimal model of $(X,\Delta)$ over $Z$, it is enough to prove the existence of a birational contraction $\phi \colon \overline{X} \dashrightarrow \overline{X}'$ over $Z$, where $\overline{X}'$ is $\mathbb{Q}$-factorial and projective over $Z$, such that $\phi_{*}(K_{\overline{X}}+\overline{\Delta}+\overline{B})$ is semi-ample over $Z$ and any prime divisor contracted by $\phi$ is contained in ${\rm Bs}|K_{\overline{X}}+\overline{\Delta}+\overline{B}/Z|_{\mathbb{R}}$. 

In this paragraph, we will check that $(\overline{X},\overline{\Delta}+\overline{B})$ is $\mathbb{Q}$-factorial dlt on a neighborhood of ${\rm Supp}\,\overline{F}$. 
Since ${\rm Supp}\,\overline{F}= {\rm Bs}|f^{*}(K_{X}+\Delta)/Z|_{\mathbb{R}}$ and ${\rm Bs}|K_{X}+\Delta/Z|_{\mathbb{R}} \cap {\rm Nlc}(X,\Delta)=\emptyset$, we have
${\rm Supp}\,\overline{F} \cap {\rm Supp}\,\overline{\Delta}^{>1} = \emptyset.$
From this and the definition of $\overline{B}$, we see that
$${\rm Supp}\,\overline{F} \cap {\rm Supp}\,(\overline{\Delta}+\overline{B})^{>1}={\rm Bs}|f^{*}(K_{X}+\Delta)/Z|_{\mathbb{R}} \cap {\rm Supp}\,\overline{\Delta}^{>1} = \emptyset.$$
By construction of $\overline{B}$, outside ${\rm Supp}\,(\overline{\Delta}+\overline{B})^{>1}$ we have $\overline{\Delta}+\overline{B} \leq \overline{\Gamma}$, where
$$\overline{\Gamma}=f_{*}^{-1}\Delta+\sum_{\text{$E_{i}$:$f$-exceptional}}E_{i}+{\rm Supp}\,\overline{F}.$$
Since $(\overline{X},\overline{\Gamma}^{<1}+{\rm Supp}\,\overline{\Gamma}^{\geq 1})$ is a $\mathbb{Q}$-factorial dlt pair, the following pair 
$$(\overline{X},(\overline{\Delta}+\overline{B})^{<1}+{\rm Supp}\,(\overline{\Delta}+\overline{B})^{\geq 1})$$
is dlt outside ${\rm Supp}\,(\overline{\Delta}+\overline{B})^{>1}$. 
Hence, $(\overline{X},\overline{\Delta}+\overline{B})$ is $\mathbb{Q}$-factorial dlt on a neighborhood of ${\rm Supp}\,\overline{F}$. 

From now on, we use the notations of the relative Nakayama--Zariski decomposition in Definition \ref{defn--relative-nakayama-zariski-decom}. 
We may write $\overline{M}=\sum_{k}r_{k}\overline{M}^{(k)}$ for some positive real numbers $r_{k}$ and Cartier divisors $\overline{M}^{(k)}$ on $\overline{X}$ that are semi-ample over $Z$. 
We take $\alpha \in \mathbb{R}_{>0}$ so that $\alpha r_{k}>2 \cdot{\rm dim}\,\overline{X}$ for all $k$.   
Since ${\rm Supp}\,(\overline{B}+\overline{E}) \subset {\rm Supp}\, \overline{F}$, choosing $\alpha$ appropriately we may assume $\alpha \overline{F}\geq \overline{B}+\overline{E}$. 
Since $f^{*}(K_{X}+\Delta)\sim_{\mathbb{R},\,Z}\overline{M}+\overline{F}$, we have
$$K_{\overline{X}}+ \overline{\Delta}+\overline{B}+\alpha \overline{M} \sim_{\mathbb{R},\,Z}(\alpha+1)f^{*}(K_{X}+\Delta)-(\alpha \overline{F}-\overline{E}-\overline{B}).$$
We have the equality $\overline{F}=N_{\sigma}(f^{*}(K_{X}+\Delta);\overline{X}/Z)$, which is the definition of $\overline{F}$, and we also have the obvious relation $(\alpha+1)\overline{F} \geq \alpha \overline{F}-\overline{E}-\overline{B}$.  
By \cite[III, 4.2 (1) Lemma]{nakayama} or \cite[Lemma 3.7 (5)]{liuxie-relative-nakayama}, we obtain
\begin{equation*}
\begin{split}
&N_{\sigma}(K_{\overline{X}}+ \overline{\Delta}+\overline{B}+\alpha \overline{M};\overline{X}/Z)\\
=&N_{\sigma}((\alpha+1)f^{*}(K_{X}+\Delta)-(\alpha \overline{F}-\overline{E}-\overline{B});\overline{X}/Z)\\
=&(\alpha+1)\overline{F}-(\alpha \overline{F}-\overline{E}-\overline{B}) 
=\overline{F}+\overline{E}+\overline{B}.
\end{split}
\end{equation*}
By using ${\rm Supp}\,\overline{F}={\rm Supp}\,(\overline{F}+\overline{E}+\overline{B})$ and ${\rm Supp}\,\overline{F} ={\rm Bs}|f^{*}(K_{X}+\Delta)/Z|_{\mathbb{R}}$, we also have
\begin{equation*}
\begin{split}
{\rm Supp}\,\overline{F} =&{\rm Bs}|f^{*}(K_{X}+\Delta)/Z|_{\mathbb{R}}\cup {\rm Supp}\,(\overline{E}+\overline{B})\\
\supset&{\rm Bs}|f^{*}(K_{X}+\Delta)+\overline{E}+\overline{B}/Z|_{\mathbb{R}}={\rm Bs}|K_{\overline{X}}+ \overline{\Delta}+\overline{B}/Z|_{\mathbb{R}}\\
\supset &{\rm Bs}|K_{\overline{X}}+ \overline{\Delta}+\overline{B}+\alpha \overline{M}/Z|_{\mathbb{R}}\supset {\rm Supp}\,N_{\sigma}(K_{\overline{X}}+\overline{\Delta}+\overline{B}+\alpha \overline{M};\overline{X}/Z),
\end{split}
\end{equation*}
where the final inclusion follows from 
$$\overline{X} \setminus {\rm Bs}|K_{\overline{X}}+ \overline{\Delta}+\overline{B}+\alpha \overline{M}/Z|_{\mathbb{R}}\subset \overline{X}\setminus{\rm Supp}\,N_{\sigma}(K_{\overline{X}}+ \overline{\Delta}+\overline{B}+\alpha \overline{M};\overline{X}/Z).$$
From these discussions, we obtain 
$${\rm Bs}|K_{\overline{X}}+ \overline{\Delta}+\overline{B}/Z|_{\mathbb{R}}=  {\rm Supp}\,N_{\sigma}(K_{\overline{X}}+\overline{\Delta}+\overline{B}+\alpha \overline{M};\overline{X}/Z)={\rm Supp}\, \overline{F}.$$

Since $(\overline{X},\overline{\Delta}+\overline{B})$ is $\mathbb{Q}$-factorial dlt on a neighborhood of ${\rm Supp}\,\overline{F}$, by Corollary \ref{cor--mmp-nomralpair-Qfacdlt}, we can run a $(K_{\overline{X}}+\overline{\Delta}+\overline{B}+\alpha \overline{M})$-MMP over $Z$ with scaling of an ample divisor $\overline{A}$
$$(\overline{X},\overline{\Delta}+\overline{B}+\alpha \overline{M})\dashrightarrow \cdots\dashrightarrow (\overline{X}_{l},\overline{\Delta}_{l}+\overline{B}_{l}+\alpha \overline{M}_{l})\dashrightarrow \cdots.$$
Now we have
$$K_{\overline{X}}+\overline{\Delta}+\overline{B}+\alpha \overline{M}\sim_{\mathbb{R},\,Z}(\alpha+1)\overline{M}+\overline{F}+\overline{E}+\overline{B}.$$
By the choice of $\alpha$ and the length of extremal rays (\cite[Theorem 1.5 (iii)]{fujino-morihyper}), we see that $\overline{M}_{l}$ trivially intersects the curves contracted by the $(K_{\overline{X}_{l}}+\overline{\Delta}_{l}+\overline{B}_{l}+\alpha \overline{M}_{l})$-negative extremal contraction of the MMP. 
Therefore the non-isomorphic locus of the MMP is contained in ${\rm Supp}\,\overline{F}={\rm Bs}|K_{\overline{X}}+ \overline{\Delta}+\overline{B}/Z|_{\mathbb{R}}$. 
We recall that ${\rm coeff}_{F_{j}}(\overline{\Delta}+\overline{B})=1$ for any component $F_{j}$ of $\overline{F}$, which follows from the definition of $\overline{B}$. 
For each $l \geq 1$, we put
$$\lambda_{l}:={\rm inf}\set{\mu \in \mathbb{R}_{\geq0}|\text{$K_{\overline{X}_{l}}+\overline{\Delta}_{l}+\overline{B}_{l}+\alpha \overline{M}_{l}+\mu \overline{A}_{l}$ is nef over $Z$}}$$
and $\lambda:={\rm lim}_{l \to \infty} \lambda_{l}$. 
If $\lambda>0$, by construction of the MMP in Corollary \ref{cor--mmp-nomralpair-Qfacdlt}, we have $\lambda \neq \lambda_{l}$ for any $l$. 
Then the MMP is also a $(K_{\overline{X}}+\overline{\Delta}+\overline{B}+\alpha \overline{M}+\lambda \overline{A})$-MMP over $Z$ with scaling of $\overline{A}$. 
Since the non-isomorphic locus of the MMP is contained in ${\rm Supp}\,\overline{F}$
and we have ${\rm coeff}_{F_{j}}(\overline{\Delta}+\overline{B})=1$ for any component $F_{j}$ of $\overline{F}$, we may apply the argument of \cite[Proof of Lemma 3.10.11 (2)]{bchm} and the special termination \cite{fujino-sp-ter} with the aid of \cite[Theorem E]{bchm} and the $(\pi \circ f)$-ampleness of $\lambda \overline{A}$. 
We see that the MMP terminates. 
However, it contradicts the fact $\lambda>0$.  
Thus $\lambda =0$. 
Since 
$${\rm Supp}\,N_{\sigma}(K_{\overline{X}}+\overline{\Delta}+\overline{B}+\alpha \overline{M};\overline{X}/Z)={\rm Supp}\, \overline{F}={\rm Supp}\,(\overline{F}+\overline{E}+\overline{B}),$$
we can find an index $l'$ such that $\overline{X} \dashrightarrow \overline{X}_{l'}$ contracts $\overline{F}+\overline{E}+\overline{B}$. 
Then 
$$K_{\overline{X}_{l'}}+\overline{\Delta}_{l'}+\overline{B}_{l'}=(K_{\overline{X}_{l'}}+\overline{\Delta}_{l'}+\overline{B}_{l'}+\alpha \overline{M}_{l'})-\alpha \overline{M}_{l'}\sim_{\mathbb{R},\,Z} (\alpha+1)\overline{M}_{l'}-\alpha \overline{M}_{l'} = \overline{M}_{l'},$$
which is semi-ample over $Z$ because $\overline{M}$ is semi-ample over $Z$. 
Putting $\overline{X}':=\overline{X}_{l'}$ and denoting $\overline{X} \dashrightarrow \overline{X}'$ by $\phi$, the birational contraction $\phi \colon \overline{X} \dashrightarrow \overline{X}'$ over $Z$ satisfies
\begin{itemize}
\item
$\overline{X}'$ is $\mathbb{Q}$-factorial, 
\item
$\phi_{*}(K_{\overline{X}}+\overline{\Delta}+\overline{B})$ is semi-ample over $Z$, and
\item
any prime divisor contracted by $\phi$ is a component of ${\rm Bs}|K_{\overline{X}}+ \overline{\Delta}+\overline{B}/Z|_{\mathbb{R}}$. 
\end{itemize}
As discussed before, Proposition \ref{prop--minmodel-biratmodel} implies the existence of a $\mathbb{Q}$-factorial good minimal model of $(X,\Delta)$ over $Z$. 
We finish the proof. 
\end{proof}

\begin{thm}\label{thm--from-nonvanish-to-minmodel}
Let $\pi \colon X \to Z$ be a projective morphism of normal quasi-projective varieties. 
Let $(X,\Delta)$ be a normal pair and let $A$ be an effective $\pi$-ample $\mathbb{R}$-divisor on $X$ such that ${\rm Nklt}(X,\Delta) \neq \emptyset$ and ${\rm Bs}|K_{X}+\Delta+A/Z|_{\mathbb{R}} \cap {\rm Nklt}(X,\Delta+A) = \emptyset$. 
Then $(X,\Delta+A)$ has a $\mathbb{Q}$-factorial good minimal model over $Z$. 
\end{thm}

\begin{proof}
Since $(X,\Delta)$ is not klt, we see that ${\rm Bs}|K_{X}+\Delta+A/Z|_{\mathbb{R}} \neq X$. 
In particular, $|K_{X}+\Delta+A/Z|_{\mathbb{R}} \neq \emptyset$.  
The idea of the proof is the same as that of Proposition \ref{prop--exist-minmodel-boundary-change}. 
However, the proof of Theorem \ref{thm--from-nonvanish-to-minmodel} is more complicated because we need to deal with many $\mathbb{R}$-divisors. 
Therefore we divide the proof into several steps. 

\begin{step1}\label{step1--thm--from-nonvanish-to-minmodel}
In this step, we reduce the theorem to the case where $X$ is $\mathbb{Q}$-factorial and ${\rm Nklt}(X,\Delta)= {\rm Nklt}(X,\Delta+A)$. 

Let $h \colon (\overline{X},\overline{\Delta}) \to (X,\Delta)$ be a dlt blow-up of $(X,\Delta)$. 
Since $h^{*}A$ is nef and big over $Z$, we may write $h^{*}A \sim_{\mathbb{R},\,Z} \overline{A}+\overline{G}$ for some general $(\pi \circ h)$-ample $\mathbb{R}$-divisor $\overline{A}$ on $\overline{X}$ and effective $\mathbb{R}$-divisor $\overline{G}$ on $\overline{X}$ such that ${\rm Nklt}(\overline{X},\overline{\Delta})={\rm Nklt}(\overline{X},\overline{\Delta}+\overline{G}+\overline{A})$ as closed subschemes of $\overline{X}$. 
Since 
$$K_{\overline{X}}+\overline{\Delta}+\overline{G}+\overline{A} \sim_{\mathbb{R},\,Z} h^{*}(K_{X}+\Delta+A)$$
and 
$${\rm Nklt}(\overline{X},\overline{\Delta}+\overline{G}+\overline{A})={\rm Nklt}(\overline{X},\overline{\Delta})\subset {\rm Nklt}(\overline{X},\overline{\Delta}+h^{*}A)\subset h^{-1}({\rm Nklt}(X,\Delta+A))$$
set theoretically, we see that
\begin{equation*}
\begin{split}
&{\rm Bs}|K_{\overline{X}}+(\overline{\Delta}+\overline{G})+\overline{A}/Z|_{\mathbb{R}} \cap {\rm Nklt}(\overline{X},\overline{\Delta}+\overline{G}+\overline{A})\\
 \subset & {\rm Bs}|h^{*}(K_{X}+\Delta+A)/Z|_{\mathbb{R}} \cap h^{-1}({\rm Nklt}(X,\Delta+A)) = \emptyset.
\end{split}
\end{equation*} 
By Lemma \ref{lem--min-model-nonlc-crepant}, $(X,\Delta+A)$ has a good minimal model over $Z$ if $(\overline{X},\overline{\Delta}+h^{*}A)$ has a good minimal model over $Z$. 
By Proposition \ref{prop--exist-minmodel-boundary-change}, $(\overline{X},\overline{\Delta}+h^{*}A)$ has a good minimal model over $Z$ if $(\overline{X},\overline{\Delta}+\overline{G}+\overline{A})$ has a good minimal model over $Z$. 
By replacing $(X,\Delta)$ and $A$ with $(\overline{X},\overline{\Delta}+\overline{G})$ and $\overline{A}$ respectively, we may assume that $X$ is $\mathbb{Q}$-factorial and ${\rm Nklt}(X,\Delta)= {\rm Nklt}(X,\Delta+A)$. 
\end{step1}

\begin{step1}\label{step2--thm--from-nonvanish-to-minmodel}
In this step, we construct a projective birational morphism $f \colon X' \to X$ and some $\mathbb{R}$-divisors on $X'$. 

By applying Proposition \ref{prop--dlt-resol-linsystem} to $\pi \colon X \to Z$, $(X,\Delta)$, and $K_{X}+\Delta+A$, we obtain a projective birational morphism $f \colon X' \to X$ from a $\mathbb{Q}$-factorial variety $X'$, an $\mathbb{R}$-divisor $M'$ on $X'$, and an effective $\mathbb{R}$-divisor $F'$ on $X'$ satisfying the following properties.
\begin{itemize}
\item
$f^{*}(K_{X}+\Delta+A)\sim_{\mathbb{R},\,Z} M'+F'$, 

\item
$M'$ is semi-ample over $Z$ and ${\rm Supp}\,F'= {\rm Bs}|f^{*}(K_{X}+\Delta+A)/Z|_{\mathbb{R}}$, 

\item
putting
$$\Gamma'=f_{*}^{-1}\Delta+\sum_{\text{$E_{i}$:$f$-exceptional}}E_{i}+{\rm Supp}\,F',$$
then $(X',\Gamma'^{<1}+{\rm Supp}\,\Gamma'^{\geq 1})$ is a $\mathbb{Q}$-factorial dlt pair, and

\item
for any $f$-exceptional prime divisor $E_{i}$ on $X'$, at least one of $a(E_{i},X,\Delta) \leq  -1$ and $E_{i} \subset {\rm Bs}|f^{*}(K_{X}+\Delta+A)/Z|_{\mathbb{R}}$ holds. 
\end{itemize}
We may write
$$K_{X'}+\Delta'=f^{*}(K_{X}+\Delta)+E'$$
for some effective $\mathbb{R}$-divisors $\Delta'$ and $E'$ that have no common components. 
We put 
$$B':=\sum_{F_{j}\subset {\rm Supp}\,F'}(1-{\rm coeff}_{F_{j}}(\Delta'))F_{j}.$$
Since ${\rm Bs}|K_{X}+\Delta+A/Z|_{\mathbb{R}} \cap {\rm Nklt}(X,\Delta)=\emptyset$, the inequality $a(F_{j},X,\Delta) > -1$ holds for any component $F_{j}$ of $F'$. Therefore $B'$ is effective. 
Since $X$ is $\mathbb{Q}$-factorial, there exists an effective $f$-exceptional $\mathbb{R}$-divisor $G'$ on $X'$ such that $-G'$ is $f$-ample. 
Rescaling $G'$, we may assume that $-G'+f^{*}A$ is ample over $Z$. 
We can write 
$$G'=\Theta'+\Phi'$$
for some $\Theta' \geq 0$ and $\Phi' \geq 0$ such that every component of $\Theta'$ is a component of $B'$ and any component of $\Phi'$ is not a component of $B'$. 
By rescaling $G'$ again, we may further assume that $B' -\Theta' \geq0$. 
Finally, we take a general member $A' \in |-G'+f^{*}A/Z|_{\mathbb{R}}$ such that $(X',\Gamma'^{<1}+{\rm Supp}\,\Gamma'^{\geq 1}+A')$ is a $\mathbb{Q}$-factorial dlt pair. Then $A'$ is ample over $Z$. 

By the above discussion, we obtain a projective birational morphism $f \colon X' \to X$ satisfying the properties of Proposition \ref{prop--dlt-resol-linsystem}, an $\mathbb{R}$-divisor $M'$ on $X'$, effective $\mathbb{R}$-divisors $F'$, $\Gamma'$, $\Delta'$, $E'$, $B'$, $G'$, $\Theta'$, and $\Phi'$ on $X'$, and a general $(\pi \circ f)$-ample $\mathbb{R}$-divisor $A'$ on $X'$ such that
\begin{itemize}
\item
we can write 
$$f^{*}(K_{X}+\Delta+A)\sim_{\mathbb{R},\,Z} M'+F'$$ 
such that $M'$ is semi-ample over $Z$ and ${\rm Supp}\,F'= {\rm Bs}|f^{*}(K_{X}+\Delta+A)/Z|_{\mathbb{R}}$, 

\item
$\Delta'$ and $E'$ satisfy
$$K_{X'}+\Delta'=f^{*}(K_{X}+\Delta)+E'$$
and the condition that $\Delta'$ and $E'$ have no common components, 
\item
$B'$ is defined by
$$B':=\sum_{F_{j}\subset {\rm Supp}\,F'}(1-{\rm coeff}_{F_{j}}(\Delta'))F_{j},$$
\item
$G'$ is $f$-exceptional and we may write
$$G'=\Theta'+\Phi'$$
such that $B'-\Theta' \geq 0$ and any component of $\Phi'$ is not a component of $B'$,
\item
$\Gamma'$ is defined by
$$\Gamma'=f_{*}^{-1}\Delta+\sum_{\text{$E_{i}$:$f$-exceptional}}E_{i}+{\rm Supp}\,F',$$
and 
\item
we have
$$A' \sim_{\mathbb{R},\,Z}-G'+f^{*}A$$ 
such that $(X',\Gamma'^{<1}+{\rm Supp}\,\Gamma'^{\geq 1}+A')$ is a $\mathbb{Q}$-factorial dlt pair. 
\end{itemize}
The $\mathbb{R}$-divisor $\Gamma'$ and its property will be used in Step \ref{step5--thm--from-nonvanish-to-minmodel}. 
\end{step1}

\begin{step1}\label{step3--thm--from-nonvanish-to-minmodel}
In this step, we study properties of the effective $\mathbb{R}$-divisors defined in Step \ref{step2--thm--from-nonvanish-to-minmodel}. 

By construction, the discrepancy of any component of $E'$ with respect to $(X,\Delta)$ is positive. 
Since at least $a(E_{i},X,\Delta) \leq  -1$ and $E_{i} \subset {\rm Bs}|f^{*}(K_{X}+\Delta+A)/Z|_{\mathbb{R}}$ holds for any $f$-exceptional prime divisor $E_{i}$ on $X'$, we have ${\rm Supp}\,E'\subset {\rm Bs}|f^{*}(K_{X}+\Delta+A)/Z|_{\mathbb{R}}$. 
Now ${\rm Supp}\,B' \subset {\rm Supp}\,F'$ by construction, and $B'-\Theta' \geq 0$ implies ${\rm Supp}\,\Theta' \subset {\rm Supp}\,B'$. 
From these facts and the property ${\rm Supp}\,F'= {\rm Bs}|f^{*}(K_{X}+\Delta+A)/Z|_{\mathbb{R}}$, we have
$${\rm Supp}\,(E'+B'+\Theta') \subset {\rm Supp}\,F'={\rm Bs}|f^{*}(K_{X}+\Delta+A)/Z|_{\mathbb{R}}.$$

Since ${\rm Bs}|K_{X}+\Delta+A/Z|_{\mathbb{R}} \cap {\rm Nklt}(X,\Delta)=\emptyset$, the inequality $a(F_{j},X,\Delta) > -1$ holds for any component $F_{j}$ of $F'$. 
By the definitions of $B'$ and $F'$, we have
$${\rm Supp}\,B' = {\rm Supp}\,F'={\rm Bs}|f^{*}(K_{X}+\Delta+A)/Z|_{\mathbb{R}}.$$

We also have ${\rm Supp}\,B' \cap {\rm Supp}\,\Delta'^{\geq 1} = \emptyset$ by construction of $B'$. 
The definition of $\Phi'$ shows that $\Phi'$ is $f$-exceptional and any component of $\Phi'$ is not an irreducible component of ${\rm Supp}\,B'$. 
Since at least $a(E_{i},X,\Delta) \leq  -1$ and $E_{i} \subset {\rm Bs}|f^{*}(K_{X}+\Delta+A)/Z|_{\mathbb{R}}$ holds for any $f$-exceptional prime divisor $E_{i}$ on $X'$, the discrepancy of any component of $\Phi'$ with respect to $(X,\Delta)$ is not greater than $-1$. 
Then
$${\rm Supp}\,B' \cap {\rm Supp}\,(\Delta'^{\geq 1}+\Phi')\subset {\rm Bs}|f^{*}(K_{X}+\Delta+A)/Z|_{\mathbb{R}} \cap f^{-1}({\rm Nklt}(X,\Delta)) =\emptyset.$$

By the definition of $B'$ and the equality ${\rm Supp}\,B' = {\rm Supp}\,F'$, it follows that
$${\rm coeff}_{F_{j}}(\Delta'+B'+\Phi')={\rm coeff}_{F_{j}}(\Delta'+B')=1$$
 for any component $F_{j}$ of $F'$. 
 
By construction in Step \ref{step2--thm--from-nonvanish-to-minmodel}, we can write
\begin{equation*}
\begin{split}
K_{X'}+\Delta'+f^{*}A+(B'-\Theta')=f^{*}(K_{X}+\Delta+A)+E'+(B'-\Theta')
\end{split}
\end{equation*}
and $B'-\Theta'$ is effective. 
We can also write
\begin{equation*}
\begin{split}
K_{X'}+\Delta'+f^{*}A+(B'-\Theta')\sim_{\mathbb{R},\,Z}K_{X'}+\Delta'+B'+\Phi'+ A'.
\end{split}
\end{equation*}
\end{step1}

From the above arguments, the following properties hold.

\begin{enumerate}[(a)]
\item \label{proof-thm--from-nonvanish-to-minmodel-(a)}
$B'-\Theta'$ is effective and 
$${\rm Supp}\,(E'+B'+\Theta')={\rm Supp}\,B'= {\rm Supp}\,F'={\rm Bs}|f^{*}(K_{X}+\Delta+A)/Z|_{\mathbb{R}},$$
\item \label{proof-thm--from-nonvanish-to-minmodel-(b)}
we have
$${\rm Supp}\,B' \cap {\rm Supp}\,(\Delta'^{\geq 1}+\Phi')=\emptyset,$$
\item \label{proof-thm--from-nonvanish-to-minmodel-(c)}
the equality
$${\rm coeff}_{F_{j}}(\Delta'+B'+\Phi')=1$$
holds for any component $F_{j}$ of $F'$, and 
\item \label{proof-thm--from-nonvanish-to-minmodel-(d)}
we have
 \begin{equation*}
\begin{split}
&K_{X'}+\Delta'+f^{*}A+(B'-\Theta')=f^{*}(K_{X}+\Delta+A)+E'+(B'-\Theta'), \qquad {\rm and}\\
&K_{X'}+\Delta'+f^{*}A+(B'-\Theta')\sim_{\mathbb{R},\,Z}K_{X'}+\Delta'+B'+\Phi'+ A'.
\end{split}
\end{equation*}
\end{enumerate}

\begin{step1}\label{step4--thm--from-nonvanish-to-minmodel}
By (\ref{proof-thm--from-nonvanish-to-minmodel-(a)}) and (\ref{proof-thm--from-nonvanish-to-minmodel-(d)}), Proposition \ref{prop--minmodel-biratmodel} for $\pi \colon X \to Z$, $(X,\Delta+A)$, $f \colon X' \to X$, and $(X',\Delta'+f^{*}A+(B'-\Theta'))$ implies that to prove the existence of a $\mathbb{Q}$-factorial good minimal model of $(X,\Delta+A)$ over $Z$, it is sufficient to prove the existence of a birational contraction $\phi \colon X' \dashrightarrow X''$ over $Z$, where $X''$ is $\mathbb{Q}$-factorial and projective over $Z$, such that the $\mathbb{R}$-divisor $\phi_{*}(K_{X'}+\Delta'+f^{*}A+(B'-\Theta'))$ is semi-ample over $Z$ and any prime divisor contracted by $\phi$ is an irreducible component of ${\rm Bs}|K_{X'}+\Delta'+f^{*}A+(B'-\Theta')/Z|_{\mathbb{R}}$. 
\end{step1}

\begin{step1}\label{step5--thm--from-nonvanish-to-minmodel}
In this step, we will check that $(X',\Delta'+B'+\Phi'+ A')$ is $\mathbb{Q}$-factorial dlt on a neighborhood of ${\rm Supp}\,F'$. 

We set $U':=X' \setminus {\rm Supp}\,(\Delta'^{\geq 1}+\Phi')$. 
By the definitions of $B'$, $\Delta'$, and $\Gamma'$ in Step \ref{step2--thm--from-nonvanish-to-minmodel}, the $\mathbb{R}$-divisor $(\Delta'+B')|_{U'}$ is a boundary $\mathbb{R}$-divisor and we have 
$$(\Delta'+B')|_{U'} \leq \Gamma'|_{U'}.$$
Since $(X',\Gamma'^{<1}+{\rm Supp}\,\Gamma'^{\geq 1}+A')$ is a $\mathbb{Q}$-factorial dlt pair, which is the final property in the list of properties in Step \ref{step2--thm--from-nonvanish-to-minmodel}, we see that $(U',(\Delta'+B'+\Phi'+ A')|_{U'})$ is a $\mathbb{Q}$-factorial dlt pair. 
By (\ref{proof-thm--from-nonvanish-to-minmodel-(a)}) and (\ref{proof-thm--from-nonvanish-to-minmodel-(b)}) in Step \ref{step3--thm--from-nonvanish-to-minmodel}, we have
$$ {\rm Supp}\,F' \cap {\rm Supp}\,(\Delta'^{\geq 1}+\Phi') = {\rm Supp}\,B' \cap {\rm Supp}\,(\Delta'^{\geq 1}+\Phi')=\emptyset.$$
This shows that ${\rm Supp}\,F' \subset U'$. 
Therefore, $(X',\Delta'+B'+\Phi'+ A')$ is $\mathbb{Q}$-factorial dlt on a neighborhood of ${\rm Supp}\,F'$. 
\end{step1}

\begin{step1}\label{step6--thm--from-nonvanish-to-minmodel}
In this step, we will construct the birational contraction $\phi \colon X' \dashrightarrow X''$ over $Z$ mentioned in Step \ref{step4--thm--from-nonvanish-to-minmodel}. 

Let $M'$ be the $\mathbb{R}$-divisor on $X'$ defined in the construction of $f \colon X' \to X$ in Step \ref{step2--thm--from-nonvanish-to-minmodel}. 
We may write $M'=\sum_{k}r_{k}M'^{(k)}$ for some positive real numbers $r_{k}$ and Cartier divisors $M'^{(k)}$ on $X'$ that are semi-ample over $Z$. 
We take $\alpha \in \mathbb{R}_{>0}$ so that $\alpha r_{k}>2 \cdot{\rm dim}\,X'$ for all $k$. 
We fix a general element $N' \in |M'/Z|_{\mathbb{R}}$ such that 
$${\rm Nklt}(X', \Delta'+B'+\Phi'+ A'+\alpha N')={\rm Nklt}(X', \Delta'+B'+\Phi'+ A').$$
By (\ref{proof-thm--from-nonvanish-to-minmodel-(a)}) and (\ref{proof-thm--from-nonvanish-to-minmodel-(d)}) in Step \ref{step3--thm--from-nonvanish-to-minmodel}, we have
\begin{equation*}
\begin{split}
&{\rm Bs}|K_{X'}+\Delta'+B'+\Phi'+ A'+\alpha N'/Z|_{\mathbb{R}}\\
\subset&{\rm Bs}|K_{X'}+\Delta'+B'+\Phi'+ A'/Z|_{\mathbb{R}}\\
= & {\rm Bs}|f^{*}(K_{X}+\Delta+A)+E'+(B'-\Theta')/Z|_{\mathbb{R}}\\
\subset& {\rm Bs}|f^{*}(K_{X}+\Delta+A)/Z|_{\mathbb{R}} \cup {\rm Supp}\,(E'+(B'-\Theta'))={\rm Supp\,}F'. 
\end{split}
\end{equation*}
By this property and the dlt property of $(X',\Delta'+B'+\Phi'+ A')$ on a neighborhood of ${\rm Supp}\,F'$, the non-nef locus of $K_{X'}+\Delta'+B'+\Phi'+ A'+\alpha N'$ over $Z$ does not intersect ${\rm Nlc}(X', \Delta'+B'+\Phi'+ A'+\alpha N')$. 
 
By Corollary \ref{cor--mmp-nomralpair-Qfacdlt}, we can run a $(K_{X'}+\Delta'+B'+\Phi'+ A'+\alpha N')$-MMP over $Z$ with scaling of an ample divisor 
$$(X',\Delta'+B'+\Phi'+ A'+\alpha N') \dashrightarrow \cdots \dashrightarrow (X'_{l},\Delta'_{l}+B'_{l}+\Phi'_{l}+ A'_{l}+\alpha N'_{l}) \dashrightarrow \cdots.$$
By the choice of $\alpha$ and the length of extremal rays (\cite[Theorem 1.5 (iii)]{fujino-morihyper}), every $N'_{l}$ trivially intersects the curves contracted by the extremal contraction of the MMP. 
Hence, the $(K_{X'}+\Delta'+B'+\Phi'+ A'+\alpha N')$-MMP is a $(K_{X'}+\Delta'+B'+\Phi'+ A')$-MMP and $N'_{l}$ is semi-ample over $Z$ for every $l$. 
By this fact and (\ref{proof-thm--from-nonvanish-to-minmodel-(d)}) in Step \ref{step3--thm--from-nonvanish-to-minmodel}, prime divisors contracted in this MMP are contained in 
$${\rm Bs}|K_{X'}+\Delta'+B'+\Phi'+ A'/Z|_{\mathbb{R}}={\rm Bs}|K_{X'}+\Delta'+f^{*}A+(B'-\Theta')/Z|_{\mathbb{R}}.$$
By (\ref{proof-thm--from-nonvanish-to-minmodel-(d)}) in Step \ref{step3--thm--from-nonvanish-to-minmodel} and the relation $f^{*}(K_{X}+\Delta+A)\sim_{\mathbb{R},\,Z}M'+F'$, which is the first property of $f \colon X' \to X$ in Step \ref{step2--thm--from-nonvanish-to-minmodel}, we have
\begin{equation*}
\begin{split}
K_{X'}+\Delta'+B'+\Phi'+ A'+\alpha N'\sim_{\mathbb{R},\,Z} &\;f^{*}(K_{X}+\Delta+A)+E'+(B'-\Theta')+\alpha N'\\
\sim_{\mathbb{R},\,Z} & (\alpha+1)N'+(F'+E'+(B'-\Theta')).
\end{split}
\end{equation*}
By (\ref{proof-thm--from-nonvanish-to-minmodel-(a)}) in Step \ref{step3--thm--from-nonvanish-to-minmodel}, we have 
$${\rm Supp}\,(F'+E'+(B'-\Theta'))={\rm Supp}\,F'.$$ 
Since every $N'_{l}$ trivially intersects the curves contracted by the extremal contraction of the $(K_{X'}+\Delta'+B'+\Phi'+ A'+\alpha N')$-MMP, the non-isomorphic locus of the MMP is contained in ${\rm Supp}\,F'$. 
By (\ref{proof-thm--from-nonvanish-to-minmodel-(c)}) in Step \ref{step3--thm--from-nonvanish-to-minmodel} and the dlt property of $(X',\Delta'+B'+\Phi'+ A')$ around ${\rm Supp}\,F'$, we may apply the special termination \cite{fujino-sp-ter} to each lc center contained in ${\rm Supp}\,F'$ with the aid of \cite[Theorem E]{bchm} and the $(\pi \circ f)$-ampleness of $A'$.  
We note that we may apply \cite[Theorem E]{bchm} to our setting because we can use $A'_{l}$ and \cite[Proof of Lemma 3.10.11 (2)]{bchm} to produce an ample $\mathbb{R}$-divisor on $X'_{l}$ for every $l$. 
We see that the $(K_{X'}+\Delta'+B'+\Phi'+ A'+\alpha N')$-MMP over $Z$ terminates with a $\mathbb{Q}$-factorial minimal model over $Z$
$$(X',\Delta'+B'+\Phi'+ A'+\alpha N') \dashrightarrow (X'_{n},\Delta'_{n}+B'_{n}+\Phi'_{n}+ A'_{n}+\alpha N'_{n}).$$
Moreover, $N'_{n}$ is semi-ample over $Z$ by construction. 

Let $F'_{n}$ and $E'_{n}$ be the birational transforms of $F'$ and $E'$ on $X'_{n}$, respectively. 
By Corollary \ref{cor--mmpwithscaling-normalpair}, we may run a $(K_{X'_{n}}+\Delta'_{n}+B'_{n}+\Phi'_{n}+ A'_{n})$-MMP over $Z$ with scaling of $\alpha N'_{n}$. 
By the same argument as above, we may apply the special termination \cite{fujino-sp-ter} to each lc center contained in ${\rm Supp}\,F'_{n}$ with the aid of \cite[Theorem E]{bchm}. 
We note that we have
$$K_{X'_{n}}+\Delta'_{n}+B'_{n}+\Phi'_{n}+ A'_{n}+ t N'_{n}\sim_{\mathbb{R},\,Z} (t+1)N'_{n}+(F'_{n}+E'_{n}+(B'_{n}-\Theta'_{n}))$$ 
for all $0 \leq t \leq \alpha$, and ${\rm Supp}\,(F'_{n}+E'_{n}+(B'_{n}-\Theta'_{n}))={\rm Supp}\,F'_{n}$ by (\ref{proof-thm--from-nonvanish-to-minmodel-(a)}) in Step \ref{step3--thm--from-nonvanish-to-minmodel}. 
Hence, the non-isomorphic locus of the $(K_{X'_{n}}+\Delta'_{n}+B'_{n}+\Phi'_{n}+ A'_{n})$-MMP is contained in ${\rm Supp}\,F'_{n}$. 
Furthermore, if all components of $F'_{n}$ are contracted after finitely many steps then the MMP terminates.  
We see that the MMP terminates with a $\mathbb{Q}$-factorial minimal model over $Z$ 
$$(X'_{n}, \Delta'_{n}+B'_{n}+\Phi'_{n}+ A'_{n}) \dashrightarrow (X'', \Delta''+B''+\Phi''+ A'').$$
Let $\phi \colon X' \dashrightarrow X'_{n} \dashrightarrow X''$ be the composition. 
Then
$$(X', \Delta'+B'+\Phi'+ A') \dashrightarrow (X'', \Delta''+B''+\Phi''+ A'')$$
is a sequence of steps of a $(K_{X'}+\Delta'+B'+\Phi'+ A')$-MMP over $Z$ to a minimal model over $Z$. 

We check that $K_{X''}+\Delta''+B''+\Phi''+ A''$ is semi-ample over $Z$. 
By the argument in \cite[Proof of Lemma 3.10.11 (2)]{bchm}, we may write $A'' \sim_{\mathbb{R},\,Z}\Omega''+L''$ such that $\Omega''$ is effective, $L''$ is ample over $Z$, and 
$${\rm Nlc}(X'', \Delta''+B''+\Phi''+ A'')={\rm Nlc}(X'', \Delta''+B''+\Phi''+ \Omega'').$$
Moreover, since 
$$K_{X''}+\Delta''+B''+\Phi''+ A'' \sim_{\mathbb{R},\,Z} \phi_{*}N'+\phi_{*}(F'+E'+(B'-\Theta'))$$
and ${\rm Nlc}(X', \Delta'+B'+\Phi'+ A') \cap {\rm Supp}\,F'= \emptyset$ (see Step \ref{step5--thm--from-nonvanish-to-minmodel}), the semi-ampleness of $N'$ over $Z$ and Remark \ref{rem--mmp-basic} imply that the $\mathbb{R}$-line bundle
$$(K_{X''}+\Delta''+B''+\Phi''+ \Omega''+L'')|_{{\rm Nlc}(X'', \Delta''+B''+\Phi''+ \Omega'')}$$
is semi-ample over $Z$. 
Now we apply Theorem \ref{thm--abundance-quasi-log} to $(X'', \Delta''+B''+\Phi''+ \Omega'')$ and $L''$, and we see that
$$K_{X''}+\Delta''+B''+\Phi''+ A'' \sim_{\mathbb{R},\,Z}K_{X''}+\Delta''+B''+\Phi''+ \Omega''+L''$$
is semi-ample over $Z$. 

By (\ref{proof-thm--from-nonvanish-to-minmodel-(d)}) in Step \ref{step3--thm--from-nonvanish-to-minmodel}, the following properties hold. 
\begin{itemize}
\item
$X''$ is $\mathbb{Q}$-factorial, 
\item
$\phi$ is a sequence of steps of a $(K_{X'}+\Delta'+f^{*}A+(B'-\Theta'))$-MMP over $Z$ that only contracts prime divisors contained in ${\rm Bs}|K_{X'}+\Delta'+f^{*}A+(B'-\Theta')/Z|_{\mathbb{R}}$, and
\item
the $\mathbb{R}$-divisor
$$\phi_{*}(K_{X'}+\Delta'+f^{*}A+(B'-\Theta')) \sim_{\mathbb{R},\,Z}\phi_{*}(K_{X'}+\Delta'+B'+\Phi'+ A')$$
is semi-ample over $Z$.
\end{itemize}
\end{step1}
As discussed in Step \ref{step4--thm--from-nonvanish-to-minmodel}, Proposition \ref{prop--minmodel-biratmodel} implies the existence of a $\mathbb{Q}$-factorial good minimal model of $(X,\Delta+A)$ over $Z$. 
We finish the proof. 
\end{proof}

\begin{cor}\label{cor-minmodel-explicit-construction}
Let $\pi \colon X \to Z$ be a projective morphism of normal quasi-projective varieties. 
Let $(X,\Delta)$ be a normal pair and let $A$ be an effective $\pi$-ample $\mathbb{R}$-divisor on $X$ such that ${\rm Nklt}(X,\Delta) \neq \emptyset$ and ${\rm Bs}|K_{X}+\Delta+A/Z|_{\mathbb{R}} \cap {\rm Nklt}(X,\Delta+A) = \emptyset$. 
Then there exist a projective birational morphism $f \colon X' \to X$ from a $\mathbb{Q}$-factorial variety $X'$ and a birational contraction $\phi \colon X' \dashrightarrow X''$ over $Z$ to a $\mathbb{Q}$-factorial variety $X''$, which is projective over $Z$, satisfying the following.
\begin{itemize}
\item
At least one of $a(E_{i},X,\Delta) \leq  -1$ and $E_{i} \subset {\rm Bs}|f^{*}(K_{X}+\Delta+A)/Z|_{\mathbb{R}}$ holds for any $f$-exceptional prime divisor $E_{i}$ on $X'$, 
\item
$\phi$ contracts all divisorial components of ${\rm Bs}|f^{*}(K_{X}+\Delta+A)/Z|_{\mathbb{R}}$ and $\phi$ is an isomorphism on $X' \setminus {\rm Bs}|f^{*}(K_{X}+\Delta+A)/Z|_{\mathbb{R}}$, and 
\item
writing
$$K_{X'}+\Delta'=f^{*}(K_{X}+\Delta)+E'$$
for some effective $\mathbb{R}$-divisors $\Delta'$ and $E'$ on $X'$ having no common components, then $(X'',\phi_{*}(\Delta'+f^{*}A))$ is a $\mathbb{Q}$-factorial good minimal model of $(X,\Delta+A)$ over $Z$. 
\end{itemize}
\end{cor}

\begin{proof}
By Theorem \ref{thm--from-nonvanish-to-minmodel}, $(X,\Delta+A)$ has a good minimal model over $Z$. 
Thus, $K_{X}+\Delta+A$ birationally has the Nakayama--Zariski decomposition over $Z$ whose positive part is semi-ample over $Z$. 
Applying Proposition \ref{prop--dlt-resol-linsystem-2} to $\pi \colon X \to Z$, $(X,\Delta)$ and $K_{X}+\Delta+A$, we get a projective birational morphism $f \colon X' \to X$ from a $\mathbb{Q}$-factorial variety $X'$ such that putting $M'$ and $F'$ as the positive and the negative parts of the Nakayama--Zariski decomposition of $f^{*}(K_{X}+\Delta+A)$ over $Z$, respectively, then the following properties hold.
\begin{itemize}
\item
$f^{*}(K_{X}+\Delta+A)\sim_{\mathbb{R},\,Z} M'+F'$, 

\item
$M'$ is semi-ample over $Z$ and ${\rm Supp}\,F'= {\rm Bs}|f^{*}(K_{X}+\Delta+A)/Z|_{\mathbb{R}}$, 

\item
putting
$$\Gamma'=f_{*}^{-1}\Delta+\sum_{\text{$E_{i}$:$f$-exceptional}}E_{i}+{\rm Supp}\,F',$$
then $(X',\Gamma'^{<1}+{\rm Supp}\,\Gamma'^{\geq 1})$ is a $\mathbb{Q}$-factorial dlt pair, and

\item
for any $f$-exceptional prime divisor $E_{i}$ on $X'$, at least one of $a(E_{i},X,\Delta) \leq  -1$ and $E_{i} \subset {\rm Bs}|f^{*}(K_{X}+\Delta+A)/Z|_{\mathbb{R}}$ holds. 
\end{itemize}
We write
$$K_{X'}+\Delta'=f^{*}(K_{X}+\Delta)+E'$$
for some effective $\mathbb{R}$-divisors $\Delta'$ and $E'$ that have no common components. 
We put 
$$B':=\sum_{F_{j}\subset {\rm Supp}\,F'}(1-{\rm coeff}_{F_{j}}(\Delta'))F_{j}.$$
By the same argument as in the proof of Proposition \ref{prop--exist-minmodel-boundary-change}, we get a sequence of a steps of a
$(K_{X'}+\Delta'+B'+f^{*}A+\alpha M')$-MMP 
$$\phi \colon X' \dashrightarrow X''$$
over $Z$ such that $\phi$ is also a
$(K_{X'}+\Delta'+B'+f^{*}A)$-MMP and $\phi_{*}(K_{X'}+\Delta'+B'+f^{*}A)$ is semi-ample over $Z$. 
This $\phi$ satisfies the following.
\begin{itemize}
\item
$X''$ is $\mathbb{Q}$-factorial, 
\item
$\phi_{*}(K_{X'}+\Delta'+B'+f^{*}A)$ is semi-ample over $Z$, and
\item
$\phi$ contracts all divisorial components of ${\rm Bs}|f^{*}(K_{X}+\Delta+A)/Z|_{\mathbb{R}}$ and $\phi$ is an isomorphism on $X' \setminus {\rm Bs}|f^{*}(K_{X}+\Delta+A)/Z|_{\mathbb{R}}$. 
\end{itemize}
Therefore, $f \colon X' \to X$ and $\phi \colon X' \dashrightarrow X''$ are what we wanted (see Proposition \ref{prop--minmodel-biratmodel}). 
\end{proof}

\subsection{Termination of minimal model program}

In this subsection, we study the termination of MMP with scaling of an ample divisor. 

\begin{thm}\label{thm--from-minimodel-to-termi}
Let $\pi \colon X \to Z$ be a projective morphism of normal quasi-projective varieties. 
Let $(X,\Delta)$ be a normal pair and let $A$ be an effective $\pi$-ample $\mathbb{R}$-divisor on $X$ such that ${\rm Nklt}(X,\Delta) \neq \emptyset$ and ${\rm Bs}|K_{X}+\Delta+A/Z|_{\mathbb{R}} \cap {\rm Nklt}(X,\Delta) = \emptyset$. 
Let 
$$(X,\Delta+A)=:(X_{1},\Delta_{1}+A_{1}) \dashrightarrow \cdots \dashrightarrow (X_{i},\Delta_{i}+A_{i}) \dashrightarrow \cdots$$
be a sequence of steps of a $(K_{X}+\Delta+A)$-MMP over $Z$ with scaling of $A$ such that if we put $$\lambda_{i}:=\{\mu \in \mathbb{R}_{\geq 0}\,| \,\text{$(K_{X_{i}}+\Delta_{i}+A_{i})+\mu A_{i}$ is nef over $Z$} \}$$ for each $i \geq 1$, then 
${\rm lim}_{i \to \infty} \lambda_{i}=0$.  
Then $\lambda_{m}=0$ for some $m$ and $(X_{m},\Delta_{m}+A_{m})$ is a good minimal model of $(X,\Delta+A)$ over $Z$. 
\end{thm}

\begin{proof}
Replacing $A$ if necessary, we may assume ${\rm Nklt}(X,\Delta)={\rm Nklt}(X,\Delta+A)$ as closed subschemes of $X$. 
Then ${\rm Bs}|K_{X}+\Delta+A/Z|_{\mathbb{R}} \cap {\rm Nklt}(X,\Delta+A) = \emptyset$. 

We divide the proof into several steps.

\begin{step2}\label{step1--thm--from-minimodel-to-termi}
In this step, we will construct some varieties and $\mathbb{R}$-divisors, and we will discuss properties of them. 

By Corollary \ref{cor-minmodel-explicit-construction}, there exists a projective birational morphism $f \colon X' \to X$ from a $\mathbb{Q}$-factorial variety $X'$ and a birational contraction $\phi \colon X' \dashrightarrow X''$ over $Z$ to a $\mathbb{Q}$-factorial variety $X''$, which is projective over $Z$, satisfying the following.
\begin{itemize}
\item
At least one of $a(E_{i},X,\Delta) \leq  -1$ and $E_{i} \subset {\rm Bs}|f^{*}(K_{X}+\Delta+A)/Z|_{\mathbb{R}}$ holds for any $f$-exceptional prime divisor $E_{i}$ on $X'$, 
\item
$\phi$ contracts all divisorial components of ${\rm Bs}|f^{*}(K_{X}+\Delta+A)/Z|_{\mathbb{R}}$ and $\phi$ is an isomorphism on $X' \setminus {\rm Bs}|f^{*}(K_{X}+\Delta+A)/Z|_{\mathbb{R}}$, and 
\item
writing
$$K_{X'}+\Delta'=f^{*}(K_{X}+\Delta)+E'$$
for some effective $\mathbb{R}$-divisors $\Delta'$ and $E'$ on $X'$ having no common components, then $(X'',\phi_{*}(\Delta'+f^{*}A))$ is a $\mathbb{Q}$-factorial good minimal model of $(X, \Delta+A)$ over $Z$. 
\end{itemize}
We put $\Delta''=\phi_{*}\Delta'$ and $A''=\phi_{*}f^{*}A$. 
By the equality $K_{X'}+\Delta'=f^{*}(K_{X}+\Delta)+E'$ and the first property, we have ${\rm Supp}\,E' \subset {\rm Bs}|f^{*}(K_{X}+\Delta+A)/Z|_{\mathbb{R}}$. 
By this inclusion and the equality ${\rm Bs}|K_{X}+\Delta+A/Z|_{\mathbb{R}} \cap {\rm Nklt}(X,\Delta+A) = \emptyset$, we have
\begin{equation*}
\begin{split}
&{\rm Bs}|f^{*}(K_{X}+\Delta+A)/Z|_{\mathbb{R}}\cap f^{-1}({\rm Nklt}(X,\Delta+A))=\emptyset, \quad {\rm and}\\
&{\rm Supp}\,E' \cap f^{-1}({\rm Nklt}(X,\Delta+A))=\emptyset.
\end{split}
\end{equation*}
By the definition of good minimal models (Definition \ref{defn--minmodel}) of $(X,\Delta+A)$ over $Z$ and the negativity lemma, the inequality
$$a(P,X,\Delta+A) \leq a(P,X'',\Delta''+A'')$$
holds for any prime divisor $P$ over $X$. 
By these facts and the fact that $\phi \colon X' \dashrightarrow X''$ is an isomorphism on $X' \setminus {\rm Bs}|f^{*}(K_{X}+\Delta+A)/Z|_{\mathbb{R}}$, we see that $\phi^{-1}$ is an isomorphism on a neighborhood of ${\rm Nklt}(X'',\Delta''+A'')$ and the following properties hold.
\begin{enumerate}[(i)]\setcounter{enumi}{0}
\item\label{thm--from-minimodel-to-termi-(I)}
${\rm Bs}|A''/Z|_{\mathbb{R}} \cap {\rm Nklt}(X'',\Delta''+A'')=\emptyset$, and
\item\label{thm--from-minimodel-to-termi-(II)}
for any prime divisor $P$ over $X''$, if $a(P,X'',\Delta''+A'') \leq -1$ then 
$$a(P,X,\Delta+A) = a(P,X'',\Delta''+A'').$$ 
\end{enumerate}

We can write $A'' \sim_{\mathbb{R},\,Z}H''+B''$ for some effective $\mathbb{R}$-divisor $B''$ on $X''$ and an effective $\mathbb{R}$-divisor $H''$ on $X''$ which is ample over $Z$. 
We fix $u \in \mathbb{R}_{>0}$ such that 
\begin{itemize}
\item
${\rm Nklt}(X'',\Delta''+A'')={\rm Nklt}(X'',\Delta''+A''+uB''+uH'')$, 
\item
${\rm Nklt}(X'',\Delta''+A'')={\rm Nklt}(X'',\Delta''+(1+u)A'')$, and 
\item
for any prime divisor $D$ on $X$ which is exceptional over $X''$, we have
$$a(D,X,\Delta+(1+u)A) < a(D,X'', \Delta''+(1+u)A'').$$
\end{itemize}
Then $\phi^{-1}$ is an isomorphism on a neighborhood of ${\rm Nklt}(X'',\Delta''+(1+u)A'')$. 
By the equality $K_{X'}+\Delta'=f^{*}(K_{X}+\Delta)+E'$ and the definitions of $A''$ and $\Delta''$, we see that 
\begin{enumerate}[(i)]\setcounter{enumi}{2}
\item \label{thm--from-minimodel-to-termi-(III)}
for any prime divisor $Q$ over $X''$, if  $a(Q,X'',\Delta''+(1+u)A'') \leq -1$ or $Q$ is not exceptional over $X''$ then 
$$a(Q,X,\Delta+(1+u)A)=a(Q,X'',\Delta''+(1+u)A'').$$ 
\end{enumerate}

Let $X'' \to W$ be the contraction over $Z$ induced by $K_{X''}+\Delta''+A''$. 
Then 
$${\rm Bs}|K_{X''}+(\Delta''+uB''+A'')+uH''/W|_{\mathbb{R}}\subset {\rm Bs}|u(H''+B'')/Z|_{\mathbb{R}}={\rm Bs}|A''/Z|_{\mathbb{R}}.$$
By the relation ${\rm Nklt}(X'',\Delta''+A'')={\rm Nklt}(X'',\Delta''+A''+uB''+uH'')$
and (\ref{thm--from-minimodel-to-termi-(I)}), we see that
$${\rm Bs}|K_{X''}+(\Delta''+uB''+A'')+uH''/W|_{\mathbb{R}}\cap {\rm Nklt}(X'',(\Delta''+A''+uB'')+uH'')=\emptyset.$$
By Theorem \ref{thm--from-nonvanish-to-minmodel}, the pair $(X'',(\Delta''+uB''+A'')+uH'')$ has a good minimal model over $W$. 
We also have
\begin{equation*}
\begin{split}
&{\rm Bs}|K_{X''}+\Delta''+(1+u)A''/W|_{\mathbb{R}}\cap {\rm Nklt}(X'',\Delta''+(1+u)A'')\\
\subset&{\rm Bs}|K_{X''}+(\Delta''+uB''+A'')+uH''/W|_{\mathbb{R}}\cap {\rm Nklt}(X'',\Delta''+A''+uB''+uH'')=\emptyset. 
\end{split}
\end{equation*}
Since $\Delta''+(1+u)A'' \sim_{\mathbb{R},\,Z}(\Delta''+uB''+A'')+uH''$, the pair $(X'',\Delta''+(1+u)A'')$ has a good minimal model over $W$ by Proposition \ref{prop--exist-minmodel-boundary-change}. 
\end{step2}

\begin{step2}\label{step2--thm--from-minimodel-to-termi}
In this step we define a diagram and an $\mathbb{R}$-divisor used in the proof. 

Let $(X''',\Gamma_{u}''')$ be a good minimal model of $(X'',\Delta''+(1+u)A'')$ over $W$. 
We denote $X'' \to W$ and $X''' \to W$ by $\tau''$ and $\tau'''$, respectively. 
Let $g \colon Y \to X$, $g_{2} \colon Y \to X''$, and $g_{3} \colon Y \to X'''$ be a common resolution of $X\dashrightarrow X'' \dashrightarrow X'''$. 
Now we have the following diagram 
$$\xymatrix{&&&Y \ar[llld]_{g}\ar[ld]^{g_{2}}\ar[rd]^{g_{3}}\\X \ar[ddr]_{\pi}  \ar@{-->}[rr]_{\phi \circ f^{-1}}&&X''\ar[ldd] \ar@{-->}[rr]\ar[rd]_{\tau''}&&X'''\ar[ld]^{\tau'''}\\&&&W.\ar[lld]\\&Z}$$
We define an $\mathbb{R}$-divisor $\Gamma_{0}'''$ on $X'''$ by 
$$K_{X'''}+\Gamma_{0}'''=g_{3*}g_{2}^{*}(K_{X''}+\Delta''+A'').$$ 
Note that $\Gamma_{0}'''$ may not be effective. 
Since $\tau''$ is the contraction induced by $K_{X''}+\Delta''+A''$, we see that $K_{X'''}+\Gamma_{0}'''$ is $\mathbb{R}$-Cartier and $K_{X'''}+\Gamma_{0}'''\sim_{\mathbb{R},\,W}0$. 
In particular, the negativity lemma implies
$$g_{2}^{*}(K_{X''}+\Delta''+A'')=g_{3}^{*}(K_{X'''}+\Gamma_{0}''').$$
Moreover, there exists $\epsilon \in (0,1)$ such that 
$$K_{X'''}+\epsilon \Gamma_{u}'''+(1-\epsilon )\Gamma_{0}''' =\epsilon (K_{X'''}+\Gamma_{u}''')+(1-\epsilon )(K_{X'''}+\Gamma_{0}''')$$
is semi-ample over $Z$. 
Then, for any $\delta \in [0, \epsilon]$, the $\mathbb{R}$-divisor
$$K_{X'''}+\delta \Gamma_{u}'''+(1-\delta )\Gamma_{0}'''$$
is semi-ample over $Z$. 
For any $\delta \in [0, \epsilon]$, we define $E_{Y}^{(\delta)}$ by
\begin{equation*}\tag{$\clubsuit$}\label{thm--from-minimodel-to-termi-(clubsuit)}
\begin{split}
E_{Y}^{(\delta)}:=&g^{*}(K_{X}+\Delta+(1+\delta u)A)-g_{3}^{*}(K_{X'''}+\delta \Gamma_{u}'''+(1-\delta )\Gamma_{0}''')\\
=& (1-\delta)\bigl(g^{*}(K_{X}+\Delta+A)-g_{2}^{*}(K_{X''}+\Delta''+A'')\bigr)\\
&\qquad\qquad \quad \quad\; +\delta \bigl(g^{*}(K_{X}+\Delta+(1+u)A)-g_{3}^{*}(K_{X'''}+ \Gamma_{u}''')\bigr).
\end{split}
\end{equation*}
\end{step2}

\begin{step2}\label{step3--thm--from-minimodel-to-termi}
In this step, we will prove that $E_{Y}^{(\delta)}$ in (\ref{thm--from-minimodel-to-termi-(clubsuit)}) is effective and $g_{3}$-exceptional for any $\delta \in [0,\epsilon]$. 

By construction of $E_{Y}^{(\delta)}$, it follows that
\begin{equation*}
\begin{split}
E_{Y}^{(\delta)}=&g^{*}(K_{X}+\Delta+(1+\delta u)A)-g_{3}^{*}(K_{X'''}+\delta \Gamma_{u}'''+(1-\delta )\Gamma_{0}''')\\
=& (1-\delta)\sum_{P}\bigl(a(P,X'',\Delta''+A'')-a(P,X,\Delta+A)\bigr)P\\
&\qquad\qquad \quad \quad\;+\delta \sum_{P}\bigl(a(P,X''',\Gamma'''_{u})-a(P,X,\Delta+(1+u)A)\bigr)P,
\end{split}
\end{equation*}
where $P$ runs over prime divisors on $Y$. 

We first prove $E_{Y}^{(\delta)} \geq 0$. 
For any prime divisor $\bar{D}$ on $X$, the definition of good minimal model (Definition \ref{defn--minmodel}) and the negativity lemma show
$$a(\bar{D},X,\Delta+A) \leq a(\bar{D},X'',\Delta''+A''),$$ 
and 
$$a(\bar{D},X'',\Delta''+(1+u)A'') \leq a(\bar{D},X''',\Gamma_{u}''').$$ 
If $\bar{D}$ is not exceptional over $X''$, then 
$$a(\bar{D},X,\Delta+(1+u)A)=a(\bar{D},X'',\Delta''+(1+u)A'')$$
 by (\ref{thm--from-minimodel-to-termi-(III)}). 
If $\bar{D}$ is exceptional over $X''$, then 
$$a(\bar{D},X,\Delta+(1+u)A)<a(\bar{D},X'',\Delta''+(1+u)A'')$$
 by the choice of $u \in \mathbb{R}_{>0}$. 
In both cases, we have
$$a(\bar{D},X,\Delta+(1+u)A) \leq a(\bar{D},X'',\Delta''+(1+u)A'') \leq a(\bar{D},X''',\Gamma_{u}''').$$
Therefore, for any component $\bar{E}$ of $E_{Y}^{(\delta)}$ that is not exceptional over $X$, we have 
$${\rm coeff}_{\bar{E}}(E_{Y}^{(\delta)}) \geq 0.$$
By applying the negativity lemma to $g \colon Y \to X$ and $E_{Y}^{(\delta)}$, we obtain $E_{Y}^{(\delta)} \geq 0$. 

Next, we prove that $E_{Y}^{(\delta)}$ is $g_{3}$-exceptional. 
We pick a prime divisor $\tilde{D}$ on $Y$ that is not exceptional over $X'''$. 
Then the definition of good minimal model (Definition \ref{defn--minmodel}) shows 
$$a(\tilde{D},X''',\Gamma'''_{u})=a(\tilde{D},X'',\Delta''+(1+u)A'').$$
If $\tilde{D}$ is not exceptional over $X''$, then 
$$a(\tilde{D},X,\Delta+A)=a(\tilde{D},X'',\Delta''+A''),$$ 
by the definition of good minimal model (Definition \ref{defn--minmodel}), and (\ref{thm--from-minimodel-to-termi-(III)}) shows
$$a(\tilde{D},X,\Delta+(1+u)A)=a(\tilde{D},X'',\Delta''+(1+u)A'')=a(\tilde{D},X''',\Gamma'''_{u}).$$
Therefore we have ${\rm coeff}_{\tilde{D}}(E_{Y}^{(\delta)}) = 0$ if $\tilde{D}$ is not exceptional over $X''$. 
If $\tilde{D}$ is exceptional over $X''$, then 
$$a(\tilde{D},X''',\Gamma'''_{u}) =a(\tilde{D},X'',\Delta''+(1+u)A'')\leq -1$$
by the definition of good minimal model (Definition \ref{defn--minmodel}). 
By (\ref{thm--from-minimodel-to-termi-(II)}), we have 
$$a(\tilde{D},X,\Delta+A)=a(\tilde{D},X'',\Delta''+A''),$$ 
and furthermore (\ref{thm--from-minimodel-to-termi-(III)}) implies 
$$a(\tilde{D},X,\Delta+(1+u)A)=a(\tilde{D},X'',\Delta''+(1+u)A'')=a(\tilde{D},X''',\Gamma'''_{u}).$$ 
Therefore ${\rm coeff}_{\tilde{D}}(E_{Y}^{(\delta)}) = 0$ even if $\tilde{D}$ is exceptional over $X''$. 
Thus $E_{Y}^{(\delta)}$ is $g_{3}$-exceptional. 
From these arguments, $E_{Y}^{(\delta)}$ is effective and $g_{3}$-exceptional for any $\delta \in [0,\epsilon]$.
\end{step2} 

\begin{step2}\label{step4--thm--from-minimodel-to-termi}
With this step we complete the proof. 

By (\ref{thm--from-minimodel-to-termi-(clubsuit)}), for any $\delta \in [0,\epsilon]$ we obtain
$$g^{*}(K_{X}+\Delta+(1+\delta u)A)=g_{3}^{*}(K_{X'''}+\delta \Gamma_{u}'''+(1-\delta )\Gamma_{0}''')+E_{Y}^{(\delta)}$$
such that $K_{X'''}+\delta \Gamma_{u}'''+(1-\delta )\Gamma_{0}'''$ is semi-ample over $Z$ and $E_{Y}^{(\delta)}\geq 0$ is $g_{3}$-exceptional. 

We use the $(K_{X}+\Delta+A)$-MMP over $Z$ and $\lambda_{i}$ in Theorem \ref{thm--from-minimodel-to-termi}. 
Fix $m$ such that $\lambda_{m}<{\rm min}\{\lambda_{m-1}, \epsilon u\}$. 
Such $m$ exists since we have $u>0$, $\epsilon>0$, and ${\rm lim}_{i \to \infty}\lambda_{i}=0$. 
Let $h \colon Y' \to Y$ and $h_{m} \colon Y' \to X_{m}$ be a common resolution of $Y \dashrightarrow X_{m}$. 
We have the following diagram
$$
\xymatrix{
&Y \ar[ld]_{g}\ar[rd]^{g_{3}}&&Y' \ar[ll]_{h}\ar[rd]^{h_{m}}\\
X   \ar@{-->}[rr] &&X''' \ar@{-->}[rr]&&X_{m}
}
$$
over $Z$. 
By construction, for any $t \in (\lambda_{m},{\rm min}\{\lambda_{m-1}, \epsilon u\})$, we can write 
$$h^{*}g^{*}(K_{X}+\Delta+(1+t)A)=h_{m}^{*}(K_{X_{m}}+\Delta_{m}+(1+t)A_{m})+F_{Y'}^{(t)}$$
for some effective $h_{m}$-exceptional $\mathbb{R}$-divisor $F_{Y'}^{(t)}$ on $Y'$. 
Now $t= \frac{t}{u} \cdot u$ and $ \frac{t}{u} \leq \epsilon$. 
Hence 
$$h_{m}^{*}(K_{X_{m}}+\Delta_{m}+(1+t)A_{m})+F_{Y'}^{(t)}=h^{*}g_{3}^{*}(K_{X'''}+\frac{t}{u} \Gamma_{u}'''+(1-\frac{t}{u} )\Gamma_{0}''')+h^{*}E_{Y}^{^{(\frac{t}{u})}}.$$
Since $h_{m}^{*}(K_{X_{m}}+\Delta_{m}+(1+t)A_{m})$ and $K_{X'''}+\frac{t}{u} \Gamma_{u}'''+(1-\frac{t}{u} )\Gamma_{0}'''$ are nef over $Z$ and $F_{Y'}^{(t)}$ (resp.~$E_{Y}^{^{(\frac{t}{u})}}$) is effective and $h_{m}$-exceptional (resp.~$g_{3}$-exceptional), by the negativity lemma, we see that $F_{Y'}^{(t)}= h^{*}E_{Y}^{^{(\frac{t}{u})}}$. 
Therefore, we get the equality
$$h_{m}^{*}(K_{X_{m}}+\Delta_{m}+(1+t)A_{m})=h^{*}g_{3}^{*}(K_{X'''}+\frac{t}{u} \Gamma_{u}'''+(1-\frac{t}{u} )\Gamma_{0}''')$$
for all $t \in (\lambda_{m},{\rm min}\{\lambda_{m-1}, \epsilon u\})$. 
Then 
$$h_{m}^{*}(K_{X_{m}}+\Delta_{m}+A_{m})=h^{*}g_{3}^{*}(K_{X'''}+\Gamma_{0}'''),$$
and the right hand side is semi-ample over $Z$ by construction of $\Gamma_{0}'''$. 
In particular, we have $\lambda_{m}=0$, and $(X_{m},\Delta_{m}+A_{m})$ is a good minimal model of $(X,\Delta+A)$ over $Z$. 
\end{step2}
We finish the proof. 
\end{proof}

\begin{cor}\label{cor--mmp-termi-kltlocus}
Let $\pi \colon X \to Z$ be a projective morphism of normal quasi-projective varieties. 
Let $(X,\Delta)$ be a normal pair and let $A$ be an effective $\pi$-ample $\mathbb{R}$-divisor on $X$ such that ${\rm Nklt}(X,\Delta) \neq \emptyset$ and ${\rm Bs}|K_{X}+\Delta+A/Z|_{\mathbb{R}} \cap {\rm Nklt}(X,\Delta) = \emptyset$. 
Let 
$$(X,\Delta+A)=:(X_{1},\Delta_{1}+A_{1}) \dashrightarrow \cdots \dashrightarrow (X_{i},\Delta_{i}+A_{i}) \dashrightarrow \cdots$$
be a sequence of steps of a $(K_{X}+\Delta+A)$-MMP over $Z$ with scaling of $A$ constructed by using Corollary \ref{cor--mmp-nomralpair-Qfacdlt}. 
We put 
$$\lambda_{i}:=\{\mu \in \mathbb{R}_{\geq 0}\,| \,\text{$(K_{X_{i}}+\Delta_{i}+A_{i})+\mu A_{i}$ is nef over $Z$} \}$$
for each $i \geq 1$. 
Then $\lambda_{m}=0$ for some $m$ and $(X_{m},\Delta_{m}+A_{m})$ is a good minimal model of $(X,\Delta+A)$ over $Z$. 
\end{cor}

\begin{proof}
Assume by contradiction that $\lambda_{i} \neq 0$ for any $i \geq 1$. 
We put $\lambda:={\rm lim}_{i \to \infty} \lambda_{i}$. 
By Corollary \ref{cor--mmp-nomralpair-Qfacdlt}, we have 
$\lambda\neq \lambda_{i}$ for all $i \geq 1$. 
To get a contradiction, replacing $A$ with $(1+\lambda) A$, we may assume $\lambda=0$. 
Then Theorem \ref{thm--from-minimodel-to-termi} implies $\lambda_{m}=0$ for some $m$, which contradicts the fact that $\lambda\neq \lambda_{i}$ for all $i \geq 1$. 
Thus, Corollary \ref{cor--mmp-termi-kltlocus} holds. 
\end{proof}

\subsection{Non-vanishing}

In this subsection we prove the non-vanishing theorem.

\begin{prop}\label{prop--non-vanishing-kltmmp}
Let $\pi \colon X \to Z$ be a projective morphism of normal quasi-projective varieties. 
Let $(X,\Delta)$ be a normal pair such that $\Delta$ is a $\mathbb{Q}$-divisor on $X$. 
Let $A$ be a $\pi$-ample $\mathbb{Q}$-divisor on $X$ such that $K_{X}+\Delta+A$ is $\pi$-pseudo-effective. 
Suppose that ${\rm NNef}(K_{X}+\Delta+A/Z) \cap {\rm Nklt}(X,\Delta) =\emptyset$. 
Suppose in addition that $(K_{X}+\Delta+A)|_{{\rm Nklt}(X,\Delta)}$, which we think of a $\mathbb{Q}$-line bundle on ${\rm Nklt}(X,\Delta)$, is semi-ample over $Z$. 
Then the equality ${\rm Bs}|K_{X}+\Delta+A/Z|_{\mathbb{R}} \cap {\rm Nklt}(X,\Delta) = \emptyset$ holds. 
\end{prop}

\begin{proof}
We may assume that $(X,\Delta)$ is not klt. 
By replacing $A$ with a general member of $|A/Z|_{\mathbb{Q}}$, we may assume $A \geq 0$. 
Let 
$$(X,\Delta+A)=:(X_{1},\Delta_{1}+A_{1}) \dashrightarrow \cdots \dashrightarrow (X_{i},\Delta_{i}+A_{i}) \dashrightarrow \cdots$$
be a sequence of steps of a $(K_{X}+\Delta+A)$-MMP over $Z$ with scaling of $A$ constructed by using Corollary \ref{cor--mmp-nomralpair-Qfacdlt}. 
Then the non-isomorphic locus of the $(K_{X}+\Delta+A)$-MMP is disjoint from ${\rm Nklt}(X,\Delta)$ (Remark \ref{rem--mmp-basic}). 
We define 
$$\lambda_{i}:=\{\mu \in \mathbb{R}_{\geq 0}\,| \,\text{$(K_{X_{i}}+\Delta_{i}+A_{i})+\mu A_{i}$ is nef over $Z$} \}$$
for each $i \geq 1$ and $\lambda := {\rm lim}_{i \to \infty}\lambda_{i}$. 
By construction, the MMP terminates after finitely many steps or otherwise $\lambda \neq \lambda_{i}$ for every $i \geq 1$, and $K_{X_{i}}+\Delta_{i}+A_{i}+t A_{i}$ is semi-ample over $Z$ for every $i \geq 1$ and $t \in (\lambda_{i},\lambda_{i-1}]$. 
Assuming ${\rm Bs}|K_{X}+\Delta+A+\lambda A/Z|_{\mathbb{R}} \cap {\rm Nklt}(X,\Delta) = \emptyset$, then Theorem \ref{thm--from-minimodel-to-termi} implies $\lambda_{j}=\lambda$ for some $j$. 
Then the equality $\lambda=0$ holds because otherwise we get a contradiction, and ${\rm Bs}|K_{X}+\Delta+A/Z|_{\mathbb{R}} \cap {\rm Nklt}(X,\Delta) = \emptyset$. 
From this discussion, it is sufficient to prove ${\rm Bs}|K_{X}+\Delta+A+\lambda  A/Z|_{\mathbb{R}} \cap {\rm Nklt}(X,\Delta) = \emptyset$. 

We fix nonnegative rational numbers $q$ and $q'$ such that 
$$q' \leq \lambda \leq q.$$ 
We note that we may take $q=q'=\lambda$ when $\lambda$ is a rational number. 
By the hypothesis of Proposition \ref{prop--non-vanishing-kltmmp} that $(K_{X}+\Delta+A)|_{{\rm Nklt}(X,\Delta)}$ is semi-ample over $Z$, we can find an integer $p >1 $ such that 
\begin{itemize}
\item
$pq \in \mathbb{Z}$ and $p(K_{X}+\Delta+A)$, $pA$, and $pqA$ are all Cartier, and
\item
$p(K_{X}+\Delta+A)|_{{\rm Nklt}(X,\Delta)}$, $pA|_{{\rm Nklt}(X,\Delta)}$, and $pqA|_{{\rm Nklt}(X,\Delta)}$
are globally generated over $Z$. 
\end{itemize} 

In this paragraph, we will prove ${\rm Bs}|K_{X}+\Delta+A+qA/Z|_{\mathbb{R}} \cap {\rm Nklt}(X,\Delta) = \emptyset$. 
Since $\lambda={\rm lim}_{i \to \infty}\lambda_{i}$ and $\frac{1+pq}{p-1} > q \geq \lambda$, we can find $m$ such that the birational contraction 
$$X \dashrightarrow X_{m}$$
is a sequence of steps of a $(K_{X}+\Delta+A+qA)$-MMP over $Z$ and the $\mathbb{Q}$-divisor $$K_{X_{m}}+\Delta_{m}+A_{m}+\frac{1+pq}{p-1}A_{m}$$
is semi-ample over $Z$. 
By Remark \ref{rem--mmp-basic}, the birational map $X \dashrightarrow X_{m}$ is an isomorphism on ${\rm Nklt}(X,\Delta)$. 
We put
$$L:=p(K_{X_{m}}+\Delta_{m}+A_{m}+qA_{m}).$$
Then $L$ is a Weil divisor on $X_{m}$ and Cartier on a neighborhood of ${\rm Nklt}(X_{m},\Delta_{m})$. 
We can write
$$L=K_{X_{m}}+\Delta_{m}+(p-1)\left(K_{X_{m}}+\Delta_{m}+A_{m}+\frac{1+pq}{p-1}A_{m}\right)$$
and $(K_{X_{m}}+\Delta_{m}+A_{m}+\frac{1+pq}{p-1}A_{m})|_{{\rm Nklt}(X_{m},\Delta_{m})}$ is ample over $Z$ because $X \dashrightarrow X_{m}$ is an isomorphism on ${\rm Nklt}(X,\Delta)$, the inequality $\frac{1+pq}{p-1} > 0$ holds, and $p(K_{X}+\Delta+A)|_{{\rm Nklt}(X,\Delta)}$ is globally generated over $Z$. 
Therefore $L- (K_{X_{m}}+\Delta_{m})$ is nef and log big over $Z$ with respect to $(X_{m},\Delta_{m})$ (\cite[Definition 3.7]{fujino-bpf-quasi-log}). 
From these facts, we can use Theorem \ref{thm--vanishing-quasi-log} to $(X_{m},\Delta_{m})$ and $L$. 
Putting $\pi_{m}\colon X_{m} \to Z$ and $\mathcal{I}$ as the structure morphism and the defining ideal sheaf of ${\rm Nklt}(X_{m},\Delta_{m})$ respectively, then Theorem \ref{thm--vanishing-quasi-log} implies that
$$R^{1}\pi_{m*}(\mathcal{I} \otimes_{\mathcal{O}_{X_{m}}} \mathcal{O}_{X_{m}}(L))=0.$$
We consider the exact sequence 
$$0 \longrightarrow \mathcal{I} \otimes_{\mathcal{O}_{X_{m}}} \mathcal{O}_{X_{m}}(L) \longrightarrow \mathcal{O}_{X_{m}}(L) \longrightarrow \mathcal{O}_{{\rm Nklt}(X_{m},\Delta_{m})}(L|_{{\rm Nklt}(X_{m},\Delta_{m})}) \longrightarrow 0.$$
By taking the pushforward to $Z$, we get the following exact sequence
\begin{equation*}
\begin{split}
\pi_{m*}\mathcal{O}_{X_{m}}(L)  \longrightarrow \pi_{m*}\mathcal{O}_{{\rm Nklt}(X_{m},\Delta_{m})}(L|_{{\rm Nklt}(X_{m},\Delta_{m})}) \longrightarrow R^{1}\pi_{m*}(\mathcal{I} \otimes_{\mathcal{O}_{X_{m}}} \mathcal{O}_{X_{m}}(L))=0.
\end{split}
\end{equation*}
By construction of $p$ and $L$, the line bundle $L|_{{\rm Nklt}(X_{m},\Delta_{m})}$ is globally generated over $Z$. 
Hence, we have 
$${\rm Bs}|K_{X_{m}}+\Delta_{m}+A_{m}+qA_{m}/Z|_{\mathbb{R}} \cap {\rm Nklt}(X_{m},\Delta_{m}) ={\rm Bs}|L/Z|_{\mathbb{R}} \cap {\rm Nklt}(X_{m},\Delta_{m})= \emptyset.$$
Since $X \dashrightarrow X_{m}$ is a sequence of steps of a $(K_{X}+\Delta+A+qA)$-MMP over $Z$ and the map is an isomorphism on ${\rm Nklt}(X,\Delta)$, we have ${\rm Bs}|K_{X}+\Delta+A+qA/Z|_{\mathbb{R}} \cap {\rm Nklt}(X,\Delta) = \emptyset$. 

In this paragraph, we will find a rational number $\tilde{q} \in [q',\lambda]$ such that 
$${\rm Bs}|K_{X}+\Delta+A+\tilde{q}A/Z|_{\mathbb{R}} \cap {\rm Nklt}(X,\Delta) = \emptyset.$$ 
We fix $k \in \mathbb{Z}_{>0}$ such that
$$\frac{\lfloor k\lambda \rfloor}{k}\geq q' \qquad {\rm and} \qquad k\lambda -\lfloor k\lambda\rfloor < \frac{1}{p}.$$
Such $k$ always exists by the argument using the Diophantine approximation (see, for example, \cite[Subsection 3.7]{bchm}). 
We put
$$\tilde{q}:=\frac{\lfloor k\lambda \rfloor}{k}.$$
Then $\tilde{q} \in [q',\lambda]$. 
We will prove that this $\tilde{q}$ is the desired rational number. 
By definition, we have
$$\frac{1+\lfloor k \lambda \rfloor p}{kp-1}> \frac{1+\bigl(k \lambda-\frac{1}{p}\bigr) p}{kp-1}=\frac{kp\lambda}{kp-1}>\lambda \geq \tilde{q}.$$
Since $\lambda={\rm lim}_{i \to \infty}\lambda_{i}$, we can find $n$ such that the birational contraction 
$$X \dashrightarrow X_{n}$$
 is a finite sequence of steps of a $(K_{X}+\Delta+A+\tilde{q}A)$-MMP over $Z$ and the $\mathbb{Q}$-divisor
 $$K_{X_{n}}+\Delta_{n}+A_{n}+\frac{1+\lfloor k \lambda \rfloor p}{kp-1} A_{n}$$
 is semi-ample over $Z$. 
By Remark \ref{rem--mmp-basic}, the birational map $X \dashrightarrow X_{n}$ is an isomorphism on ${\rm Nklt}(X,\Delta)$. 
We put
$$L':=kp(K_{X_{n}}+\Delta_{n}+A_{n}+\tilde{q}A_{n})=kp(K_{X_{n}}+\Delta_{n}+A_{n})+p\lfloor k \lambda \rfloor A_{n}.$$
By the definitions of $p$, $k$, and $\tilde{q}$, it follows that $L'$ is a Weil divisor on $X_{n}$ and Cartier on a neighborhood of ${\rm Nklt}(X_{n},\Delta_{n})$. 
We can write
$$L'=K_{X_{n}}+\Delta_{n}+(kp-1)\left(K_{X_{n}}+\Delta_{n}+A_{n}+\frac{1+\lfloor k \lambda \rfloor p}{kp-1}A_{n}\right)$$
and $(K_{X_{n}}+\Delta_{n}+A_{n}+\frac{1+\lfloor k \lambda \rfloor p}{kp-1}A_{n})|_{{\rm Nklt}(X_{n},\Delta_{n})}$ is ample over $Z$ because $X \dashrightarrow X_{n}$ is an isomorphism on ${\rm Nklt}(X,\Delta)$, $\frac{1+\lfloor k \lambda \rfloor p}{kp-1} > 0$, and $p(K_{X}+\Delta+A)|_{{\rm Nklt}(X,\Delta)}$ is globally generated over $Z$. 
Therefore $L'- (K_{X_{n}}+\Delta_{n})$ is nef and log big over $Z$ with respect to $(X_{n},\Delta_{n})$ (\cite[Definition 3.7]{fujino-bpf-quasi-log}). 
From these facts, we can use Theorem \ref{thm--vanishing-quasi-log} to $(X_{n},\Delta_{n})$ and $L'$. 
Putting $\pi_{n}\colon X_{n} \to Z$ and $\mathcal{I}'$ as the structure morphism and the defining ideal sheaf of ${\rm Nklt}(X_{n},\Delta_{n})$ respectively, then Theorem \ref{thm--vanishing-quasi-log} shows
$$R^{1}\pi_{n*}(\mathcal{I}' \otimes_{\mathcal{O}_{X_{n}}} \mathcal{O}_{X_{n}}(L'))=0.$$
We consider the exact sequence 
$$0 \longrightarrow \mathcal{I}' \otimes_{\mathcal{O}_{X_{n}}} \mathcal{O}_{X_{n}}(L') \longrightarrow \mathcal{O}_{X_{n}}(L') \longrightarrow \mathcal{O}_{{\rm Nklt}(X_{n},\Delta_{n})}(L'|_{{\rm Nklt}(X_{n},\Delta_{n})}) \longrightarrow 0.$$
By taking the pushforward to $Z$, we get the following exact sequence
\begin{equation*}
\begin{split}
\pi_{n*}\mathcal{O}_{X_{n}}(L')  \longrightarrow \pi_{n*}\mathcal{O}_{{\rm Nklt}(X_{n},\Delta_{n})}(L'|_{{\rm Nklt}(X_{n},\Delta_{n})}) \longrightarrow R^{1}\pi_{n*}(\mathcal{I}' \otimes_{\mathcal{O}_{X_{n}}} \mathcal{O}_{X_{n}}(L'))=0.
\end{split}
\end{equation*}
By construction of $p$ and $L'$, the line bundle $L'|_{{\rm Nklt}(X_{n},\Delta_{n})}$ is globally generated over $Z$. 
Hence, we have 
$${\rm Bs}|K_{X_{n}}+\Delta_{n}+(1+\tilde{q})A_{n}/Z|_{\mathbb{R}} \cap {\rm Nklt}(X_{n},\Delta_{n}) ={\rm Bs}|L'/Z|_{\mathbb{R}} \cap {\rm Nklt}(X_{n},\Delta_{n})= \emptyset.$$
Since $X \dashrightarrow X_{n}$ is a sequence of steps of a $(K_{X}+\Delta+A+\tilde{q}A)$-MMP over $Z$ and the map is an isomorphism on ${\rm Nklt}(X,\Delta)$, we have ${\rm Bs}|K_{X}+\Delta+A+\tilde{q}A/Z|_{\mathbb{R}} \cap {\rm Nklt}(X,\Delta) = \emptyset$. 

Now we have nonnegative rational numbers $\tilde{q} \leq \lambda \leq q$ such that
\begin{equation*}
\begin{split}
{\rm Bs}|K_{X}+\Delta+A+qA/Z|_{\mathbb{R}} \cap {\rm Nklt}(X,\Delta) =& \emptyset, \qquad {\rm and}\\
{\rm Bs}|K_{X}+\Delta+A+\tilde{q}A/Z|_{\mathbb{R}} \cap {\rm Nklt}(X,\Delta) =& \emptyset.
\end{split}
\end{equation*}
Since $K_{X}+\Delta+A+\lambda A$ is a convex linear combination of $K_{X}+\Delta+A+qA$ and $K_{X}+\Delta+A+\tilde{q}A$, we see that
$${\rm Bs}|K_{X}+\Delta+A+\lambda A/Z|_{\mathbb{R}} \cap {\rm Nklt}(X,\Delta) = \emptyset.$$
Then Proposition \ref{prop--non-vanishing-kltmmp} holds as discussed. 
\end{proof}

\begin{prop}\label{prop--nonvan-ample-decom}
Let $\pi \colon X \to Z$ be a projective morphism of normal quasi-projective varieties. 
Let $(X,\Delta)$ be a normal pair such that $\Delta$ is a $\mathbb{Q}$-divisor on $X$. 
Let $A$ be a $\pi$-ample $\mathbb{R}$-divisor on $X$. 
Suppose that there exist positive real numbers $r_{1},\,\cdots,\,r_{l}$ and $\pi$-ample $\mathbb{Q}$-divisors $A_{1},\,\cdots ,\,A_{l}$ on $X$ such that
\begin{itemize}
\item
$\sum_{i=1}^{l} r_{i}=1$ and $A=\sum_{i=1}^{l} r_{i}A_{i}$, and 
\item
for every $1 \leq i \leq l$, the $\mathbb{Q}$-line bundle $(K_{X}+\Delta+A_{i})|_{{\rm Nklt}(X,\Delta)}$ on ${\rm Nklt}(X,\Delta)$ is semi-ample over $Z$. 
\end{itemize}
Then there exist positive real numbers $r'_{1},\,\cdots,\,r'_{m}$, positive integers $p_{1},\,\cdots,\,p_{m}$ that are greater than one, and $\pi$-ample $\mathbb{Q}$-divisors $A'_{1},\,\cdots ,\,A'_{m}$ on $X$ such that 
\begin{itemize}
\item
$\sum_{j=1}^{m} r'_{j}=1$ and $A=\sum_{j=1}^{m} r'_{j}A'_{j}$,  
\item
$p_{j}(K_{X}+\Delta+A'_{j})$ and $p_{j}A'_{j}$ are Cartier for every $1 \leq j \leq m$,    
\item
$p_{j}(K_{X}+\Delta+A'_{j})|_{{\rm Nklt}(X,\Delta)}$ and $p_{j}A'_{j}|_{{\rm Nklt}(X,\Delta)}$ are globally generated over $Z$ for every $1 \leq j \leq m$, and
\item
$\frac{p_{j}}{p_{j}-1}A'_{j} - A$ is ample over $Z$ for every $1 \leq j \leq m$. 
\end{itemize}
\end{prop}

\begin{proof}
This is an application of \cite[Lemma 3.7.7]{bchm}. 
By the hypothesis, we can find an integer $p \geq 2$ such that 
\begin{itemize}
\item
$p(K_{X}+\Delta)$ and all $pA_{i}$ are Cartier, and 
\item
all $p(K_{X}+\Delta+A_{i})|_{{\rm Nklt}(X,\Delta)}$ and $pA_{i}|_{{\rm Nklt}(X,\Delta)}$ are globally generated over $Z$. 
\end{itemize}
Put $\boldsymbol{r}=(r_{1},\,\cdots,\, r_{l})$. 
We apply \cite[Lemma 3.7.7]{bchm} to $p$, $\frac{1}{2}{\rm min}\{r_{i}\}_{1\leq i \leq l}$, and the following set 
$$\mathcal{S}:=\left\{(a_{1},\,\cdots,\,a_{l}) \in \mathbb{R}^{l}\,\middle|\,\text{$a_{i} \geq 0$ for all $i$ and $\sum_{i=1}^{l}a_{i}=1$}\right\} \ni \boldsymbol{r}.$$ 
There exist rational points $\boldsymbol{q}^{(1)},\,\cdots,\, \boldsymbol{q}^{(m)}$ of $\mathcal{S}$ and positive integers $p_{1},\,\cdots ,\,p_{m}$, which are divisible by $p$, such that $\boldsymbol{r}$ is a convex linear combination of $\boldsymbol{q}^{(1)},\,\cdots,\, \boldsymbol{q}^{(m)}$ and for every $1 \leq j \leq m$ the following conditions 
$$||\boldsymbol{r}-\boldsymbol{q}^{(j)}||<\frac{1}{p_{j}}\cdot \frac{1}{2}{\rm min}\{r_{i}\}_{1\leq i \leq l}\qquad {\rm and} \qquad \frac{p_{j}}{p}\boldsymbol{q}^{(j)} \in \mathbb{Z}^{l}$$
hold. 
Then $p_{j} \geq 2$ for all $1 \leq j \leq m$ since $p \geq 2$, and the first property implies that $\boldsymbol{q}^{(j)}$ is not the origin for all $1 \leq j \leq m$.  
For each $1\leq j \leq m$, we put 
 $$(q^{(j)}_{1},\,\cdots,\,q^{(j)}_{l}):=\boldsymbol{q}^{(j)} \qquad {\rm and} \qquad A'_{j}:=\sum_{i=1}^{l}q^{(j)}_{i}A_{i}.$$
Since $\boldsymbol{r}$ is a convex linear combination of $\boldsymbol{q}^{(1)},\,\cdots,\, \boldsymbol{q}^{(m)}$, we can write $\boldsymbol{r}=\sum_{j=1}^{m}r'_{j}\boldsymbol{q}^{(j)}$ with nonnegative real numbers $r'_{1},\,\cdots,\,r'_{m}$ such that $\sum_{j=1}^{m}r'_{j}=1$. 
By removing indices $j$ such that $r'_{j}=0$, we may assume that all $r'_{j}$ are positive. 

From now on we prove that these $r'_{1},\,\cdots,\,r'_{m}$, $p_{1},\,\cdots ,\,p_{m}$, and $A'_{1},\,\cdots ,\,A'_{m}$ satisfy the properties of Proposition \ref{prop--nonvan-ample-decom}. 
Note that all $A'_{j}$ are $\pi$-ample $\mathbb{Q}$-divisors by construction. 
It is obvious by definition that $\sum_{j=1}^{m}r'_{j}=1$, and $\boldsymbol{r}=\sum_{j=1}^{m}r'_{j}\boldsymbol{q}^{(j)}$ shows $r_{i}=\sum_{j=1}^{m}r'_{j}q^{(j)}_{i}$ for every $1 \leq i \leq l$. 
Hence, we have
$$A=\sum_{i=1}^{l} r_{i}A_{i}=\sum_{i=1}^{l} \sum_{j=1}^{m}r'_{j}q^{(j)}_{i}A_{i}=\sum_{j=1}^{m} r'_{j}A'_{j}.$$
Therefore the first property of Proposition \ref{prop--nonvan-ample-decom} holds. 

The condition $\frac{p_{j}}{p}\boldsymbol{q}^{(j)} \in \mathbb{Z}^{l}$ implies $p_{j}q^{(j)}_{i}\in p\mathbb{Z}_{\geq 0}$ for all $1 \leq i \leq l$. 
Since all $pA_{i}$ are Cartier and all $pA_{i}|_{{\rm Nklt}(X,\Delta)}$ are globally generated over $Z$, the relation 
$$p_{j}A'_{j}=\sum_{i=1}^{l}p_{j}q^{(j)}_{i}A_{i}$$
 implies that all $p_{j}A'_{j}$ is Cartier and all $p_{j}A'_{j}|_{{\rm Nklt}(X,\Delta)}$ are globally generated over $Z$. 
 Since $p_{j} \in p\mathbb{Z}$ and $p(K_{X}+\Delta)$ is Cartier, we see that $p_{j}(K_{X}+\Delta)$ is Cartier for all $1\leq j \leq m$. 
 In particular, $p_{j}(K_{X}+\Delta+A'_{j})$ and $p_{j}A'_{j}$ are Cartier for every $1 \leq j \leq m$. 
This is the second property of Proposition \ref{prop--nonvan-ample-decom}. 
 
In the previous paragraph, we have already proved that $p_{j}A'_{j}|_{{\rm Nklt}(X,\Delta)}$ is globally generated over $Z$ for every $1 \leq j \leq m$. 
By the definition of $\mathcal{S}$, we see that $\sum_{i=1}^{l}q^{(j)}_{i}=1$ holds for all $1 \leq j \leq m$. 
This implies that $\sum_{i=1}^{l}p_{j}q^{(j)}_{i}=p_{j}$ for every $1 \leq j \leq m$.
Then
\begin{equation*}
\begin{split}
p_{j}(K_{X}+\Delta+A'_{j})=p_{j}(K_{X}+\Delta)+p_{j}A'_{j}=&\sum_{i=1}^{l}p_{j}q^{(j)}_{i}(K_{X}+\Delta)+\sum_{i=1}^{l}p_{j}q^{(j)}_{i}A_{i}\\
=&\sum_{i=1}^{l}p_{j}q^{(j)}_{i}(K_{X}+\Delta+A_{i}).
\end{split}
\end{equation*}
As explained above, $p_{j}q^{(j)}_{i} \in p \mathbb{Z}_{\geq 0}$ for every $1 \leq i \leq l$ and $1 \leq j \leq m$. 
Since all the $\mathbb{Q}$-line bundles $p(K_{X}+\Delta+A_{i})|_{{\rm Nklt}(X,\Delta)}$ are globally generated over $Z$, it follows that $p_{j}(K_{X}+\Delta+A'_{j})|_{{\rm Nklt}(X,\Delta)}$ is globally generated over $Z$ for every $1 \leq j \leq m$. 
Therefore, the third property of Proposition \ref{prop--nonvan-ample-decom} holds.

Finally, since $||\boldsymbol{r}-\boldsymbol{q}^{(j)}||<\frac{1}{p_{j}}\cdot \frac{1}{2}{\rm min}\{r_{i}\}_{1\leq i \leq l}$, we have $|r_{i}-q^{(j)}_{i}|<\frac{1}{2p_{j}}{\rm min}\{r_{i}\}_{1\leq i \leq l}$ for all $1 \leq i \leq l$ and $1 \leq j \leq m$. 
Then 
$$q^{(j)}_{i} \geq r_{i} >\frac{1}{2}{\rm min}\{r_{i}\}_{1\leq i \leq l} \qquad {\rm or} \qquad q^{(j)}_{i}>r_{i}-\frac{1}{2p_{j}}{\rm min}\{r_{i}\}_{1\leq i \leq l} \geq \frac{1}{2}{\rm min}\{r_{i}\}_{1\leq i \leq l}.$$
This shows
$|r_{i}-q^{(j)}_{i}|<\frac{1}{2p_{j}}{\rm min}\{r_{i}\}_{1\leq i \leq l}< \frac{q^{(j)}_{i}}{p_{j}}< \frac{q^{(j)}_{i}}{p_{j}-1}.$
By simple calculations, for every $1 \leq i \leq l$ and $1 \leq j \leq m$ we have
$$r_{i}<q^{(j)}_{i}+\frac{q^{(j)}_{i}}{p_{j}-1}= \frac{p_{j}}{p_{j}-1}q^{(j)}_{i}.$$
Since $A'_{j}=\sum_{i=1}^{l}q^{(j)}_{i}A_{i}$ and $A=\sum_{i=1}^{l}r_{i}A_{i}$ by definition, $\frac{p_{j}}{p_{j}-1}A'_{j} - A$ is ample over $Z$ for every $1 \leq j \leq m$. 
Therefore, the final property of Proposition \ref{prop--nonvan-ample-decom} holds. 
\end{proof}

\begin{thm}\label{thm--non-vanishing-kltmmp-main}
Let $\pi \colon X \to Z$ be a projective morphism of normal quasi-projective varieties. 
Let $(X,\Delta)$ be a normal pair and let $A$ be a $\pi$-ample $\mathbb{R}$-divisor on $X$ such that $K_{X}+\Delta+A$ is $\pi$-pseudo-effective. 
Suppose that ${\rm NNef}(K_{X}+\Delta+A/Z) \cap {\rm Nklt}(X,\Delta) = \emptyset$. 
Suppose in addition that $(K_{X}+\Delta+A)|_{{\rm Nklt}(X,\Delta)}$, which we think of an $\mathbb{R}$-line bundle on ${\rm Nklt}(X,\Delta)$, is semi-ample over $Z$. 
Then ${\rm Bs}|K_{X}+\Delta+A/Z|_{\mathbb{R}} \cap {\rm Nklt}(X,\Delta) = \emptyset$. 
\end{thm}

\begin{proof}
We prove the theorem in several steps. 

\begin{step3}\label{step1-thm--non-vanishing-kltmmp-main}
In this step, we reduce Theorem \ref{thm--non-vanishing-kltmmp-main} to the case where $\Delta$ is a $\mathbb{Q}$-divisor. 

By the argument from convex geometry, the set
$$\{\overline{\Delta} \in {\rm WDiv}_{\mathbb{R}}(X)\,|\,\text{$K_{X}+\overline{\Delta}$ is $\mathbb{R}$-Cartier and ${\rm Nklt}(X,\Delta)={\rm Nklt}(X,\overline{\Delta})$}\}$$
contains a rational polytope $\mathcal{D} \ni \Delta$ in the $\mathbb{R}$-vector space spanned by the components of $\Delta$. 
We take an effective $\mathbb{Q}$-divisor $\overline{\Delta} \in \mathcal{D}$ such that $(K_{X}+\Delta)-(K_{X}+\overline{\Delta})+A$ is $\pi$-ample. 
Putting $\overline{A}=(K_{X}+\Delta)-(K_{X}+\overline{\Delta})+A$, then $(X,\overline{\Delta})$ is a normal pair, $\overline{A}$ is $\pi$-ample, ${\rm Nklt}(X,\overline{\Delta})={\rm Nklt}(X,\Delta)$, and $K_{X}+\overline{\Delta}+\overline{A}=K_{X}+\Delta+A$. 
From them, we may replace $(X,\Delta)$ and $A$ by $(X,\overline{\Delta})$ and $\overline{A}$ respectively.  
Thus, we may assume that $\Delta$ is a $\mathbb{Q}$-divisor. 
\end{step3} 

\begin{step3}\label{step2-thm--non-vanishing-kltmmp-main}
In this step, we discuss a decomposition of $A$ into $\pi$-ample $\mathbb{Q}$-divisors. 

We can write
$$A=\sum_{i=1}^{l}r_{i}A_{i},$$
for some distinct $\pi$-ample $\mathbb{Q}$-divisors $A_{i}$ on $X$ and positive real numbers $r_{1},\,\cdots,\,r_{l}$ that are linearly independent over $\mathbb{Q}$. 
We consider the set
$$\left\{(a_{1},\,\cdots,\,a_{l}) \in (\mathbb{R}_{>0})^{l}\,\middle| \begin{array}{l}
\text{$(K_{X}+\Delta+\sum_{i=1}^{l}a_{i}A_{i})|_{{\rm Nklt}(X,\Delta)}$ is semi-ample over $Z$} 
\end{array}\right\}.$$
By the argument from convex geometry, we can find a rational polytope $\mathcal{C}$ in the set such that $(r_{1},\,\cdots,\,r_{l})$ is in the interior of $\mathcal{C}$. 
Let $\boldsymbol{q}_{1},\,\cdots,\,\boldsymbol{q}_{k}$ be the vertices of $\mathcal{C}$. 
For each $1\leq j \leq k$, we can write $\boldsymbol{q}_{j}=(q^{(j)}_{1},\,\cdots,\,q^{(j)}_{ l})$ for some positive rational numbers $q^{(j)}_{1},\,\cdots,\,q^{(j)}_{l}$. 
For every $1\leq j \leq k$, let $\tilde{A}_{j}$ be the ample $\mathbb{Q}$-divisor corresponding to $\boldsymbol{q}_{j}$. 
In other words,
$$\tilde{A}_{j}:=\sum_{i=1}^{l}q^{(j)}_{i}A_{i}.$$
Then all $\tilde{A}_{j}$ are $\mathbb{Q}$-Cartier, $(K_{X}+\Delta+\tilde{A}_{j})|_{{\rm Nklt}(X,\Delta)}$ are finite $\mathbb{Q}_{>0}$-linear combinations of invertible sheaves on ${\rm Nklt}(X,\Delta)$ that are globally generated over $Z$, and we can write $A=\sum_{j=1}^{k}\tilde{r}_{j} \tilde{A}_{j}$ for some positive real numbers $\tilde{r}_{1},\,\cdots,\,\tilde{r}_{k}$ such that $\sum_{j=1}^{k}\tilde{r}_{j}=1$. 

By replacing the sets $\{A_{1},\,\cdots, \, A_{l}\}$ and $\{r_{1},\,\cdots,\, r_{l}\}$ by the sets $\{\tilde{A}_{1},\,\cdots, \, \tilde{A}_{k}\}$ and $\{\tilde{r}_{1},\,\cdots,\, \tilde{r}_{k}\}$ respectively, we obtain a decomposition of $A$ into $\pi$-ample $\mathbb{Q}$-divisors
$$A=\sum_{i=1}^{l}r_{i}A_{i}$$
with $r_{1},\,\cdots,\,r_{l} \in \mathbb{R}_{>0}$ such that $\sum_{i=1}^{l}r_{i}=1$ and $(K_{X}+\Delta+A_{i})|_{{\rm Nklt}(X,\Delta)}$ is semi-ample over $Z$ for every $1 \leq i \leq l$. 
By Proposition \ref{prop--nonvan-ample-decom} and replacing $r_{1},\,\cdots,\, r_{l}$ and $A_{1},\,\cdots, \, A_{l}$ again, we may assume the existence of positive integers $p_{1},\,\cdots,\,p_{l}$, which are greater than one, such that
\begin{enumerate}[(I)]
\item\label{thm--non-vanishing-kltmmp-main-(I)}
$p_{i}(K_{X}+\Delta+A_{i})$ and $p_{i}A_{i}$ are Cartier for every $1 \leq i \leq l$,  
\item\label{thm--non-vanishing-kltmmp-main-(II)}
$p_{i}(K_{X}+\Delta+A_{i})|_{{\rm Nklt}(X,\Delta)}$ and $p_{i}A_{i}|_{{\rm Nklt}(X,\Delta)}$ are globally generated over $Z$ for every $1 \leq i \leq l$, and
\item\label{thm--non-vanishing-kltmmp-main-(III)}
$\frac{p_{i}}{p_{i}-1}A_{i} - A$ is ample over $Z$ for every $1 \leq i \leq l$. 
\end{enumerate}
\end{step3}

\begin{step3}\label{step4-thm--non-vanishing-kltmmp-main}
In this step, we prove ${\rm Bs}|K_{X}+\Delta+A_{i}/Z|_{\mathbb{R}} \cap {\rm Nklt}(X,\Delta) = \emptyset$ for every $1 \leq i \leq l$. 
Throughout this step, we fix $i$. 

By (\ref{thm--non-vanishing-kltmmp-main-(III)}) and the condition ${\rm NNef}(K_{X}+\Delta+A/Z) \cap {\rm Nklt}(X,\Delta) = \emptyset$, which is the hypothesis of Theorem \ref{thm--non-vanishing-kltmmp-main}, we have ${\rm NNef}(K_{X}+\Delta+\frac{p_{i}}{p_{i}-1}A_{i}/Z) \cap {\rm Nklt}(X,\Delta) = \emptyset$. 
By (\ref{thm--non-vanishing-kltmmp-main-(III)}) and the semi-ampleness of $(K_{X}+\Delta+A)|_{{\rm Nklt}(X,\Delta)}$ over $Z$, we see that the $\mathbb{Q}$-line bundle $(K_{X}+\Delta+\frac{p_{i}}{p_{i}-1}A_{i})|_{{\rm Nklt}(X,\Delta)}$ is ample over $Z$. 
From these facts, we may apply Proposition \ref{prop--non-vanishing-kltmmp} to $\pi \colon X \to Z$, $(X,\Delta)$, and $\frac{p_{i}}{p_{i}-1}A_{i}$. 
We see that 
$${\rm Bs}|K_{X}+\Delta+\frac{p_{i}}{p_{i}-1}A_{i}/Z|_{\mathbb{R}} \cap {\rm Nklt}(X,\Delta) = \emptyset.$$ 
We take a general member $H_{i} \in |A_{i}/Z|_{\mathbb{R}}$. 
By Corollary \ref{cor--mmp-termi-kltlocus}, we get a sequence of steps of a $(K_{X}+\Delta+\frac{p_{i}}{p_{i}-1}H_{i})$-MMP over $Z$ with scaling of $\frac{p_{i}}{p_{i}-1}H_{i}$ that terminates with a good minimal model over $Z$
$$(X,\Delta+\tfrac{p_{i}}{p_{i}-1}H_{i}) \dashrightarrow (X',\Delta'+\tfrac{p_{i}}{p_{i}-1}H').$$
Note that the good minimal model depends on $i$, although we remove the index $i$ because $i$ was fixed. 
Let $A'$ be the birational transform of $A_{i}$ on $X'$. 
By construction, the MMP is a $(K_{X}+\Delta+A_{i})$-MMP over $Z$ with scaling of $A_{i}$, and $K_{X'}+\Delta'+\frac{p_{i}}{p_{i}-1}A'$ is semi-ample over $Z$. 
In particular, $K_{X'}+\Delta'+A'$ and $A'$ are both $\mathbb{Q}$-Cartier. 
By the property ${\rm NNef}(K_{X}+\Delta+\frac{p_{i}}{p_{i}-1}A_{i}/Z)  \cap {\rm Nklt}(X,\Delta) = \emptyset$ and Remark \ref{rem--mmp-basic}, the birational map $X \dashrightarrow X'$ is an isomorphism on a neighborhood of ${\rm Nklt}(X,\Delta)$. 

We put
$$L:=p_{i}(K_{X'}+\Delta'+A').$$
By (\ref{thm--non-vanishing-kltmmp-main-(I)}), $L$ is a Weil divisor on $X'$ and $L$ is Cartier on a neighborhood of ${\rm Nklt}(X',\Delta')$. 
We can write
$$L=K_{X'}+\Delta'+(p_{i}-1)\left(K_{X'}+\Delta'+\frac{p_{i}}{p_{i}-1}A' \right).$$
Then $(K_{X'}+\Delta'+\frac{p_{i}}{p_{i}-1}A')|_{{\rm Nklt}(X,\Delta)}$ is ample over $Z$. 
This is because $X \dashrightarrow X'$ is an isomorphism on a neighborhood of ${\rm Nklt}(X,\Delta)$ and $(K_{X}+\Delta+\frac{p_{i}}{p_{i}-1}A_{i})|_{{\rm Nklt}(X,\Delta)}$ is ample over $Z$, which follows from (\ref{thm--non-vanishing-kltmmp-main-(III)}) and the semi-ampleness of $(K_{X}+\Delta+A)|_{{\rm Nklt}(X,\Delta)}$ over $Z$. 
Therefore, $L- (K_{X'}+\Delta')$ is nef and log big over $Z$ with respect to $(X',\Delta')$ (see \cite[Definition 3.7]{fujino-bpf-quasi-log}). 
From these facts, we can apply Theorem \ref{thm--vanishing-quasi-log} to $(X',\Delta')$ and $L$. 
Putting $\pi' \colon X' \to Z$ and $\mathcal{I}$ as the structure morphism and the defining ideal sheaf of ${\rm Nklt}(X',\Delta')$ respectively, then Theorem \ref{thm--vanishing-quasi-log} implies
$$R^{1}\pi'_{*}(\mathcal{I} \otimes_{\mathcal{O}_{X'}} \mathcal{O}_{X'}(L))=0.$$
We consider the exact sequence 
$$0 \longrightarrow \mathcal{I} \otimes_{\mathcal{O}_{X'}} \mathcal{O}_{X'}(L) \longrightarrow \mathcal{O}_{X'}(L) \longrightarrow \mathcal{O}_{{\rm Nklt}(X',\Delta')}(L|_{{\rm Nklt}(X',\Delta')}) \longrightarrow 0.$$
By taking the pushforward to $Z$, we get the following exact sequence
\begin{equation*}
\begin{split}
\pi'_{*}\mathcal{O}_{X'}(L)  \longrightarrow \pi'_{*}\mathcal{O}_{{\rm Nklt}(X',\Delta')}(L|_{{\rm Nklt}(X',\Delta')}) \longrightarrow R^{1}\pi'_{*}(\mathcal{I} \otimes_{\mathcal{O}_{X'}} \mathcal{O}_{X'}(L))=0.
\end{split}
\end{equation*}
By (\ref{thm--non-vanishing-kltmmp-main-(II)}) and construction of $X \dashrightarrow X'$, the restriction $L|_{{\rm Nklt}(X',\Delta')}$ is a line bundle on ${\rm Nklt}(X',\Delta')$ and this is globally generated over $Z$. 
Hence, we have 
$${\rm Bs}|K_{X'}+\Delta'+A'/Z|_{\mathbb{R}} \cap {\rm Nklt}(X',\Delta') ={\rm Bs}|L/Z|_{\mathbb{R}} \cap {\rm Nklt}(X',\Delta')= \emptyset.$$
Since $X \dashrightarrow X'$ is a sequence of steps of a $(K_{X}+\Delta+A_{i})$-MMP over $Z$ and the map is an isomorphism on ${\rm Nklt}(X,\Delta)$, we have ${\rm Bs}|K_{X}+\Delta+A_{i}/Z|_{\mathbb{R}} \cap {\rm Nklt}(X,\Delta) = \emptyset$. 
\end{step3}
By construction in Step \ref{step2-thm--non-vanishing-kltmmp-main}, the $\mathbb{R}$-divisor $K_{X}+\Delta+A$ is a convex linear combination of $K_{X}+\Delta+A_{1},\,\cdots,\,K_{X}+\Delta+A_{l}$. 
Since ${\rm Bs}|K_{X}+\Delta+A_{i}/Z|_{\mathbb{R}} \cap {\rm Nklt}(X,\Delta) = \emptyset$ for all $1 \leq i \leq l$, we have
$${\rm Bs}|K_{X}+\Delta+A/Z|_{\mathbb{R}} \cap {\rm Nklt}(X,\Delta) = \emptyset.$$
We finish the proof. 
\end{proof}

\begin{thm}\label{thm--klt-minmodeltheory-main}
Let $\pi \colon X \to Z$ be a projective morphism of normal quasi-projective varieties. 
Let $(X,\Delta)$ be a normal pair and let $A$ be a $\pi$-ample $\mathbb{R}$-divisor on $X$ such that $K_{X}+\Delta+A$ is $\pi$-pseudo-effective. 
Suppose that ${\rm NNef}(K_{X}+\Delta+A/Z) \cap {\rm Nklt}(X,\Delta) = \emptyset$. 
Suppose in addition that $(K_{X}+\Delta+A)|_{{\rm Nklt}(X,\Delta)}$, which we think of an $\mathbb{R}$-line bundle on ${\rm Nklt}(X,\Delta)$, is semi-ample over $Z$. 
Let 
$$(X,\Delta+A)=:(X_{1},\Delta_{1}+A_{1}) \dashrightarrow \cdots \dashrightarrow (X_{i},\Delta_{i}+A_{i}) \dashrightarrow \cdots$$
be a sequence of steps of a $(K_{X}+\Delta+A)$-MMP over $Z$ with scaling of $A$.  We put 
$$\lambda_{i}:=\{\mu \in \mathbb{R}_{\geq 0}\,| \,\text{$(K_{X_{i}}+\Delta_{i}+A_{i})+\mu A_{i}$ is nef over $Z$} \}$$ for each $i \geq 1$. 
Then 
${\rm lim}_{i \to \infty} \lambda_{i}=\lambda_{m}$ for some $m$. 
In particular, if ${\rm lim}_{i \to \infty} \lambda_{i}=0$ then the $(K_{X}+\Delta+A)$-MMP terminates with a good minimal model over $Z$. 
\end{thm}

\begin{proof}
If $(X,\Delta)$ is klt, then the theorem follows from \cite[Theorem E]{bchm} (see also  \cite[Theorem 4.1 (iii)]{birkar-flip}).
If $(X,\Delta)$ is not klt, then we put $\lambda:={\rm lim}_{i \to \infty} \lambda_{i}$. 
Replacing $A$ by $(1+\lambda) A$, we may assume $\lambda=0$. 
Then the theorem follows from Theorem \ref{thm--non-vanishing-kltmmp-main} and Theorem \ref{thm--from-minimodel-to-termi}. 
\end{proof}

\section{Minimal model program along log canonical locus}\label{sec--mmp-lc}

In this section we study the minimal model theory for normal pairs such that the non-nef locus of the log canonical $\mathbb{R}$-divisor is disjoint from the non-lc locus of the normal pair. 

\subsection{Proof of main result}

The goal of this subsection is to prove Theorem \ref{thm--termination-lcmmp-main}.

\begin{lem}\label{lem--mmp-fibration-kltlocus}
Let $f\colon (Y,\Delta) \to [X,\omega]$ be a quasi-log scheme induced by a normal pair, and let $\pi \colon X \to Z$ be a  projective morphism of normal quasi-projective varieties. 
Let $A$ be a $\pi$-ample $\mathbb{R}$-divisor on $X$ such that $\omega+A$ is $\pi$-pseudo-effective. 
Suppose that ${\rm NNef}(\omega+A/Z) \cap{\rm Nqklt}(X,\omega) = \emptyset$. 
Suppose in addition that $(\omega+A)|_{{\rm Nqklt}(X,\omega)}$, which we think of an $\mathbb{R}$-line bundle on ${\rm Nqklt}(X,\omega)$, is semi-ample over $Z$. 
Let
$$X=:X_{1} \dashrightarrow \cdots \dashrightarrow X_{i} \dashrightarrow \cdots$$
be a sequence of steps of an $(\omega+A)$-MMP over $Z$ with scaling of $A$ such that if we put $$\lambda_{i}:=\{\mu \in \mathbb{R}_{\geq 0}\,| \,\text{$\omega_{i}+A_{i}+\mu A_{i}$ is nef over $Z$} \}$$ for each $i \geq 1$, then 
${\rm lim}_{i \to \infty} \lambda_{i}=0$.  
Then $\lambda_{m}=0$ for some $m$. 
\end{lem}

\begin{proof}
We apply Lemma \ref{lem--can-bundle-formula} to $(Y,\Delta) \to [X,\omega]$ and $\frac{1}{2}A$. 
We get a normal pair $(X,\Gamma)$ such that 
$K_{X}+\Gamma \sim_{\mathbb{R},\,Z}\omega+\frac{1}{2}A$ and ${\rm Nklt}(X,\Gamma)$ is a closed subscheme of ${\rm Nqklt}(X,\omega)$. 
We consider $(X,\Gamma)$ and $\frac{1}{2}A$. 
Since 
$$K_{X}+\Gamma+\frac{1}{2}A \sim_{\mathbb{R},\,Z} \omega+A,$$
we have ${\rm NNef}(K_{X}+\Gamma+\frac{1}{2}A/Z) \cap {\rm Nklt}(X,\Gamma) = \emptyset$. 
Since $(\omega+A)|_{{\rm Nqklt}(X,\omega)}$ is semi-ample over $Z$, it follows that $(K_{X}+\Gamma+\frac{1}{2}A)|_{{\rm Nklt}(X,\Gamma)}$ is semi-ample over $Z$. 
By Remark \ref{rem--mmp-basic}, we may regard 
$X \dashrightarrow \cdots \dashrightarrow X_{i} \dashrightarrow \cdots$ as a sequence of steps of a $(K_{X}+\Gamma+\frac{1}{2}A)$-MMP over $Z$ with scaling of $\frac{1}{2}A$. 
We may apply Theorem \ref{thm--klt-minmodeltheory-main} to $X \to Z$, $(X,\Gamma)$, and $\frac{1}{2}A$, and we see that $\lambda_{m}=0$ for some $m$. 
\end{proof}

\begin{thm}\label{thm--mmp-lclocus-main}
Let $f\colon (Y,\Delta) \to [X,\omega]$ be a quasi-log scheme induced by a normal pair, and let $\pi \colon X \to Z$ be a  projective morphism of normal quasi-projective varieties. 
Let $A$ be a $\pi$-ample $\mathbb{R}$-divisor on $X$ such that $\omega+A$ is $\pi$-pseudo-effective. 
Suppose that ${\rm NNef}(\omega+A/Z) \cap{\rm Nqlc}(X,\omega) = \emptyset$. 
Suppose in addition that $(\omega+A)|_{{\rm Nqlc}(X,\omega)}$, which we think of an $\mathbb{R}$-line bundle on ${\rm Nqlc}(X,\omega)$, is semi-ample over $Z$. 
Then there exists a sequence of steps of an $(\omega+A)$-MMP over $Z$ with scaling of $A$ 
$$X=:X_{1} \dashrightarrow  X_{2} \dashrightarrow \cdots \dashrightarrow X_{n}$$
such that $\omega_{n}+A_{n}$ is semi-ample over $Z$, where $\omega_{n}$ and $A_{n}$ are the birational transforms of $\omega$ and $A$ on $X_{n}$ respectively.
\end{thm}

\begin{proof}
By taking a dlt blow-up of $(\widetilde{Y},\widetilde{\Delta}) \to (Y,\Delta)$ and replacing $(Y,\Delta)$ with $(\widetilde{Y},\widetilde{\Delta})$, we may assume that $Y$ is $\mathbb{Q}$-factorial and $(Y,\Delta^{<1}+{\rm Supp}\,\Delta^{\geq 1})$ is dlt. 
By replacing $A$ with a general member of $|A/Z|_{\mathbb{R}}$, we may assume that $A$ is an effective $\mathbb{R}$-divisor and ${\rm Nklt}(Y,\Delta+f^{*}A)={\rm Nklt}(Y,\Delta)$ and ${\rm Nlc}(Y,\Delta+f^{*}A)={\rm Nlc}(Y,\Delta)$ hold as relations of closed subschemes of $Y$. 
We may further assume that $(Y,\Delta^{<1}+{\rm Supp}\,\Delta^{\geq 1}+f^{*}A)$ is a dlt pair whose lc centers are those of $(Y,\Delta^{<1}+{\rm Supp}\,\Delta^{\geq 1})$. 

We divide the proof into several steps. 

\begin{step4}\label{step1-thm-mmp-lclocus-main}
In this step, we construct a diagram used in the proof. 

By Theorem \ref{thm--mmp-nefthreshold-strict}, there exists a diagram
$$
\xymatrix{
\hspace{-5mm}(Y,\Delta+f^{*}A)=:(Y_{1},\Delta_{1}+H_{1})\ar@<6.5ex>[d]_{f=:f_{1}}\ar@{-->}[r]&\cdots \ar@{-->}[r]&(Y_{i},\Delta_{i}+H_{i})\ar[d]_{f_{i}}\ar@{-->}[r]&\cdots\\
[X,\omega+A]=:[X_{1},\omega_{1}+A_{1}]  \ar@{-->}[r]&\cdots \ar@{-->}[r]&[X_{i},\omega_{i}+A_{i}]\ar@{-->}[r]&\cdots
}
$$
over $Z$, where $A_{i}$ (resp.~$H_{i}$) is the birational transform of $A$ (resp.~$f^{*}A$) on $X_{i}$ (resp.~$Y_{i}$), such that 
\begin{itemize}
\item
$f_{i}\colon (Y_{i},\Delta_{i}+H_{i}) \to [X_{i},\omega_{i}+A_{i}]$ is a quasi-log scheme induced by a normal pair such that $Y_{i}$ and $X_{i}$ are projective over $Z$ and every $Y_{i}$ is $\mathbb{Q}$-factorial, 
\item
the sequence of maps
$$(Y_{1},\Delta_{1}+H_{1}) \dashrightarrow \cdots \dashrightarrow (Y_{i},\Delta_{i}+H_{i}) \dashrightarrow \cdots$$ is a sequence of steps of a $(K_{Y_{1}}+\Delta_{1}+H_{1})$-MMP over $Z$ with scaling of $H_{1}$, 
\item
the sequence of maps
$$X_{1} \dashrightarrow \cdots \dashrightarrow X_{i} \dashrightarrow \cdots$$
 is a sequence of steps of an $(\omega_{1}+A_{1})$-MMP over $Z$ with scaling of $A_{1}$, and
\item
if we put 
$$\lambda_{i}:={\rm inf}\{\mu \in \mathbb{R}_{\geq 0}\,| \,\text{$\omega_{i}+A_{i}+\mu A_{i}$ is nef over $Z$} \}$$ 
for each $i \geq 1$, then the following properties hold. 
\begin{itemize}
\item[$\circ$]
$\lambda_{i}>\lambda_{i+1}$ for all $i \geq 1$, 
\item[$\circ$]
$\omega_{i}+A_{i}+\lambda_{i-1} A_{i}$ is semi-ample over $Z$ for all $i \geq 1$, and 
\item[$\circ$]
$\omega_{i}+A_{i}+\bar{\lambda} A_{i}$ is ample over $Z$ for every $i \geq 1$ and all $\bar{\lambda} \in (\lambda_{i},\lambda_{i-1})$.  
\end{itemize}
\end{itemize}
Note that the birational map $Y_{i}\dashrightarrow Y_{i+1}$ is not necessarily a step of an MMP. 
In fact, $Y_{i}\dashrightarrow Y_{i+1}$ can be a finite sequence of steps of an MMP.  

Put $\lambda={\rm lim}_{i \to \infty}\lambda_{i}$. 
Replacing $A$ by a general member of $|(1+\lambda) A/Z|_{\mathbb{R}}$, we may assume $\lambda =0$. 
By construction, $\omega_{i}$ and $A_{i}$ are $\mathbb{R}$-Cartier for any $i \geq 1$. 
Moreover, $X_{i}\dashrightarrow X_{i+1}$ and $Y_{i}\dashrightarrow Y_{i+1}$ are birational contractions. 
By an induction on $i$, we see that 
$H_{i}=f_{i}^{*}A_{i}$ for all $i \geq 1$. 
Since ${\rm NNef}(\omega+A/Z) \cap{\rm Nqlc}(X,\omega) = \emptyset$, we have 
$${\rm NNef}(K_{Y_{1}}+\Delta_{1}+H_{1}/Z) \cap{\rm Nlc}(Y_{1},\Delta_{1}) = \emptyset.$$ 
By Remark \ref{rem--mmp-basic}, the birational map 
$Y_{1} \dashrightarrow Y_{i}$
is an isomorphism on a neighborhood of ${\rm Nlc}(Y,\Delta)$ for every $i \geq 1$. 
By construction, ${\rm Nlc}(Y_{1},\Delta_{1}+H_{1})={\rm Nlc}(Y_{1},\Delta_{1})$ as closed subschemes, and therefore ${\rm Nlc}(Y_{i},\Delta_{i}+H_{i})={\rm Nlc}(Y_{i},\Delta_{i})$ as closed subschemes for every $i \geq 1$. 
Moreover,
$$Y_{1} \dashrightarrow \cdots \dashrightarrow Y_{i} \dashrightarrow \cdots$$
is a sequence of steps of a $(K_{Y_{1}}+\Delta^{<1}_{1}+{\rm Supp}\,\Delta^{\geq 1}_{1}+H_{1})$-MMP over $Z$. 
We also see that $(Y_{i},\Delta^{<1}_{i}+{\rm Supp}\,\Delta^{\geq 1}_{i}+H_{i})$ is a dlt pair and the lc centers of the pair are those of $(Y_{i},\Delta^{<1}_{i}+{\rm Supp}\,\Delta^{\geq 1}_{i})$. 
This is because $(Y,\Delta^{<1}+{\rm Supp}\,\Delta^{\geq 1}+f^{*}A)$ is a dlt pair whose lc centers are those of $(Y,\Delta^{<1}+{\rm Supp}\,\Delta^{\geq 1})$. 
Thus ${\rm Nklt}(Y_{i},\Delta_{i}+H_{i})={\rm Nklt}(Y_{i},\Delta_{i})$ as closed subschemes of $Y_{i}$ for all $i \geq 1$. 

By these discussions, we may assume that ${\rm lim}_{i \to \infty}\lambda_{i}=0$ and  the quasi-log scheme induced by a normal pair $(Y_{i},\Delta_{i}+H_{i}) \to [X_{i},\omega_{i}+A_{i}]$ satisfies the following properties for every $i \geq 1$.
\begin{itemize}
\item
$H_{i}=f_{i}^{*}A_{i}$, and therefore $(Y_{i},\Delta_{i}) \to [X_{i},\omega_{i}]$ is a quasi-log scheme induced by a normal pair, 
\item
$(Y_{i},\Delta^{<1}_{i}+{\rm Supp}\,\Delta^{\geq 1}_{i}+H_{i})$ is a $\mathbb{Q}$-factorial dlt pair whose lc centers are those of $(Y_{i},\Delta^{<1}_{i}+{\rm Supp}\,\Delta^{\geq 1}_{i})$, and 
\item
${\rm Nlc}(Y_{i},\Delta_{i}+H_{i})={\rm Nlc}(Y_{i},\Delta_{i})$ and ${\rm Nklt}(Y_{i},\Delta_{i}+H_{i})={\rm Nklt}(Y_{i},\Delta_{i})$ as closed subschemes of $Y_{i}$. 
\end{itemize}
\end{step4}

\begin{step4}\label{step2-thm-mmp-lclocus-main}
We will apply the argument of the special termination \cite{fujino-sp-ter} to our situation. 
We will prove that for any integer $0 \leq d < {\rm dim}\,Y$, there exists $m_{d} \in \mathbb{Z}_{>0}$ such that the non-isomorphic locus of the $(K_{Y_{m_{d}}}+\Delta_{m_{d}}+H_{m_{d}})$-MMP over $Z$
$$(Y_{m_{d}},\Delta_{m_{d}}+H_{m_{d}}) \dashrightarrow  \cdots \dashrightarrow (Y_{i},\Delta_{i}+H_{i}) \dashrightarrow  \cdots$$
is disjoint from any lc center of $(Y_{m_{d}},\Delta_{m_{d}})$ whose dimension is less than or equal to $d$. 
We prove the statement by an induction on $d$. 

We can find $i_{0} \in \mathbb{Z}_{>0}$ such that for any $i \geq i_{0}$, the birational map $Y_{i} \dashrightarrow Y_{i+1}$ is an isomorphism on the generic point of any lc center of $(Y_{i},\Delta_{i})$. 
Therefore, we may set $m_{0}:=i_{0}$. 
The case $d=0$ of the above statement holds true, and we may assume that  $Y_{m_{0}} \dashrightarrow Y_{i}$ is an isomorphism on the generic point of any lc center of $(Y_{m_{0}},\Delta_{m_{0}})$ for any $i \geq m_{0}$. 
\end{step4}

\begin{step4}\label{step3-thm-mmp-lclocus-main}
We assume the existence of $m_{d} \in \mathbb{Z}_{>0}$ of the statement in Step \ref{step2-thm-mmp-lclocus-main}. 
We may assume $m_{d} \geq m_{0}$. 
We fix an arbitrary $(d+1)$-dimensional lc center $S_{m_{d}}$ of $(Y_{m_{d}},\Delta_{m_{d}})$. 
By the choice of $m_{0}$, for any $i \geq m_{d}$ there is an lc center $S_{i}$ of $(Y_{i},\Delta_{i})$ such that the birational map $Y_{m_{d}} \dashrightarrow Y_{i}$ induces a birational map $S_{m_{d}} \dashrightarrow S_{i}$ over $Z$. 
In this step, we construct a normal pair $(S_{i},\Delta_{S_{i}})$ and an $\mathbb{R}$-divisor $H_{S_{i}}$ on $S_{i}$ for each $i \geq m_{d}$. 

Because $(Y_{i}, \Delta^{<1}_{i}+{\rm Supp}\,\Delta^{\geq 1}_{i})$ is a $\mathbb{Q}$-factorial dlt pair, as in Theorem \ref{thm--dlt-nonlc-locus} we can define an $\mathbb{R}$-divisor $B_{S_{i}}$ and an $\mathbb{R}$-Cartier $\mathbb{R}$-divisor $G_{S_{i}}$ on $S_{i}$ by 
\begin{equation*}
\begin{split}
K_{S_{i}}+B_{S_{i}}=(K_{Y_{i}}+\Delta^{<1}_{i}+{\rm Supp}\,\Delta^{\geq 1}_{i})|_{S_{i}}, \quad {\rm and} \quad
G_{S_{i}}:= (\Delta^{>1}_{i}-{\rm Supp}\,\Delta^{>1}_{i})|_{S_{i}}. 
\end{split}
\end{equation*}
Then $K_{S_{i}}+B_{S_{i}}+G_{S_{i}}=(K_{Y_{i}}+\Delta_{i})|_{S_{i}}$. 
We define an $\mathbb{R}$-Cartier $\mathbb{R}$-divisor $H_{S_{i}}$ on $S_{i}$ by 
$$H_{S_{i}}:=H_{i}|_{S_{i}}.$$
Since $(Y_{i},\Delta^{<1}_{i}+{\rm Supp}\,\Delta^{\geq 1}_{i}+H_{i})$ is a dlt pair and the lc centers of the pair are those of $(Y_{i},\Delta^{<1}_{i}+{\rm Supp}\,\Delta^{\geq 1}_{i})$, we see that $(S_{i},B_{S_{i}}+H_{S_{i}})$ is a dlt pair whose lc centers are those of $(S_{i},B_{S_{i}})$. 
Since the inclusion ${\rm Supp}\,G_{S_{i}} \subset {\rm Supp}\,\Delta^{>1}_{i}|_{S_{i}}$ holds and the non-isomorphic locus of the $(K_{Y_{1}}+\Delta_{1}+H_{1})$-MMP is disjoint from ${\rm Nlc}(Y_{1},\Delta_{1})$, the birational map $S_{i} \dashrightarrow S_{i+1}$ is an isomorphism on an open subset containing ${\rm Supp}\,G_{S_{i}}$. 
Therefore, the birational transform of $G_{S_{i}}$ on $S_{i+1}$ is equal to $G_{S_{i+1}}$. 
By using the argument of the special termination \cite{fujino-sp-ter} and replacing $m_{d}$, we may assume that for any $i \geq m_{d}$, the birational map $S_{i} \dashrightarrow S_{i+1}$ is small and the birational transforms of $B_{S_{i}}$ and $H_{S_{i}}$ on $S_{i+1}$ are equal to $B_{S_{i+1}}$ and $H_{S_{i+1}}$, respectively. 
Since any lc center of $(S_{i},B_{S_{i}}+G_{S_{i}}+H_{S_{i}})$ is an lc center of $(Y_{i},\Delta_{i}+H_{i})$ whose dimension is less than or equal to $d$, the birational map $S_{i} \dashrightarrow S_{i+1}$ is an isomorphism on an open subset containing all the lc centers of $(S_{i}, B_{S_{i}}+G_{S_{i}})$. 
By recalling that $S_{i} \dashrightarrow S_{i+1}$ is an isomorphism around ${\rm Supp}\,G_{S_{i}}$, we see that $S_{i} \dashrightarrow S_{i+1}$ is an isomorphism on a neighborhood of ${\rm Nklt}(S_{i},B_{S_{i}}+G_{S_{i}})$. 

We set $\Delta_{S_{i}} =B_{S_{i}}+G_{S_{i}}$. 
By the above argument and Theorem \ref{thm--dlt-nonlc-locus}, for any $i \geq m_{d}$, the $\mathbb{R}$-divisors $\Delta_{S_{i}}$ and $H_{S_{i}}$ satisfy the following properties.
\begin{itemize}
\item
$(S_{i}, \Delta_{S_{i}})$ is a normal pair defined by adjunction
$K_{S_{i}}+\Delta_{S_{i}}=(K_{Y_{i}}+\Delta_{i})|_{S_{i}},$
\item
$(S_{i}, \Delta^{<1}_{S_{i}}+{\rm Supp}\,\Delta^{\geq 1}_{S_{i}})$ is a dlt pair defined by adjunction
$$K_{S_{i}}+\Delta^{<1}_{S_{i}}+{\rm Supp}\,\Delta^{\geq 1}_{S_{i}}=(K_{Y_{i}}+\Delta^{<1}_{i}+{\rm Supp}\,\Delta^{\geq 1}_{i})|_{S_{i}},$$ 
\item
the birational map $S_{i} \dashrightarrow S_{i+1}$ is small and an isomorphism on a neighborhood of ${\rm Nklt}(S_{i}, \Delta_{S_{i}})$, and 
\item
the birational transforms of $\Delta_{S_{i}}$ and $H_{S_{i}}$ on $S_{i+1}$ are equal to $\Delta_{S_{i+1}}$ and $H_{S_{i+1}}$, respectively. 
\end{itemize}
\end{step4}

\begin{step4}\label{step4-thm-mmp-lclocus-main}
In this step, we consider the case when the image of $S_{m_{d}}$ on $X_{m_{d}}$ is contained in ${\rm Nqlc}(X_{m_{d}}, \omega_{m_{d}})$. 

By Remark \ref{rem--mmp-basic} and the hypothesis ${\rm NNef}(\omega+A/Z) \cap{\rm Nqlc}(X,\omega) = \emptyset$ of Theorem \ref{thm--mmp-lclocus-main}, the non-isomorphic locus of the  $(\omega+A)$-MMP is disjoint from ${\rm Nqlc}(X,\omega)$. 
Thus, we have ${\rm NNef}(\omega_{i}+A_{i}/Z) \cap{\rm Nqlc}(X_{i},\omega_{i}) = \emptyset$. 
Since $K_{Y_{i}}+\Delta_{i}+H_{i}\sim_{\mathbb{R}}f_{i}^{*}(\omega_{i}+A_{i})$, if the image of $S_{m_{d}}$ on $X_{m_{d}}$ is contained in ${\rm Nqlc}(X_{m_{d}},\omega_{m_{d}})$, then  $K_{S_{i}}+\Delta_{S_{i}}+H_{S_{i}}$ is nef over $Z$ for any $i \geq m_{d}$. 
For any $m_{d} \leq i \leq j$, by taking a common resolution of $S_{i} \dashrightarrow S_{j}$ and using the negativity lemma, we see that any prime divisor $P'$ over $S_{i}$ satisfies
$$a(P',S_{i}, \Delta_{S_{i}}+H_{S_{i}})=a(P',S_{j}, \Delta_{S_{j}}+H_{S_{j}}).$$
Since $Y_{m_{d}} \dashrightarrow  \cdots \dashrightarrow Y_{i} \dashrightarrow  \cdots$
is a sequence of steps of a $(K_{Y_{m_{d}}}+\Delta_{m_{d}}+H_{m_{d}})$-MMP over $Z$, by  the same argument as \cite[Proof of Lemma 3.38]{kollar-mori}, the above equality on the discrepancies implies that $Y_{i} \dashrightarrow Y_{j}$ is an isomorphism on a neighborhood of $S_{i}$ for any $m_{d} \leq i \leq j$. 
\end{step4}

\begin{step4}\label{step5-thm-mmp-lclocus-main}
From this step to Step \ref{step7-thm-mmp-lclocus-main}, we consider the case when the image of $S_{m_{d}}$ on $X_{m_{d}}$ is not contained in ${\rm Nqlc}(X_{m_{d}},\omega_{m_{d}})$. 
Then the image of $S_{i}$ on $X_{i}$ is not contained in ${\rm Nqlc}(X_{i},\omega_{i})$ for any $i \geq m_{d}$. 
We note that this condition is used for Theorem \ref{thm--pair-inductive-adjunction}. 
In this step we construct a sequence of small birational maps 
$$T_{m_{d}} \dashrightarrow \cdots \dashrightarrow T_{i} \dashrightarrow T_{i+1} \dashrightarrow \cdots$$
that will be used to prove the existence of $m_{d+1} \in \mathbb{Z}_{>0}$ of the statement in Step \ref{step2-thm-mmp-lclocus-main}. 

For each $i \geq m_{d}$, we have the diagram 
$$
\xymatrix{
(Y_{i},\Delta_{i}+H_{i})\ar[d]_{f_{i}}\ar@{-->}[rr]&& (Y_{i+1},\Delta_{i+1}+H_{i+1})\ar[d]^{f_{i+1}} \\
[X_{i},\omega_{i}+A_{i}] \ar@{-->}[rr]\ar[dr]_{\varphi_{i}}&& [X_{i+1},\omega_{i+1}+A_{i+1}] \ar[dl]^{\varphi'_{i}}\\
&V_{i}
}
$$
over $Z$ such that $\varphi_{i} \colon X_{i} \to V_{i}$ and $\varphi'_{i} \colon X_{i+1} \to V_{i}$ form a step of an $(\omega_{i}+A_{i})$-MMP over $Z$ in Definition \ref{defn--mmp-fullgeneral}.
In other words, $\varphi_{i}$ is a projective birational morphism, $\varphi'_{i}$ is a projective small birational morphism, and both $-(\omega_{i}+A_{i})$ and $\omega_{i+1}+A_{i+1}$ are ample over $V_{i}$. 
Let $g_{i} \colon S_{i} \to T_{i}$ (resp.~$g_{i+1} \colon S_{i+1} \to T_{i+1}$) be the Stein factorization of the restriction $f_{i}|_{S_{i}}$ (resp.~$f_{i+1}|_{S_{i+1}}$). 
Let $\psi_{i} \colon T_{i} \to W_{i}$ be the Stein factorization of $T_{i} \to V_{i}$. 
Then the morphism $T_{i+1} \to V_{i}$ factors through a contraction $T_{i+1} \to W_{i}$, which we denote by $\psi'_{i}$. 
Let $\omega_{T_{i}}$ (resp.~$\omega_{T_{i+1}}$) be the pullback of $\omega_{i}$ (resp.~$\omega_{i+1}$) to $T_{i}$ (resp.~$T_{i+1}$), and 
let $A_{T_{i}}$ (resp.~$A_{T_{i+1}}$) be the pullback of $A_{i}$ (resp.~$A_{i+1}$) to $T_{i}$ (resp.~$T_{i+1}$). 
Then we have
$$g_{i}^{*}A_{T_{i}} \sim_{\mathbb{R}} H_{S_{i}} \qquad {\rm and} \qquad g_{i+1}^{*}A_{T_{i+1}} \sim_{\mathbb{R}} H_{S_{i+1}},$$
and we have the following diagram
$$
\xymatrix{
(S_{i},\Delta_{S_{i}}+H_{S_{i}})\ar[d]_{g_{i}}\ar@{-->}[rr]&& (S_{i+1},\Delta_{S_{i+1}}+H_{S_{i+1}})\ar[d]^{g_{i+1}} \\
[T_{i},\omega_{T_{i}}+A_{T_{i}}] \ar@{-->}[rr]\ar[dr]_{\psi_{i}}&& [T_{i+1},\omega_{T_{i+1}}+A_{T_{i+1}}] \ar[dl]^{\psi'_{i}}\\
&W_{i},
}
$$
where $(S_{i}, \Delta_{S_{i}}) \overset{g_{i}}{\longrightarrow} [T_{i},\omega_{T_{i}}]$ and $(S_{i+1}, \Delta_{S_{i+1}}) \overset{g_{i+1}}{\longrightarrow}  [T_{i+1},\omega_{T_{i+1}}]$ 
are quasi-log schemes induced by normal pairs (Theorem \ref{thm--pair-inductive-adjunction}), and the divisors $-(\omega_{T_{i}}+A_{T_{i}})$ and $\omega_{T_{i+1}}+A_{T_{i+1}}$ are ample over $W_{i}$. 
We recall that ${\rm lim}_{j \to \infty}\lambda_{j}=0$, where 
$$\lambda_{j}:={\rm inf}\{\mu \in \mathbb{R}_{\geq 0}\,| \,\text{$\omega_{j}+A_{j}+\mu A_{j}$ is nef over $Z$} \}.$$
Then $K_{S_{j}}+\Delta_{S_{j}}+H_{S_{j}}+t H_{S_{j}}$ is semi-ample over $Z$ for all $j > i$ and $t \in (\lambda_{j},\lambda_{j-1}]$. 
Because $S_{i} \dashrightarrow S_{j}$ is small and the birational transforms of $\Delta_{S_{i}}$ and $H_{S_{i}}$ on $S_{j}$ are equal to $\Delta_{S_{j}}$ and $H_{S_{j}}$ respectively, $K_{S_{i}}+\Delta_{S_{i}}+H_{S_{i}}$ is the limit of movable $\mathbb{R}$-divisors over $Z$. 
Hence, $\omega_{T_{i}}+A_{T_{i}}$ is pseudo-effective over $Z$. 
This implies that $\psi_{i}$ is birational. 
By construction of the diagram, $\psi'_{i} \circ g_{i+1}$ coincides with $\psi_{i} \circ g_{i}$ over the generic point of $W_{i}$. 
Thus $\psi'_{i}$ is also birational. 
If there exists a $\psi_{i}$-exceptional prime divisor $Q$ on $T_{i}$, then we can find a prime divisor $P$ on $S_{i}$ such that $g_{i}(P)=Q$ and $S_{i} \dashrightarrow S_{i+1}$ is not an isomorphism on the generic point of $P$. 
This contradicts the fact that $S_{i} \dashrightarrow S_{i+1}$ is small. 
Therefore $\psi_{i}$ is small. 
By the same argument, we see that $\psi'_{i}$ is small. 
Hence, $T_{i} \dashrightarrow T_{i+1}$ is small, and the birational transform of $\omega_{T_{i}}$ (resp.~$A_{T_{i}}$) on $T_{i+1}$ is equal to $\omega_{T_{i+1}}$ (resp.~$A_{T_{i+1}}$). 
Therefore, the sequence of birational maps
$$T_{m_{d}} \dashrightarrow \cdots \dashrightarrow T_{i} \dashrightarrow T_{i+1}\dashrightarrow \cdots$$
is a sequence of steps of an $(\omega_{T_{m_{d}}}+A_{T_{m_{d}}})$-MMP over $Z$ in Definition \ref{defn--mmp-fullgeneral}. 
Moreover, it follows that the $(\omega_{T_{m_{d}}}+A_{T_{m_{d}}})$-MMP is an $(\omega_{T_{m_{d}}}+A_{T_{m_{d}}})$-MMP over $Z$ with scaling of $A_{T_{m_{d}}}$ in Definition \ref{defn--mmp-fullgeneral}. 
Indeed, define
$$\nu_{i}:={\rm inf}\{\mu \in \mathbb{R}_{\geq 0}\,| \,\text{$\omega_{T_{i}}+A_{T_{i}}+\mu A_{T_{i}}$ is nef over $Z$} \}$$ 
for each $i \geq m_{d}$, and pick any curve $C_{i} \subset T_{i}$ contracted by $\psi_{i} \colon T_{i} \to W_{i}$. 
Then $\nu_{i} \leq \lambda_{i}$, where $\lambda_{i}$ is the nonnegative real number defined in Step \ref{step1-thm-mmp-lclocus-main}.
Now $(\omega_{T_{i}}+A_{T_{i}})\,\cdot\, C_{i}<0$ and $(\omega_{T_{i}}+A_{T_{i}}+\lambda_{i} A_{T_{i}})\,\cdot\, C_{i}=0$ by construction. 
They imply $(\omega_{T_{i}}+A_{T_{i}}+\nu_{i} A_{T_{i}})\,\cdot \, C_{i} \leq 0$. 
We also have $(\omega_{T_{i}}+A_{T_{i}}+\nu_{i} A_{T_{i}}) \, \cdot \, C_{i} \geq 0$ by the definition of $\nu_{i}$. 
Therefore, it follows that $(\omega_{T_{i}}+A_{T_{i}}+\nu_{i} A_{T_{i}})\, \cdot \, C_{i} = 0,$ which is the condition of the MMP with scaling in Definition \ref{defn--mmp-fullgeneral}. 

By these discussions, the sequence of small birational maps
$$T_{m_{d}} \dashrightarrow \cdots \dashrightarrow T_{i} \dashrightarrow T_{i+1}\dashrightarrow \cdots$$
is a sequence of steps of an $(\omega_{T_{m_{d}}}+A_{T_{m_{d}}})$-MMP over $Z$ with scaling of $A_{T_{m_{d}}}$.  
\end{step4}

\begin{step4}\label{step6-thm-mmp-lclocus-main}
As in Step \ref{step5-thm-mmp-lclocus-main}, we define 
$$\nu_{i}:={\rm inf}\{\mu \in \mathbb{R}_{\geq 0}\,| \,\text{$\omega_{T_{i}}+A_{T_{i}}+\mu A_{T_{i}}$ is nef over $Z$} \}$$ 
for each $i \geq m_{d}$. 
By construction of the diagram in Step \ref{step1-thm-mmp-lclocus-main} (see also Theorem \ref{thm--mmp-nefthreshold-strict}), there exists $\lambda' \in (\lambda_{m_{d}},\lambda_{m_{d}-1})$ such that $\omega_{T_{m_{d}}}+A_{T_{m_{d}}}+\lambda' A_{T_{m_{d}}}$ is ample over $Z$. 
In this step, we prove $\nu_{m'}=0$ for some $m' \geq m_{d}$. 
In other words, we prove that $T_{i} \dashrightarrow T_{i+1}$ is an isomorphism for every $i \geq m'$ (see also Remark \ref{rem--mmp-fullgeneral}). 
We will use Lemma \ref{lem--mmp-fibration-kltlocus}. 

We first prove that ${\rm NNef}(\omega_{T_{m_{d}}}+A_{T_{m_{d}}}/Z) \cap {\rm Nqklt}(T_{m_{d}}, \omega_{T_{m_{d}}})=\emptyset$. 
We pick an arbitrary $t \in (0,\lambda')$. 
We may write
$$(\omega_{T_{m_{d}}}+A_{T_{m_{d}}})+t(\omega_{T_{m_{d}}}+A_{T_{m_{d}}}+\lambda' A_{T_{m_{d}}})=(1+t)\left(\omega_{T_{m_{d}}}+A_{T_{m_{d}}}+\frac{\lambda't}{1+t} A_{T_{m_{d}}}\right).$$
Since ${\rm lim}_{i \to \infty}\lambda_{i}=0$, we can find $i'$ such that $\frac{\lambda't}{1+t} \in (\lambda_{i'},\lambda_{i'-1}]$. 
Since $\omega_{T_{i'}}+A_{T_{i'}}+\frac{\lambda't}{1+t} A_{T_{i'}}$ is semi-ample over $Z$ and 
$$K_{S_{i'}}+\Delta_{S_{i'}}+H_{S_{i'}}+\frac{\lambda't}{1+t} H_{S_{i'}}\sim_{\mathbb{R}}g_{i'}^{*}(\omega_{T_{i'}}+A_{T_{i'}}+\frac{\lambda't}{1+t} A_{T_{i'}}),$$
it follows that $K_{S_{i'}}+\Delta_{S_{i'}}+H_{S_{i'}}+\frac{\lambda't}{1+t} H_{S_{i'}}$ is semi-ample over $Z$. 
Recall that the birational map $S_{m_{d}} \dashrightarrow S_{i'}$ is small and an isomorphism on a neighborhood of ${\rm Nklt}(S_{m_{d}}, \Delta_{S_{m_{d}}})$. 
Thus
$${\rm Bs}|K_{S_{m_{d}}}+\Delta_{S_{m_{d}}}+H_{S_{m_{d}}}+\frac{\lambda't}{1+t} H_{S_{m_{d}}}/Z|_{\mathbb{R}} \cap {\rm Nklt}(S_{m_{d}}, \Delta_{S_{m_{d}}}) = \emptyset.$$
Since 
$$K_{S_{m_{d}}}+\Delta_{S_{m_{d}}}+H_{S_{m_{d}}}+\frac{\lambda't}{1+t} H_{S_{m_{d}}}\sim_{\mathbb{R}}g_{m_{d}}^{*}(\omega_{T_{m_{d}}}+A_{T_{m_{d}}}+\frac{\lambda't}{1+t} A_{T_{m_{d}}})$$
and $g_{m_{d}}({\rm Nklt}(S_{m_{d}}, \Delta_{S_{m_{d}}}))={\rm Nqklt}(T_{m_{d}}, \omega_{T_{m_{d}}})$ set-theoretically, we have
\begin{equation*}
\begin{split}
&{\rm Bs}|(\omega_{T_{m_{d}}}+A_{T_{m_{d}}})+t(\omega_{T_{m_{d}}}+A_{T_{m_{d}}}+\lambda' A_{T_{m_{d}}})/Z|_{\mathbb{R}} \cap {\rm Nqklt}(T_{m_{d}}, \omega_{T_{m_{d}}})\\
 =&{\rm Bs}|\omega_{T_{m_{d}}}+A_{T_{m_{d}}}+\frac{\lambda't}{1+t} A_{T_{m_{d}}}/Z|_{\mathbb{R}} \cap {\rm Nqklt}(T_{m_{d}}, \omega_{T_{m_{d}}}) = \emptyset.
\end{split}
\end{equation*}
Since $t \in \mathbb{R}_{>0}$ is arbitrary, ${\rm NNef}(\omega_{T_{m_{d}}}+A_{T_{m_{d}}}/Z) \cap {\rm Nqklt}(T_{m_{d}}, \omega_{T_{m_{d}}})=\emptyset$. 

Next we prove that $(\omega_{T_{m_{d}}}+A_{T_{m_{d}}})|_{{\rm Nqklt}(T_{m_{d}}, \omega_{T_{m_{d}}})}$ is semi-ample over $Z$. 
We recall the hypothesis of Theorem \ref{thm--mmp-lclocus-main} that 
${\rm NNef}(\omega+A/Z) \cap {\rm Nqlc}(X,\omega) = \emptyset$. 
This condition and Remark \ref{rem--mmp-basic} show that the non-isomorphic locus of the  $(\omega+A)$-MMP is disjoint from ${\rm Nqlc}(X,\omega)$. 
Recall also that $(\omega+A)|_{{\rm Nqlc}(X,\omega)}$ is semi-ample over $Z$, which is the hypothesis of Theorem \ref{thm--mmp-lclocus-main}.
Therefore, $(\omega_{m_{d}}+A_{m_{d}})|_{{\rm Nqlc}(X_{m_{d}},\omega_{m_{d}})}$ is semi-ample over $Z$. 
Applying Theorem \ref{thm--pair-inductive-adjunction} to $[Y_{m_{d}},\Delta_{m_{d}}] \to (X_{m_{d}},\omega_{m_{d}})$ and $S_{m_{d}}$, we see that $T_{m_{d}} \to X_{m_{d}}$ induces a morphism ${\rm Nqlc}(T_{m_{d}}, \omega_{T_{m_{d}}}) \to {\rm Nqlc}(X_{m_{d}},\omega_{m_{d}})$ of closed subschemes. Then $(\omega_{T_{m_{d}}}+A_{T_{m_{d}}})|_{{\rm Nqlc}(T_{m_{d}}, \omega_{T_{m_{d}}})}$ is the pullback of $(\omega_{m_{d}}+A_{m_{d}})|_{{\rm Nqlc}(X_{m_{d}},\omega_{m_{d}})}$ by this morphism. 
Thus, $(\omega_{T_{m_{d}}}+A_{T_{m_{d}}})|_{{\rm Nqlc}(T_{m_{d}}, \omega_{T_{m_{d}}})}$ is semi-ample over $Z$. 
If ${\rm Nqklt}(T_{m_{d}}, \omega_{T_{m_{d}}})$ coincides with ${\rm Nqlc}(T_{m_{d}}, \omega_{T_{m_{d}}})$ as a closed subscheme of $T_{m_{d}}$, then $(\omega_{T_{m_{d}}}+A_{T_{m_{d}}})|_{{\rm Nqklt}(T_{m_{d}}, \omega_{T_{m_{d}}})}$ is clearly semi-ample over $Z$. If ${\rm Nqklt}(T_{m_{d}}, \omega_{T_{m_{d}}})$ does not coincide with ${\rm Nqlc}(T_{m_{d}}, \omega_{T_{m_{d}}})$ as a closed subscheme of $T_{m_{d}}$, then there exists a qlc center of $[T_{m_{d}}, \omega_{T_{m_{d}}}]$. 
We put 
$$T'={\rm Nqklt}(T_{m_{d}}, \omega_{T_{m_{d}}})$$ 
and we apply the adjunction for quasi-log schemes (\cite[Theorem 6.1.2 (i)]{fujino-book}). We may define the structure of a quasi-log scheme $[T',\omega_{T'}]$ on $T'$ such that $\omega_{T'}\sim_{\mathbb{R}} \omega_{T_{m_{d}}}|_{T'}$, ${\rm Nqlc}(T', \omega_{T'}) = {\rm Nqlc}(T_{m_{d}}, \omega_{T_{m_{d}}})$, and the qlc centers of $[T',\omega_{T'}]$ are exactly the qlc centers of $[T_{m_{d}}, \omega_{T_{m_{d}}}]$ contained in $T'$. Now $(\omega_{T_{m_{d}}}+A_{T_{m_{d}}})|_{T'}$ is nef over $Z$ because the non-nef locus of $\omega_{T_{m_{d}}}+A_{T_{m_{d}}}$ over $Z$ does not intersect $T'={\rm Nqklt}(T_{m_{d}}, \omega_{T_{m_{d}}})$. 
Then $$(1+\lambda')(\omega_{T_{m_{d}}}+A_{T_{m_{d}}})|_{T'}\sim_{\mathbb{R}} \lambda'\left(\omega_{T'}+\frac{1}{\lambda'}(\omega_{T_{m_{d}}}+A_{T_{m_{d}}}+\lambda' A_{T_{m_{d}}})|_{T'}\right).$$ Hence, we may apply Theorem \ref{thm--abundance-quasi-log} to $[T',\omega_{T'}]$ and $\frac{1}{\lambda'}(\omega_{T_{m_{d}}}+A_{T_{m_{d}}}+\lambda' A_{T_{m_{d}}})|_{T'}$, and we see that $$(\omega_{T_{m_{d}}}+A_{T_{m_{d}}})|_{{\rm Nqklt}(T_{m_{d}}, \omega_{T_{m_{d}}})}=(\omega_{T_{m_{d}}}+A_{T_{m_{d}}})|_{T'}$$ is semi-ample over $Z$. 

We have
$$(1+\lambda')(\omega_{T_{m_{d}}}+A_{T_{m_{d}}}) = \lambda'\left(\omega_{T_{m_{d}}}+\frac{1}{\lambda'}(\omega_{T_{m_{d}}}+A_{T_{m_{d}}}+\lambda' A_{T_{m_{d}}})\right).$$
We put $A':=\frac{1}{\lambda'}(\omega_{T_{m_{d}}}+A_{T_{m_{d}}}+\lambda' A_{T_{m_{d}}})$.  
Then
$$\omega_{T_{m_{d}}}+A'=\frac{1+\lambda'}{\lambda'}(\omega_{T_{m_{d}}}+A_{T_{m_{d}}}).$$
By Remark \ref{rem--mmp-basic}, the $(\omega_{T_{m_{d}}}+A_{T_{m_{d}}})$-MMP over $Z$ with scaling of $A_{T_{m_{d}}}$
$$T_{m_{d}} \dashrightarrow \cdots \dashrightarrow T_{i}\dashrightarrow \cdots$$
is also a sequence of steps of an $(\omega_{T_{m_{d}}}+A')$-MMP over $Z$ with scaling of $A'$. 

By the above discussions, we may apply Lemma \ref{lem--mmp-fibration-kltlocus} to the $(\omega_{T_{m_{d}}}+A')$-MMP over $Z$ with scaling of $A'$. 
By Lemma \ref{lem--mmp-fibration-kltlocus}, we have $\nu_{m'} =0$ for some $m' \geq m_{d}$. 
\end{step4}

\begin{step4}\label{step7-thm-mmp-lclocus-main}
In this step we prove that $Y_{i} \dashrightarrow Y_{j}$ is an isomorphism on a neighborhood of $S_{i}$ for any $m' \leq i \leq j$, where $m'$ is the integer in Step \ref{step6-thm-mmp-lclocus-main}, which satisfies $\nu_{m'} =0$. 

We recall that
$$(Y_{m_{d}},\Delta_{m_{d}}+H_{m_{d}}) \dashrightarrow  \cdots \dashrightarrow (Y_{i},\Delta_{i}+H_{i}) \dashrightarrow  \cdots$$
is a sequence of steps of a $(K_{Y_{m_{d}}}+\Delta_{m_{d}}+H_{m_{d}})$-MMP over $Z$ and the restriction of this sequence to $S_{m_{d}}$ is a sequence of small birational maps
$$(S_{m_{d}},\Delta_{S_{m_{d}}}+H_{S_{m_{d}}}) \dashrightarrow  \cdots \dashrightarrow (S_{i},\Delta_{S_{i}}+H_{S_{i}}) \dashrightarrow  \cdots$$
such that 
$$K_{S_{i}}+\Delta_{S_{i}}+H_{S_{i}} \sim_{\mathbb{R}} g_{i}^{*}(\omega_{T_{i}}+A_{T_{i}}),$$
where $(S_{i}, \Delta_{S_{i}}) \overset{g_{i}}{\longrightarrow} [T_{i},\omega_{T_{i}}]$ is the structure of a quasi-log scheme induced by a normal pair. 
By Step \ref{step6-thm-mmp-lclocus-main}, the divisor $K_{S_{i}}+\Delta_{S_{i}}+H_{S_{i}}$ is nef over $Z$ for any $i \geq m'$. 
By the same argument as in Step \ref{step4-thm-mmp-lclocus-main}, the birational map $Y_{i} \dashrightarrow Y_{j}$ is an isomorphism on a neighborhood of $S_{i}$ for any $m' \leq i \leq j$. 
\end{step4}

\begin{step4}\label{step8-thm-mmp-lclocus-main}
In this step we complete the argument of the special termination by finding $m_{d+1} \in \mathbb{Z}_{>0}$ in the statement of Step \ref{step2-thm-mmp-lclocus-main}. 

By Step \ref{step4-thm-mmp-lclocus-main} and Step \ref{step7-thm-mmp-lclocus-main}, for any $(d+1)$-dimensional lc center of $(Y_{m_{d}},\Delta_{m_{d}})$, we can find $m_{d+1} \in \mathbb{Z}_{>0}$ in the statement of Step \ref{step2-thm-mmp-lclocus-main}. 
In other words, the non-isomorphic locus of the $(K_{Y_{m_{d+1}}}+\Delta_{m_{d+1}}+H_{m_{d+1}})$-MMP over $Z$
$$(Y_{m_{d+1}},\Delta_{m_{d+1}}+H_{m_{d+1}}) \dashrightarrow  \cdots \dashrightarrow (Y_{i},\Delta_{i}+H_{i}) \dashrightarrow  \cdots$$
is disjoint from any lc center of $(Y_{m_{d+1}},\Delta_{m_{d+1}})$ whose dimension is less than or equal to $d+1$. 
By an induction on $0 \leq d < {\rm dim}\,Y$, there exists $m \in \mathbb{Z}_{>0}$ such that each step of the $(K_{Y_{m}}+\Delta_{m}+H_{m})$-MMP over $Z$
$$(Y_{m},\Delta_{m}+H_{m}) \dashrightarrow  \cdots \dashrightarrow (Y_{i},\Delta_{i}+H_{i}) \dashrightarrow  \cdots$$
is an isomorphism on an open subset containing all lc centers of $(Y_{m},\Delta_{m})$. 
Since the non-isomorphic locus of the $(K_{Y_{m}}+\Delta_{m}+H_{m})$-MMP is disjoint from ${\rm Nlc}(Y_{m},\Delta_{m})$, the non-isomorphic locus of the $(K_{Y_{m}}+\Delta_{m}+H_{m})$-MMP is disjoint from ${\rm Nklt}(Y_{m},\Delta_{m})$. 
\end{step4} 

\begin{step4}\label{step9-thm-mmp-lclocus-main}
In this step we prove $\lambda_{n}=0$ for some $n \geq m$. 

By the same argument as in Step \ref{step6-thm-mmp-lclocus-main}, we have ${\rm NNef}(\omega_{m}+A_{m}/Z) \cap {\rm Nqklt}(X_{m}, \omega_{m})=\emptyset$ and $(\omega_{m}+A_{m})|_{{\rm Nqklt}(X_{m}, \omega_{m})}$ is semi-ample over $Z$. 
We pick $\lambda'' \in (\lambda_{m},\lambda_{m-1})$ such that $\omega_{m}+A_{m}+\lambda'' A_{m}$ is ample over $Z$. 
We have
$$(1+\lambda'')(\omega_{m}+A_{m}) = \lambda''\left(\omega_{m}+\frac{1}{\lambda''}(\omega_{m}+A_{m}+\lambda'' A_{m})\right).$$
We put 
$$A'':=\frac{1}{\lambda''}(\omega_{m}+A_{m}+\lambda'' A_{m}).$$
By Remark \ref{rem--mmp-basic}, the $(\omega_{m}+A_{m})$-MMP over $Z$ with scaling of $A_{m}$
$$X_{m} \dashrightarrow \cdots \dashrightarrow X_{i}\dashrightarrow \cdots$$
is also a sequence of steps of an $(\omega_{m}+A'')$-MMP over $Z$ with scaling of $A''$. 
We may apply Lemma \ref{lem--mmp-fibration-kltlocus} to the $(\omega_{m}+A'')$-MMP over $Z$ with scaling of $A''$. 
By Lemma \ref{lem--mmp-fibration-kltlocus}, we have $\lambda_{n} =0$ for some $n \geq m$. 
\end{step4}

We fix the smallest $n \in \mathbb{Z}_{>0}$ such that $\lambda_{n}=0$. 
By the same argument as in Step \ref{step9-thm-mmp-lclocus-main} in this proof, we have 
$$(1+\frac{\lambda_{n-1}}{2})(\omega_{n}+A_{n}) = \frac{\lambda_{n-1}}{2}\left(\omega_{n}+\frac{2}{\lambda_{n-1}}(\omega_{n}+A_{n}+\frac{\lambda_{n-1}}{2} A_{n})\right)$$
and $(\omega_{n}+A_{n})|_{{\rm Nqklt}(X_{n}, \omega_{n})}$ is semi-ample over $Z$. 
Now $\lambda_{n-1}>0$ and $\omega_{n}+A_{n}+\frac{\lambda_{n-1}}{2} A_{n}$ is ample over $Z$ by construction in Step \ref{step1-thm-mmp-lclocus-main}.  
By applying Theorem \ref{thm--abundance-quasi-log} to $[X_{n}, \omega_{n}]$ and $\frac{2}{\lambda_{n-1}}(\omega_{n}+A_{n}+\frac{\lambda_{n-1}}{2} A_{n})$, we see that the $\mathbb{R}$-divisor $\omega_{n}+A_{n}$ is semi-ample over $Z$. 
We finish the proof. 
\end{proof}

\begin{thm}\label{thm--termination-lcmmp-main}
Let $\pi \colon X \to Z$ be a projective morphism of normal quasi-projective varieties. 
Let $(X,\Delta)$ be a normal pair and let $A$ be a $\pi$-ample $\mathbb{R}$-divisor on $X$ such that $K_{X}+\Delta+A$ is $\pi$-pseudo-effective. 
Suppose that ${\rm NNef}(K_{X}+\Delta+A/Z) \cap {\rm Nlc}(X,\Delta) = \emptyset$. 
Suppose in addition that $(K_{X}+\Delta+A)|_{{\rm Nlc}(X,\Delta)}$, which we think of an $\mathbb{R}$-line bundle on ${\rm Nlc}(X,\Delta)$, is semi-ample over $Z$. 
We put $(X_{1},B_{1}):=(X,\Delta+A)$. 
Let $H$ be a $\pi$-ample $\mathbb{R}$-divisor on $X$. 
Then there exists a sequence of steps of a $(K_{X_{1}}+B_{1})$-MMP over $Z$ with scaling of $H$ 
$$
\xymatrix{
(X_{1},B_{1})\ar@{-->}[r]&\cdots \ar@{-->}[r]& (X_{i},B_{i})\ar@{-->}[rr]\ar[dr]&& (X_{i+1},B_{i+1})\ar[dl] \ar@{-->}[r]&\cdots\ar@{-->}[r]&(X_{m},B_{m}), \\
&&&W_{i}
}
$$
where $(X_{i},B_{i}) \to W_{i} \leftarrow (X_{i+1},B_{i+1})$ is a step of a $(K_{X_{i}}+B_{i})$-MMP over $Z$, such that
\begin{itemize}
\item
the non-isomorphic locus of the MMP is disjoint from ${\rm Nlc}(X,\Delta)$, 
\item
$\rho(X_{i}/W_{i})=1$ for every $i \geq 1$, and
\item
$K_{X_{m}}+B_{m}$ is semi-ample over $Z$. 
\end{itemize}
Moreover, if $X$ is $\mathbb{Q}$-factorial, then all $X_{i}$ in the MMP are also $\mathbb{Q}$-factorial. 
\end{thm}

\begin{proof}
We pick $\epsilon \in \mathbb{R}_{>0}$ such that $A-\epsilon H$ is $\pi$-semi-ample. 
Then we can find a member $\Delta'$ of $|\Delta+A-\epsilon H/Z|_{\mathbb{R}}$ such that ${\rm Nlc}(X,\Delta')={\rm Nlc}(X,\Delta)$ as closed subschemes of $X$. 
By replacing $\Delta$, $A$, and $H$ with $\Delta'$, $\epsilon H$, and $\epsilon H$ respectively, we may assume $A= H$. 

By Corollary \ref{cor--mmp-nomralpair-Qfacdlt}, there exists a sequence of steps of a $(K_{X}+B)$-MMP over $Z$ with scaling of $A$ 
$$
\xymatrix{
(X_{1},B_{1})\ar@{-->}[r]&\cdots \ar@{-->}[r]& (X_{i},B_{i})\ar@{-->}[rr]\ar[dr]&& (X_{i+1},B_{i+1})\ar[dl] \ar@{-->}[r]&\cdots, \\
&&&W_{i}
}
$$
where $(X_{i},B_{i}) \to W_{i} \leftarrow (X_{i+1},B_{i+1})$ is a step of a $(K_{X_{i}}+B_{i})$-MMP over $Z$, satisfying the following properties.
\begin{itemize}
\item
The non-isomorphic locus of the MMP is disjoint from ${\rm Nlc}(X,\Delta)$, 
\item
$\rho(X_{i}/W_{i})=1$ for every $\geq 1$, 
\item
if we put 
$$\lambda_{i}:={\rm inf}\{\mu \in \mathbb{R}_{\geq 0}\,| \,\text{$K_{X_{i}}+B_{i}+\mu A_{i}$ is nef over $Z$} \}$$  
for each $i \geq 1$ and $\lambda:={\rm lim}_{i \to \infty}\lambda_{i}$, then the MMP terminates after finitely many steps or otherwise $\lambda \neq \lambda_{i}$ for every $i \geq 1$, and
\item 
if $X$ is $\mathbb{Q}$-factorial, then all $X_{i}$ in the MMP are also $\mathbb{Q}$-factorial. 
\end{itemize}
It is sufficient to prove that $\lambda_{m}=0$ for some $m$ and $K_{X_{m}}+B_{m}$ is semi-ample over $Z$. 

By Theorem \ref{thm--mmp-lclocus-main}, there exists a sequence of steps of an $(K_{X}+B)$-MMP over $Z$ with scaling of $A$ 
$$(X_{1},B_{1})=:(X'_{1},B'_{1})  \dashrightarrow  (X'_{2},B'_{2}) \dashrightarrow \cdots \dashrightarrow (X'_{n}, B'_{n})$$
such that $K_{X'_{n}}+B'_{n}$ is semi-ample over $Z$. 
We put 
$$\lambda'_{j}:={\rm inf}\{\mu \in \mathbb{R}_{\geq 0}\,| \,\text{$K_{X'_{j}}+B'_{j}+\mu A'_{j}$ is nef over $Z$} \}$$  
for each $1 \leq j \leq n$, where $A'_{j}$ is the birational transform of $A$ on $X'_{j}$. 
Then we can find $l$ such that $\lambda'_{l} \leq \lambda <\lambda'_{l-1}$. 
For this $l$, we can find $m$ such that $\lambda_{m}<\lambda'_{l-1}$ and $\lambda_{m}<\lambda_{m-1}$. 
Such $m$ exists since $\lambda ={\rm lim}_{i \to \infty}\lambda_{i}$. 
We put $\lambda'':= {\rm min}\{\lambda_{m-1}, \lambda'_{l-1}\}$. 
Then $\lambda_{m}<\lambda''$.  
Let $f_{1}\colon Y \to X_{1}$, $f_{m} \colon Y \to X_{m}$, and $f'_{l} \colon Y \to X'_{l}$ be a common resolution of $X_{1} \dashrightarrow X_{m}\dashrightarrow X'_{l}$. 
We have 
$$\lambda'_{l}\leq \lambda \leq \lambda_{m}<\lambda''\leq \lambda'_{l-1}.$$
For any $t \in (\lambda_{m},\lambda'')$, we can write
$$f_{1}^{*}(K_{X_{1}}+B_{1}+t A_{1})=f_{m}^{*}(K_{X_{m}}+B_{m}+t A_{m})+E^{(t)}$$
with an effective $f_{m}$-exceptional $\mathbb{R}$-divisor $E^{(t)}$ on $Y$. 
We can also write
$$f_{1}^{*}(K_{X_{1}}+B_{1}+t A_{1})=f'^{*}_{l}(K_{X'_{l}}+B'_{l}+t A'_{l})+F^{(t)}$$
with an effective $f'_{l}$-exceptional $\mathbb{R}$-divisor $F^{(t)}$ on $Y$. 
Then 
$$f_{m}^{*}(K_{X_{m}}+B_{m}+t A_{m})+E^{(t)}=f'^{*}_{l}(K_{X'_{l}}+B'_{l}+t A'_{l})+F^{(t)}.$$
Since $f_{m}^{*}(K_{X_{m}}+B_{m}+t A_{m})$ and $f'^{*}_{l}(K_{X'_{l}}+B'_{l}+t A'_{l})$ are nef over $Z$, the negativity lemma implies $E^{(t)}=F^{(t)}$ for any $t \in (\lambda_{m},\lambda'')$. 
Then 
$$f_{m}^{*}(K_{X_{m}}+B_{m}+t A_{m})=f'^{*}_{l}(K_{X'_{l}}+B'_{l}+t A'_{l})$$
for any $t \in (\lambda_{m},\lambda'')$, and therefore 
$$f_{m}^{*}(K_{X_{m}}+B_{m}+\lambda'_{l} A_{m})=f'^{*}_{l}(K_{X'_{l}}+B'_{l}+\lambda'_{l} A'_{l}).$$
This shows that $K_{X_{m}}+B_{m}+\lambda'_{l} A_{m}$ is nef over $Z$. 
By the choice of $l$ and the definition of $\lambda_{m}$, we have $\lambda_{m}\leq \lambda'_{l} \leq \lambda$. 
This implies $\lambda_{m}=\lambda=0$ because the $(K_{X}+B)$-MMP 
$$(X_{1},B_{1}) \dashrightarrow\cdots \dashrightarrow  (X_{i},B_{i}) \dashrightarrow \cdots $$
terminates after finitely many steps or otherwise $\lambda \neq \lambda_{i}$ for every $i \geq 1$. 
Then $l=n$ since $\lambda'_{l} \leq \lambda=0$. 
Now we have
$$f_{m}^{*}(K_{X_{m}}+B_{m})=f'^{*}_{n}(K_{X'_{n}}+B'_{n})$$
and $K_{X'_{n}}+B'_{n}$ is semi-ample over $Z$. 
Thus, $K_{X_{m}}+B_{m}$ is semi-ample over $Z$. 
We finish the proof. 
\end{proof}

\begin{thm}\label{thm--termination-all-lcmmp}
Let $\pi \colon X \to Z$ be a projective morphism of normal quasi-projective varieties. 
Let $(X,\Delta)$ be a normal pair and let $A$ be a $\pi$-ample $\mathbb{R}$-divisor on $X$ such that $K_{X}+\Delta+A$ is $\pi$-pseudo-effective. 
Suppose that ${\rm NNef}(K_{X}+\Delta+A/Z) \cap {\rm Nlc}(X,\Delta) = \emptyset$. 
Suppose in addition that $(K_{X}+\Delta+A)|_{{\rm Nlc}(X,\Delta)}$, which we think of an $\mathbb{R}$-line bundle on ${\rm Nlc}(X,\Delta)$, is semi-ample over $Z$. 
Let $H$ be a $\pi$-ample $\mathbb{R}$-divisor on $X$, and let 
$$(X,\Delta+A)=:(X_{1},\Delta_{1}+A_{1}) \dashrightarrow \cdots \dashrightarrow (X_{i},\Delta_{i}+A_{i}) \dashrightarrow \cdots$$
be a sequence of steps of a $(K_{X}+\Delta+A)$-MMP over $Z$ with scaling of $H$ such that if we put 
$$\lambda_{i}:=\{\mu \in \mathbb{R}_{\geq 0}\,| \,\text{$(K_{X_{i}}+\Delta_{i}+A_{i})+\mu H_{i}$ is nef over $Z$} \}$$
for each $i \geq 1$, then ${\rm lim}_{i \to \infty} \lambda_{i}=0$. 
Then $\lambda_{m}=0$ for some $m$ and $(X_{m},\Delta_{m}+A_{m})$ is a good minimal model of $(X,\Delta+A)$ over $Z$. 
\end{thm}

\begin{proof}
The proof of Theorem \ref{thm--termination-lcmmp-main} works without any change. 
\end{proof}

\subsection{Proofs of corollaries}

In this subsection we prove corollaries.

\begin{cor}\label{cor--nonvan-lc-main}
Let $\pi \colon X \to Z$ be a projective morphism of normal quasi-projective varieties. 
Let $(X,\Delta)$ be a normal pair and let $A$ be a $\pi$-ample $\mathbb{R}$-divisor on $X$ such that $K_{X}+\Delta+A$ is $\pi$-pseudo-effective. 
Suppose that ${\rm NNef}(K_{X}+\Delta+A/Z) \cap {\rm Nlc}(X,\Delta) = \emptyset$. 
Suppose in addition that $(K_{X}+\Delta+A)|_{{\rm Nlc}(X,\Delta)}$, which we think of an $\mathbb{R}$-line bundle on ${\rm Nlc}(X,\Delta)$, is semi-ample over $Z$. 
Then ${\rm Bs}|K_{X}+\Delta+A/Z|_{\mathbb{R}} \cap {\rm Nlc}(X,\Delta) = \emptyset$. 
\end{cor}

\begin{proof}
This immediately follows from Theorem \ref{thm--termination-lcmmp-main}. 
\end{proof}

\begin{cor}\label{cor--mmp-from-nonvan-fullgeneral}
Let $\pi \colon X \to Z$ be a projective morphism of normal quasi-projective varieties. 
Let $(X,\Delta)$ be a normal pair and let $A$ be a $\pi$-ample $\mathbb{R}$-divisor on $X$ such that $K_{X}+\Delta+A$ is $\pi$-pseudo-effective. 
Suppose that ${\rm Bs}|K_{X}+\Delta+A/Z|_{\mathbb{R}} \cap {\rm Nlc}(X,\Delta) = \emptyset$. 
We put $(X_{1},B_{1}):=(X,\Delta+A)$. 
Let $H$ be a $\pi$-ample $\mathbb{R}$-divisor on $X$. 
Then there exists a sequence of steps of a $(K_{X_{1}}+B_{1})$-MMP over $Z$ with scaling of $H$ 
$$
\xymatrix{
(X_{1},B_{1})\ar@{-->}[r]&\cdots \ar@{-->}[r]& (X_{i},B_{i})\ar@{-->}[rr]\ar[dr]&& (X_{i+1},B_{i+1})\ar[dl] \ar@{-->}[r]&\cdots\ar@{-->}[r]&(X_{m},B_{m}), \\
&&&W_{i}
}
$$
where $(X_{i},B_{i}) \to W_{i} \leftarrow (X_{i+1},B_{i+1})$ is a step of a $(K_{X_{i}}+B_{i})$-MMP over $Z$, such that
\begin{itemize}
\item
the non-isomorphic locus of the MMP is disjoint from ${\rm Nlc}(X,\Delta)$, 
\item
$\rho(X_{i}/W_{i})=1$ for every $i \geq 1$, and
\item
$K_{X_{m}}+B_{m}$ is semi-ample over $Z$. 
\end{itemize}
Moreover, if $X$ is $\mathbb{Q}$-factorial, then all $X_{i}$ in the MMP are also $\mathbb{Q}$-factorial. 
\end{cor}

\begin{proof}
This immediately follows from Theorem \ref{thm--termination-lcmmp-main}. 
\end{proof}

\begin{cor}\label{cor--mmp-lc-strictnefthreshold-main}
Let $\pi \colon X \to Z$ be a projective morphism of normal quasi-projective varieties. 
Let $(X,\Delta)$ be a normal pair such that $K_{X}+\Delta$ is $\pi$-pseudo-effective. 
Suppose that ${\rm NNef}(K_{X}+\Delta/Z) \cap {\rm Nlc}(X,\Delta) = \emptyset$. 
Let $A$ be a $\pi$-ample $\mathbb{R}$-divisor on $X$. 
Then there exists a sequence of steps of a $(K_{X}+\Delta)$-MMP over $Z$ with scaling of $A$ 
$$
\xymatrix{
(X,\Delta)=:(X_{1},\Delta_{1})\ar@{-->}[r]&\cdots \ar@{-->}[r]& (X_{i},\Delta_{i})\ar@{-->}[rr]\ar[dr]_{\varphi_{i}}&& (X_{i+1},\Delta_{i+1})\ar[dl]^{\varphi'_{i}} \ar@{-->}[r]&\cdots, \\
&&&W_{i}
}
$$
where $(X_{i},\Delta_{i}) \to W_{i} \leftarrow (X_{i+1},\Delta_{i+1})$ is a step of the $(K_{X}+\Delta)$-MMP over $Z$, such that
\begin{itemize}
\item
the non-isomorphic locus of the MMP is disjoint from ${\rm Nlc}(X,\Delta)$, 
\item
$\rho(X_{i}/W_{i})=1$ for every $i \geq 1$, and
\item
if we put 
$$\lambda_{i}:={\rm inf}\{\mu \in \mathbb{R}_{\geq 0}\,| \,\text{$K_{X_{i}}+\Delta_{i}+\mu A_{i}$ is nef over $Z$} \}$$  
for each $i \geq 1$, then ${\rm lim}_{i \to \infty}\lambda_{i}=0$. 
\end{itemize}
Moreover, if $X$ is $\mathbb{Q}$-factorial, then all $X_{i}$ in the MMP are also $\mathbb{Q}$-factorial. 
\end{cor}

\begin{proof}
By Corollary \ref{cor--mmp-nomralpair-Qfacdlt}, there exists a sequence of steps of a $(K_{X}+\Delta)$-MMP over $Z$ with scaling of $A$ 
$$
\xymatrix{
(X,\Delta)=:(X_{1},\Delta_{1})\ar@{-->}[r]&\cdots \ar@{-->}[r]& (X_{i},\Delta_{i})\ar@{-->}[rr]\ar[dr]&& (X_{i+1},\Delta_{i+1})\ar[dl] \ar@{-->}[r]&\cdots, \\
&&&W_{i}
}
$$
where $(X_{i},\Delta_{i}) \to W_{i} \leftarrow (X_{i+1},\Delta_{i+1})$ is a step of a $(K_{X_{i}}+\Delta_{i})$-MMP over $Z$, such that
\begin{itemize}
\item
the non-isomorphic locus of the MMP is disjoint from ${\rm Nlc}(X,\Delta)$, 
\item
$\rho(X_{i}/W_{i})=1$ for every $i \geq 1$, 
\item
if we put 
$$\lambda_{i}:={\rm inf}\{\mu \in \mathbb{R}_{\geq 0}\,| \,\text{$K_{X_{i}}+\Delta_{i}+\mu A_{i}$ is nef over $Z$} \}$$  
for each $i \geq 1$ and $\lambda:={\rm lim}_{i \to \infty}\lambda_{i}$, then the MMP terminates after finitely many steps or otherwise $\lambda \neq \lambda_{i}$ for every $i \geq 1$, and
\item 
if $X$ is $\mathbb{Q}$-factorial, then all $X_{i}$ in the MMP are also $\mathbb{Q}$-factorial. 
\end{itemize}

Then the MMP is a $(K_{X}+\Delta+\lambda A)$-over $Z$ with scaling of $A$. 
If $\lambda>0$, then we have ${\rm NNef}(K_{X}+\Delta+\lambda A/Z) \cap {\rm Nlc}(X,\Delta) = \emptyset$, and $(K_{X}+\Delta+\lambda A)|_{{\rm Nlc}(X,\Delta)}$ is ample over $Z$ because $(K_{X}+\Delta)|_{{\rm Nlc}(X,\Delta)}$ is nef over $Z$ and we may use \cite[Proof of Theorem 1.38]{kollar-mori} and Nakai--Moishezon's criterion for the ampleness. 
This contradicts Theorem \ref{thm--termination-all-lcmmp}, and therefore $\lambda=0$. 
This MMP satisfies the conditions of Corollary \ref{cor--mmp-lc-strictnefthreshold-main}. 
\end{proof}

\begin{cor}[cf.~{\cite{gongyo1}}]\label{cor--mmp-numericaldim-zero} 
Let $(X,\Delta)$ be a projective normal pair such that $K_{X}+\Delta$ is pseudo-effective and the non-nef locus of $K_{X}+\Delta$ is disjoint from ${\rm Nlc}(X,\Delta)$. 
Suppose that the numerical dimension $\kappa_{\sigma}(K_{X}+\Delta)$ (see \cite[V, 2.5. Definition]{nakayama}) is zero. 
Put $X_{1}=X$ and $\Delta_{1}=\Delta$. 
Then there exists a sequence of steps of a $(K_{X}+\Delta)$-MMP
$$
\xymatrix{
(X_{1},\Delta_{1})\ar@{-->}[r]&\cdots \ar@{-->}[r]& (X_{i},\Delta_{i})\ar@{-->}[rr]\ar[dr]&& (X_{i+1},\Delta_{i+1})\ar[dl] \ar@{-->}[r]&\cdots\ar@{-->}[r]&(X_{m},\Delta_{m}), \\
&&&W_{i}
}
$$
where $(X_{i},\Delta_{i}) \to W_{i} \leftarrow (X_{i+1},\Delta_{i+1})$ is a step of the $(K_{X}+\Delta)$-MMP, such that
\begin{itemize}
\item
the non-isomorphic locus of the MMP is disjoint from ${\rm Nlc}(X,\Delta)$, 
\item
$\rho(X_{i}/W_{i})=1$ for every $i \geq 1$, and
\item
$K_{X_{m}}+\Delta_{m} \equiv 0$. 
\end{itemize}
Moreover, if $X$ is $\mathbb{Q}$-factorial, then all $X_{i}$ in the MMP are also $\mathbb{Q}$-factorial. 
\end{cor}

\begin{proof}
By Corollary \ref{cor--mmp-lc-strictnefthreshold-main}, there exists a sequence of steps of a $(K_{X}+\Delta)$-MMP with scaling of an ample $\mathbb{R}$-divisor $A$ 
$$
\xymatrix{
(X_{1},\Delta_{1})\ar@{-->}[r]&\cdots \ar@{-->}[r]& (X_{i},\Delta_{i})\ar@{-->}[rr]\ar[dr]&& (X_{i+1},\Delta_{i+1})\ar[dl] \ar@{-->}[r]&\cdots, \\
&&&W_{i}
}
$$
where $(X_{i},\Delta_{i}) \to W_{i} \leftarrow (X_{i+1},\Delta_{i+1})$ is a step of the $(K_{X}+\Delta)$-MMP over $Z$, such that
\begin{itemize}
\item
the non-isomorphic locus of the MMP is disjoint from ${\rm Nlc}(X,\Delta)$, 
\item
$\rho(X_{i}/W_{i})=1$ for every $i \geq 1$, and
\item
if we put 
$$\lambda_{i}:={\rm inf}\{\mu \in \mathbb{R}_{\geq 0}\,| \,\text{$K_{X_{i}}+\Delta_{i}+\mu A_{i}$ is nef over $Z$} \}$$  
for each $i \geq 1$, then ${\rm lim}_{i \to \infty}\lambda_{i}=0$. 
\end{itemize}
Moreover, if $X$ is $\mathbb{Q}$-factorial, then all $X_{i}$ in the MMP are $\mathbb{Q}$-factorial. 
Prime divisors contracted by the MMP are components of the negative part of the Nakayama--Zariski decomposition of $K_{X}+\Delta$. 
Hence, there exists $m$ such that $K_{X_{m}}+\Delta_{m}$ is the limit of movable $\mathbb{R}$-divisors.
By the same argument as in \cite[Proof of Theorem 5.1]{gongyo1}, we see that $K_{X_{m}}+\Delta_{m} \equiv 0$. 
The MMP satisfies the properties of Corollary \ref{cor--mmp-numericaldim-zero}. 
\end{proof}

\begin{cor}[cf.~{\cite[Theorem A]{bbp}, \cite[Theorem A]{tsakanikas-xie}}]\label{cor--nonnef-diminished}
Let $(X,\Delta)$ be a projective normal pair such that $K_{X}+\Delta$ is pseudo-effective and the non-nef locus of $K_{X}+\Delta$ is disjoint from ${\rm Nlc}(X,\Delta)$. 
Then 
$${\rm NNef}(K_{X}+\Delta)=\boldsymbol{\rm B}_{-}(K_{X}+\Delta).$$
Furthermore, every irreducible component of the non-nef locus of $K_{X}+\Delta$ is uniruled. 
\end{cor}

\begin{proof}
The argument in \cite{bbp} or \cite{tsakanikas-xie} works with no changes because we may use Corollary \ref{cor--mmp-lc-strictnefthreshold-main} and \cite[Theorem 1.12]{fujino-morihyper}. 
\end{proof}

\begin{cor}\label{cor--finite-generation-main}
Let $\pi \colon X \to Z$ be a projective morphism of normal quasi-projective varieties. 
Let $(X,\Delta)$ be a normal pair such that $\Delta$ is a $\mathbb{Q}$-divisor on $X$. 
Let $A$ be a $\pi$-ample $\mathbb{Q}$-divisor on $X$ such that $K_{X}+\Delta+A$ is $\pi$-pseudo-effective. 
Suppose that ${\rm NNef}(K_{X}+\Delta+A/Z) \cap {\rm Nlc}(X,\Delta) = \emptyset$ and that $(K_{X}+\Delta+A)|_{{\rm Nlc}(X,\Delta)}$, which we think of a $\mathbb{Q}$-line bundle on ${\rm Nlc}(X,\Delta)$, is semi-ample over $Z$. 
Then the sheaf of graded $\pi_{*}\mathcal{O}_{X}$-algebra
$$\underset{m \in \mathbb{Z}_{\geq 0}}{\bigoplus}\pi_{*}\mathcal{O}_{X}(\lfloor m(K_{X}+\Delta+A)\rfloor)$$
is finitely generated. 
\end{cor}

\begin{proof}
This immediately follows from Theorem \ref{thm--termination-lcmmp-main}.  
\end{proof}

\section{Examples}\label{sec--example}

In this section we collect some examples. 

The first example shows that a step of an MMP for normal pairs does not always exist and the abundance conjecture for normal pairs does not hold in general.

\begin{exam}\label{const--for-exam}
Let $E$ be an elliptic curve with a very ample divisor $H$ and define $$X:=\mathbb{P}_{E}(\mathcal{O}_{E}\oplus \mathcal{O}_{E}(-H))\overset{f}{\longrightarrow} E.$$ 
Let $S$ be the unique section of $\mathcal{O}_{X}(1)$. 
Then $K_{X}+S+(S+f^{*}H) \sim0$ and $S+f^{*}H$ is base point free. 
The contraction 
$$\pi \colon X \to Z$$ induced by $S+ f^{*}H$
 is birational and we have ${\rm Ex}(\pi)=S$. 
Thus $-S$ is $\pi$-ample. 
Let $D$ be a non-torsion Cartier divisor on $E$ of degree zero. 
Since $f^{*}D+S+f^{*}H$ is nef and big, we can write 
$f^{*}D+S+f^{*}H \sim_{\mathbb{Q}}A+B$
such that $A$ is an effective ample $\mathbb{Q}$-divisor and $B$ is an effective $\mathbb{Q}$-divisor on $X$. 
Now we have
$$K_{X}+S+B+A \sim_{\mathbb{Q}}f^{*}D$$
and $D$ is not semi-ample. 

\begin{itemize}
\item
With notation as above, we consider $\pi \colon X \to Z$ and the pair $(X,2S+B+A)$. 
Since $-S$ is $\pi$-ample, $\pi$ is a $(K_{X}+2S+B+A)$-negative extremal contraction. 
However, the divisor
$\pi_{*}(K_{X}+2S+B+A)\sim_{\mathbb{Q}}\pi_{*}f^{*}D$
is not $\mathbb{R}$-Cartier since $rH$ is not $\mathbb{R}$-linearly equivalent to $D$ for any $r \in \mathbb{R}$ (cf.~\cite[Proposition 7.2.8]{fujino-book}). 
This example implies that we cannot always construct a step of an MMP even in the case of surfaces.
\item
With notation as above, consider the pair $(X,S+B+A)$. 
By construction, $K_{X}+S+B+A$ is nef but not semi-ample. 
This pair shows that the abundance conjecture does not always hold for normal pairs with polarizations. 
\end{itemize}
\end{exam}

The next example shows that the non-nef locus of the log canonical $\mathbb{R}$-divisor of a normal pair is not necessarily Zariski closed. 
In particular, minimal models of normal pairs do not always exist. 

\begin{exam}\label{exam-nonneflocus}
For a blow up $V \to \mathbb{P}^{3}$ at nine very general points and
$$X:=\mathbb{P}_{V}(\mathcal{O}_{V}\oplus \mathcal{O}_{V}(1))\overset{f}{\longrightarrow} V,$$
where $\mathcal{O}_{V}(1)$ is very ample line bundle on $V$, Lesieutre \cite[Theorem 1.1]{lesieutre} constructed a big $\mathbb{R}$-divisor $D$ on $X$ such that
$$\boldsymbol{\rm B}_{-}(D)=\bigcup_{\substack{\text{$A$: ample}\\\text{$D+A$: $\mathbb{Q}$-Cartier}}}{\rm Bs}|D+A|_{\mathbb{Q}}$$
is a countable union of curves. 
By construction, there is an effective $\mathbb{R}$-divisor $\Delta_{V}$ on $V$ such that $K_{V}+\Delta_{V} \sim_{\mathbb{R}}0$, and therefore $K_{X}+\Delta \sim_{\mathbb{R}}0$ for some effective $\mathbb{R}$-divisor $\Delta$ on $X$. 
We can write $D\sim_{\mathbb{R}} A+B$ such that $A$ is an effective ample $\mathbb{R}$-divisor and $B$ is an effective $\mathbb{R}$-divisor on $X$. 
Then
$$K_{X}+\Delta+B+A \sim_{\mathbb{R}}D.$$
Moreover, by \cite[V, 1.3 Theorem]{nakayama}, it follows that $\boldsymbol{\rm B}_{-}(D)$ coincides with the non-nef locus of $D$. 
Therefore, the pair $(X,\Delta+B+A)$ satisfies the condition that the non-nef locus of $K_{X}+\Delta+B+A$ is not Zariski closed. 
If $(X,\Delta+B+A)$ has a minimal model $(X',\Gamma')$, taking a common resolution $g \colon Y \to X$ and $g' \colon Y \to X'$ of $X \dashrightarrow X'$, then we can write
$$g^{*}(K_{X}+\Delta+B+A)=g'^{*}(K_{X'}+\Gamma')+E$$
for some effective $g'$-exceptional $\mathbb{R}$-divisor $E$ on $Y$. 
Since $K_{X'}+\Gamma'$ is nef, we have
$${\rm NNef}(g^{*}(K_{X}+\Delta+B+A))={\rm NNef}(g'^{*}(K_{X'}+\Gamma')+E)={\rm Supp}\,E.$$
Then ${\rm NNef}(K_{X}+\Delta+B+A)=g({\rm NNef}(g^{*}(K_{X}+\Delta+B+A)))=g({\rm Supp}\,E)$ is Zariski closed, a contradiction. 
Therefore, $(X,\Delta+B+A)$ does not have a minimal model.  
\end{exam}



\begin{thebibliography}{BCHM10}



\bibitem[A04]{ambro1}F.~Ambro, Shokurov's boundary property, J. Differential Geom. {\textbf{67}} (2004), no.~2, 229--255. 



\bibitem[AK17]{ambrokollar} F.~Ambro, J.~Koll\'ar, Minimal models of semi-log-canonical pairs. In Moduli of K-stable varieties, Springer INdAM Ser., {\textbf{31}}, Springer, Cham (2019), 1--13.






\bibitem[B12]{birkar-flip}
C.~Birkar, 
Existence of log canonical flips and a special LMMP, 
Publ. Math. Inst. Hautes \'Etudes Sci. {\textbf{115}} (2012), no.~1, 325--368.




\bibitem[BCHM10]{bchm}C.~Birkar, P.~Cascini, C.~D.~Hacon, J.~M\textsuperscript{c}Kernan, Existence of minimal models for varieties of log general type, J. Amer. Math. Soc. {\textbf{23}} (2010), no.~2, 405--468.


\bibitem[BH14]{bhzariski} C.~Birkar, Z.~Hu, Polarized pairs, log minimal models, and Zariski decompositions, Nagoya Math. J. {\textbf{215}} (2014), 203--224. 







\bibitem[BZ16]{bz} C.~Birkar, D.~Q.~Zhang, Effectivity of Iitaka fibrations and pluricanonical systems of polarized pairs, Publ. Math. Inst. Hautes \'Etudes Sci. {\textbf{123}} (2016), no.~1, 283--331. 

\bibitem[BBP13]{bbp} S.~Boucksom, A.~Broustet, G.~Pacienza, Uniruledness of stable base loci of adjoint linear systems via Mori theory, Math. Z. {\textbf{275}} (2013), no.~1--2, 499--507. 







\bibitem[CLX23]{clx--vanishing-g-pair} B.~Chen, J.~Liu, and L.~Xie, Vanishing theorems for generalized pairs, preprint (2023), arXiv:2305.12337v1.














\bibitem[F07a]{fujino-what-log-ter}O.~Fujino, {\it What is log terminal?} In Flips for $3$-folds and $4$-folds, Oxford University Press (2007).

\bibitem[F07b]{fujino-sp-ter}O.~Fujino, {\it Special termination and reduction to pl flips.} In Flips for $3$-folds and $4$-folds, Oxford University Press (2007).



\bibitem[F11]{fujino-fund}O.~Fujino, Fundamental theorems for the log minimal model program, Publ. Res. Inst. Math. Sci. {\textbf{47}} (2011), no.~3, 727--789. 









\bibitem[F14]{fujino-fund-slc}O.~Fujino, Fundamental theorems for semi log canonical pairs, Algebraic Geom. {\textbf{1}} (2014), no.~2, 194--228. 

\bibitem[F16]{fujino-bpf-quasi-log} O.~Fujino, Basepoint-free Theorem of Reid--Fukuda Type for Quasi-log Schemes. Publ. Res. Inst. Math. Sci. {\textbf{52}} (2016), no.~1, 63--81. 

\bibitem[F17]{fujino-book}O.~Fujino, {\em Foundations of the minimal model program}, MSJ Mem. \textbf{35}, Mathematical Society of Japan, Tokyo, 2017. 


\bibitem[F25]{fujino-morihyper} O.~Fujino, Cone theorem and Mori hyperbolicity, J. Differential Geom. {\textbf{129}} (2025), no.~3, 617--693.  










\bibitem[FH23]{fujino-hashizume-adjunction}O.~Fujino, K.~Hashizume, Adjunction and inversion of adjunction, Nagoya Math. J. {\textbf{249}} (2023), 119--147. 



\bibitem[FST11]{fst-suppli-nonlc-ideal}
O.~Fujino, K.~Schwede, S.~Takagi, Supplements to non-lc ideal sheaves. RIMS Kokyuroku Bessatsu, B24, Res. Inst. Math. Sci. (RIMS), Kyoto, 2011, 1--46.






\bibitem[G11]{gongyo1}Y.~Gongyo, On the minimal model theory for dlt pairs of numerical log kodaira dimension zero, Math. Res. Lett., {\textbf{18}} (2011) , no.~5, 991--1000.











\bibitem[HL23]{hacon-liu}C.~D.~Hacon, J.~Liu, Existence of flips for generalized lc pairs, Camb. J. Math. {\textbf{11}} (2023), no.~4, 795--828.



\bibitem[HX13]{haconxu-lcc}C.~D.~Hacon, C.~Xu, Existence of log canonical closures, Invent. Math. {\textbf{192}} (2013), no.~1, 161--195. 









\bibitem[H19]{has-mmp}
K.~Hashizume, 
Remarks on special kinds of the relative log minimal model program,
Manuscripta Math. {\textbf{160}} (2019), no.~3, 285--314. 



\bibitem[H22a]{has-iitakafibration}
K.~Hashizume, 
Iitaka fibrations for dlt pairs polarized by a nef and log big divisor, Forum Math. Sigma.  {\textbf{10}} (2022), e85.  


\bibitem[H22b]{has-nonvan-gpair} K.~Hashizume, Non-vanishing theorem for generalized log canonical pairs with a polarization, Selecta Math. {\textbf{77}} (2022), article number 77.


\bibitem[HH20]{hashizumehu}
K.~Hashizume, Z.~Hu, 
On minimal model theory for log abundant lc pairs, J. Reine Angew. Math.,  {\textbf{767}} (2020), 109--159.





















\bibitem[K13]{kollar-mmp}J.~Koll\'ar, {\em Singularities of the Minimal Model Program}, Cambridge Tracts in Mathematics {\textbf{200}}. Cambridge University Press, Cambridge, 2013.


\bibitem[K21]{kollar-nonQfac-mmp}J.~Koll\'ar, Relative MMP without $\mathbb{Q}$-factoriality, Electronic Research Archive {\textbf{29}} (2021), no.~5, 3193--3203. 



\bibitem[KM98]{kollar-mori} J.~Koll\'ar, S.~Mori, {\em{Birational geometry of algebraic varieties}}. With the collaboration of C.~H.~Clemens and A.~Corti. Translated from the 1998 Japanese original. Cambridge Tracts in Mathematics {\textbf{134}}. Cambridge University Press, Cambridge, 1998.





\bibitem[LP20]{lp-g-abund-I} V.~Lazi\'c, T.~Peternell, On generalised abundance, I, Publ. Res. Inst. Math. Sci. {\textbf{56}} (2020), no.~2, 353--389.

\bibitem[LT22]{ltj}
V.~Lazi\'c, N.~Tsakanikas, Special MMP for log canonical generalised pairs, with an appendix joint with Xiaowei Jiang, Selecta Math. (N.S.) (2022), Article no.~89.


\bibitem[L14]{lesieutre} J.~Lesieutre, The diminished base locus is not always closed, Compositio Math. {\textbf{150}} (2014), 1729--1741. 





\bibitem[LX23a]{liuxie-relative-nakayama}
J.~Liu, L.~Xie, Relative Nakayama--Zariski decomposition and minimal models of generalized pairs, Peking Math. J. (2023).

\bibitem[LX23b]{liuxie-semiample-gpair}
J.~Liu, L.~Xie, Semi-ampleness of NQC generalized log canonical pairs, Adv. Math.  {\textbf{427}} (2023), 109126. 



\bibitem[N04]{nakayama}N.~Nakayama, {\em Zariski-decomposition and abundance}, MSJ Mem. {\textbf{14}}, Mathematical Society of Japan, Tokyo, 2004. 








\bibitem[TX23]{tsakanikas-xie} N.~Tsakanikas, Z.~Xie, Comparison and uniruledness of asymptotic base loci, to appear in Michigan Math. J.


\bibitem[TX24]{tsakanikas-xie--remarks-mmp} N.~Tsakanikas, L.~Xie, Remarks on the existence of minimal models of log canonical generalized pairs, Math. Z. {\textbf{307}} (2024), article number 20.



\bibitem[X24]{xie-contraction-gpair} L.~Xie, Contraction theorem for generalized pairs, Algebraic Geometry and Physics, {\textbf{1}} (2024), no.~1, 101--124. 

\end{thebibliography}
\end{document}